\documentclass[10pt, a4paper, reqno,dvipsnames]{article}
\usepackage{fullpage}
\usepackage{mathtools}

\usepackage{enumitem}

\RequirePackage[bookmarksopen=true, urlcolor=blue!80!black,colorlinks, linkcolor=green!40!black,citecolor=red!50!black]{hyperref}

\usepackage{amsthm,amsfonts,amssymb,euscript,mathrsfs}
\usepackage{yhmath,tikz,bbding}

\usepackage{xcolor,comment}
\usepackage{bm}
\usepackage{bbm}

\theoremstyle{plain}

\newtheorem{lemma}{Lemma}[section]
\newtheorem{theorem}{Theorem}[section]
\newtheorem{prop}{Proposition}[section]
\newtheorem{corollary}{Corollary}[section]

\newtheorem{remark}{Remark}[section]
\numberwithin{equation}{section}

\theoremstyle{plain}
\newtheorem{notation}{Notation}[section]
\newtheorem{definition}{Definition}[section]
\newtheorem{example}{Example}[section]

\usepackage{stmaryrd}
\date{}

\renewcommand{\d}{\mathrm{d}}
\newcommand{\R}{\mathbb{R}}
\newcommand{\N}{\mathbb{N}}
\newcommand{\Z}{\mathbb{Z}}
\newcommand{\C}{\mathbb{C}}
\newcommand{\ffi}{\varphi}

\newcommand{\e}{\varepsilon}
\newcommand{\dr}{\partial}

\newcommand{\half}{\frac{1}{2}}

\newcommand{\B}{\mathbf{B}}

\renewcommand{\a}{\alpha}
\renewcommand{\b}{\beta}
\newcommand{\De}{\Delta}
\newcommand{\de}{\delta}
\newcommand{\om}{\omega}
\newcommand{\Om}{\Omega}
\newcommand{\E}{\mathbb{E}}
\newcommand{\si}{\sigma}
\newcommand{\an}[1]{\langle #1 \rangle}
\newcommand{\ga}{\gamma}
\newcommand{\ka}{\kappa}
\newcommand{\La}{\Lambda}
\newcommand{\K}{\mathcal{K}}
\newcommand{\kfr}{\mathfrak{K}}
\newcommand{\Vect}{\mathrm{Vect}}

\newcommand{\non}{\nonumber}
\newcommand{\low}{\mathrm{low}}
\newcommand{\lowl}{{\mathrm{low-}\ell}}
\newcommand{\lowm}{{\mathrm{low-}m}}
\newcommand{\lowr}{{\mathrm{low-}r}}
\newcommand{\high}{\mathrm{high}}
\newcommand{\st}{\mathrm{st}}
\newcommand{\scb}{\mathrm{scb}}

\newcommand{\children}{\mathtt{children}}
\newcommand{\parent}{\mathtt{parent}}
\newcommand{\siblings}{\mathtt{siblings}}
\newcommand{\offspring}{\mathtt{offspring}}
\newcommand{\bush}{\mathrm{bush}}

\renewcommand{\l}{\left\|}
\renewcommand{\r}{\right\|}
\newcommand{\enstq}[2]{\left\{#1~\middle|~#2\right\}} 
\newcommand{\saut}{\par\leavevmode\par}
\newcommand{\pth}[1]{\left( #1 \right)}

\newcommand{\ps}[1]{\left\langle #1 \right\rangle}

\newcommand{\init}{{\rm in}}

\newcommand{\T}{\mathbb T}

\usepackage{scalerel,stackengine}
\stackMath
\newcommand\reallywidehat[1]{%
\savestack{\tmpbox}{\stretchto{%
  \scaleto{%
    \scalerel*[\widthof{\ensuremath{#1}}]{\kern-.6pt\bigwedge\kern-.6pt}%
    {\rule[-\textheight/2]{1ex}{\textheight}}
  }{\textheight}%
}{0.5ex}}%
\stackon[1pt]{#1}{\tmpbox}%
}
\newcommand{\lin}{\mathrm{lin}}

\newcommand{\bi}{\mathrm{bi}}
\newcommand{\ter}{\mathrm{ter}}
\newcommand{\tot}{\mathrm{tot}}

\DeclarePairedDelimiter\floor{\lfloor}{\rfloor}

\newcommand{\Xres}[1]{X_{\mathtt{res},#1}}
\newcommand{\lef}{\mathtt{left}}
\newcommand{\righ}{\mathtt{right}}
\renewcommand{\Im}{\mathrm{Im}}

\DeclareFontFamily{U}{mathx}{\hyphenchar\font45}
\DeclareFontShape{U}{mathx}{m}{n}{
      <5> <6> <7> <8> <9> <10>
      <10.95> <12> <14.4> <17.28> <20.74> <24.88>
      mathx10
      }{}
\DeclareSymbolFont{mathx}{U}{mathx}{m}{n}
\DeclareFontSubstitution{U}{mathx}{m}{n}
\DeclareMathAccent{\widecheck}{0}{mathx}{"71}

\title{Wave turbulence for a semilinear Klein-Gordon system}

\author{Anne-Sophie de Suzzoni\footnote{Centre de Mathématiques Laurent Schwartz, Ecole Polytechnique, France (\href{mailto:anne-sophie.de-suzzoni@polytechnique.edu}{anne-sophie.de-suzzoni@polytechnique.edu})}, Annalaura Stingo\footnote{Centre de Mathématiques Laurent Schwartz, Ecole Polytechnique, France (\href{mailto:annalaura.stingo@polytechnique.edu}{annalaura.stingo@polytechnique.edu})}  \,and Arthur Touati\footnote{CNRS \& Institut de Mathématiques de Bordeaux, France (\href{mailto:arthur.touati@math.u-bordeaux.fr}{arthur.touati@math.u-bordeaux.fr})}}

\begin{document}

\maketitle

\begin{abstract}
In this article we consider a system of two Klein-Gordon equations, set on the $d$-dimensional box of size $L$, coupled through quadratic semilinear terms of strength $\e$ and evolving from well-prepared random initial data. We rigorously derive the effective dynamics for the correlations associated to the solution, in the limit where $L\rightarrow \infty$ and $\e\rightarrow 0$ according to some power law. The main novelty of our work is that, due to the absence of invariances, trivial resonances always take precedence over quasi-resonances. The derivation of the nonlinear effective dynamics is justified up time to $\delta T$, where $T=\e^{-2}$ is the appropriate timescale and $\delta$ is independent of $L$ and $\e$. 
We use Feynmann interaction diagrams, here adapted to a normal form reduction and to the coupled nature of our real-valued system. We also introduce a frequency decomposition at the level of the diagrammatic and develop a new combinatorial tool which allows us to work with the Klein-Gordon dispersion relation.
\end{abstract}

\tableofcontents


\section{Introduction}

\subsection{Background and motivations}

In this article, we study the nonlinear evolution of correlations associated to a semilinear Klein-Gordon system, which models a wave turbulence phenomenon. We begin by reviewing the relevant literature and then highlight the main novelty of our work: the impact of trivial resonances on the effective dynamics.

\subsubsection{Wave turbulence}

Wave turbulence studies how the distribution of a solution to a wave equation evolves when the initial data are random variables or when the equation includes a random forcing term. The concept was first introduced in \cite{Peierls1,Brout-Prigo} and later, in the 1960s, emerged in the context of plasma physics \cite{Vedenov1967,ZS67} and fluid mechanics \cite{Hass1,Hass2,Benney}. Around the same time, it was popularized by Zakharov, one of his major contributions being the introduction of what is now called Kolmogorov--Zakharov spectrum, which models an energy transfer across spatial scales, see the article \cite{KZspectra}. Wave turbulence is a very active field in mathematical physics, as reflected in the work of Zakharov and others \cite{zakharovBook,Naz,Galtier}.

Starting in the late 2000s, the mathematical community began to take an interest in weak turbulence, a subfield of wave turbulence where the initial data are assumed to be small. A major discovery by the aforementioned physicists was that the statistics of certain solutions to wave equations could be described using the so-called kinetic equations. This led to the derivation and study of kinetic equations that emerge from the statistical properties of wave equations.
 The pioneer work \cite{LS11} describes a mixing effect for the dynamics of discrete Schr\"odinger equations. This was followed by works on atmospheric primitive equations \cite{dSTont,ASkineq} and  the Schrödinger equation \cite{BGHS,CoG19,denghani19,SchroQuint,hani2024inhomogeneous}. In parallel, a stochastic approach was developed in \cite{DyKuk1,DyKuk2,Dykuk3}. This line of research culminated in the works \cite{DHPropag,denghani2023,denghani2021,deng2023long}, where the authors rigorously derive the kinetic equation governing the evolution of correlations associated to the cubic Schrödinger equation and extend the derivation up to the time where the equation remains well-posed. The ideas developed in these papers were then applied to the long-time derivation of the Boltzmann equation in both Euclidean space and on the torus \cite{deng2024long,deng2025hilbert}. Meanwhile, in \cite{staffilanitran}, the authors derive a kinetic equation for atmospheric primitive equations.
 
The aforementioned mathematical works have been strongly influenced by the study of the Cauchy problem for various equations with random data, in particular the works \cite{bouPer,burq1,burq2,DNY22,bringmann2024invariant}; the derivation of collisional and diffusive regimes for the Schrödinger equation with random potentials, in particular by \cite{ELY,ELSMY}; and the derivation theorem for the Boltzmann equation for classical particles, for instance \cite{Boltzmann}.
 
In addition, several recent works address related issues. For example, we mention \cite{ampatzoglou2021derivation,HRST,faou2024scattering} in inhomogeneous settings,  \cite{ma2022almost,de2023discrete} in the context of fluids,  \cite{grande2024rigorous} with damping and forcing and \cite{dubach} for toy models in finite dimension. The study of the kinetic equation itself has been the subject of \cite{EMV1,EMV2}, \cite{EVbook}, \cite{collot2022stability,menegaki2022l2stability}, and \cite{Faou} for a stochastic version. For a linear turbulence model, we mention the work of \cite{beck}. Related to the present work, we also mention the derivation of the continuous resonant system \cite{faouCRequation}. Finally, some questions related to propagation of chaos were adressed in \cite{PropagChaos}; some related to graph theory in \cite{bruned2024cancellations}.

\subsubsection{Trivial resonances and discrete wave turbulence}

In the physics literature, many equations studied through the prism of wave turbulence are quadratic and real-valued, such as fluid primitive equations, see for instance the pioneer work on capillary waves by Zakharov and Filonenko \cite{Zakharov1967}. In the present work, we investigate the derivation of effective equations for the statistical dynamics of weakly interacting nonlinear waves starting from a \emph{system} of equations with \emph{real-valued} solutions. We believe that, in the generic case, there are not enough symmetries in the (system of) equation(s) to ensure that the effective dynamics is \emph{kinetic}. We believe that, instead, they behave more like in the \emph{discrete} wave turbulence or finite-box effects regime (see \cite{expFiniteBox,KARTASHOVA2,KARTASHOVA1,DymKukDiscrete,SchroQuint}) because trivial (exact) resonances do not cancel each other and play a significant part in the effective dynamics.

To better illustrate the significance of invariances for the correlations' dynamics, let us come back to the well-understood case of the cubic nonlinear Schrödinger equation. In \cite{denghani2021,denghani2023}, Deng and Hani explain the different roles played by exact resonances versus quasi-resonances in the derivation of the effective dynamics. The idea is to compare the ‘‘size" of the sets 
\[
E = \enstq{ (k_1,k_2,k_3)\in \frac1{L} \Z^d }{ \omega(k_1-k_2+k_3) - \omega(k_1) + \omega(k_2) -\omega(k_3) = 0  }
\]
and 
\[
Q =\enstq{ (k_1,k_2,k_3)\in \frac1{L} \Z^d }{ 0\neq | \omega(k_1-k_2+k_3) - \omega(k_1) + \omega(k_2) -\omega(k_3)| \leq \varepsilon^2  },
\]
which represent respectively exact resonances and quasi-resonances. Here, $L$ is the size of the box, $\varepsilon$ is the strength of the nonlinearity and $\omega(k) = |k|^2$ is the dispersion relation of the Schrödinger equation. We remark that if $L$ is fixed, the size of $E$ is larger than that of $Q$, which is actually empty for small enough $\varepsilon$. This explains why the precedence of exact resonances over quasi-resonances is called finite-box effects. Set $E$ always contains trivial resonances, that is $k_1 = k_2$ or $k_2 = k_3$. As it happens, due to the $U(1)$ invariance of the Schrödinger equation, the contributions of these trivial resonances cancel out and do not affect the effective dynamics. This is seen in the fact that the equation may be ‘‘Wick-renormalized" (see Section 2.1 in \cite{denghani2021}). In this case, the right set to consider is 
\[
\tilde E = \enstq{ (k_1,k_2,k_3)\in \frac1{L} \Z^d }{ \omega(k_1-k_2+k_3) - \omega(k_1) + \omega(k_2) -\omega(k_3) = 0 \, \wedge \, k_2 \neq k_1 \, \wedge\, k_2 \neq k_3 }.
\]
It is then possible to find regimes between $\varepsilon$ and $L$ for which the size of $\tilde E$ is smaller than the size of $Q$ and the effective dynamics is of kinetic type. Note that in the case of fluid primitive equations, the invariance under $U(1)$ is substituted by the invariance of the average (or $0$-mode) of the solution.

In the framework we consider here and which we describe in the next sections, these invariances no longer hold. Therefore, the set of exact resonances $E$ (or rather its analogue in our context) has to be considered instead of $\tilde E$. Due to the trivial resonances, \emph{there are no regimes in which $Q$ takes precedence over $E$.} This impacts dramatically the effective dynamics, which in particular is not of kinetic type. Still, this effective dynamics is \emph{nonlinear} and driven by the \emph{trivial} resonances.

\begin{remark}
Note that in the context of the Schrödinger equation, exact resonances also play a role in the long time dynamics of the correlations but on timescales longer than the kinetic one, see the works on the so-called CR equation \cite{faouCRequation,FaouMouzard}.
\end{remark}

\subsubsection{The toy model}

One possibility to investigate how the absence of invariance can affect the effective dynamics of correlations would be to consider two coupled nonlinear Schrödinger equations. However, we have chosen a different route motivated by the theoretical description in \cite{galtier2017turbulence} of a turbulent behaviour of gravitational waves in the early age of the universe. Gravitational waves are famous vacuum solutions of the field equations of general relativity, the so-called Einstein equations. Despite the energy cascade being potentially mitigated by the expansion of the early universe (see \cite{clough2018difficulty}), deriving an effective dynamics for such waves is a mathematical challenge due to the nature of the Einstein equations.

In some coordinates, the Einstein vacuum equations rewrite as a system of quasilinear wave equations for the metric of the spacetime, with quadratic semilinear interaction terms. A good toy model for the Einstein vacuum equations is thus a system of wave equations on Minkowski spacetime with quadratic nonlinearities. Nevertheless, such a wave system presents low-frequency singularities, which induce specific analytic difficulties on which we do not wish to focus.  
With this in mind, we instead study a system of two coupled quadratic Klein--Gordon equations, which may be viewed as a low-frequency regularization of wave equations. More precisely, we consider the following system 
\begin{equation}\label{first system}
    \begin{aligned}
    \pth{\partial_t^2  - \De + m} u & = \e Q^0( v, v),
    \\ \pth{\partial_t^2  - \De + m} v  & = \e Q^1( u, u),
    \end{aligned} 
\end{equation}
set on the $d$-dimensional torus of size $L$ and with $m>0$ (see Section \ref{section KG system} below for the precise assumptions on $Q^0$ and $Q^1$ and on the initial data). System \eqref{first system} obviously misses the geometric nature of the Einstein equations and understanding how general covariance might affect the effective dynamics of turbulent gravitational waves would be of great interest.

\begin{remark}
We chose to consider this specific class of nonlinearities because the resulting effective dynamics exhibit a particularly convenient structure. While other choices were possible, we selected this formulation to maintain clarity and avoid unnecessary complexity.
\end{remark}

\subsubsection{Presentation of the result}

Our main result (see Theorem \ref{main theorem} for a precise statement) is the rigorous derivation of the effective dynamics for the correlations associated with \eqref{first system}, in the limit $L\to\infty$, $\e\to 0$ and where the relation between $L$ and $\e$ is given in Theorem \ref{main theorem}. The Klein-Gordon system \eqref{first system} does not display exact resonances and a normal form argument allows us to transform it into an equivalent system with cubic nonlinearities. This new cubic system does display exact resonances and moreover trivial resonances drive the nonlinear effective dynamics. As explained above, this discrete regime is ultimately linked to the fact that we consider a system and not a scalar equation (as opposed for instance to \cite{Spohnphonon}, where a single Klein-Gordon equation is considered on a lattice).  

\begin{remark} Note that our result does not compare to 
\begin{itemize} 
\item \cite{galtier2017turbulence}, where \emph{kinetic} wave turbulence is predicted for gravitational waves. This is mainly due to the gauge choice for the Einstein vacuum equations made in \cite{galtier2017turbulence}, which reduces the dynamics to a scalar wave equation;
\item \cite{ampatzoglou2021derivation}, where, although there is no invariance under $U(1)$, kinetic wave turbulence is again observed. This is due to their setting being the Euclidean space, in which the set of exact resonances has zero measure.
\end{itemize}
\end{remark}

After the normal form procedure, the correct timescale $T$ to observe the effective dynamics can be proved to be $\e^{-2}$. One of the strength of our result is that we justify the approximation of the correlations by the effective dynamics up to times of order $\de T$, where $\de>0$ is a small quantity \emph{independent of $L$ and $\e$}. In other words, our derivation extends the validity of the effective dynamics up to a nontrivial timescale. In the context of wave turbulence, such a rigorous derivation has only been obtained for the Schrödinger equation in the pioneer work \cite{denghani2021} (and subsequently improved in \cite{denghani2023}), where the authors reach the so-called kinetic timescale.

An overview of our proof will be given in Section \ref{section sketch of proof}, but we can already say that our result follows from a precise diagrammatic representation of the solution to the Klein-Gordon system, as in the aforementioned articles on other equations. However, the new features of the system we consider force us to develop a new diagrammatic adapted to the normal form procedure and an analysis adapted to the Klein-Gordon dispersion relation (which, as opposed to the Schrödinger one, forbids the application of number theory results). For this last aspect, we rely on a low/high-frequency analysis and we develop new combinatorial tools that take this decomposition into account. 

\begin{remark}
Note that a low/high-frequency decomposition is also utilized in \cite{LS11} and in the work on inhomogeneous settings \cite{ampatzoglou2021derivation}, but in both cases it does not show up in the diagrammatic. See also Remark \ref{remark low LS} below.
\end{remark}

\subsection{The Klein-Gordon system}\label{section KG system}

In this section we present the system of two coupled Klein-Gordon equations we study in this article and introduce all the notations that are required to state Theorem \ref{main theorem}.

\paragraph{The system.} We consider the Cauchy problem
\begin{equation}\label{KG system}\tag{KG}
    \left\{ 
    \begin{aligned}
    \pth{\partial_t^2  - \De + m} u & = \e Q^0( v, v)
    \\ \pth{\partial_t^2  - \De + m} v  & = \e Q^1( u, u)
    \\ (u,\dr_t u)_{|_{\{t=0\}}} & = (u_\init,u'_\init)
    \\ (v,\dr_t v)_{|_{\{t=0\}}} & = (v_\init,v'_\init)
    \end{aligned} \right.
\end{equation}
posed on the torus $L\T^d$ of size $L>0$, with $m> 0$ and $\De=\sum_{i=1}^d\dr_i^2$ the flat Laplacian. The solutions $u$ and $v$ are real-valued and the small parameter $\e>0$ will be ultimately linked to $L$, see Theorem \ref{main theorem}.

\paragraph{The nonlinearities.} For $\eta\in\{0,1\}$, the nonlinearity $Q^\eta$ is a symmetric bilinear map defined in Fourier space by 
\begin{align*}
\widehat{Q^\eta (f,g)} (k) = \frac1{(2\pi L)^{ \frac{d}{2} }} \sum_{k_1+k_2 = k} Q^\eta(k_1,k_2) \hat f (k_1) \hat g(k_2),
\end{align*}
where we consider the following definition of the Fourier transform for complex-valued functions defined on $L\T^d$:
\begin{align*}
    \hat{u}(k) & = \frac{1}{(2\pi L)^{\frac{d}{2}}} \int_{L\T^d}e^{-ik\cdot y} u(y) \d y, \qquad u(x) = \frac{1}{(2\pi L)^{\frac{d}{2}}} \sum_{k\in \Z_L^d}e^{ik\cdot x}\hat{u}(k), 
\end{align*}
where $x\in L\T^d$ and $k\in \Z_L^d$ (with $\Z_L^d\vcentcolon = \frac{1}{L}\Z^d$). 

\begin{remark}
Note that, with this convention, convolution and Parseval formulas read
\begin{align*}
    \widehat{uv}(k)=\sum_{k_1+k_2=k}\frac{\hat{u}(k_1)\hat{v}(k_2)}{(2\pi L)^{\frac{d}{2}}}, \qquad \l u\r_{L^2(L\T^d)}=\l \hat{u}\r_{\ell^2(\Z_L^d)}.
\end{align*}
\end{remark}

The functions $Q^\eta(\xi_1,\xi_2)$ defined on $(\R^d)^2$ are assumed to be symmetric, bounded together with their gradient and satisfy
\begin{align*}
Q^\eta(-\xi_1,-\xi_2)=\overline{Q^\eta(\xi_1,\xi_2)},
\end{align*}
so that $Q^0( v, v)$ and $Q^1( u, u)$ in \eqref{KG system} are real-valued. We also assume
\begin{align}\label{assumption Q}
Q^{\eta}(\xi,-\xi) = 0,
\end{align}
to ensure that the zero-mode of the solution is not affected by the nonlinearities.

\begin{remark}
We impose the assumption that nonlinearities do not affect the zero-mode to avoid additional technical complications that do not contribute to a deeper understanding of the underlying phenomenon. Without this assumption, the effective dynamics and the final stages of the proof become significantly more complex (see Remark \ref{rm:FinStageZero}). Moreover, the main comparison estimates \eqref{main estimate 1}-\eqref{main estimate 2} fail to hold for $k=0$.
Nevertheless, the general analysis remains valid without \eqref{assumption Q} (see Remark \ref{rm:AnalysisZero}).
\end{remark}

\paragraph{The initial data.}
The initial data for $u$ and $v$ in \eqref{KG system} are defined as follows. We consider $\pth{\pth{\mu^0_k,\mu^1_k}}_{k\in \Z_L^d}$ a family of independent Gaussians variables in $\mathbb{C}^2$ such that, for all $k,k'\in\Z_L^d$ and $\eta,\eta'\in\{0,1\}$, we have
\begin{align}\label{assumption ID}
\E\pth{ \mu^\eta_k} & =0, & \E\pth{ \mu^\eta_k \mu^{\eta'}_{k'}} & = 0 , & \E\pth{ \mu^\eta_k\overline{ \mu^{\eta'}_{k'}} } & = \de_{k-k'} M^{\eta,\eta'}(k) ,
\end{align}
where $M^{\eta,\eta'}\in W^{1,\infty}(\R^d)\cap W^{1,1}(\R^d)$ and satisfy $M^{\eta,\eta'}=\overline{M^{\eta',\eta}}$ and $\left| M^{\eta,\eta'} \right|^2 \leq \left| M^{\eta,\eta} \right| \left| M^{\eta',\eta'} \right| $. We also assume that functions $M^{\eta,\eta'}$ are supported in the ball $B(0,R)$ centered at the origin of radius $R>0$. We then define $(u_\init,u'_\init)$ and $(v_\init,v'_\init)$ in Fourier by
\begin{equation}\label{ID pour uv}
\begin{aligned}
\widehat{u_\init}(k) & \vcentcolon = \frac{1}{2\ps{k}}\pth{ \mu^0_k + \overline{\mu^0_{-k}}}, & \widehat{u'_\init}(k) & \vcentcolon = \frac{1}{2i}\pth{ \mu^0_k - \overline{\mu^0_{-k}}},
\\ \widehat{v_\init}(k) & \vcentcolon = \frac{1}{2\ps{k}}\pth{ \mu^1_k + \overline{\mu^1_{-k}}}, & \widehat{v'_\init}(k) & \vcentcolon = \frac{1}{2i}\pth{ \mu^1_k - \overline{\mu^1_{-k}}},
\end{aligned}
\end{equation}
where the Japanese bracket is defined by $\an k \vcentcolon = \sqrt{m+|k|^2}$. This choice of initial data is such that $u_\init$, $u'_\init$, $v_\init$ and $v'_\init$ are real-valued.

\paragraph{First order formulation.} In order to transform the second-order system \eqref{KG system} into a first-order one, we introduce $H \vcentcolon= \sqrt{-\De +m}$ and 
\begin{align}\label{def U V}
    U & \vcentcolon = H u + i \dr_t u , \qquad V  \vcentcolon = H v + i \dr_t v. 
\end{align}
These new unknowns satisfy the following half-Klein-Gordon system
\begin{equation}\label{half wave system}
    \left\{
    \begin{aligned}
    i\dr_t U & = H U - \frac{\e}{4} Q^0\pth{H^{-1} \pth{V+\bar{V}},H^{-1} \pth{V+\bar{V}}},
    \\ i\dr_t V & = H V - \frac{\e}{4} Q^1\pth{H^{-1} \pth{U+\bar{U}},H^{-1} \pth{U+\bar{U}}},
    \\ (U,V)_{|_{\{t=0\}}} & =(U_\init,V_\init).
    \end{aligned}
    \right.
\end{equation}
From \eqref{ID pour uv} and \eqref{def U V} we deduce that the Fourier transforms of $U_\init$ and $V_\init$ are simply given by 
\begin{align*}
    \widehat{U_\init}(k) =  \mu_k^0 , \qquad \widehat{V_\init}(k) =  \mu_k^1 .
\end{align*}
In order to screen the linear evolution in system \eqref{half wave system}, we introduce 
\begin{align*}
X^{0}(t) \vcentcolon = e^{itH}U(t) , \qquad X^1(t) \vcentcolon = e^{itH}V(t).
\end{align*}
Moreover, in order to get the most concise system for $X^0$ and $X^1$, we make use of two notations: for $\eta\in\{0,1\}$ we set $\bar{\eta}=1-\eta$ and if $z\in\mathbb{C}$ and $\iota\in\{\pm\}$ then $z^\iota$ denotes $z$ if $\iota=+$ or $\bar{z}$ if $\iota=-$. Note that in Fourier space this translates to $\widehat{u^\iota}(k)=\widehat{u}(\iota k)^\iota$. Using these notations, system \eqref{half wave system} for $(U,V)$ rewrites as the following integral equation
\begin{equation}\label{fixed point}
\widehat{X^\eta}(t,k)  = \mu_k^\eta + \frac{i\e}{L^{\frac{d}{2}}} \int_0^t  \sum_{\substack{k_1+k_2=k\\\iota_1,\iota_2\in\{\pm\}}} q^\eta(k_1,k_2) e^{i\tau\De^{(\iota_1,\iota_2)}_{k_1,k_2}} \widehat{X^{\overline{\eta}}}(\tau,\iota_1k_1)^{\iota_1}\widehat{X^{\overline{\eta}}}(\tau,\iota_2k_2)^{\iota_2}\d\tau,
\end{equation}
where we defined
\begin{align}
    \De^{(\iota_1,\iota_2)}_{k_1,k_2} & \vcentcolon = \ps{k_1+k_2}-\iota_1\ps{k_1}-\iota_2\ps{k_2},\label{De double}
    \\ q^\eta(k_1,k_2) & \vcentcolon =  \frac{Q^\eta\pth{k_1,k_2}}{4(2\pi )^{\frac{d}{2}}\an{k_1}\an{k_2}} .\label{q double}
\end{align}

\paragraph{Normal form.} The choice of the Klein-Gordon dispersion relation with a non-zero mass implies the absence of first order resonances as shown in the next lemma.

\begin{lemma}\label{lem De double}
For all $\iota_1,\iota_2\in\{\pm\}$ and $a,b\in\R^d$ we have
\begin{align}\label{estim De}
\frac{1}{\left| \De^{(\iota_1,\iota_2)}_{a,b} \right|} \lesssim \an{a}^\half \an{b}^\half.
\end{align}
\end{lemma}

\begin{proof}
If $(\iota_1,\iota_2)=(-,-)$ then \eqref{estim De} simply follows from $\an{a}\geq \sqrt{m}$, which implies both $\De^{(-,-)}_{a,b}\geq 3\sqrt{m}$ and $\an{a}^\half \an{b}^\half\geq \sqrt{m}$. For $(\iota_1,\iota_2)=(+,+)$, we first note the identity
\begin{align}\label{De++ De --}
-\De^{(+,+)}_{a,b} \De^{(-,-)}_{a,b} &  = m  + 2 \pth{ \ps{a}\ps{b} - a\cdot b }.
\end{align}
Using also the triangle inequality $\ps{a+b}\lesssim \ps{a}+\ps{b}$ to get $\De^{(-,-)}_{a,b}\lesssim  \ps{a}+\ps{b}$, \eqref{De++ De --} implies
\begin{align*}
\frac{1}{\left| \De^{(+,+)}_{a,b} \right|}  & \lesssim \frac{ \ps{a}+\ps{b}}{ \ps{a}\ps{b} - |a| |b| }.
\end{align*}
Multiplying by the conjugate quantity and using $ |a| |b| \leq \ps{a}\ps{b}$ we get
\begin{align*}
\frac{1}{\left| \De^{(+,+)}_{a,b} \right|}  & \lesssim \frac{ \pth{\ps{a}+\ps{b}} \ps{a}\ps{b}}{ \ps{a}^2\ps{b}^2 - |a|^2 |b|^2 } \lesssim \frac{ \pth{\ps{a}+\ps{b}} \ps{a}\ps{b}}{ \ps{a}^2+\ps{b}^2  },
\end{align*}
where we also used $\ps{a}^2+\ps{b}^2 \leq 2 (m + |a|^2 +  |b|^2)$. Using now $\pth{\ps{a}+\ps{b}}^2\leq 2( \ps{a}^2+\ps{b}^2)$ and $2\ps{a}^\half \ps{b}^\half\leq \ps{a}+\ps{b}$ we finally obtain
\begin{align*}
\frac{1}{\left| \De^{(+,+)}_{a,b} \right| } & \lesssim \frac{  \ps{a}\ps{b}}{ \ps{a}+\ps{b}  } \lesssim \ps{a}^\half\ps{b}^\half,
\end{align*}
which proves \eqref{estim De} in the case $(\iota_1,\iota_2)=(+,+)$. Now for $(\iota_1,\iota_2)=(-,+)$, we distinguish two cases. First if $|b|\leq |a|$, then $\De^{(-,+)}_{a,b}\geq \ps{a+b} \geq \sqrt{m}$ so that \eqref{estim De} holds. If $|a|\leq |b|$ then we use $\De^{(-,+)}_{a,b} = - \De^{(+,+)}_{a+b,-a}$ and \eqref{estim De} in the case $(\iota_1,\iota_2)=(+,+)$, which gives
\begin{align*}
\frac{1}{\left| \De^{(-,+)}_{a,b} \right|}  \lesssim \ps{a+b}^\half \ps{a}^\half \lesssim \ps{b}^\half \ps{a}^\half.
\end{align*}
The case $(\iota_1,\iota_2)=(+,-)$ is deduced from the case $(\iota_1,\iota_2)=(-,+)$ using $\De^{(+,-)}_{a,b}=\De^{(-,+)}_{b,a}$. This concludes the proof of \eqref{estim De}.
\end{proof}

The above lemma allows us to perform a normal form procedure on \eqref{fixed point}. More precisely, using integration by parts
one can show that \eqref{fixed point} is equivalent to the fixed point problem
\begin{equation}\label{fixed point after normal form}
\begin{aligned}
&\widehat{X^\eta}(t,k) 
\\& =  \widehat{X^\eta}_\init(k)
\\&\quad +  \frac{\e}{L^{\frac{d}{2}}}  \sum_{\substack{k_1+k_2=k\\\iota_1,\iota_2\in\{\pm\}}} \frac{q^\eta(k_1,k_2)}{\De^{(\iota_1,\iota_2)}_{k_1,k_2}}\pth{   e^{it\De^{(\iota_1,\iota_2)}_{k_1,k_2}} \widehat{X^{\overline{\eta}}}(t,\iota_1k_1)^{\iota_1}\widehat{X^{\overline{\eta}}}(t,\iota_2k_2)^{\iota_2} - \widehat{X^{\overline{\eta}}}_\init(\iota_1k_1)^{\iota_1}\widehat{X^{\overline{\eta}}}_\init(\iota_2k_2)^{\iota_2}}
\\&\quad -\frac{i\e^2}{L^{d}} \int_0^t    \sum_{\substack{k_1+k_2+k_3=k \\\iota_1,\iota_2,\iota_3\in\{\pm\}}} e^{i\tau\De^{(\iota_1,\iota_2,\iota_3)}_{k_1,k_2,k_3}} q^\eta_{\iota_2}(k_1,k_2,k_3) \widehat{X^\eta}(\tau,\iota_1k_1)^{\iota_1}\widehat{X^{\overline{\eta}}}(\tau,\iota_2k_2)^{\iota_2}\widehat{X^\eta}(\tau,\iota_3k_3)^{\iota_3}\d\tau,
\end{aligned}
\end{equation}
where $  \widehat{X^\eta}_\init(k)\vcentcolon = \mu^\eta_k$ and where we defined
\begin{align}
    \De^{(\iota_1,\iota_2,\iota_3)}_{k_1,k_2,k_3} & \vcentcolon = \ps{k_1+k_2+k_3}-\iota_1\ps{k_1}-\iota_2\ps{k_2}-\iota_3\ps{k_3}, \label{De triple}
    \\q^\eta_{\iota}(k_1,k_2,k_3) & \vcentcolon =  \frac{4 \ps{k_1+k_3}}{ m+2\pth{(k_1+k_2+k_3)\cdot k_2-\iota\ps{k_1+k_2+k_3}\ps{k_2}} } q^{\overline{\eta}}(k_1,k_3)q^\eta(k_1+k_3,k_2) .\label{q triple}
\end{align}

\begin{remark}
Lemma \ref{lem De double} shows that there are no exact first order resonances but also that there might be high-frequency first order quasi-resonances. In this article they play no role since they are compensated by the decay of $q^\eta(a,b)$, which follows from the boundedness of $Q^\eta(a,b)$ in \eqref{q double}.
\end{remark}

\subsection{Statement of the main theorem}

Our main result concerns the derivation of an effective dynamics for the correlations associated to \eqref{fixed point after normal form}. This dynamics is governed by the following system satisfied by functions $\rho^\eta$ (for $\eta=0,1$) and $\rho^\times$ defined on $\R^d$: 
\begin{equation}\label{cross system} 
\left\{
\begin{aligned}
\dr_t \rho^0 (t,\xi) & = 4 \rho^0 (t,\xi) \int_{\R^d} \mathrm{Im}\pth{ N^0(\xi,\zeta) \overline{\rho^\times(t,\zeta)}}  \d\zeta ,
\\ \dr_t \rho^1 (t,\xi) & = 4 \rho^1 (t,\xi) \int_{\R^d} \mathrm{Im}\pth{ N^1(\xi,\zeta) \rho^\times(t,\zeta)}  \d\zeta ,
\\ \dr_t \rho^\times (t,\xi) & = 2i \rho^\times (t,\xi) \int_{\R^d}  \pth{ P_+(\xi,\zeta) \overline{\rho^\times (t,\zeta)} +  P_-(\xi,\zeta) \rho^\times (t,\zeta)} \d\zeta,
\\ (\rho^0,\rho^1,\rho^\times)_{|_{\{t=0\}}} & = \pth{ M^{0,0},M^{1,1}, M^{0,1}},
\end{aligned}
\right.
\end{equation}
where the kernels $N^\eta$ and $P_\pm$ are defined by
\begin{align}
N^\eta(\xi,\zeta) & \vcentcolon =  q^\eta_{+}(\xi,\zeta,-\zeta) - \overline{q^\eta_{-}(\xi,-\zeta,\zeta)} \label{kernel N},
\\ P_\pm(\xi,\zeta) & \vcentcolon =   \overline{q^1_{\pm}(\xi,\pm\zeta,\mp\zeta)} - q^0_{\pm}(\xi,\pm\zeta,\mp\zeta). \label{kernel P} 
\end{align}
We now state the main result of our article. 

\begin{theorem}\label{main theorem}
Let $s>\frac{d}{2}$, $\b>2$ and $A>0$. There exist $\de_0,L_0>0$ such that if $0<\de\leq \de_0$, $L\geq L_0$ and $\varepsilon \vcentcolon = L^{-\frac1\beta}$, the following holds:
\begin{itemize}
    \item[(i)] There exists a set $\mathcal{E}_{L,A}$ of probability larger than $1-L^{-A}$ such that, if the initial datum $(U_\init, V_\init) $ introduced in the previous subsection belongs to $\mathcal{E}_{L,A}$, equation \eqref{fixed point after normal form} admits a unique solution $X^\eta\in \mathcal{C}\pth{ [0,\de\e^{-2}], H^s(L\T^d) }$;
    \item[(ii)] There exists a unique solution $\pth{\rho^\eta,\rho^\times}\in \mathcal{C}\pth{ [0,\de], W^{1,\infty}(\R^d)\cap W^{1,1}(\R^d)}$ of \eqref{cross system};
    \item[(iii)] We have
    \begin{align}\label{main estimate 1}
    \lim_{L\to \infty}\sup_{t\in[0,\de]}\sup_{k\in\Z_L^d}\left| \E\pth{ \mathbbm{1}_{\mathcal{E}_{L,A}} \left| \widehat{X^\eta}(\e^{-2}t,k) \right|^2 } - \rho^\eta(t,k) \right| & = 0,
    \end{align}
    and
    \begin{align}\label{main estimate 2}
    \lim_{L\to \infty}\sup_{t\in[0,\de]}\sup_{k\in\Z_L^d}\left| \E\pth{ \mathbbm{1}_{\mathcal{E}_{L,A}} \widehat{X^0}(\e^{-2}t,k) \overline{\widehat{X^1}(\e^{-2}t,k)} } - \rho^\times(t,k) \right| & = 0.
    \end{align}
\end{itemize}
\end{theorem}

\begin{remark}
We remark that, although $U_\init$ and $V_\init$ belong to any Sobolev space, they are homogeneous. In particular, their $H^s$ norms are of order $L^{\frac{d}{2}}$. Consequently, relying solely on deterministic methods would only allow us to solve \eqref{fixed point after normal form} up to the time scale $\e^{-2}L^{-d}$, which vanishes as $L\to\infty$, in contrast to the timescale achieved in Theorem 1.1.
\end{remark}

\noindent We now make some comments on Theorem \ref{main theorem}.

\paragraph{The correlations of the Klein-Gordon system.} We can deduce from \eqref{main estimate 1}-\eqref{main estimate 2} the behavior of the correlations associated to the initial system \eqref{KG system}. First, note that the proof of \eqref{main estimate 1}-\eqref{main estimate 2} can be easily adapted to show that, for all $\eta,\eta'\in\{0,1\}$, 
\begin{align}\label{main estimate 3}
    \lim_{L\to \infty}\sup_{t\in[0,\de]}\sup_{k,k'\in\Z_L^d} \E\pth{ \mathbbm{1}_{\mathcal{E}_{L,A}}  \widehat{X^\eta}(\e^{-2}t,k) \widehat{X^{\eta'}}(\e^{-2}t, k')  } & = 0.
\end{align}
 Using \eqref{main estimate 1}-\eqref{main estimate 2}-\eqref{main estimate 3} and \eqref{def U V} we can obtain in the limit $L\to \infty$ the following:
\begin{align*}
    \E\pth{\mathbbm{1}_{\mathcal{E}_{L,A}}  \left| \hat{u}(\e^{-2}t,k) \right|^2 } &\longrightarrow \frac{\rho^0(t,k)+\rho^0(t,-k)}{4\ps{k}^2} ,
    \\ \E\pth{\mathbbm{1}_{\mathcal{E}_{L,A}}  \left| \hat{v}(\e^{-2}t,k) \right|^2 }& \longrightarrow \frac{\rho^1(t,k)+\rho^1(t,-k)}{4\ps{k}^2} ,
    \\ \E\pth{\mathbbm{1}_{\mathcal{E}_{L,A}}   \hat{u}(\e^{-2}t,k) \overline{ \hat{v}(\e^{-2}t,k) }  } &\longrightarrow \frac{\rho^\times(t,k)+\overline{\rho^\times(t,-k)}}{4\ps{k}^2},
\end{align*}
where all the above limits are uniform in $t\in[0,\de]$ and $k\in\Z_L^d$. Note that the exact same limits hold if we replace $\hat{u}(\e^{-2}t,k)$ (resp. $\hat{v}(\e^{-2}t,k)$) by $\ps{k}^{-1} \widehat{\dr_t u}(\e^{-2}t,k)$ (resp. $\ps{k}^{-1} \widehat{\dr_t v}(\e^{-2}t,k)$). We also obtain
\begin{align*}
    \E\pth{\mathbbm{1}_{\mathcal{E}_{L,A}} \hat{u}(\e^{-2}t,k) \overline{\widehat{\dr_t u}(\e^{-2}t,k)} } & \longrightarrow \frac{\rho^0(t,-k) - \rho^0(t,k)}{4i\ps{k}},
    \\ \E\pth{\mathbbm{1}_{\mathcal{E}_{L,A}} \hat{v}(\e^{-2}t,k) \overline{\widehat{\dr_t v}(\e^{-2}t,k)} } & \longrightarrow \frac{\rho^1(t,-k) - \rho^1(t,k)}{4i\ps{k}},
    \\ \E\pth{\mathbbm{1}_{\mathcal{E}_{L,A}} \hat{u}(\e^{-2}t,k) \overline{\widehat{\dr_t v}(\e^{-2}t,k)} } & \longrightarrow \frac{\overline{\rho^\times(t,-k)}-\rho^\times(t,k)}{4i\ps{k}},
    \\ \E\pth{\mathbbm{1}_{\mathcal{E}_{L,A}} \hat{v}(\e^{-2}t,k) \overline{\widehat{\dr_t u}(\e^{-2}t,k)} } & \longrightarrow \frac{\rho^\times(t,-k)-\overline{\rho^\times(t,k)}}{4i\ps{k}}.
\end{align*}

\paragraph{The effective dynamics.} System \eqref{cross system} has a triangular structure since the third equation for $\rho^\times$ does not involve $\rho^0$ and $\rho^1$, and $\rho^\times$ only appears in the source term of the first and second equations. We can actually solve these equations and obtain $(\rho^0,\rho^1)$ in terms of $\rho^\times$:
\begin{equation}\label{expression rho0 rho1}
\begin{aligned}
    \rho^0(t,\xi) & = M^{0,0}(\xi) \exp\pth{ 4 \int_0^t  \int_{\R^d} \mathrm{Im}\pth{ N^0(\xi,\zeta) \overline{\rho^\times(\tau,\zeta)}}  \d\zeta  \d\tau },
    \\ \rho^1(t,\xi) & = M^{1,1}(\xi) \exp\pth{ 4 \int_0^t    \int_{\R^d} \mathrm{Im}\pth{ N^1(\xi,\zeta) \rho^\times(\tau,\zeta)}  \d\zeta  \d\tau },
\end{aligned}
\end{equation}
where using \eqref{q double}, \eqref{q triple} and \eqref{kernel N} one can show that
\begin{align*}
    N^\eta(\xi,\zeta) & = \frac1{4(2\pi)^{d}\an \xi \an{\zeta}^2} \pth{ \frac{Q^{\overline \eta}(\xi,-\zeta)Q^{\eta}(\xi-\zeta,\zeta)}{m+2(\xi\cdot \zeta - \an \xi \an \zeta)} - \frac{\overline{Q^{\overline \eta}(\xi,\zeta)Q^{\eta}(\xi+\zeta,-\zeta)}}{m-2(\xi\cdot \zeta - \an \xi \an \zeta)}}.
\end{align*}
The expressions \eqref{expression rho0 rho1} suggests exponential decay or growth for $\rho^0$ and $\rho^1$, which might be seen as an obstruction to a kinetic regime.

Regarding the equation for $\rho^\times$, we note that the kernels $P_\pm$ are generically non-trivial since, using \eqref{q double}, \eqref{q triple} and \eqref{kernel P}, one can show that
\begin{align*}
    4(2\pi )^{d} P_\pm(\xi,\zeta) & = \frac{  \an{\xi}^{-1}\an{\zeta}^{-2}}{ m\pm2\pth{\xi\cdot \zeta-\ps{\xi}\ps{\zeta}} }  \pth{ \overline{Q^0\pth{\xi,\mp\zeta}}   \overline{Q^1\pth{\xi\mp\zeta, \pm\zeta}} -  Q^1\pth{\xi,\mp\zeta}  Q^0\pth{\xi\mp\zeta, \pm\zeta} }.
\end{align*}
Definitions \eqref{kernel N}-\eqref{kernel P} actually imply the relation $N^0-\overline{N^1}=\overline{P_-}-P_+$, which in turn yields that any solution $(\rho^0,\rho^1,\rho^\times)$ of \eqref{cross system} satisfies
\begin{align}\label{angle constant}
    \dr_t \pth{ \frac{\left| \rho^\times \right|}{\sqrt{ \rho^0\rho^1}} } & = 0,
\end{align}
on the support of $M^{0,0}M^{0,1}$. This suggests that, despite the exponential decay or growth for $\rho^0$ and $\rho^1$, the statistical correlation between $\widehat{X^0}$ and $\widehat{X^1}$ has constant modulus.  

The expressions \eqref{expression rho0 rho1} and the identity \eqref{angle constant} crucially rely on the triangular structure of \eqref{cross system}, which would not hold without assumption \eqref{assumption Q} (see Section \ref{section heuristic} below) or if we had considered different coupling terms in \eqref{KG system} such as $(Q^0(u,v),Q^1(u,v))$. In these other situations, the effective dynamics one obtains in the limiting process is much more involved than \eqref{cross system}.

\subsection{Sketch of proof}\label{section sketch of proof}

The proof of Theorem \ref{main theorem} is based on a diagrammatic representation of the solution $X^\eta$ to \eqref{fixed point after normal form} in terms of trees. It is spread over four sections:
\begin{itemize}
\item In Section \ref{subsec:diag} we define a sequence of approximate solutions to \eqref{fixed point after normal form} and introduce our main combinatorial tools, that is signed coloured trees, couplings and bushes;
\item In Section \ref{section convergence} we prove the main estimate for the sequence of approximate solutions;
\item In Section \ref{section remainder} we construct an exact solution to \eqref{fixed point after normal form} as an appropriately truncated Dyson series plus a remainder term;
\item In Section \ref{section resonant system} we solve \eqref{cross system} and prove \eqref{main estimate 1}-\eqref{main estimate 2}.
\end{itemize}
In Sections \ref{section intro approximate}, \ref{section intro exact}, \ref{section intro correlations} and \ref{section heuristic} below we give an overview of the proof and provide a heuristic derivation of the \eqref{cross system}.

\subsubsection{Approximate solution}\label{section intro approximate}

We start by presenting the construction of a sequence of approximate solutions to \eqref{fixed point after normal form}, thus covering the general diagrammatic from Section \ref{subsec:diag} and the estimates from Section \ref{section convergence}. 

The standard way to exploit the randomness of the initial data, and in particular \eqref{assumption ID}, is to formally expand the solution $X^\eta$ to \eqref{fixed point after normal form} and obtain the so-called Dyson series $\sum_{n\geq 0} X^\eta_n$, which is a formal solution to \eqref{fixed point after normal form}. In order to estimate each term $X^\eta_n$ and prove convergence of the Dyson series, we write each $X^\eta_n$ as a sum over trees representing the combinatorial structure of the nonlinear terms in \eqref{fixed point after normal form} when iterating Duhamel's formula. Compared to the case of the cubic Schr\"odinger equation, we need to introduce new features to the diagrammatic in order to capture the peculiarities of our system:
\begin{itemize}
    \item[(i)] our trees are made of binary and/or ternary nodes, as \eqref{fixed point after normal form} contains both quadratic (coming from the normal form) and cubic interaction terms;
    \item[(ii)] we consider colour maps $c$ that associate ``colours'' in $\{0,1\}$ to each node. This is because we started with a \textit{system} of equations and need to distinguish between $X^0$ and $X^1$ in \eqref{fixed point after normal form};
    \item[(iii)] we consider sign maps $\varphi$ that associate a sign $\{\pm\}$ to each node. This is because we started with a system with real-valued solutions and need to take complex conjugation into account when passing to the first order half-Klein-Gordon system \eqref{half wave system} and to \eqref{fixed point after normal form}.
\end{itemize}
In conclusion, the natural combinatorial objects representing $X^\eta_n$ are \emph{signed coloured trees with binary and ternary nodes}, and in Proposition \ref{prop:XtoF} we prove that
\begin{align*}
X^{\eta,\iota}_n = \sum_{(A,\ffi,c) \in \mathcal A_n^{\eta,\iota}} F_{(A,\ffi,c)} ,
\end{align*}
where $\mathcal A_n^{\eta,\iota}$ is the set of signed coloured trees of size $n$ whose root has colour $\eta$ and sign $\iota$, and where $F_{(A,\ffi,c)}$ is defined recursively in Definition \ref{def FA}. The estimate for $X^\eta_n$ is inferred from the estimate of the correlations $\E\pth{ \big| \widehat{X^\eta_n} \big|^2}$, which are computed from the representation formula written above. In fact, the structure of the initial data and the Wick formula (see Equation \eqref{eq:Wick} or \cite{Janson}) for the product of Gaussians allow us to represent the correlations as a sum over \emph{couplings}, i.e. ordered pairs of signed coloured trees, with the property that the number of overall positive leaves equals that of overall negative leaves, equipped with a coupling map $\si$ pairing leaves of different sign. In Proposition \ref{prop:correlation} we obtain
\begin{align*}
\E\pth{ \left| \widehat{X^\eta_n}(\e^{-2}t,k) \right|^2 } & = \sum_{C\in\mathcal{C}^{\eta,\eta,+,-}_{n,n}} \widehat{F_C}(\e^{-2}t,k),
\end{align*}
where $\mathcal{C}^{\eta,\eta,+,-}_{n,n}$ is a set of couplings. Naively estimating $\widehat{F_C}(\e^{-2}t,k)$ over the time interval $t\in [0,\de]$ by $\de^n$ would lead to
\begin{align}\label{naive bound intro}
    \E\pth{ \left| \widehat{X^\eta_n}(\e^{-2}t,k) \right|^2} \lesssim \de^{n} n!,
\end{align}
where the factorial comes from counting the number of couplings of size $n$. This factorial growth, which is now a standard obstruction in the wave turbulence literature, cannot be beaten by the geometric factor $\de^n$ and this bound is useless to prove the convergence of the Dyson series $\sum_{n\geq 0} X^\eta_n$ over relevant timescales. 

To get a more refined estimate than \eqref{naive bound intro}, with $\de$ still independent of $L$ and $\e$, is the hardest part of the proof. Similarly to \cite{denghani2021}, our goal is to identify the contributions that lead to a systematic $\e$ gain in the estimate for $\widehat{F_C}$. As for trees, couplings have binary and/or ternary nodes. The quadratic terms in \eqref{fixed point after normal form} show that each binary node naturally contributes to an $\e$ gain in the estimate for $\widehat{F_C}$. This is \textit{a priori} not the case for ternary nodes and we employ instead a low/high-frequency argument based on a cut-off function located at $\e^\ga$ (for some $\ga\in (0,2)$). Introduced at the level of the diagrammatic, it distinguishes between \emph{low and high ternary nodes}. For instance, the first node in the example below is low while the second is high:
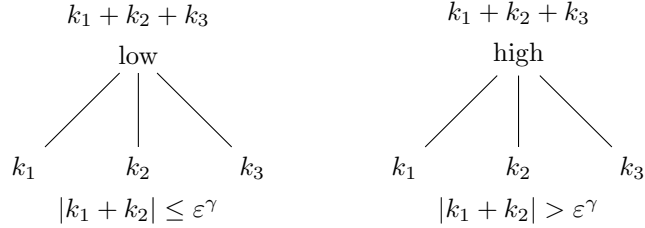
\begin{center}
\begin{tikzpicture}
\node[label=above:{$k_1+k_2+k_3$}]{\textup{low}}
    child{node{$k_1$}}
    child{node[label=below:{$|k_1 + k_2| \leq \e^\gamma$}]{$k_2$}}
    child{node{$k_3$}}
;
\node at (5,0) [label=above:{$k_1+k_2+k_3$}]{\textup{high}}
    child{node{$k_1$}}
    child{node[label=below:{$|k_1 + k_2|>\e^{\gamma}$}]{$k_2$}}
    child{node{$k_3$}}
;
\end{tikzpicture}
\end{center}

\noindent Contributions of high and low nodes are estimated differently:
\begin{itemize}
    \item on the one hand, we estimate high nodes following the strategy of \cite{LS11}, namely via oscillating integrals and spanning trees. We show that most high node contributes to $\widehat{F_C}$ with a gain of $\e^{2-\gamma}$. Note that when estimating oscillating integrals, we cannot rely on number theory arguments as in \cite{BGHS} and subsequent works since the Klein-Gordon dispersion relation is inhomogeneous;
    \item on the other hand, we estimate low nodes using the constraint $|k_1 + k_2| \leq \e^\gamma$, which in most cases yields a gain of $\e^{\gamma d}$. However, when the pair of trees below the children contributing to the low-frequency interaction forms a coupling, $k_1 + k_2$ vanishes structurally and the low-frequency constraint does not provide any gain. It is then crucial to identify the non-trivial and independent low-frequency constraints. This is done with a new combinatorial object we call \emph{bush} (see Section \ref{section bushes}), which moreover replaces the notion of skeleton from \cite{denghani2021}. Low ternary nodes with a self-coupled bush -- i.e. a bush where all leaves are paired together -- do not produce any gain in $\e$ and are thus the main obstruction to the convergence of the Dyson series. Nevertheless, such nodes induce structural constraints on the coupling and do not contribute to the factorial loss of \eqref{naive bound intro}.
\end{itemize}
Using these ideas, we obtain the following key estimate for $\widehat{F_C}(\e^{-2}t,k)$ (see Proposition \ref{prop |FC|}): 
\begin{align}\label{refined estimate FC}
\left| \widehat{F_C}(\e^{-2}t,k) \right| \lesssim \delta^{n} \varepsilon^{\alpha \pth{n-n_\scb}},
\end{align}
where $n$ is the number of nodes, $n_\scb$ is the number of self-coupled bushes and $\a>0$ is universal. The refined estimate \eqref{refined estimate FC} holds for $n\leq |\ln\e|^3$ and ultimately removes the factorial growth in the estimate for the correlations. This leads to the following estimate with high probability for $X^\eta_n$ (see Corollary \ref{coro Xn}):
\begin{equation*}
    \l X^\eta_n \r_{\mathcal{C}([0,\e^{-2}\de],H^s(L\T^d))}\lesssim \delta^{\frac{n}{2}}L^{\frac{d}{2}+\frac1\beta + A},
\end{equation*}
where $\b$ and $A$ are as in Theorem \ref{main theorem}. This estimate shows that truncating the Dyson series $\sum_{n\geq 0} X^\eta_n$ at some $N\sim \ln L$ indeed produces an approximate solution of \eqref{fixed point after normal form}. 

\begin{remark}\label{remark low LS}
In \cite{LS11}, the authors use a low/high-frequency cut-off similar to ours, except that it is not integrated in the diagrammatic but only in the equation. This is only possible because in \cite{LS11} low-frequency interactions always yield a gain in $\e$, due to the continuous nature of the phase space there.
\end{remark}

\subsubsection{Exact solution}\label{section intro exact}

Here we give an overview of Section \ref{section remainder}, which is concerned with the resolution of \eqref{fixed point after normal form} thus proving the first part of Theorem \ref{main theorem}.

We look for a solution $X^\eta$ of \eqref{fixed point after normal form} of the form 
\begin{align*}
   X^\eta =  X_{\leq N}^\eta + v^\eta, \quad \eta=0,1
\end{align*}
where $X_{\leq N} \vcentcolon =\sum_{n\leq N}X_n$ and $v$ is a remainder that solves an equation of the form
\begin{align*}
(\mathrm{Id}-\mathfrak{L})(v) & = \mathfrak{S} + \mathfrak{N}(v).
\end{align*}
In the above equation, $\mathfrak{L}$ is a linear operator, $\mathfrak{S}$ is shown to be a small source term and $\mathfrak{N}$ contains nonlinear terms, all depending on $X_{\leq N}$. The norm of the linear operator $\mathfrak{L}$ is of order $N^{s+d}L^{\frac{1}{4}+d+A}$, which \textit{a priori} forbids the inversion of $\mathrm{Id}-\mathfrak{L}$. This obstruction was already present in \cite{denghani2021} and we use the same strategy to overcome it, namely we show that the operator norm of the $N$-th iterates $\mathfrak{L}$ can be made arbitrarily small if the truncation threshold $N$ grows logarithmically with $L$. More precisely, we can show (see Proposition \ref{prop norme itérées}) that for an arbitrary truncation threshold $N$, for $m\geq 0$ and $s>0$ we have
\begin{align}\label{estim norm op intro}
\l \mathfrak{L}^m \r_{\mathrm{op}} \lesssim {\de}^{\frac{m}{2}} m^{s+d} N^{s+d} L^{\frac{1}{4} + d + A},
\end{align}
where the operator norm refers to the action of $\mathfrak{L}$ on $\mathcal{C}([0,\e^{-2}\de],H^s(L\T^d))$. Therefore, if both $N$ and $m$ are of order $\ln L$, we can choose $\de$ small enough to have $\l \mathfrak{L}^m \r_{\mathrm{op}} \lesssim L^{-\mu} (\ln L)^{2(s+d)}$ for some $\mu >0$. This allows us to invert $\mathrm{Id}-\mathfrak{L}$ for $L$ large enough and to solve for $v$. We thus get a solution $X^\eta$ to \eqref{fixed point after normal form}.

The proof of \eqref{estim norm op intro} is the key part of Section \ref{section remainder}, and it is once again based on a tree expansion of the iterates of $\mathfrak{L}$:
\begin{align*}
\mathfrak{L}^{n}(v)^{\eta,\iota} & = \sum_{(A,\ffi,c)\in \mathcal{A}^{\lin}_{n,\eta,\iota}} F^{\lin}_{(A,\ffi,c)}.
\end{align*}
Here, $\mathcal{A}^{\lin}_{n,\eta,\iota}$ is a set of signed coloured \emph{linear} trees (see Definition \ref{def sign map colour map}), i.e. trees where the contribution of one of the leaves is $v$. The above representation formula encapsulates the transition between the recursive structure naturally associated to the iterates of $\mathfrak{L}$ to the recursive structure of mixed trees with binary and ternary nodes. We adapt the diagrammatic to linear trees and estimate $\mathfrak{L}^{n}(v)^{\eta,\iota}$ similarly to  $X^{\eta,\iota}_n$.

\subsubsection{Dynamics of the correlations}\label{section intro correlations}

Finally, we discuss the content of Section \ref{section resonant system}, which is concerned with the last two parts of Theorem \ref{main theorem}.

Our first task is to identify the correct resonant structures in the correlations that ultimately dominate in the limit $L\to \infty$. For this, we define \emph{resonant nodes} and \emph{resonant couplings}. On the one hand, resonant nodes are low ternary nodes whose bush is self-coupled and whose children satisfy particular sign conditions (we have already introduced them in Section \ref{section intro approximate} as obstructions to the convergence of the Dyson series). On the other hand, resonant couplings are couplings containing only resonant nodes. These are the relevant couplings to look at since couplings having at least one non-resonant node always allow an $\e$ gain in the estimates, see Proposition \ref{prop nonresonant}, and their contribution becomes negligible in the limit. Up to symmetries and cancellations, all nodes in a resonant coupling $C = (A,A',\sigma)$ have the following form
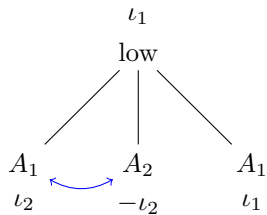
\begin{center}
\begin{tikzpicture}
\node[label=above:{$\iota_1$}]{low}
child{node(A1)[label=below:{$\iota_2$}]{$A_1$}}
child{node(A2)[label=below:{$-\iota_2$}]{$A_2$}}
child{node[label=below:{$\iota_1$}]{$A_1$}}
;
\path (A1) edge [bend right,<->,blue] (A2) ; 
\end{tikzpicture}
\end{center}
where $\iota_1$ and $\iota_2$ are signs and where the node's bush is self-coupled. The triplet $(A_1,A_2,\sigma')$, where $\sigma'$ is the restriction of $\sigma$ to the leaves of $A_1$ and $A_2$, is itself a coupling (see Lemma \ref{lemma 1 resonant nodes}), which we represent with a blue arrow in the above figure. In Lemma \ref{lemma res coupling reduction} and Proposition \ref{prop:reduceRes} we show that resonant coupling have almost a recursive structure. Namely, a resonant coupling is either trivial or of one the following non-exclusive forms:

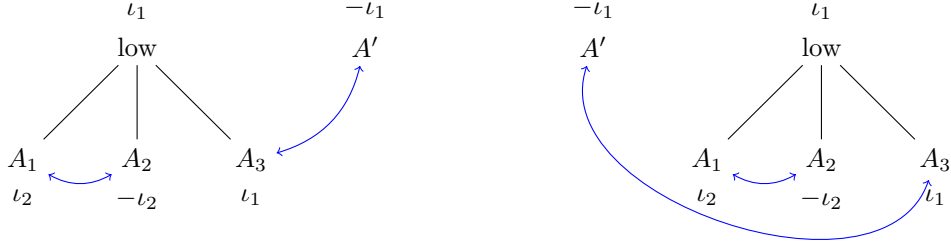
\begin{center}
\begin{tikzpicture}
\node[label=above:{$\iota_1$}]{low}
child{node(A1)[label=below:{$\iota_2$}]{$A_1$}}
child{node(A2)[label=below:{$-\iota_2$}]{$A_2$}}
child{node(A3)[label=below:{$\iota_1$}]{$A_3$}}
;
\path (A1) edge [bend right,<->,blue] (A2) ; 
\draw (3,0) node(A')[label=above:{$-\iota_1$}]{$A'$} ; 
\path (A3) edge [bend right,<->,blue] (A') ;

\node at (9,0) [label=above:{$\iota_1$}]{low}
child{node(A1)[label=below:{$\iota_2$}]{$A_1$}}
child{node(A2)[label=below:{$-\iota_2$}]{$A_2$}}
child{node(A3)[label=below:{$\iota_1$}]{$A_3$}}
;
\path (A1) edge [bend right,<->,blue] (A2) ; 
\draw (6,0) node(A')[label=above:{$-\iota_1$}]{$A'$} ; 
\path (A') edge [bend right=90,<->,blue] (A3) ;
\end{tikzpicture}
\end{center}
where $(A_1,A_2)$ and $(A_3,A')$ together with the restriction of coupling map to their leaves form resonant couplings. To get a truly recursive structure, we show in Section \ref{section btor rtob} that resonant couplings with ordered nodes are in bijection with binary trees with ordered nodes. As a consequence, resonant couplings can be endowed with the recursive structure of binary trees. This structure is the one of quadratic equations, which explains the quadratic nature of the effective dynamics described by \eqref{cross system}. Finally, note that the quantities in \eqref{main estimate 1}-\eqref{main estimate 2} are proved to converge to zero polynomially with respect to $L$.

\subsubsection{Heuristic derivation of the effective dynamics}\label{section heuristic}

We conclude this introduction by providing a heuristic derivation of the effective dynamics \eqref{cross system}. We look at the solution to \eqref{fixed point after normal form} at the appropriate timescale and introduce $\widehat{Z^\eta}(t,k) \vcentcolon = \widehat{X^\eta}(\e^{-2}t,k)$ together with 
\begin{align*}
    \rho^\eta(t,k) & \vcentcolon = \E\pth{ \left| \widehat{Z^\eta}(t,k) \right|^2 }, \qquad \rho^\times(t,k)  \vcentcolon = \E \pth{ \widehat{Z^0}(t,k) \overline{\widehat{Z^1}(t,k)}  }.
\end{align*}
We write $\widehat{Z^\eta}(t,k)$ as the sum of the following three terms:
\begin{align*}
    \widehat{Z^\eta}_0(t,k) & \vcentcolon=  \widehat{Z^\eta}_\init(k) \\
      \widehat{Z^\eta}_1(t,k) & \vcentcolon= \frac{\e}{L^{\frac{d}{2}}}  \sum_{\substack{k_1+k_2=k\\\iota_1,\iota_2\in\{\pm\}}} \frac{q^\eta(k_1,k_2)}{\De^{(\iota_1,\iota_2)}_{k_1,k_2}}\pth{   e^{i\e^{-2}t\De^{(\iota_1,\iota_2)}_{k_1,k_2}} \widehat{Z^{\overline{\eta}}}(t,\iota_1k_1)^{\iota_1}\widehat{Z^{\overline{\eta}}}(t,\iota_2k_2)^{\iota_2} - \widehat{Z^{\overline{\eta}}}_\init(\iota_1k_1)^{\iota_1}\widehat{Z^{\overline{\eta}}}_\init(\iota_2k_2)^{\iota_2}},
    \\ \widehat{Z^\eta}_2(t,k) & \vcentcolon =  -\frac{i}{L^{d}} \int_0^t    \sum_{\substack{k_1+k_2+k_3=k \\\iota_1,\iota_2,\iota_3\in\{\pm\}}} e^{i\e^{-2}\tau\De^{(\iota_1,\iota_2,\iota_3)}_{k_1,k_2,k_3}} q^\eta_{\iota_2}(k_1,k_2,k_3) \widehat{Z^\eta}(\tau,\iota_1k_1)^{\iota_1}\widehat{Z^{\overline{\eta}}}(\tau,\iota_2k_2)^{\iota_2}\widehat{Z^\eta}(\tau,\iota_3k_3)^{\iota_3}\d\tau.
\end{align*}
In the limit $\e\to0$, $\widehat{Z^\eta}_1(t,k)$ becomes negligible and can be discarded. The same can be said for the contributions to $\widehat{Z^\eta}_2$ for which the resonant factor $\De^{(\iota_1,\iota_2,\iota_3)}_{k_1,k_2,k_3}$ is non-zero, since integration by parts produces a factor $\e^2$. While quasi-resonances would only be observed at times of order $\e^{-4}$ (since the strength of the cubic nonlinearities is $\e^2$), the dynamics of $\widehat{Z^\eta}$ is lead by the exact resonances $\De^{(\iota_1,\iota_2,\iota_3)}_{k_1,k_2,k_3}=0$, so we can assume that
\begin{align*}
    \dr_t \widehat{Z^\eta}(t,k) & = -\frac{i}{L^{d}}  \sum_{\substack{k_1+k_2+k_3=k \\\iota_1,\iota_2,\iota_3\in\{\pm\}\\ \De^{(\iota_1,\iota_2,\iota_3)}_{k_1,k_2,k_3}=0}}  q^\eta_{\iota_2}(k_1,k_2,k_3) \widehat{Z^\eta}(t,\iota_1k_1)^{\iota_1}\widehat{Z^{\overline{\eta}}}(t,\iota_2k_2)^{\iota_2}\widehat{Z^\eta}(t,\iota_3k_3)^{\iota_3}.
\end{align*}
Furthermore, Wick formula and propagation of chaos suggest that only trivial resonances contribute to the dynamics of the correlations $\rho^\eta$ and  $\rho^\times$. Trivial resonances are such that $k_a+k_b=0$, $k_c=k$, $\iota_a=-\iota_b$ and $\iota_c=+$ where $a,b,c\in\{1,2,3\}$ are all distinct. Therefore, the above sum can be split into three terms and we obtain
\begin{align*}
    \dr_t\widehat{Z^\eta}(t,k) & =   -\frac{i}{L^d} \widehat{Z^\eta}(t,k) \sum_{\ell,\iota}   q^\eta_\iota(k,\ell,-\ell)    \widehat{Z^{\overline{\eta}}}(t,\iota\ell)^{\iota} \widehat{Z^\eta}(t,\iota\ell)^{-\iota} 
    \\&\quad -\frac{i}{L^d} \widehat{Z^{\overline{\eta}}}(t,k) \sum_{\ell,\iota}   q^\eta_+(\ell,k,-\ell)   \widehat{Z^\eta}(t,\iota\ell)^{\iota}  \widehat{Z^\eta}(t,\iota\ell)^{-\iota} 
    \\&\quad -\frac{i}{L^d} \widehat{Z^\eta}(t,k) \sum_{\ell,\iota}   q^\eta_{-\iota}(\ell,-\ell,k)   \widehat{Z^\eta}(t,\iota\ell)^{\iota} \widehat{Z^{\overline{\eta}}}(t,\iota\ell)^{-\iota}. 
\end{align*}
The above right hand side can be written in a more compact form. We first expand the sum over $\iota$, make a change of variable $\ell'=-\ell$ in the sums where $\iota=-$, and then regroup terms using the symmetry of $q$ in its first and third variables to get:
\begin{align*}
    \dr_t\widehat{Z^\eta}(t,k) & =    -\frac{2i}{L^d} \widehat{Z^\eta}(t,k) \sum_{\ell}  q^\eta_+(k,\ell,-\ell)     \widehat{Z^{\overline{\eta}}}(t,\ell) \overline{\widehat{Z^\eta}(t,\ell)} 
    \\
     & - \frac{2i}{L^d} \widehat{Z^{\overline{\eta}}}(t,k) \sum_{\ell}   q^\eta_+(\ell,k,-\ell)  \left|\widehat{Z^{\eta}}(t,\ell)\right|^2 
    \\& -\frac{2i}{L^d} \widehat{Z^\eta}(t,k) \sum_{\ell}  q^\eta_-(k,-\ell,\ell)  \overline{\widehat{Z^{\overline{\eta}}}(t,\ell)} \widehat{Z^\eta}(t,\ell) .
\end{align*}
We now compute the equations satisfied by the correlations:
\begin{align*}
    \dr_t\rho^0 (t,k) & =  2 \mathrm{Re}\pth{ \E\pth{\overline{\widehat{Z^0}(t,k)} \dr_t\widehat{Z^0}(t,k)}}
    \\ & =  \frac{4}{L^d}   \rho^0 (t,k) \sum_{\ell}\mathrm{Im}\pth{ q^0_+(k,\ell,-\ell)\overline{\rho^\times(\ell)} + q^0_-(k,-\ell,\ell) \rho^\times(\ell) } 
    \\&\quad + \frac4{L^d}\mathrm{Im}\pth{ \overline{\rho^\times (t,k)} \sum_\ell q_+^0 (\ell, k ,-\ell) \rho^0 (t,l)}.
\end{align*}
Recalling \eqref{kernel N}, we recognize the kernel $N^0(k,\ell)$ in the following expression
\begin{align*}
    \mathrm{Im}\Big( q^0_+(k,\ell,-\ell)\overline{\rho^\times(\ell)} + q^0_-(k,-\ell,\ell) \rho^\times(\ell) \Big)  & =  \mathrm{Im}\pth{N^0(k,\ell) \overline{\rho^\times(\ell)}},
\end{align*}
which leads to
\begin{align} \label{eq intro rho0}
    \dr_t \rho^0 (t,k) & =  \frac{4}{L^d}   \rho^0 (t,k) \sum_{\ell}\mathrm{Im}\pth{ N^0(k,\ell)\overline{\rho^\times(t,\ell)}  }  + \frac4{L^d}\mathrm{Im}\pth{ \overline{\rho^\times (t,k)} \sum_\ell q_+^0 (\ell, k ,-\ell) \rho^0 (t,\ell)}.
\end{align}
Exchanging the roles of $0$ and $1$, we obtain similarly the equation satisfied by $\rho^1$
\begin{align}\label{eq intro rho1}
    \dr_t \rho^1 (t,k) & =  \frac{4}{L^d}   \rho^1 (t,k) \sum_{\ell}\mathrm{Im}\pth{ N^1(k,\ell)\rho^\times(t,\ell)  }  + \frac4{L^d}\mathrm{Im}\pth{ \rho^\times (t,k)\sum_\ell q_+^1 (\ell, k ,-\ell) \rho^1 (t,\ell)}.
\end{align}
As for the equation satisfied by $\rho^\times$, we obtain
\begin{equation}\label{eq intro rhotimes}
    \begin{aligned}
    \dr_t\rho^\times (t,k) & =  \E\pth{ \widehat{Z^0}(t,k) \overline{\dr_t\widehat{Z^1}(t,k)}  } + \E\pth{ \overline{\widehat{Z^1}(t,k)} \dr_t \widehat{Z^0}(t,k)   }
    \\ & = \frac{2i}{L^d} \rho^\times (t,k) \sum_{\ell}\pth{ P_+(k,\ell) \overline{\rho^\times (t,\ell)} + P_-(k,\ell) \rho^\times (t,\ell)}
    \\&\quad +\frac{2i}{L^d} \sum_\ell \pth{\rho^0(t,k)\overline{q_+^1(\ell, k, -\ell)}\rho^1(t,\ell) - \rho^1(t,k) q_+^0(\ell, k, -\ell)\rho^0(t,\ell)}
    \end{aligned}
\end{equation}
where we recognized $P_+(k,\ell)$ and $P_-(k,\ell)$ as defined in \eqref{kernel P}. Collecting \eqref{eq intro rho0}, \eqref{eq intro rho1} and \eqref{eq intro rhotimes} we have formally derived a system for the correlations $(\rho^0,\rho^1,\rho^\times)$. Formally replacing the Riemann sums in these equations by integrals in the limit $L\to \infty$, we get
\begin{equation}\label{CS2}
\left\{
    \begin{aligned}
    \dr_t \rho^0 (t,\xi) & =  4   \rho^0 (t,\xi) \int_{\R^d}\mathrm{Im}\pth{ N^0(\xi,\zeta)\overline{\rho^\times(t,\zeta)}  }\d\zeta  + 4\mathrm{Im}\pth{ \overline{\rho^\times (t,\xi)} \int_{\R^d} q_+^0 (\zeta, \xi ,-\zeta) \rho^0 (t,\zeta)\d\zeta},
    \\ \dr_t \rho^1 (t,\xi) & = - 4   \rho^1 (t,\xi) \int_{\R^d}\mathrm{Im}\pth{ \overline{N^1(\xi,\zeta)}\overline{\rho^\times(\zeta)}  }\d\zeta  + 4\mathrm{Im}\pth{ \rho^\times (t,\xi)\int_{\R^d} q_+^1 (\zeta, \xi ,-\zeta) \rho^1 (t,\zeta)\d\zeta},
    \\ \dr_t\rho^\times (t,\xi) & =  2i \rho^\times (t,\xi) \int_{\R^d}\pth{ P_+(\xi,\zeta) \overline{\rho^\times (t,\zeta)} + P_-(\xi,\zeta) \rho^\times (t,\zeta)}\d\zeta
    \\&\quad +2i \int_{\R^d} \pth{\rho^0(t,\xi)\overline{q_+^1(\zeta, \xi, -\zeta)}\rho^1(t,\zeta) - \rho^1(t,\xi) q_+^0(\zeta, \xi, -\zeta)\rho^0(t,\zeta)}\d\zeta.
    \end{aligned}
\right.
\end{equation}
Compared to \eqref{cross system}, we see that the system \eqref{CS2} contains extra coupling terms all involving the quantity $q^\eta_+(\zeta,\xi,-\zeta)$. Thanks to \eqref{q double} and \eqref{q triple} we have
\begin{align*}
q^\eta_+(\zeta,\xi,-\zeta) = - \frac{Q^{\overline \eta}(\zeta,-\zeta) Q^\eta(0,\xi)}{4(2\pi)^{d} m\an{\xi}\an{\zeta}^2}.
\end{align*}
Therefore, our assumption \eqref{assumption Q} on the nonlinearities $Q^0$ and $Q^1$ appearing in the starting system \eqref{first system} reduces indeed \eqref{CS2} to \eqref{cross system}. As suggested after the statement of Theorem \ref{main theorem}, the interactions between $\rho^0$, $\rho^1$ and $\rho^\times$, though simplified by \eqref{assumption Q}, remain non-trivial.


\section{Diagrammatic}\label{subsec:diag}

In this section we introduce our main tools to represent nonlinear interactions in the system \eqref{fixed point after normal form} and to build an approximate solution:
\begin{itemize}
    \item In Section \ref{section formal expansion} we define a formal solution $X^\eta=\sum_{n\geq 0}X^\eta_n$. In Section \ref{section representation Xn}, each $X^\eta_n$ is expressed as a sum over signed and coloured trees previously introduced in Sections \ref{section trees} and \ref{section signed coloured trees}.
    \item In Sections \ref{section decoration} and \ref{section fourier FA} we introduce and use the necessary materials to go into Fourier space.
    \item In Section \ref{section coupling between trees} we introduce couplings of signed and coloured trees, which will be used to estimate with high probability each $X^\eta_n$ as well as the correlations.
    \item Finally, in Section \ref{section bushes} we introduce bushes which will play a major combinatorics role in our proof.
\end{itemize}

\subsection{Formal series expansion}\label{section formal expansion}

In order to analyse the fixed point \eqref{fixed point after normal form} and in particular the cubic interactions we will distinguish between the frequency regimes, high or low, of the interactions. In the case of the low frequency interactions, we also distinguish cases according to the position of the low interaction. To this purpose we introduce four smooth cutoff functions in frequency space. 

\begin{definition}\label{def cutoff}
Given a fixed $\ga>0$ and a fixed smooth function $\chi:\R_+\longrightarrow[0,1]$ supported on $[0,2]$ and equal to $1$ on $[0,1]$, we define
\begin{align*}
\chi_{\lowl} (k_1,k_2,k_3) & \vcentcolon = \chi(\varepsilon^{-\gamma}|k_2+k_3|)\times(1-\chi)(\varepsilon^{-\gamma}|k_2+k_1|)\times(1-\chi)(\varepsilon^{-\gamma}|k_3+k_1|),\\
\chi_{\lowm} (k_1,k_2,k_3) & \vcentcolon = \chi(\varepsilon^{-\gamma}|k_3+k_1|)\times(1-\chi)(\varepsilon^{-\gamma}|k_3+k_2|)\times(1-\chi)(\varepsilon^{-\gamma}|k_1+k_2|),\\
\chi_{\lowr} (k_1,k_2,k_3) & \vcentcolon = \chi(\varepsilon^{-\gamma}|k_1+k_2|)\times(1-\chi)(\varepsilon^{-\gamma}|k_1+k_3|)\times(1-\chi)(\varepsilon^{-\gamma}|k_2+k_3|) , \\
\chi_\high & \vcentcolon = 1-\chi_{\lowl} - \chi_{\lowm} - \chi_{\lowr}.
\end{align*}
\end{definition}

Note that 
\begin{align*}
\chi_\high (k_1,k_2,k_3) = & (1-\chi)(\varepsilon^{-\gamma}|k_1+k_2|)\times(1-\chi)(\varepsilon^{-\gamma}|k_1+k_3|)\times(1-\chi)(\varepsilon^{-\gamma}|k_2+k_3|) \\ & +
\chi(\varepsilon^{-\gamma}|k_1+k_2|)\times\chi(\varepsilon^{-\gamma}|k_1+k_3|)\times\chi(\varepsilon^{-\gamma}|k_2+k_3|)\\ & +
\chi(\varepsilon^{-\gamma}|k_1+k_2|)\times\chi(\varepsilon^{-\gamma}|k_1+k_3|)\times(1-\chi)(\varepsilon^{-\gamma}|k_2+k_3|)\\ & +\chi(\varepsilon^{-\gamma}|k_2+k_3|)\times\chi(\varepsilon^{-\gamma}|k_2+k_1|)\times(1-\chi)(\varepsilon^{-\gamma}|k_3+k_1|)\\ & +\chi(\varepsilon^{-\gamma}|k_3+k_1|)\times\chi(\varepsilon^{-\gamma}|k_3+k_2|)\times(1-\chi)(\varepsilon^{-\gamma}|k_1+k_2|).
\end{align*}
The cutoffs $\chi_{\lowl}$, $\chi_{\lowm}$ and $\chi_{\lowr}$ allow only one of the combinations to be low frequency. The case of $\chi_\high$ is the negation of this, that is either all the combinations are in high frequencies or at least two of them are in low frequencies. The cutoff $\chi_\high$ thus selects either high or very low frequencies. 

\saut
We introduce the following time-dependent bilinear and trilinear operators acting on functions defined on the torus of size $L$:
\begin{align}
\widehat{B^\eta_{(\iota_1,\iota_2)}}(t,f,g) (k) & \vcentcolon =  \frac1{L^{\frac{d}{2}}} \sum_{k_1+k_2 = k}  \frac{q^\eta (k_1,k_2)}{\De_{k_1,k_2}^{(\iota_1,\iota_2)}} e^{it \Delta_{k_1,k_2}^{(\iota_1,\iota_2)}} \hat f( k_1) \hat g( k_2),\label{def B Fourier}
\\ \widehat{C^\eta_{*,(\iota_1,\iota_2,\iota_3)}}(t,f,g,h) (k) & \vcentcolon = \frac1{L^{d}} \sum_{k_1+k_2+k_3 = k} \chi_*(k_1,k_2,k_3) q_{\iota_2}^\eta (k_1,k_2,k_3)e^{it \Delta_{k_1,k_2,k_3}^{(\iota_1,\iota_2,\iota_3)}} \hat f( k_1) \hat g( k_2) \hat h(k_3), \label{def C Fourier}
\end{align}
where $*$ stands either for $\lowl$, $\lowm$, $\lowr$ or $\high$.

\begin{remark}
Throughout the article, a sum over $*$ always denote a sum over $\lowl$, $\lowm$, $\lowr$ or $\high$, i.e
\begin{align*}
\sum_{*}= \sum_{*\in \{\lowl,\lowm,\lowr,\high\}}.
\end{align*}
This also applies to other mathematical symbols such as $\bigsqcup$.
\end{remark}

We can now rewrite \eqref{fixed point after normal form} in a more schematic form:
\begin{align*}
X^\eta(t) & = X^\eta_\init +  \e \sum_{\substack{\iota_1,\iota_2\in\{\pm\}}} \pth{B^\eta_{(\iota_1,\iota_2)}\pth{t, X^{\bar{\eta},\iota_1}(t),X^{\bar{\eta},\iota_2}(t) } - B^\eta_{(\iota_1,\iota_2)}\pth{0, X^{\bar{\eta},\iota_1}_\init,X^{\bar{\eta},\iota_2}_\init }}
\\&\quad - i \e^2 \int_0^t \sum_*\sum_{\iota_1,\iota_2,\iota_3\in\{\pm\}}  C^\eta_{*,(\iota_1,\iota_2,\iota_3)}\pth{ \tau , X^{\eta,\iota_1}(\tau),X^{\bar{\eta},\iota_2}(\tau),X^{\eta,\iota_3}(\tau) } \d\tau,
\end{align*}
and define a formal solution as follows. 

\begin{definition}\label{def X_n}
The sequence $\pth{X^{\eta}_n}_{n\in\mathbb{N},\eta=0,1}$ is defined recursively by
\begin{align}
X_0^\eta(t) & \vcentcolon = X_\init^\eta,\label{def X0}
\\ X_1^\eta (t) & \vcentcolon = \varepsilon \sum_{\iota_1,\iota_2\in\{\pm\}}\pth{B^\eta_{(\iota_1,\iota_2)}\pth{t,X_0^{\overline\eta,\iota_1}(t),X_0^{\overline \eta, \iota_2}(t)} - B^\eta_{(\iota_1,\iota_2)}\pth{0,X_0^{\overline\eta,\iota_1}(t),X_0^{\overline \eta, \iota_2}(t)}},\label{def X1}
\\ X_{n+2}^\eta(t) & \vcentcolon = -i \varepsilon^2 \int_{0}^t \sum_{*}\sum_{\substack{n_1+n_2+n_3=n\\ \iota_1,\iota_2,\iota_3\in\{\pm\}}} C^\eta_{*, (\iota_1,\iota_2,\iota_3)} \pth{\tau,X_{n_1}^{\eta,\iota_1}(\tau), X_{n_2}^{\overline\eta, \iota_2 }(\tau),X_{n_3}^{\eta, \iota_3}(\tau)}\d\tau \label{def Xn+2}
\\&\quad + \varepsilon \sum_{\substack{n_1+n_2 = n+1\\ \iota_1,\iota_2\in\{\pm\}}} B_{(\iota_1,\iota_2)}^\eta \pth{t, X_{n_1}^{\bar{\eta},\iota_1}(t), X_{n_2}^{\bar{\eta},\iota_2}(t)},\non
\end{align}
for $n\geq 0$.
\end{definition}

\begin{remark} 
\begin{itemize}
    \item One can indeed check easily that $X^\eta(t) = \sum_{n\geq 0} X^\eta_n(t)$ formally solves \eqref{fixed point after normal form}.
    \item By induction and thanks to our assumption \eqref{assumption Q} on the nonlinearities $Q^\eta$, for any $n\geq 1$, $\eta\in \{0,1\}$ and $t\in \R$, we have $\widehat{X_n^\eta} (t,0) = 0$.
\end{itemize}
\end{remark}

\subsection{Trees}\label{section trees}

We define recursively (rooted) trees adapted to our problem. They have four types $T_{\lowl}$, $T_{\lowm}$, $T_{\lowr}$, and $T_\high$  of ternary nodes and two types $B_\st$ (standard), $B_\init$ (initial) of binary nodes. We write $\bot$ the trivial tree (reduced to one leaf) and we set
\begin{align}
\tilde{ \mathcal A}_0 & \vcentcolon = \{\bot \},\label{def Atilde0}
\\ \tilde{\mathcal A}_1 & \vcentcolon = \{B_\init (\bot,\bot),B_\st(\bot,\bot)\},\label{def Atilde1}
\end{align}
and 
\begin{equation}\label{def Atilden+2}
\begin{aligned}
\tilde{\mathcal A}_{n+2} & \vcentcolon = \bigsqcup_* \enstq{ T_*(A_1,A_2,A_3) }{(A_1,A_2,A_3) \in \tilde{ \mathcal A}_{n_1} \times \tilde{ \mathcal A}_{n_2} \times\tilde{ \mathcal A}_{n_3} , \; n_1+n_2+n_3 = n}
\\&\quad \sqcup \enstq{ B_\st(A_1,A_2) }{ (A_1,A_2) \in \tilde{ \mathcal A}_{n_1} \times \tilde{ \mathcal A}_{n_2}  , \; n_1+n_2 = n+1} ,
\end{aligned}
\end{equation}
for $n\geq 0$. The following standard definition introduces the different types of nodes and the leaves of a tree $A$.

\begin{definition}\label{def ensembles liés à un arbre}
Let $A\in\tilde{\mathcal{A}}_n$ be a tree. We define
\begin{align*}
\mathcal N_\st(A) & \vcentcolon = \left\{ \text{standard binary nodes of $A$} \right\},
\\ \mathcal N_\init(A) & \vcentcolon = \left\{ \text{initial binary nodes of $A$} \right\},
\\ \mathcal N_B(A) & \vcentcolon = \mathcal N_\init (A) \sqcup \mathcal N_\st (A).
\end{align*}
For $*\in\{ \lowl, \lowm,\lowr,\high \}$ we also define
\begin{align*}
\mathcal N_*(A) & \vcentcolon = \left\{ \text{ternary nodes of $A$ of type $*$} \right\},
\end{align*}
and
\begin{align*}
\mathcal N_\low(A) & \vcentcolon = \mathcal N_\lowl(A) \sqcup \mathcal N_\lowm(A) \sqcup \mathcal N_\lowr(A),
\\ \mathcal N_T(A) & \vcentcolon = \mathcal N_\high(A)\sqcup \mathcal N_\low(A).
\end{align*}
Moreover, we define
\begin{align*}
\mathcal{L}(A) & \vcentcolon = \left\{ \text{leaves of $A$} \right\}.
\end{align*}
We denote by  $n_\st(A)$, $n_\init(A)$, $n_B(A)$, $n_*(A)$, $n_\low(A)$, $n_T(A)$ and $n_\ell(A)$ their respective cardinals. Finally, we define
\begin{align*}
    \mathcal N(A) & \vcentcolon =  \mathcal N_B(A) \sqcup \mathcal N_T(A),
\\ \mathcal{R}(A) & \vcentcolon = \left\{ r_A \right\},
\end{align*}
where $r_A$ denotes the root of $A.$
\end{definition}

\begin{remark} \label{remark tree cardinal}
By induction, if $A\in \tilde{\mathcal{A}}_n$ then we have $n_\ell(A) = n+1$ and $n = n_B(A) +  2n_T(A)$. Therefore there exists a constant $\Lambda>0$ such that for all $n\in \N$ we have
\begin{align*}
\#\tilde{\mathcal A}_n\leq \La^n.
\end{align*}
Indeed, a tree in $\tilde{\mathcal A}_n$ is uniquely determined by a sequence of symbols with $n_B$ symbols of binary nodes and $n_T$ symbols of ternary nodes and $n+1$ symbols of leaves. Since there are two types of binary nodes and four types of ternary nodes, we deduce that the cardinal of  $\tilde{\mathcal{A}}_n$ is bounded by  
\begin{align*}
\sum_{n_B+2n_T=n} 2^{n_B}4^{n_T}\begin{pmatrix} n_B+n_T+n+1\\ n_B\end{pmatrix} \begin{pmatrix} n_T+n+1\\ n_T\end{pmatrix}\leq \sum_{n_B+2n_T=n} 2^n 2^{n_B+2n_T + 2n+2} \leq 4 (n+1)16^n. 
\end{align*}
We deduce that the cardinal is bounded by $4 (32)^n$. For $n\geq 1$, this is less that $(128)^n$. For $n=0$, the cardinal of $A\in \tilde{\mathcal{A}}_0$ is $1$ such that we get the bound $\Lambda \geq 128$.
\end{remark}

\begin{definition}
Let $A\in\tilde{A}_n$ be a tree.
\begin{itemize}
\item If $j\in\mathcal{N}(A)$, we denote by
\begin{align*}
\mathtt{children}(j)
\end{align*}
the set of its children. More precisely, if $j = B_*(A_1, A_2)$ for $* = \st, \init$ then $\mathtt{children}(j) =\{r_{A_1}, r_{A_2}\}$; if $j=T_*(A_1, A_2, A_3)$ for $*=\lowl, \lowm,\lowr, \high$ then $\children(j) = \{r_{A_1}, r_{A_2}, r_{A_3}\}$.
\item If $j\in(\mathcal{N}(A)\sqcup\mathcal{L}(A))\setminus \mathcal{R}(A)$, we denote by
\begin{align*}
\mathtt{parent}(j)
\end{align*}
the parent of $j$ in $A$, and if $\mathcal{S}\subset (\mathcal{N}(A)\sqcup\mathcal{L}(A))\setminus \mathcal{R}(A)$ we use the notation
\begin{align*}
\mathtt{parent}(\mathcal{S})=\enstq{\mathtt{parent}(j)}{j\in\mathcal{S}}.
\end{align*}
\item If $j\in(\mathcal{N}(A)\sqcup\mathcal{L}(A))\setminus \mathcal{R}(A)$, we define
\begin{align*}
\mathtt{siblings}(j)=\mathtt{children}\pth{\mathtt{parent}(j)}.
\end{align*}
Note that $j\in \siblings(j)$.
\end{itemize} 
\end{definition}

\begin{remark}\label{remark parent}
We note that if $\mathcal{S}\subset (\mathcal{N}(A)\sqcup\mathcal{L}(A))\setminus \mathcal{R}(A)$ then
\begin{align*}
\#\mathcal{S} & = \# \mathtt{p}(\mathcal{S})_1 + 2 \#\mathtt{p}(\mathcal{S})_2  + 3 \#\mathtt{p}(\mathcal{S})_3,
\\ \#\mathtt{parent}(\mathcal{S}) & =  \# \mathtt{p}(\mathcal{S})_1 +  \#\mathtt{p}(\mathcal{S})_2  +  \#\mathtt{p}(\mathcal{S})_3,
\end{align*}
where $\mathtt{p}(\mathcal{S})_i=\enstq{ j \in \mathcal{N}(A)}{\#\pth{ \mathtt{children}(j)\cap\mathcal{S}}=i}$.
\end{remark}

We introduce the standard parentality partial order relation $<$ on the nodes and leaves of a tree.

\begin{definition}
Let $A\in\tilde{\mathcal{A}}_n$ be a tree. 
\begin{itemize}
\item For $j,j'\in\mathcal{N}(A)$, we say that $j<j'$ if there exists $n\in \N^*$ and $j_1,\dots, j_{n+1}\in \mathcal{N}(A)$ with $j_1=j$, $j_{n+1}=j'$ and such that
\begin{align*}
\forall i\in\llbracket 1,n \rrbracket, \; j_i\in \mathtt{children}{(j_{i+1})}.
\end{align*}
\item Moreover, if $\ell\in\mathcal{L}(A)$ and $j\in\mathcal{N}(A)$ we say that $\ell<j$ if there exists $j'\in\mathcal{N}(A)$ such that $\ell\in\mathtt{children}(j')$ and $j'\leq j$.
\end{itemize}
\end{definition}

Note that if $j,j'\in\mathcal{N}(A)\sqcup\mathcal{L}(A)$ are such that $j< j'$, then the sequence $j_2,\dots,j_{n}\in\mathcal{N}(A)$ linking $j$ to $j'$ (i.e such that $j\in\mathtt{children}(j_2)$, $j_i\in\mathtt{children}(j_{i+1})$ and $j_n\in\mathtt{children}(j')$) is unique. For this reason, for a given $j_0\in \mathcal N(A) \sqcup \mathcal L(A)$, the set $\{j\in \mathcal{N}(A) : j_0\le j\}$ -- which corresponds to the path from $j_0$ to the root -- is totally ordered. 

\subsection{Signed and coloured trees}\label{section signed coloured trees}

To incorporate the fact that we work with a system of equations, with complex conjugates in the nonlinearities, we introduce the notion of a signed and coloured tree. Recall that if $\eta\in\{0,1\}$, then $\overline{\eta}=1-\eta$.

\begin{definition}\label{def sign map colour map}
Let $A\in\tilde{\mathcal{A}}_n$ be a tree. 
\begin{itemize}
\item A \textbf{sign map} on $A$ is a map 
\begin{align*}
\ffi: \mathcal{L}(A)\sqcup \mathcal{N}(A) \longrightarrow \{\pm\}.
\end{align*}
If $j$ is a leaf or a branching node of $A$, then $\ffi(j)$ is the sign of $j$.
\item A \textbf{colour map} on $A$ is a map
\begin{align*}
c: \mathcal{L}(A)\sqcup \mathcal{N}(A) \longrightarrow \{0,1\},
\end{align*}
such that
\begin{itemize}
\item[(i)] if $j$ is a binary node, then $c(j')=\overline{c(j)}$ where $j'$ is any child of $j$,
\item[(ii)] if $j$ is a ternary node, then $c(\ell)=c(r)=c(j)$ and $c(m)=\overline{c(j)}$ where $\ell$, $m$ and $r$ denote respectively the left, middle and right child of $j$.
\end{itemize}
If $j\in\mathcal{L}(A)\sqcup \mathcal{N}(A)$, then $c(j)$ is the colour of $j$.
\item If $\ffi$ and $c$ are a sign map and a colour map on $A$, then $(A,\ffi,c)$ is called a \textbf{signed and coloured tree}.
\item For $\iota\in\{\pm\}$ and $\eta\in\{0,1\}$ we set
\begin{align*}
\mathcal A^{\eta,\iota}_n & \vcentcolon =\enstq{ (A, \ffi,c) \text{ signed and coloured tree}} {A\in \tilde{\mathcal{A}}_n, \; \ffi(r_A)=\iota, \; c(r_A)=\eta }  .
\end{align*}
\end{itemize}
\end{definition}

\begin{remark}\label{rm:cardinal}
We note the following simple facts:
\begin{itemize}
\item A colour map $c$ on a tree $A$ is uniquely determined by its value $c(r_A)$ at the root of $A$, opposite to a sign map $\ffi$ on $A$ which does not satisfy any recursive property (in particular there is no relation between the sign of a node and that of its children).
\item There exists some constant $\Lambda>0$ such that for all $n\in \N$, the cardinal of $\mathcal A_n^{\eta, \iota}$ is bounded by $\Lambda^n$. Indeed, the sign map is determined on the root, but there remains $n_B+n_T+n$ signs to attribute. We deduce that $\#\mathcal A_n^{\eta, \iota}\leq 4^n \# \tilde{\mathcal A}_n$ from which we deduce the bound.
\end{itemize}
\end{remark}

\begin{example}\label{premier exemple} Below we give a graphical representation of the signed and coloured tree 
\begin{align*}
(T_\high(B_\init(\bot,\bot),B_\st(\bot,\bot),B_\st(\bot,\bot)),\ffi, c_0)
\end{align*}
where $c_0$ is the colour map obtained by attributing the colour $0$ to the root. We have chosen the colour red to encode $\eta=0$, the colour yellow to encode $\eta=1$. The choice of the signs of the nodes and leaves is completely arbitrary.
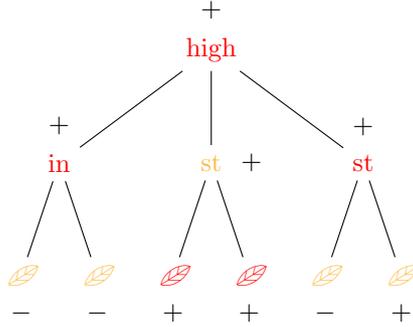
\begin{center}
\begin{tikzpicture}[level 1/.style={sibling distance=2cm},
                    level 2/.style={sibling distance=1cm}]
\node[label=above:{$+$}]{\textcolor{red}{$\textup{high}$}}
child { node[label=above:{$+$}]{\textcolor{red}{ \textup{in}}}
    child{node[label=below:{$-$}]{\textcolor{Dandelion}{\textleaf}}}
    child{node[label=below:{$-$}]{\textcolor{Dandelion}{\textleaf}}}
}
child { node[label=right: {$+$}]{\textcolor{Dandelion}{$\textup{st}$}}
    child{node[label=below:{$+$}]{\textcolor{red}{\textleaf}}}
    child{node[label=below:{$+$}]{\textcolor{red}{\textleaf}}}
}
child { node[label=above: {$+$}]{\textcolor{red}{$\textup{st}$}}
    child{node[label=below:{$-$}]{\textcolor{Dandelion}{\textleaf}}}
    child{node[label=below:{$+$}]{\textcolor{Dandelion}{\textleaf}}}
}
;
\end{tikzpicture}
\end{center}
\end{example}

\subsection{\texorpdfstring{Representation of $X_n$ in terms of trees}{Representation of Xn in terms of trees}} \label{section representation Xn}

The introduction of all this diagrammatic allows us to represent each element of the sequence $\pth{ X\eta_n}_{n\in\mathbb{N}}$ defined in \ref{def X_n} as the sum over trees representing nonlinear interactions. We first define recursively the contribution of a given coloured and signed tree.

\begin{definition}\label{def FA}
If $(A,\ffi,c)$ is a signed and coloured tree, we define $F_{(A,\ffi,c)}(t)$ recursively:
\begin{itemize}
\item We set 
\begin{align}\label{def Fbot}
F_{(\bot,\iota,\eta)}(t) & \vcentcolon = X^{\eta,\iota}_\init ,
\end{align}
where the sign and colour maps are identified with their single values $\iota\in\{\pm\}$ and $\eta\in\{0,1\}$, and
\begin{align}\label{def FBinbot}
F_{(B_\init(\bot,\bot),\ffi,c)} (t) & \vcentcolon = - \varepsilon   B^{\eta, \iota_r}_{\iota_r(\iota_1,\iota_2)} \pth{0,X^{\bar{\eta},\iota_r\iota_1}_\init,X^{\bar{\eta},\iota_r\iota_2}_\init } ,
\end{align}
where $\iota_r$, $\iota_1$, $\iota_2$ are the signs of the root, the left and right child, and $\eta$ is the colour of the root.
\item If $A = B_\st(A_1,A_2)$, we set
\begin{align}\label{def FBst}
F_{\pth{A,\ffi, c}} (t) & \vcentcolon = \varepsilon  B_{\iota_r(\iota_1,\iota_2)}^{\eta,\iota_r}\pth{t, F_{\pth{A_1,\iota_r \ffi_{|_{A_1}},c_{|_{A_1}}}} (t), F_{\pth{A_2,\iota_r \ffi_{|_{A_2}},c_{|_{A_2}}}}(t)},
\end{align}
where $\iota_r$ and $\iota_1$, $\iota_2$ are respectively the signs of the root of $A$ and of the roots of $A_1$, $A_2$ in $A$, and $\eta$ is the colour of the root of $A$.
\item If $A = T_*(A_1,A_2,A_3)$, we set 
\begin{align}
&F_{\pth{A,\ffi, c}} (t) \non
\\& \vcentcolon = -i \iota_r \varepsilon^2 \int_0^t C_{*,\iota_r(\iota_1,\iota_2,\iota_3)}^{\eta,\iota_r}\pth{\tau, F_{\pth{A_1,\iota_r \ffi_{|_{A_1}},c_{|_{A_1}}}} (\tau), F_{\pth{A_2,\iota_r \ffi_{|_{A_2}},c_{|_{A_2}}} }(\tau), F_{\pth{A_3,\iota_r \ffi_{|_{A_3}}, c_{|_{A_3}}}} (\tau)}\d\tau ,\label{def FT}
\end{align}
where $\iota_r$ and $\iota_1$, $\iota_2$, $\iota_3$ are the signs of the root of $A$ and the roots of $A_1$, $A_2$, $A_3$ in $A$, and $\eta$ is the colour of the root of $A$.
\end{itemize}
In \eqref{def FBst} and \eqref{def FT}, we used as a shortcut the notations $c_{|_{A_i}}$ (resp. $\varphi_{|_{A_i}}$) for $i=1,2,3$ to denote $c_{|_{\mathcal N(A_i)\sqcup \mathcal L(A_i)}}$ (resp. $\varphi_{|_{\mathcal N(A_i)\sqcup \mathcal L(A_i)}}$).
\end{definition}

The following lemma will be useful in the proof of Proposition \ref{prop:FFourier} and it simply follows from the dependence of $F_{(A,\ffi,c)}$ on $\ffi$.

\begin{lemma}
If $(A,\ffi,c)$ is a signed and coloured tree and if $\iota\in\{\pm\}$, then
\begin{align}\label{F^iota}
F_{\pth{A,\ffi,c}}^\iota = F_{\pth{A,\iota\ffi,c}}.
\end{align}
\end{lemma}

\begin{proof}
We prove \eqref{F^iota} directly for each tree $A$ using definitions \eqref{def Fbot}-\eqref{def FBst}. If $A=\bot$ and if we denote by $\iota'$ the single value of its sign map (and $\eta$ the single value of its colour map) we use \eqref{def Fbot} and get
\begin{align*}
F_{\pth{\bot,\iota',c}}^\iota & = \pth{X^{\eta,\iota'}_\init  }^\iota = X^{\eta,\iota\iota'}_\init = F_{\pth{\bot,\iota\iota',c}}.
\end{align*}
If $A=B_\st(A_1,A_2)$ or $A=B_\init(\bot,\bot)$, \eqref{def FBinbot} and \eqref{def FBst} imply that $F_{(A,\ffi,c)}$ is of the form
\begin{align*}
F_{(A,\ffi,c)}=Q^{\ffi(r_A)}(A,\ffi(r_A)\ffi,c),
\end{align*}
for some quantity $Q$. Therefore, using $\iota^2=1$ we obtain
\begin{align*}
F^\iota_{(A,\ffi,c)} = Q^{\iota\ffi(r_A)}(A,\ffi(r_A)\ffi,c) = Q^{\iota\ffi(r_A)}(A,\iota\ffi(r_A)\iota\ffi,c) = F_{(A,\iota\ffi,c)}.
\end{align*}
Similarly, if $A=T_*(A_1,A_2,A_3)$, \eqref{def FT} implies that $F_{(A,\ffi,\eta)}$ is of the form
\begin{align*}
F_{(A,\ffi,c)}=\ffi(r_A)iQ^{\ffi(r_A)}(A,\ffi(r_A)\ffi,c),
\end{align*}
for some quantity $Q$. Therefore, using again $\iota^2=1$ and $\pth{iz}^\iota=\iota i z^\iota$ we obtain
\begin{align*}
F^\iota_{(A,\ffi,c)} = \iota\ffi(r_A)iQ^{\iota\ffi(r_A)}(A,\ffi(r_A)\ffi,c) = \iota\ffi(r_A)iQ^{\iota\ffi(r_A)}(A,\iota\ffi(r_A)\iota\ffi,c) =  F_{(A,\iota\ffi,c)},
\end{align*}
which concludes the proof of the lemma.
\end{proof}

The following key proposition gives a representation of each element of the sequence $(X^\eta_n)_{n\in\mathbb{N}}$ from Definition \ref{def X_n} as a sum over signed and coloured trees.

\begin{prop}\label{prop:XtoF} 
For all $\eta \in \{0,1\}$, $\iota \in \{\pm\}$ and $n\in \N$, we have
\begin{align}
X^{\eta,\iota}_n = \sum_{(A,\ffi,c) \in \mathcal A_n^{\eta,\iota}} F_{(A,\ffi,c)} .  \label{XtoF}
\end{align}
\end{prop}

\begin{proof} 
We first introduce the following notation: if $A\in\tilde{\mathcal{A}}_n$, $\iota\in\{\pm\}$ and $\eta\in\{0,1\}$ we denote by $\mathrm{SM}(A,\iota)$ the set of sign maps defined on $A$ such that the sign of the root of $A$ is $\iota$. If $Q$ is now any quantity depending on a tree $A$, a sign map $\ffi$ and a colour map $c$, we have
\begin{align}\label{fact SM 2}
\sum_{(A,\ffi,c)\in\mathcal{A}_n^{\eta,\iota}} Q(A,\ffi,c) = \sum_{A\in\tilde{\mathcal{A}}_n}\sum_{\ffi\in\mathrm{SM}(A,\iota)} Q(A,\ffi,c_\eta),
\end{align}
where $c_\eta$ is the unique colour map on $A$ such that $c(r_A)=\eta$.

\saut
We now prove \eqref{XtoF} by induction over $n\in\mathbb{N}$. For $n=0$, \eqref{XtoF} is proved by using \eqref{def X0}, \eqref{def Atilde0} and \eqref{def Fbot}. For $n=1$, we start by expanding the LHS using \eqref{def X1} and \eqref{def X0}:
\begin{align*}
X^{\eta,\iota}_1(t) & = \varepsilon \sum_{\iota'_1,\iota'_2 \in \{\pm\}}\pth{B^{\eta,\iota}_{\pth{\iota_1',\iota_2'}}\pth{t,X_\init^{\overline\eta,\iota_1'},X_\init^{\overline \eta, \iota_2'}} - B^{\eta,\iota}_{\pth{\iota_1',\iota_2'}}\pth{0,X_\init^{\overline\eta,\iota_1'},X_\init^{\overline \eta, \iota_2'}}}
\end{align*}
We perform the change of variable $\iota_i=\iota \iota'_i$ for $i=1,2$ and use \eqref{def Fbot}, \eqref{def FBinbot} and \eqref{def FBst} to get
\begin{align*}
X^{\eta,\iota}_1(t) & = \varepsilon \sum_{\iota_1,\iota_2 \in \{\pm\}}\pth{B^{\eta,\iota}_{\iota\pth{\iota_1,\iota_2}}\pth{t,F_{(\bot,\iota\iota_1,\bar\eta)}(t),F_{(\bot,\iota\iota_2,\bar\eta)}(t)} - B^{\eta,\iota}_{\iota\pth{\iota_1,\iota_2}}\pth{0,X_\init^{\overline\eta,\iota\iota_1},X_\init^{\overline \eta,\iota\iota_2}}}
\\& =  \sum_{\ffi\in\mathrm{SM}(B_\st(\bot,\bot),\iota)}F_{\pth{B_\st(\bot,\bot),\ffi,c_\eta}}(t)+ \sum_{\ffi\in\mathrm{SM}(B_\init(\bot,\bot),\iota)}F_{\pth{B_\init(\bot,\bot),\ffi,c_\eta}}(t)
\\& = \sum_{(A,\ffi,c)\in\mathcal{A}^{\eta,\iota}_1}F_{(A,\ffi,c)}(t),
\end{align*}
where we used the definition \eqref{def Atilde1} of $\tilde{\mathcal{A}}_1$ and \eqref{fact SM 2}. This concludes the proof of \eqref{XtoF} in the case $n=1$. Now assume that \eqref{XtoF} holds for all $k\in\llbracket0,n+1\rrbracket$ for some $n\in\mathbb{N}$. Thanks to \eqref{def Xn+2} we have
\begin{align*}
X_{n+2}^{\eta,\iota}(t) & = - \iota i \varepsilon^2 \int_{0}^t \sum_{\substack{n_1+n_2+n_3=n\\ \iota'_1,\iota'_2,\iota'_3 \in \{\pm\}\\ * }} C^{\eta,\iota}_{*, (\iota'_1,\iota'_2,\iota'_3)} \pth{\tau,X_{n_1}^{\eta,\iota'_1}(\tau), X_{n_2}^{\overline\eta, \iota'_2 }(\tau),X_{n_3}^{\eta, \iota'_3}(\tau)}\d\tau 
\\&\quad + \varepsilon \sum_{\substack{n_1+n_2 = n+1\\ \iota'_1,\iota'_2 \in \{\pm\}}} B_{\pth{\iota'_1,\iota'_2}}^{\eta,\iota} \pth{t, X_{n_1}^{\bar{\eta},\iota'_1}(t), X_{n_2}^{\bar{\eta},\iota'_2}(t)}.
\end{align*}
Since in both sums each $n_i$ satisfies $n_i\leq n+1$ we can use our induction assumption and use \eqref{XtoF}, \eqref{fact SM 2} and the change of variable $\iota_i=\iota \iota'_i$ in both terms to obtain
\begin{align*}
X_{n+2}^{\eta,\iota}(t) & = - \hspace{-0.2cm}\sum_{\substack{n_1+n_2+n_3=n \\ A_1\in\tilde{\mathcal{A}}_{n_1} \\ A_2\in\tilde{\mathcal{A}}_{n_2}\\ A_3\in\tilde{\mathcal{A}}_{n_3} \\ * }}\iota i \varepsilon^2 \int_{0}^t \sum_{\substack{\iota_1,\iota_2,\iota_3 \in \{\pm\}\\ \ffi_1\in\mathrm{SM}(A_1,\iota\iota_1) \\ \ffi_2\in\mathrm{SM}(A_2,\iota\iota_2) \\ \ffi_3\in\mathrm{SM}(A_3,\iota\iota_3) }} C^{\eta,\iota}_{*, \iota(\iota_1,\iota_2,\iota_3)} \pth{\tau, F_{\pth{A_1,\ffi_1,c_\eta}}(t) , F_{\pth{A_2,\ffi_2,c_{\bar{\eta}}}}(t) , F_{\pth{A_3,\ffi_3,c_\eta}}(t) } \d\tau 
\\&\quad + \varepsilon \sum_{\substack{n_1+n_2 = n+1\\ A_1\in\tilde{\mathcal{A}}_{n_1} \\ A_2\in\tilde{\mathcal{A}}_{n_2} }} \sum_{\substack{\iota_1,\iota_2 \in \{\pm\}\\ \ffi_1\in\mathrm{SM}(A_1,\iota\iota_1) \\ \ffi_2\in\mathrm{SM}(A_2,\iota\iota_2)}} B_{\iota\pth{\iota_1,\iota_2}}^{\eta,\iota} \pth{t, F_{\pth{A_1,\ffi_1,c_{\bar{\eta}}}}(t) , F_{\pth{A_2,\ffi_2,c_{\bar{\eta}}}}(t) } .
\end{align*}
We now perform a change of variables in the sums over sign maps, by setting $\ffi'_i=\iota \ffi_i$ in both terms. We get
\begin{align*}
X_{n+2}^{\eta,\iota}(t) & = - \hspace{-0.3cm}\sum_{\substack{n_1+n_2+n_3=n \\ A_1\in\tilde{\mathcal{A}}_{n_1} \\ A_2\in\tilde{\mathcal{A}}_{n_2}\\ A_3\in\tilde{\mathcal{A}}_{n_3} \\ * }}\iota i \varepsilon^2 \int_{0}^t\hspace{-0.1cm} \sum_{\substack{\iota_1,\iota_2,\iota_3 \in \{\pm\}\\ \ffi'_1\in\mathrm{SM}(A_1,\iota_1) \\ \ffi'_2\in\mathrm{SM}(A_2,\iota_2) \\ \ffi'_3\in\mathrm{SM}(A_3,\iota_3) }} C^{\eta,\iota}_{*, \iota(\iota_1,\iota_2,\iota_3)} \pth{\tau, F_{\pth{A_1,\iota\ffi_1,c_\eta}}(t) , F_{\pth{A_2,\iota\ffi_2,c_{\bar{\eta}}}}(t) , F_{\pth{A_3,\iota\ffi_3,c_\eta}}(t) } \d\tau 
\\&\quad + \varepsilon \sum_{\substack{n_1+n_2 = n+1\\ A_1\in\tilde{\mathcal{A}}_{n_1} \\ A_2\in\tilde{\mathcal{A}}_{n_2} }} \sum_{\substack{\iota_1,\iota_2 \in \{\pm\}\\ \ffi'_1\in\mathrm{SM}(A_1,\iota_1) \\ \ffi'_2\in\mathrm{SM}(A_2,\iota_2)}} B_{\iota\pth{\iota_1,\iota_2}}^{\eta,\iota} \pth{t, F_{\pth{A_1,\iota\ffi_1,c_{\bar{\eta}}}}(t) , F_{\pth{A_2,\iota\ffi_2,c_{\bar{\eta}}}}(t) } .
\end{align*}
We now use the following facts: 
\begin{align}\label{fact SM bi}
&\sum_{\substack{\iota_1,\iota_2 \in \{\pm\} \\ \ffi'_1\in\mathrm{SM}(A_1,\iota_1) \\ \ffi'_2\in\mathrm{SM}(A_2,\iota_2)}} P\pth{\iota_1,\iota_2,\ffi'_1,\ffi'_2}  = \sum_{\substack{\ffi\in\mathrm{SM}\pth{ B_\st(A_1,A_2),\iota}}} P\pth{ \ffi\pth{\text{root of $A_1$}}, \ffi\pth{\text{root of $A_2$}},\ffi_{|_{A_1}},\ffi_{|_{A_2}}}
\end{align}
and 
\begin{align}
& \sum_{\substack{\iota_1,\iota_2,\iota_3 \in \{\pm\} \\ \ffi'_1\in\mathrm{SM}(A_1,\iota_1) \\ \ffi'_2\in\mathrm{SM}(A_2,\iota_2) \\ \ffi'_3\in\mathrm{SM}(A_3,\iota_3)}} P\pth{\iota_1,\iota_2,\iota_3,\ffi'_1,\ffi'_2,\ffi'_3}  \label{fact SM ter}
\\&\hspace{2cm} = \sum_{\substack{\ffi\in\mathrm{SM}\pth{ T_*(A_1,A_2,A_3),\iota}}} P\pth{ \ffi\pth{\text{root of $A_1$}}, \ffi\pth{\text{root of $A_2$}},\ffi\pth{\text{root of $A_3$}},\ffi_{|_{A_1}},\ffi_{|_{A_2}},\ffi_{|_{A_3}}},\nonumber
\end{align}
where $P$ denotes any quantity depending on the signs and the sign maps. We thus obtain, recalling also the definitions \eqref{def FT} and \eqref{def FBst} and the recursive rules defining a colour map depending on the nature of a node:
\begin{align*}
X_{n+2}^{\eta,\iota}(t) & = \sum_{\substack{n_1+n_2+n_3=n \\ A_1\in\tilde{\mathcal{A}}_{n_1} \\ A_2\in\tilde{\mathcal{A}}_{n_2}  \\ A_3\in\tilde{\mathcal{A}}_{n_3}  \\ * }} \hspace{0.4cm}\sum_{\substack{\ffi\in\mathrm{SM}\pth{ T_*(A_1,A_2,A_3),\iota}}} \hspace{-0.2cm}F_{\pth{ T_*(A_1,A_2,A_3),\ffi,c_\eta }}(t) 
\\&\quad + \sum_{\substack{ n_1+n_2 = n+1 \\ A_1\in\tilde{\mathcal{A}}_{n_1} \\ A_2\in\tilde{\mathcal{A}}_{n_2}  }}\hspace{0.4cm} \sum_{\substack{\ffi\in\mathrm{SM}\pth{ B_\st(A_1,A_2),\iota}}} \hspace{-0.2cm}F_{\pth{ B_\st(A_1,A_2),\ffi,c_\eta }}(t).
\end{align*}
Using now the definition \eqref{def Atilden+2} of $\tilde{\mathcal{A}}_{n+2}$ and \eqref{fact SM 2} we finally obtain
\begin{align*}
X^{\eta,\iota}_{n+2}(t) = \sum_{(A,\ffi,c) \in \mathcal A_{n+2}^{\eta,\iota}} F_{(A,\ffi,c)}(t),
\end{align*}
which proves that \eqref{XtoF} holds for $n+2$. This concludes the induction and the proof of the proposition.
\end{proof}

\subsection{Decoration of a tree}\label{section decoration}

In order to express the Fourier transform $\widehat{X^{\eta,\iota}_n}$ using \eqref{XtoF}, we need to associate wavenumbers to the leaves and nodes of a tree. We talk of decorations of a tree.

\begin{definition}\label{def decoration tree}
Let $A\in\tilde{\mathcal{A}}_n$ be a tree. 
\begin{itemize}
\item A \textbf{decoration map} of $A$ is a map 
\begin{align*}
\ka : \mathcal{L}(A)\sqcup \mathcal{N}(A) \longrightarrow \R^d
\end{align*}
satisfying that for any node $j$ of $A$ 
\begin{align}\label{récurrence décoration}
\ka(j) & = \sum_{j' \in \mathtt{children}(j)}\ka(j').
\end{align}
\item For $k\in\Z_L^d$, we say that a decoration $\ka$ of $A$ is a $k$-decoration of $A$ if $\ka(r_A)=k$ and $\ka\pth{\mathcal{L}(A)}\subset \Z_L^d$. We define
\begin{align*}
\mathcal{D}_k(A) \vcentcolon = \left\{ \text{$k$-decoration of $A$} \right\}.
\end{align*}
\end{itemize}
\end{definition}

Thanks to \eqref{récurrence décoration}, a decoration map $\ka$ of $A$ is uniquely determined by its restriction $\ka_{|_{\mathcal{L}(A)}}$ to the leaves of $A$. More precisely, for a given tree, the map $\ka \longmapsto \ka_{|_{\mathcal{L}(A)}}$ is a bijection.

\begin{example}
We consider the same tree as in Example \ref{premier exemple} and decorate it to illustrate Definition \ref{def decoration tree}. We remove all signs from leaves and nodes for clarity.
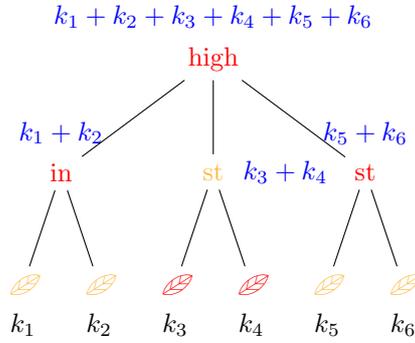
\begin{center}
\begin{tikzpicture}[level 1/.style={sibling distance=2cm},
                    level 2/.style={sibling distance=1cm}]
\node[label=above:{\textcolor{blue}{$k_1+k_2+k_3+k_4+k_5+k_6$}}]{\textcolor{red}{$\textup{high}$}}
child { node[label=above:{\textcolor{blue}{$k_1+k_2$}}]{\textcolor{red}{ \textup{in}}}
    child{node[label=below:{$k_1$}]{\textcolor{Dandelion}{\textleaf}}}
    child{node[label=below:{$k_2$}]{\textcolor{Dandelion}{\textleaf}}}
}
child { node[label=right: {\textcolor{blue}{$k_3+k_4$}}]{\textcolor{Dandelion}{$\textup{st}$}}
    child{node[label=below:{$k_3$}]{\textcolor{red}{\textleaf}}}
    child{node[label=below:{$k_4$}]{\textcolor{red}{\textleaf}}}
}
child { node[label=above: {\textcolor{blue}{$k_5+k_6$}}]{\textcolor{red}{$\textup{st}$}}
    child{node[label=below:{$k_5$}]{\textcolor{Dandelion}{\textleaf}}}
    child{node[label=below:{$k_6$}]{\textcolor{Dandelion}{\textleaf}}}
}
;
\end{tikzpicture}
\end{center}
Note that the blue wavenumbers are deduced from the black ones following the recursive rule \eqref{récurrence décoration}.
\end{example}

In the following definition, we associate to each node a double or triple $\De$. 

\begin{definition}\label{def:deltaj}
Let $A\in\tilde{\mathcal{A}}_n$ be a tree, $\ffi$ a sign map on $A$ and $\ka$ a decoration map on $A$.
\begin{itemize}
\item If $j\in\mathcal{N}_T(A)\sqcup \mathcal{N}_\st(A)$, we define 
\begin{align}
\De_j (A,\ffi,\ka) \vcentcolon = \ffi(j) \an{\ka(j)} - \sum_{j' \in \mathtt{children}(j)} \ffi(j') \an{\ka(j')}.\label{def De j}
\end{align}
\item If $j\in\mathcal{N}_\init(A)$, we define
\begin{align}
\De_j (A,\ffi,\ka) & \vcentcolon = 0,\label{def De j init}
\\ \De'_j(A,\ffi,\ka) & \vcentcolon = \ffi(j) \an{\ka(j)} - \sum_{j' \in \mathtt{children}(j)} \ffi(j') \an{\ka(j')}.\label{def De' j init}
\end{align}
\end{itemize}
\end{definition}

As it can be seen in the Definition \ref{def X_n} of the sequence $(X^\eta_n)_{n\in\mathbb{N}}$ or in \ref{def FA} of $F_A$, interactions associated to binary nodes are not directly integrated in time, while interactions associated to ternary nodes are. However, they are integrated in the sense that there might have a ternary node above them in the tree. 
We associate to any binary node the closest ternary node above them to make this clear.

\begin{definition}\label{def int}
Let $A\in\tilde{\mathcal{A}}_n$ and let $j\in\mathcal{N}(A)$.
\begin{itemize}
\item If the set
\begin{align}\label{set intermediaire qui définit int}
\enstq{j\le j' }{ j'\in \mathcal N_T(A)}
\end{align}
is empty we define $\mathtt{int}(j) = -1$.
\item Otherwise, the above set is totally ordered and we define
\begin{align*}
\mathtt{int}(j)=\min \enstq{j\le j' }{ j'\in \mathcal N_T(A) }.
\end{align*}
\end{itemize}
\end{definition}

\begin{remark} 
Definition \ref{def int} actually defines a map $\mathtt{int}:\mathcal{N}(A)\longrightarrow \mathcal{N}_T(A)\sqcup\{-1\}$ such that $\mathtt{int}_{|_{ \mathcal{N}_T(A) }}$ is the identity. For $j\in\mathcal{N}_T(A)\sqcup\{-1\}$, we will also use the notation (for instance in Section \ref{section induced orders})
\begin{align}\label{def int-1}
\mathtt{int}^{-1}(j)\vcentcolon=\enstq{j'\in\mathcal{N}_B(A)}{\mathtt{int}(j')=j}.
\end{align}
Note that this differs slightly from the usual preimage of a given $j\in\mathcal{N}_T(A)$, which is equal to $\mathtt{int}^{-1}(j)\sqcup \{j\}$, where $\mathtt{int}^{-1}(j)$ is given by \eqref{def int-1}. This is due to the fact that we have chosen to define the map $\mathtt{int}$ on $\mathcal{N}(A)$ and not just on $\mathcal{N}_B(A)$.
\end{remark}

With this notation we can define the moduli of resonance that will be integrated.

\begin{definition}\label{def:Omegaj}
Let $A\in\tilde{\mathcal{A}}_n$ be a tree, $\ffi$ a sign map on $A$ and $\ka$ a decoration map on $A$.
\begin{itemize}
\item If $j\in\mathcal{N}_T(A)$, we set
\begin{equation}\label{Omegaj}
\begin{aligned}
    \Omega_j(A,\ffi,\ka) &\vcentcolon = \Delta_j(A,\ffi,\ka)+ \sum_{\substack{j'\in\mathcal{N}_B(A) \\\text{with }\mathtt{int}(j')=j}} \Delta_{j'}(A,\ffi,\ka) 
    \\&= \sum_{\substack{j'\in\mathcal{N}(A) \\\text{with }\mathtt{int}(j')=j}} \Delta_{j'}(A,\ffi,\ka).
\end{aligned}
\end{equation}
\item We also define
\begin{equation}\label{Omega-1}
   \Omega_{-1}(A,\ffi,\ka) \vcentcolon = \sum_{\substack{j\in\mathcal{N}_B(A)\\\text{with }\mathtt{int}(j)=-1}} \Delta_{j}(A,\ffi,\ka). 
\end{equation}
\end{itemize}
\end{definition}

Finally, we also need to associate nonlinear interactions between wavenumbers to every node of a tree, depending on their nature.

\begin{definition}\label{def:qj}
Let $A\in\tilde{\mathcal{A}}_n$ be a tree, $\ffi$ a sign map on $A$, $c$ a colour map of $A$ and $\ka$ a decoration map on $A$. We associate quadratic or cubic interations between wavenumbers  to all nodes in the following way:
\begin{itemize}
    \item If $j\in\mathcal{N}_\st(A)$ we set 
        \begin{align}\label{qj Bst}
        q_j(A,\ffi,c,\ka) \vcentcolon = \ffi(j)\frac{q^{c(j)}(\ka(\ell),\ka(r))}{\Delta_j(A,\ffi,\ka)},
        \end{align}
    where $\ell$ and $r$ are respectively the left and right child of $j$.
    \item If $j\in\mathcal{N}_\init(A)$ we set
        \begin{align}\label{qj Bin}
        q_j(A,\ffi,c,\ka) \vcentcolon = -\ffi(j)\frac{q^{c(j)}(\ka(\ell),\ka(r))}{\Delta'_j(A,\ffi,\ka)} ,
        \end{align}
    where $\ell$ and $r$ are respectively the left and right child of $j$.
    \item If $j\in\mathcal{N}_*(A)$ with $*\in\{ \lowl, \lowm,\lowr,\high \}$ we set
        \begin{align}\label{qj T}
        q_j(A,\ffi,c,\ka) & \vcentcolon = \ffi(j) \chi_j(A,\ka) q_{\ffi(j)\ffi(m)}^{c(j)}(\ka(\ell),\ka(m),\ka(r)),
        \\ \chi_j (A,\ka) & \vcentcolon = \chi_*(\ka(\ell),\ka(m),\ka(r)), \label{chi j}
        \end{align}
    where $\ell$, $m$ and $r$ are respectively the left, middle and right child of $j$.
\end{itemize}
\end{definition}

Based on the assumptions we made on the nonlinearity of the original Klein-Gordon system \eqref{KG system}, we estimate each $q_j$ in the following lemma.

\begin{lemma}\label{lem qj}
Let $A\in\tilde{\mathcal{A}}_n$ be a tree, $\ffi$ a sign map on $A$, $c$ a colour map of $A$ and $\ka$ a decoration map on $A$. If $j\in\mathcal{N}_B(A)$ then
\begin{align*}
\left| q_j(A,\ffi,c,\ka) \right| & \lesssim \prod_{j'\in\children(j)}\ps{\ka(j')}^{-\half}.
\end{align*}
If $j\in\mathcal{N}_T(A)$ then
\begin{align*}
\left| q_j(A,\ffi,c,\ka) \right| & \lesssim \chi_j(A,\ka) \prod_{j'\in\children(j)}\ps{\ka(j')}^{-1}.
\end{align*}
\end{lemma}

\begin{proof}
Since $Q^\eta(k_1,k_2)$ is bounded by assumption, \eqref{q double} implies that 
\begin{align}\label{estim q double}
|q^\eta(k_1,k_2)| \lesssim  \ps{k_1}^{-1} \ps{k_2}^{-1}.
\end{align}
If $j\in\mathcal{N}_B(A)$ and $\ell$ and $r$ denote its left and right children, then \eqref{qj Bst}, \eqref{qj Bin}, \eqref{def De j} and \eqref{def De' j init} imply that
\begin{align*}
\left| q_j(A,\ffi,c,\ka) \right| & \lesssim \left| \frac{q^\eta(\ka(\ell),\ka(r))}{\De^{(\iota_1,\iota_2)}_{\ka(\ell),\ka(r)}} \right|
\end{align*}
for some $\eta$, $\iota_1$ and $\iota_2$ depending on $j$. Using now \eqref{estim q double} and \eqref{estim De} we obtain 
\begin{align*}
\left| q_j(A,\ffi,c,\ka) \right| & \lesssim \ps{\ka(\ell)}^{-\half} \ps{\ka(r)}^{-\half}.
\end{align*}
If now $j\in\mathcal{N}_T(A)$ and $\ell$, $m$ and $r$ denote its left, middle and right children, then \eqref{q triple}, \eqref{qj T} and \eqref{estim q double} imply that
\begin{align*}
\left| q_j(A,\ffi,c,\ka) \right| & \lesssim \chi_j(A,\ka) \ps{\ka(\ell)}^{-1}\ps{\ka(m)}^{-1}\ps{\ka(r)}^{-1}
\\&\quad \times \Big| m+2\pth{(\ka(\ell)+\ka(m)+\ka(r))\cdot \ka(m) - \ffi(j)\ffi(m) \ps{\ka(\ell)+\ka(m)+\ka(r)}\ps{\ka(m)}} \Big|^{-1}.
\end{align*}
Note that applying the Cauchy-Schwarz inequality in $\R^{d+1}$ to $(\sqrt{m},a)$ and $(\sqrt{m},b)$ (for some $a,b\in\R^d$) gives $a\cdot b + m \leq \ps{a}\ps{b}$. In particular, this shows that 
\begin{align*}
&\Big| m+2\pth{(\ka(\ell)+\ka(m)+\ka(r))\cdot \ka(m) - \ffi(j)\ffi(m) \ps{\ka(\ell)+\ka(m)+\ka(r)}\ps{\ka(m)}} \Big| \geq m,
\end{align*}
so that we obtain
\begin{align*}
\left| q_j(A,\ffi,c,\ka) \right| & \lesssim \chi_j(A,\ka) \ps{\ka(\ell)}^{-1}\ps{\ka(m)}^{-1}\ps{\ka(r)}^{-1}.
\end{align*}
This concludes the proof of the lemma.
\end{proof}

\begin{example}
We compute the different $q_j$s and $\De_j$s for the tree in Example \ref{premier exemple}. We arbitrarily label each node with blue integers, recall the signs of nodes and leaves but remove the wavenumbers for seek of readability.
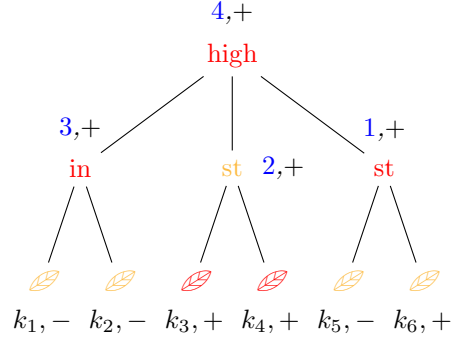
\begin{center}
\begin{tikzpicture}[level 1/.style={sibling distance=2cm},
                    level 2/.style={sibling distance=1cm}]
\node[label=above:{\textcolor{blue}{$4$},$+$}]{\textcolor{red}{$\textup{high}$}}
child { node[label=above:{\textcolor{blue}{$3$},$+$}]{\textcolor{red}{ \textup{in}}}
    child{node[label=below:{$k_1,-$}]{\textcolor{Dandelion}{\textleaf}}}
    child{node[label=below:{$k_2,-$}]{\textcolor{Dandelion}{\textleaf}}}
}
child { node[label=right: {\textcolor{blue}{$2$},$+$}]{\textcolor{Dandelion}{$\st$}}
    child{node[label=below:{$k_3,+$}]{\textcolor{red}{\textleaf}}}
    child{node[label=below:{$k_4,+$}]{\textcolor{red}{\textleaf}}}
}
child { node[label=above: {\textcolor{blue}{$1$},$+$}]{\textcolor{red}{$\st$}}
    child{node[label=below:{$k_5,-$}]{\textcolor{Dandelion}{\textleaf}}}
    child{node[label=below:{$k_6,+$}]{\textcolor{Dandelion}{\textleaf}}}
}
;
\end{tikzpicture}
\end{center}
We have 
\begin{align*}
\Delta_{\textcolor{blue}{1}} & = \an{k_5+k_6}+\an{k_5}-\an{k_6}, & \Delta_{\textcolor{blue}{2}} & = \an{k_3+k_4}-\an{k_3}-\an{k_4}  ,
\\ \Delta_{\textcolor{blue}{3}} & =  0, &  \Delta'_{\textcolor{blue}{3}} & = \an{k_1+k_2}+\an{k_1}+\an{k_2}, 
\end{align*}
\begin{align*}
\Delta_{\textcolor{blue}{4}} & = \an{k_1+k_2+k_3+k_4+k_5+k_6}-\an{k_1+k_2}-\an{k_3+k_4}-\an{k_5+k_6},
\end{align*}
and
\begin{align*}
q_{\textcolor{blue}{1}} & =  \frac{q^0(k_5,k_6)}{\Delta_{\textcolor{blue}{1}}} , & q_{\textcolor{blue}{2}} & =  \frac{q^1(k_3,k_4)}{\Delta_{\textcolor{blue}{2}}},
\\ q_{\textcolor{blue}{3}} & = -\frac{q^0(k_1,k_2)}{\Delta'_{\textcolor{blue}{3}}}, & q_{\textcolor{blue}{4}} & =  q_{+}^{0}(k_1+k_2,k_3+k_4,k_5+k_6)\chi_\high(k_1+k_2,k_3+k_4,k_5+k_6).
\end{align*}
\end{example}

\begin{example}\label{exemple Omega}
To provide an example of how to compute $\Omega_j$, we consider the following tree. We remove from the graphical representation all the unnecessary information, in particular the nature of the binary or ternary nodes, and arbitrarily label each node by blue integers.
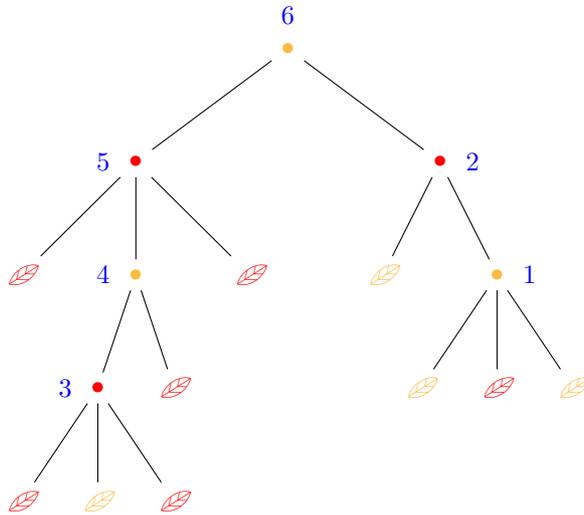
\begin{center}
\begin{tikzpicture}[level 1/.style={sibling distance=4cm},
                    level 2/.style={sibling distance=1.5cm},
                    level 3/.style={sibling distance=1cm}]
\node[label=above:{\textcolor{blue}{$6$}}]{\textcolor{Dandelion}{$\bullet$}}
    child{ node[label=left:{\textcolor{blue}{$5$}}]    {\textcolor{red}{$\bullet$}}
        child{node{\textcolor{red}{\textleaf}}}
        child{node[label=left:{\textcolor{blue}{$4$}}]{\textcolor{Dandelion}{$\bullet$}}
            child{node[label=left:{\textcolor{blue}{$3$}}]{\textcolor{red}{$\bullet$}}
                child{node{\textcolor{red}{\textleaf}}}
                child{node{\textcolor{Dandelion}{\textleaf}}}
                child{node{\textcolor{red}{\textleaf}}}
            }
            child{node{\textcolor{red}{\textleaf}}}
        }
        child{node{\textcolor{red}{\textleaf}}}
    }
    child{ node[label=right:{\textcolor{blue}{$2$}}]{\textcolor{red}{$\bullet$}}
        child{node{\textcolor{Dandelion}{\textleaf}}}
        child{ node[label=right:{\textcolor{blue}{$1$}}]{\textcolor{Dandelion}{$\bullet$}}
            child{node{\textcolor{Dandelion}{\textleaf}}}
            child{node{\textcolor{red}{\textleaf}}}
            child{node{\textcolor{Dandelion}{\textleaf}}}
        }
    }
;
\end{tikzpicture}
\end{center}
The nodes \textcolor{blue}{$5,3,1$} are ternary nodes and the nodes \textcolor{blue}{$6,4,2$} are standard binary nodes. We have 
\begin{align*}
    \mathtt{int}(\textcolor{blue}{6}) =  -1, \qquad \mathtt{int}(\textcolor{blue}{4}) =  \textcolor{blue}{5}, \qquad \mathtt{int}(\textcolor{blue}{2}) = -1 .
\end{align*}
Therefore, we have 
\begin{align*}
\Omega_{-1} & = \Delta_{\textcolor{blue}{6}}+\Delta_{\textcolor{blue}{2}} , \qquad   \Omega_{\textcolor{blue}{5}}  =   \Delta_{\textcolor{blue}{5}}+\Delta_{\textcolor{blue}{4}}, \qquad \Omega_{\textcolor{blue}{3}}  =   \Delta_{\textcolor{blue}{3}}, \qquad \Omega_{\textcolor{blue}{1}}  =   \Delta_{\textcolor{blue}{1}}.
\end{align*}

\end{example}

\subsection{\texorpdfstring{Fourier transform of $F_{(A,\ffi,c)}$}{Fourier transform of the contribution of each tree}}\label{section fourier FA}

We can now express the Fourier transform of each $F_{(A,\ffi,c)}$ defined in Definition \ref{def FA}.

\begin{prop}\label{prop:FFourier} 
For all trees $A\in\tilde{\mathcal{A}}_n$, $\ffi$ and $c$ sign and colour maps on $A$, all $t\in \R$ and all $k\in \Z_L^d$, we have
\begin{align}
\widehat{F_{(A,\ffi,c)}}(t,k) & =  \frac{\e^n}{L^{\frac{nd}{2}}}   (-i)^{ n_T(A)} \sum_{\ka\in\mathcal D_k(A)}  e^{i\Om_{-1}(A,\ffi,\ka)t}\label{fourier of F}
\\&\hspace{2cm}\times\prod_{j\in \mathcal N(A)} q_j(A,\ffi,c,\ka) \int_{I_A(t)} \prod_{j\in  \mathcal N_T(A)} e^{i \Om_j(A,\ffi,\ka) t_j} \d t_j \prod_{\ell\in \mathcal L(A)} \reallywidehat{X^{c(\ell), \ffi(\ell)}_\init} (\ka(\ell)) \non
\end{align}
where 
\begin{align*}
I_A(t) \vcentcolon= \enstq{ (t_j)_{j\in  \mathcal N_T(A)} \in [0,t]^{n_T(A)}}{ j\leq j' \Rightarrow t_{j}\leq t_{j'} }
\end{align*}
and where we use the convention that the integral over $I_A(t)$ equals 1 if $I_A(t)$ is empty.
\end{prop}

\begin{proof}
We prove \eqref{fourier of F} by induction over the size $n\in\mathbb{N}$ of $A$. If $n=0$, we only need to consider $A=\bot$ and we can identify its sign and colour maps with their single values $\iota\in\{\pm\}$ and $\eta\in\{0,1\}$. The definition \eqref{def Fbot} simply implies
\begin{align*}
\widehat{F_{(A,\iota,\eta)}}=\widehat{X^{\eta,\iota}_\init}
\end{align*}
which matches the RHS of \eqref{fourier of F}. Indeed, in this case there are no ternary nodes, $\Om_{-1}(\bot)=0$ and the only $k$-decoration is the map associating $k$ to the single leaf. 

\saut
If $n=1$ there are two cases to consider according to \eqref{def Atilde1}. We assume first that $A=B_\init(\bot,\bot)$ and denote by $\iota_r$, $\iota_1$ and $\iota_2$ respectively the signs of the root $r$, left child $\ell_1$ and right child $\ell_2$, and by $\eta$ the colour of the root. On the one hand, \eqref{def FBinbot} and \eqref{def B Fourier} imply 
\begin{align*}
\widehat{F_{(A,\ffi,c)}}(t,k)  & = - \frac{\e}{L^{\frac{d}{2}}} \sum_{\iota_rk_1+\iota_rk_2 = k}  \frac{q^\eta (k_1,k_2)^{\iota_r}}{\De_{k_1,k_2}^{\iota_r(\iota_1,\iota_2)}}  \widehat{X^{\overline{\eta}}_\init}( \iota_r\iota_1k_1)^{\iota_1} \widehat{X^{\overline{\eta}}_\init }( \iota_r\iota_2k_2)^{\iota_2}.
\end{align*}
Performing the change of variable $k'_i=\iota_rk_i$ and using
\begin{align}
q^\eta (\iota_rk'_1,\iota_rk'_2)^{\iota_r} & = q^\eta (k'_1,k'_2),\label{q(iota)}
\\ \De_{\iota_rk'_1,\iota_rk'_2}^{\iota_r(\iota_1,\iota_2)} & =\De_{k'_1,k'_2}^{\iota_r(\iota_1,\iota_2)},\label{De(iota)}
\end{align}
the above identities being consequences of \eqref{De double} and \eqref{q double}, we get
\begin{align*}
\widehat{F_{(A,\ffi,c)}}(t,k)  & =  - \frac{\e}{L^{\frac{d}{2}}} \sum_{k'_1+ k'_2 = k}  \frac{q^\eta (k'_1,k'_2)}{\De_{k'_1,k'_2}^{\iota_r(\iota_1,\iota_2)}}  \widehat{X^{\overline{\eta},\iota_1}_\init}( k'_1) \widehat{X^{\overline{\eta},\iota_2}_\init }( k'_2).
\end{align*}
On the other hand, for any quantity $P$ we have
\begin{align*}
\sum_{k'_1+k'_2=k}P(k'_1,k'_2) = \sum_{\ka\in\mathcal{D}_k(A)}P(\ka(\ell_1),\ka(\ell_2)) 
\end{align*}
where $\ell_1, \ell_2$ are the only two leaves of the tree. Together with $\De_{\ka(\ell_1),\ka(\ell_2)}^{\iota_r(\iota_1,\iota_2)} = \iota_r\De'_r$, $c(r)=\eta$, $c(\ell_1)=\overline{\eta}$, $c(\ell_2)=\overline{\eta}$ and \eqref{qj Bin}, this implies that
\begin{align*}
\widehat{F_{(A,\ffi,c)}}(t,k)  & = \frac{\e}{L^{\frac{d}{2}}} \sum_{\ka\in\mathcal{D}_k(A)} q_r(A,\ffi,c,\ka)  \widehat{X^{c(\ell_1),\ffi(\ell_1)}_\init}( \ka(\ell_1)) \widehat{X^{c(\ell_2),\ffi(\ell_2)}_\init }( \ka(\ell_2)),
\end{align*}
which matches the RHS of \eqref{fourier of F} since $\Om_{-1}(B_\init(\bot,\bot))=\De_r(B_\init(\bot,\bot))=0$ (the root is an initial binary node, see \eqref{def De j init}) and there are no ternary nodes.

\saut
If $A=B_\st(\bot,\bot)$, we use \eqref{def FBst}, \eqref{def Fbot} and \eqref{def B Fourier} and obtain (with the same notations as for the first case above):  
\begin{align*}
\widehat{F_{\pth{A,\ffi, c}}} (t,k) = \frac{\e}{L^{\frac{d}{2}}} \sum_{ k_1+  k_2 = k}  \frac{q^\eta (k_1,k_2)}{\De_{k_1,k_2}^{\iota_r(\iota_1,\iota_2)}} e^{i t \iota_r  \Delta_{k_1,k_2}^{\iota_r(\iota_1,\iota_2)}} \widehat{X^{\overline{\eta},\iota_1}_\init}(k_1) \widehat{X^{\overline{\eta},\iota_2}_\init}(k_2),
\end{align*}
where we performed the same change of variable and used \eqref{q(iota)} and \eqref{De(iota)}. The end of the proof is identical to the above case, using \eqref{qj Bst} instead of \eqref{qj Bin} and noting that $\Om_{-1}(B_\st(\bot,\bot))=\De_r(B_\st(\bot,\bot))=\iota_r  \Delta_{k_1,k_2}^{\iota_r(\iota_1,\iota_2)}$ (see \eqref{def De j}).

\saut
Assume now that \eqref{fourier of F} holds for all $k\in\llbracket0,n+1\rrbracket$ for some $n\in\mathbb{N}$ and let $A\in\tilde{\mathcal{A}}_{n+2}$. According to \eqref{def Atilden+2}, there are two cases to consider. If $A=B_\st(A_1,A_2)$ with $(A_1,A_2)\in\tilde{\mathcal{A}}_{n_1}\times \tilde{\mathcal{A}}_{n_2}$ and $n_1+n_2=n+1$, we denote by $\iota_r$, $\iota_1$ and $\iota_2$ the signs of the root of $A$, of the root of $A_1$ and of the root of $A_2$ in $A$, and also by $\eta$ the colour of the root of $A$. Then \eqref{def FBst}, \eqref{def B Fourier}, the change of variable $k'_i=\iota_r k_i$ (with $k'_i$ being immediately relabeled $k_i$) and \eqref{q(iota)}-\eqref{De(iota)} imply
\begin{align*}
\widehat{F_{\pth{A,\ffi, c}}} (t,k) & =  \frac{\e}{L^{\frac{d}{2}}} \sum_{k_1 + k_2 = k}  \frac{q^\eta (k_1,k_2)}{\De_{k_1,k_2}^{\iota_r(\iota_1,\iota_2)}} e^{it\iota_r \Delta_{k_1,k_2}^{\iota_r(\iota_1,\iota_2)}} \reallywidehat{ F^{\iota_r}_{\pth{A_1,\iota_r \ffi_{|_{A_1}},c_{|_{A_1}}}} }(t, k_1) \reallywidehat{ F^{\iota_r}_{\pth{A_2,\iota_r \ffi_{|_{A_2}},c_{|_{A_2}}}} }( t,k_2)
\\& =  \frac{\e}{L^{\frac{d}{2}}} \sum_{k_1 + k_2 = k}  \frac{q^\eta (k_1,k_2)}{\De_{k_1,k_2}^{\iota_r(\iota_1,\iota_2)}} e^{it\iota_r \Delta_{k_1,k_2}^{\iota_r(\iota_1,\iota_2)}} \reallywidehat{ F_{\pth{A_1, \ffi_{|_{A_1}},c_{|_{A_1}}}} }(t, k_1) \reallywidehat{ F_{\pth{A_2, \ffi_{|_{A_2}},c_{|_{A_2}}}} }( t,k_2) 
\end{align*}
where we used the property \eqref{F^iota}. Since $n_1,n_2\leq n+1$, we can use our induction assumption \eqref{fourier of F} on $ \reallywidehat{ F_{\pth{A_1, \ffi_{|_{A_1}},c_{|_{A_1}}}} }$ and $ \reallywidehat{ F_{\pth{A_2, \ffi_{|_{A_2}},c_{|_{A_2}}}} }$. This gives
\begin{align*}
&\widehat{F_{\pth{A,\ffi, c}}} (t,k) 
\\& =  \pth{\frac{\e}{L^{\frac{d}{2}}}}^{1+n_1+n_2} (-i)^{ n_T(A_1)+n_T(A_2)} \sum_{\substack{k_1 + k_2 = k \\ \ka_1\in\mathcal{D}_{k_1}(A_1) \\ \ka_2\in\mathcal{D}_{k_2}(A_2) }}  \frac{q^\eta (k_1,k_2)}{\De_{k_1,k_2}^{\iota_r(\iota_1,\iota_2)}} e^{it\pth{\iota_r \Delta_{k_1,k_2}^{\iota_r(\iota_1,\iota_2)} + \Om_{-1}(A_1,\ka_1) + \Om_{-1}(A_2,\ka_2) }}
\\&\quad\times \prod_{\substack{j_1\in \mathcal N(A_1) \\ j_2\in \mathcal N(A_2)}}q_{j_1}(A_1,\ka_1)q_{j_2}(A_2,\ka_2) \int_{I_{A_1}(t) \times I_{A_2}(t)}\prod_{\substack{ j_1\in  \mathcal N_T(A_1) \\ j_2\in  \mathcal N_T(A_2) }} e^{i \Om_{j_1}(A_1,\ka_1) t_{j_1}}e^{i \Om_{j_2}(A_2,\ka_2) t_{j_2}}  \d t_{j_1} \d t_{j_2}
\\&\quad \times   \prod_{\substack{\ell_1\in \mathcal L(A_1)\\ \ell_2 \in \mathcal L(A_2)}} \reallywidehat{X^{c(\ell_1), \ffi(\ell_1)}_\init} (\ka_1(\ell_1))\reallywidehat{X^{c(\ell_2), \ffi(\ell_2)}_\init} (\ka_2(\ell_2)),
\end{align*}
where in the various $\Om_{-1}$, $\Om_j$ and $q_j$ appearing in the formula, we only kept their dependence on the decoration maps, since the sign and colour maps are simply the restriction of $\ffi$ and $c$ to either $A_1$ or $A_2$. Now, if we denote by $r$ the root of $A$ and by $r_1,r_2$ the roots of $A_1$ and $A_2$ in $A$, a decoration $\ka$ on $A$ satisfies $\ka(r)=\ka(r_1)+\ka(r_2)$ and $\ka_{|_{A_1}}$ and $\ka_{|_{A_2}}$ are decorations of $A_1$ and $A_2$. Conversely, if $(\ka_1,\ka_2)\in \mathcal{D}_{k_1}(A_1)\times \mathcal{D}_{k_2}(A_2)$ then there exists a unique $\ka\in \mathcal{D}_{k_1+k_2}(A)$ such that $\ka_{|_{A_1}}=\ka_1$ and $\ka_{|_{A_2}}=\ka_2$. This shows that
\begin{align*}
\sum_{\substack{k_1 + k_2 = k \\ \ka_1\in\mathcal{D}_{k_1}(A_1) \\ \ka_2\in\mathcal{D}_{k_2}(A_2) }} Q(k_1,k_2,\ka_1,\ka_2) = \sum_{\ka\in \mathcal{D}_k(A)} Q\pth{\ka(r_1),\ka(r_2),\ka_{|_{A_1}},\ka_{|_{A_2}}}
\end{align*}
which in turn implies, since $\mathcal{N}(A)=\mathcal{N}(A_1)\sqcup\mathcal{N}(A_2)\sqcup\{r\}$ and $\mathcal{L}(A)=\mathcal{L}(A_1)\sqcup\mathcal{L}(A_2)$, that
\begin{align*}
\widehat{F_{\pth{A,\ffi, c}}} (t,k) 
 = & \pth{\frac{\e}{L^{\frac{d}{2}}}}^{1+n_1+n_2} (-i)^{ n_T(A_1)+n_T(A_2)}\\
 &\times \sum_{\ka\in \mathcal{D}_k(A)}  \frac{q^\eta (\ka(r_1),\ka(r_2))}{\De_{\ka(r_1),\ka(r_2)}^{\iota_r(\iota_1,\iota_2)}} e^{it\pth{\iota_r \Delta_{\ka(r_1),\ka(r_2)}^{\iota_r(\iota_1,\iota_2)} + \Om_{-1}\pth{A_1,\ka_{|_{A_1}}} + \Om_{-1}\pth{A_2,\ka_{|_{A_2}}} }}
\\&\times \prod_{j\in\mathcal{N}(A)\setminus\{r\}}q_{j}(A,\ka)\int_{I_{A_1}(t) \times I_{A_2}(t)}\prod_{\substack{ j_1\in  \mathcal N_T(A_1) \\ j_2\in  \mathcal N_T(A_2) }} e^{i \Om_{j_1}(A_1,\ka_{|_{A_1}}) t_{j_1}}e^{i \Om_{j_2}(A_2,\ka_{|_{A_2}}) t_{j_2}}  \d t_{j_1} \d t_{j_2} 
\\&\times \prod_{\ell\in\mathcal{L}(A)} \reallywidehat{X^{c(\ell), \ffi(\ell)}_\init} (\ka(\ell)).
\end{align*}
Thanks to \eqref{qj Bst} and \eqref{def De j} we have $\frac{q^\eta (\ka(r_1),\ka(r_2))}{\De_{\ka(r_1),\ka(r_2)}^{\iota_r(\iota_1,\iota_2)}}=q_r(A,\ka)$ and $\iota_r \Delta_{\ka(r_1),\ka(r_2)}^{\iota_r(\iota_1,\iota_2)}=\De_r(A,\ka)$. Moreover, since the root of $A$ is a binary node we have $\mathcal{N}_T(A)=\mathcal{N}_T(A_1)\sqcup\mathcal{N}_T(A_2)$ and $I_{A_1}(t) \times I_{A_2}(t)=I_A(t)$. Thus we obtain
\begin{align*}
\widehat{F_{\pth{A,\ffi, c}}} (t,k) & =  \pth{\frac{\e}{L^{\frac{d}{2}}}}^{1+n_1+n_2} (-i)^{ n_T(A)} \sum_{\ka\in \mathcal{D}_k(A)}  e^{it\pth{\De_r(A,\ka) + \Om_{-1}\pth{A_1,\ka_{|_{A_1}}} + \Om_{-1}\pth{A_2,\ka_{|_{A_2}}} }}
\\&\quad\times \prod_{j\in\mathcal{N}(A)}q_{j}(A,\ka)\int_{I_{A}(t) }\prod_{j\in  \mathcal{N}_T(A)} e^{i \Om_{j}(A,\ka) t_{j}} \d t_{j}   \prod_{\ell\in\mathcal{L}(A)} \reallywidehat{X^{c(\ell), \ffi(\ell)}_\init} (\ka(\ell)).
\end{align*}
Using then $n_1+n_2=n+1$ and $\Om_{-1}(A,\ka) = \De_r(A,\ka) + \Om_{-1}\pth{A_1,\ka_{|_{A_1}}} + \Om_{-1}\pth{A_2,\ka_{|_{A_2}}}$ concludes the proof of \eqref{fourier of F} for $A$.

\saut
Similarly, if $A=T_*(A_1,A_2,A_3)$ with $(A_1,A_2,A_3)\in\tilde{\mathcal{A}}_{n_1}\times \tilde{\mathcal{A}}_{n_2}\times\tilde{\mathcal{A}}_{n_3}$ and $n_1+n_2+n_3=n$, we use \eqref{def FT}, \eqref{def C Fourier}, \eqref{F^iota} and our induction assumption applied to $A_1$, $A_2$ and $A_3$ to get
\begin{align*}
&\widehat{F_{\pth{A,\ffi, c}}} (t,k) 
\\& =   \pth{\frac{\e}{L^{\frac{d}{2}}}}^{n_1+n_2+n_3+2} (-i)^{ n_T(A_1)+n_T(A_2)+n_T(A_3)+1}  \sum_{\substack{k_1+k_2+k_3 = k \\\ka_1\in\mathcal D_{k_1}(A_1)\\\ka_2\in\mathcal D_{k_2}(A_2) \\\ka_3\in\mathcal D_{k_3}(A_3) }} \iota_r\chi_*(k_1,k_2,k_3) q_{\iota_r\iota_2}^\eta (k_1,k_2,k_3) 
 \\&\quad \times\prod_{\substack{j_1\in \mathcal N(A_1)\\ j_2\in \mathcal N(A_2) \\ j_3\in \mathcal N(A_3) }} q_{j_1}(A_1,\ka_1)q_{j_2}(A_2,\ka_2)q_{j_3}(A_3,\ka_3)
\\&\quad \times \int_0^te^{i\tau\pth{\iota_r \Delta_{k_1,k_2,k_3}^{\iota_r(\iota_1,\iota_2,\iota_3)}+\Om_{-1}(A_1,\ka_1)+\Om_{-1}(A_2,\ka_2)+\Om_{-1}(A_3,\ka_3)}}
\\&\hspace{1cm}\times\pth{\int_{I_{A_1}(\tau)\times I_{A_2}(\tau)\times I_{A_3}(\tau)} \prod_{\substack{j_1\in  \mathcal N_T(A_1)\\ j_2\in  \mathcal N_T(A_2)\\ j_3\in  \mathcal N_T(A_3)}} e^{i \Om_{j_1}(A_1,\ka_1) t_{j_1}} e^{i \Om_{j_2}(A_2,\ka_2) t_{j_2}} e^{i \Om_{j_3}(A_3,\ka_3) t_{j_3}} \d t_{j_1} \d t_{j_2} \d t_{j_3}} \d\tau
\\&\quad \times \prod_{\substack{\ell_1\in \mathcal L(A_1) \\ \ell_2\in \mathcal L(A_2) \\ \ell_3\in \mathcal L(A_3)}} \widehat{X^{c(\ell_1), \ffi(\ell_1)}_\init} (\ka_1(\ell_1))\widehat{X^{c(\ell_2), \ffi(\ell_2)}_\init} (\ka_2(\ell_2))\widehat{X^{c(\ell_3), \ffi(\ell_3)}_\init} (\ka_3(\ell_3)),
\end{align*}
where we also used 
\begin{align*}
\chi_*(\iota_rk_1,\iota_rk_2,\iota_rk_3) & = \chi_*(k_1,k_2,k_3),
\\ q_{\iota_r\iota_2}^\eta (\iota_rk_1,\iota_rk_2,\iota_rk_3) & = q_{\iota_r\iota_2}^\eta (k_1,k_2,k_3),
\\  \Delta_{\iota_rk_1,\iota_rk_2,\iota_rk_3}^{\iota_r(\iota_1,\iota_2,\iota_3)} & =  \Delta_{k_1,k_2,k_3}^{\iota_r(\iota_1,\iota_2,\iota_3)},
\end{align*}
which follow from Definition \ref{def cutoff}, \eqref{De triple} and \eqref{q triple}. Now, as in the case of $B_\st(A_1,A_2)$ above, we have
\begin{align*}
\sum_{\substack{k_1 + k_2 + k_3 = k \\ \ka_1\in\mathcal{D}_{k_1}(A_1) \\ \ka_2\in\mathcal{D}_{k_2}(A_2) \\ \ka_3\in\mathcal{D}_{k_3}(A_3) }} Q(k_1,k_2,k_3,\ka_1,\ka_2,\ka_3) = \sum_{\ka\in \mathcal{D}_k(A)} Q\pth{\ka(r_1),\ka(r_2),\ka(r_3),\ka_{|_{A_1}},\ka_{|_{A_2}},\ka_{|_{A_3}}}
\end{align*}
where $r_1$, $r_2$ and $r_3$ are the roots of $A_1$, $A_2$ and $A_3$ in $A$. This implies, also using 
\begin{align*}
n_1+n_2+n_3+2 & = n+2,
\\ n_T(A_1)+n_T(A_2)+n_T(A_3)+1 & = n_T(A),
\\ \mathcal{N}(A) & =\mathcal{N}(A_1)\sqcup\mathcal{N}(A_2)\sqcup\mathcal{N}(A_3)\sqcup\{r\},
\\ \mathcal{L}(A) & =\mathcal{L}(A_1)\sqcup\mathcal{L}(A_2)\sqcup\mathcal{L}(A_3),
\end{align*}
that
\begin{align*}
&\widehat{F_{\pth{A,\ffi, c}}} (t,k) 
\\& =   \pth{\frac{\e}{L^{\frac{d}{2}}}}^{n+2} (-i)^{ n_T(A)}  \sum_{\ka\in \mathcal{D}_k(A)} \prod_{j\in \mathcal{N}(A)} q_{j}(A,\ka)
\\&\quad \times \int_0^te^{i\Om_r(A,\ka)\tau}
\\&\quad \times\pth{\hspace{-0.2cm}\int_{I_{A_1}(\tau)\times I_{A_2}(\tau)\times I_{A_3}(\tau)} \prod_{\substack{j_1\in  \mathcal N_T(A_1)\\ j_2\in  \mathcal N_T(A_2)\\ j_3\in  \mathcal N_T(A_3)}} e^{i \Om_{j_1}\pth{A_1,\ka_{|_{A_1}}} t_{j_1}} e^{i \Om_{j_2}\pth{A_2,\ka_{|_{A_1}}} t_{j_2}} e^{i \Om_{j_3}\pth{A_3,\ka_{|_{A_1}}} t_{j_3}} \d t_{j_1} \d t_{j_2} \d t_{j_3}\hspace{-0.2cm}} \d\tau
\\&\quad \times \prod_{\ell\in \mathcal L(A) } \widehat{X^{c(\ell), \ffi(\ell)}_\init} (\ka(\ell)),
\end{align*}
Note that we used the following facts 
\begin{align*}
\iota_r \Delta_{\ka(r_1),\ka(r_2),\ka(r_3)}^{\iota_r(\iota_1,\iota_2,\iota_3)}+\Om_{-1}\pth{A_1,\ka_{|_{A_1}}}+\Om_{-1}\pth{A_2,\ka_{|_{A_2}}}+\Om_{-1}\pth{A_3,\ka_{|_{A_3}}} & = \Om_r(A,\ka),
\\ \iota_r\chi_*(\ka(r_1),\ka(r_2),\ka(r_3)) q_{\iota_r\iota_2}^\eta (\ka(r_1),\ka(r_2),\ka(r_3)) & = q_r(A,\ka),  
\end{align*}
which follows from \eqref{def De j}, \eqref{Omegaj} and \eqref{qj T}. Now, since $\mathcal{N}_T(A) =\mathcal{N}_T(A_1)\sqcup\mathcal{N}_T(A_2)\sqcup\mathcal{N}_T(A_3)\sqcup\{r\}$ and $j<r$ for all $j\in \mathcal{N}_T(A_1)\sqcup\mathcal{N}_T(A_2)\sqcup\mathcal{N}_T(A_3)$ we have
\begin{align*}
I_A(t) = \enstq{(t_j)_{j\in\mathcal{N}_T(A)}\in[0,t]^{n_T(A)}}{t_r\in[0,t], \; \forall i=1,2,3, (t_j)_{j\in\mathcal{N}_T(A_i)}\in I_{A_i}(t_r)}.
\end{align*}
This implies, after having identified the variable $\tau$ with $t_r$, that
\begin{align*}
\widehat{F_{\pth{A,\ffi, c}}} (t,k) & =   \pth{\frac{\e}{L^{\frac{d}{2}}}}^{n+2} (-i)^{ n_T(A)}  \sum_{\ka\in \mathcal{D}_k(A)} \prod_{j\in \mathcal{N}(A)} q_{j}(A,\ka) 
\\&\hspace{3cm}\times\int_{I_A(t)}\prod_{j\in\mathcal{N}_T(A)} e^{i\Om_j(A,\ka)t_j}\d t_j  \prod_{\ell\in \mathcal L(A) } \widehat{X^{c(\ell), \ffi(\ell)}_\init} (\ka(\ell)).
\end{align*}
Since $\Om_{-1}(A,\ka)=0$ (because the root is a ternary node), this shows that \eqref{fourier of F} holds for any $A\in\tilde{\mathcal{A}}_{n+2}$, thus concluding our induction and the proof of the proposition.
\end{proof}

\begin{example}
On the tree $A_1$ of Example \ref{exemple Omega}, we have 
\begin{align*}
I_{A_1}(t) = \{(t_j)_{j\in \{\textcolor{blue}{1,3,5}\}} \; |\; t_{\textcolor{blue}{1}} \in [0,t] \, \wedge \, 0 \leq t_{\textcolor{blue}{3}} \leq t_{\textcolor{blue}{5}} \leq t \} .
\end{align*}
\end{example}

\begin{remark}\label{remark prop Fourier F}
We make two important comments on Proposition \ref{prop:FFourier}:
\begin{itemize}
\item[(i)] Thanks to our assumption \eqref{assumption Q}, if $n\geq 1$ then for any $\ka\in\mathcal{D}_0(A)$, $q_r(A,\ffi,c,\ka)=0$ and
\begin{align}\label{F(0)=0}
\widehat{F_{(A,\ffi,c)}}(t,0) = 0.
\end{align}
\item[(ii)]After a change of variables, we get
\begin{align}
&\widehat{F_{(A,\ffi,c)}}(\e^{-2}t,k) \label{fourier F e-2}
\\& \; = \frac{\e^{n_B(A)}}{L^{\frac{nd}{2}}}   (-i)^{ n_T(A)} \sum_{\ka\in\mathcal D_k(A)}  e^{i\e^{-2}\Om_{-1}t}\prod_{j\in \mathcal N(A)} q_j \int_{I_A(t)} \prod_{j\in  \mathcal N_T(A)} e^{i \Om_j \e^{-2} t_j} \d t_j \prod_{\ell\in \mathcal L(A)} \widehat{X^{c(\ell), \ffi(\ell)}_\init} (\ka(\ell)),\non
\end{align}
where we completely dropped the dependence on $A$, $\ffi$, $c$ and $\ka$ of the $\Om_{-1}$, $q_j$ and $\Om_j$ for brevity, and where we used $n=n_B(A)+2n_T(A)$.
\end{itemize}
\end{remark}

\subsection{Coupling between trees}\label{section coupling between trees}

In order to study the evolution of the correlations, we need to estimate quantities of the form
\begin{align*}
\E\pth{ \widehat{F_{(A_1,\ffi_1,c_1)}}(\e^{-2}t,k_1) \widehat{F_{(A_2,\ffi_2,c_2)}}(\e^{-2}t,k_2) }  . 
\end{align*}
According to \eqref{fourier of F}, this is will involve the following expectation
\begin{align*}
 \E\pth{ \prod_{\substack{\ell_1\in \mathcal L(A_1) \\ \ell_2\in \mathcal L(A_2) }} \widehat{X^{c_1(\ell), \ffi_1(\ell)}_\init} (\ka_1(\ell_1))  \widehat{X^{c_2(\ell), \ffi_2(\ell)}_\init} (\ka_2(\ell_2)) }.
\end{align*}
Since $\widehat{X^\eta_\init}(k)=\mu^\eta_k$, 
Wick's formula implies that the above expectation vanishes unless the leaves in $\mathcal{L}(A_1)\cup \mathcal{L}(A_2)$ are paired in a particular fashion.
To better explain this phenomenon and its consequences, we introduce below the notion of coupling.

\begin{definition}[Coupling]\label{def:coupling}
Let $(A_1,\ffi_1,c_1)$ and $(A_2,\ffi_2,c_2)$ be two signed and coloured trees and consider a map $\si:\mathcal{L}(A_1)\sqcup\mathcal{L}(A_2)\longrightarrow \mathcal{L}(A_1)\sqcup\mathcal{L}(A_2)$. 
\begin{itemize}
\item We say that the triplet $C=\big( (A_1,\ffi_1,c_1),(A_2,\ffi_2,c_2),\si \big)$ is a coupling if the following two conditions hold:
\begin{itemize}
\item[(i)] The cardinal of 
\begin{align}\label{def LC-}
\mathcal{L}(C)_- \vcentcolon = \enstq{\ell \in \mathcal{L}(A_1) }{ \ffi_1(\ell)=- }\sqcup \enstq{\ell \in \mathcal{L}(A_2) }{ \ffi_2(\ell)=- },
\end{align}
is equal to the cardinal of 
\begin{align}\label{def LC+}
\mathcal{L}(C)_+ \vcentcolon = \enstq{\ell \in \mathcal{L}(A_1) }{ \ffi_1(\ell)=+ }\sqcup \enstq{\ell \in \mathcal{L}(A_2) }{ \ffi_2(\ell)=+ }.
\end{align}
\item[(ii)] The map $\si$ is an involution and a bijection from $\mathcal{L}(C)_-$ to $\mathcal{L}(C)_+$.
\end{itemize}
\item For $n_1,n_2\in\mathbb{N}$, $\eta_1,\eta_2\in\{0,1\}$ and $\iota_1,\iota_2\in\{\pm\}$ we define
\begin{align*}
    \mathcal{C}^{\eta_1,\eta_2,\iota_1,\iota_2}_{n_1,n_2} & \vcentcolon = \enstq{\big( (A_1,\ffi_1,c_1),(A_2,\ffi_2,c_2),\si \big) \text{ coupling}}{(A_1,\ffi_1,c_1)\in\mathcal{A}^{\eta_1,\iota_1}_{n_1},\; (A_2,\ffi_2,c_2)\in\mathcal{A}^{\eta_2,\iota_2}_{n_2}  },
    \\ \mathcal{C}_{n_1,n_2} & \vcentcolon = \bigsqcup_{\substack{ \eta_1,\eta_2\in\{0,1\} \\ \iota_1,\iota_2\in\{\pm\} }} \mathcal{C}^{\eta_1,\eta_2,\iota_1,\iota_2}_{n_1,n_2}.
\end{align*}
\item If $C=\big( (A_1,\ffi_1,c_1),(A_2,\ffi_2,c_2),\si \big)$ is a coupling, then we define 
\begin{align*}
\mathcal{N}_X(C) & \vcentcolon = \mathcal{N}_X(A_1) \sqcup \mathcal{N}_X(A_2), 
\\ \mathcal{L}(C) & \vcentcolon = \mathcal{L}(A_1) \sqcup \mathcal{L}(A_2),
\\ \mathcal{N}(C) & \vcentcolon = \mathcal{N}(A_1) \sqcup \mathcal{N}(A_2),
\\ \mathcal{R}(C) & \vcentcolon = \mathcal{R}(A_1)\sqcup \mathcal{R}(A_2)
\end{align*}
for $X \in \{\init,\st,B, \high,\lowl,\lowm,\lowr, \low,T\}$ and $ n_X(C)$, $n_\ell(C)$ the respective cardinals of $\mathcal{N}_X(C)$, $\mathcal{L}(C) $. Moreover, if $C\in \mathcal{C}_{n_1,n_2}$ then we define
\begin{align*}
n(C)\vcentcolon = n_1+n_2.
\end{align*}
\end{itemize}
\end{definition}

\begin{example}\label{ex:coupling}
Set $A=B_\st(T_\high(\bot,\bot,\bot),\bot)$ with the colour of the root $0$ and the following sign map 
\begin{align*}
\ffi(r) & = -, \quad \ffi(rg)  = -, \quad \ffi(rd)  = +,\quad \ffi(rgg)  = -, \quad \ffi(rgm)  = +, \quad \ffi(rgd)  = -,
\end{align*}
where $r$ is the root, $rg$ is the left child of the root, $rd$ is the right child of the root, $rgg$ is the left child of $rg$, $rgm$ is the middle child of $rg$ and $rgd$ is the right child of $rg$. Set $A' = B_\st(\bot,\bot)$ with the colour of the root $1$ and the following sign map
\begin{align*}
\ffi'(r') = +, \quad \ffi'(r'g') = -, \quad \ffi'(r'd') = +,
\end{align*}
where $r'$ is the root of $A'$, $r'g'$ is its left child and $r'd'$ its right child. Define $\sigma$ by
\begin{align*}
\sigma(rgg)= rd , \quad \sigma(rgd)= r'd' , \quad \sigma(rg')= rgm . 
\end{align*}
The triplet $C=\big((A,\ffi,c_0),(A',\ffi',c_1),\sigma\big)$ forms a coupling, which we represent graphically below. The minus leaves are associated to the plus leaves by means of the same symbol. The signs of the branching nodes are omitted for readibility.

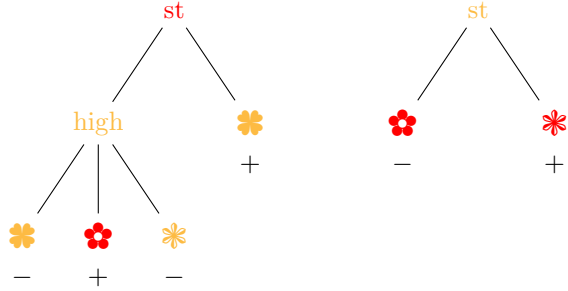
\begin{center} 
\begin{tikzpicture}[level 1/.style={sibling distance=2cm},
                    level 2/.style={sibling distance=1cm}]
\node{\textcolor{red}{$\st$}}
    child{node{\textcolor{Dandelion}{$\textup{high}$}}
        child{node[label=below:{$-$}]{\textcolor{Dandelion}{\FourClowerSolid}}}
        child{node[label=below:{$+$}]{\textcolor{red}{\FiveFlowerPetal}}}
        child{node[label=below:{$-$}]{\textcolor{Dandelion}{\SixFlowerAltPetal}}}
    }
    child{node[label=below:{$+$}]{\textcolor{Dandelion}{\FourClowerSolid}}}
;
\node at (4,0) {\textcolor{Dandelion}{$\st$}}
    child{node[label=below:{$-$}]{\textcolor{red}{\FiveFlowerPetal}}}
    child{node[label=below:{$+$}]{\textcolor{red}{\SixFlowerAltPetal}}}
;
\end{tikzpicture}
\end{center}
\end{example}

We introduce the following definition of decoration of a coupling, i.e a decoration of two trees compatible with the coupling map $\si$ and Wick's formula.

\begin{definition}[Decoration of a coupling]\label{def:decocoup} 
Let $C=\big( (A_1,\ffi_1,c_1),(A_2,\ffi_2,c_2),\si \big)$ be a coupling.
\begin{itemize}
\item A decoration map of $C$ is a map
\begin{align*}
\ka : \mathcal{L}(C)\sqcup \mathcal{N}(C) \longrightarrow \R^d
\end{align*}
such that $\ka_{|_{\mathcal{L}(A_1)\sqcup \mathcal{N}(A_1)}}$ and $\ka_{|_{\mathcal{L}(A_2)\sqcup \mathcal{N}(A_2)}}$ are decoration maps on $A_1$ and $A_2$ in the sense of Definition \ref{def decoration tree} and such that
\begin{align}\label{decoration coupling rule}
\forall \ell\in\mathcal{L}(C), \; \ka(\ell) + \ka(\si(\ell)) = 0.
\end{align} 
\item For $k\in\Z_L^d$, we say that a decoration $\ka$ of $C$ is a $k$-decoration of $C$ if $\ka(r_{A_1})=k$ (where $r_{A_1}$ is the root of $A_1$) and $\ka\pth{\mathcal{L}(C)}\subset \Z_L^d$. We define
\begin{align*}
\mathcal{D}_k(C) \vcentcolon = \left\{\text{$k$-decoration of $C$}\right\}.
\end{align*} 
\item If $\ka$ is a decoration map of $C$ and if we denote by $\ka_i = \ka_{|_{\mathcal{L}(A_i)\sqcup \mathcal{N}(A_i)}}$ for $i=1,2$, then we define
\begin{align*}
q_j & \vcentcolon = q_j(A_i,\ffi,c_i,\ka_i), \qquad \text{for $j\in\mathcal{N}(A_i)$},
\\ \Om_j & \vcentcolon = \Om_j(A_i,\ffi,\ka_i), \qquad \text{for $j\in\mathcal{N}_T(A_i)$},
\\ \Om_{-1}(C) & \vcentcolon = \Om_{-1}(A_1,\ffi_1,\ka_1)+\Om_{-1}(A_2,\ffi_2,\ka_2).
\end{align*}
\end{itemize}
\end{definition}

\begin{remark}\label{remark post def decocoup}
We make two comments to illustrate the previous definition.
\begin{itemize}
\item[(i)] A decoration map $\ka$ of a coupling $C=\big( (A_1,\ffi_1,c_1),(A_2,\ffi_2,c_2),\si \big)$ is uniquely determined by its restriction $\ka_{|_{\mathcal{L}(C)_-}}$ to the negative (or the positive) leaves of $C$. Indeed, if the image of the negative leaves is given, \eqref{decoration coupling rule} and the fact that $\si:\mathcal{L}(C)_-\longrightarrow\mathcal{L}(C)_+$ is a bijection allow us to define uniquely $\ka$ on $\mathcal{L}(A_1)\sqcup\mathcal{L}(A_2)$. Since $\ka_{|_{\mathcal{L}(A_1)\sqcup \mathcal{N}(A_1)}}$ and $\ka_{|_{\mathcal{L}(A_2)\sqcup \mathcal{N}(A_2)}}$ are asked to be decoration maps of $A_1$ and $A_2$, \eqref{récurrence décoration} then completely defines $\ka$.
\item[(ii)] If $\ka$ is a decoration map of a coupling $C=\big( (A_1,\ffi_1,c_1),(A_2,\ffi_2,c_2),\si \big)$, then $\ka(r_{A_1})+\ka(r_{A_2})=0$, where $r_{A_1}$ and $r_{A_2}$ are the roots of $A_1$ and $A_2$. In particular, if $\ka$ is a $k$-decoration of $C$ then $\ka(r_{A_2})=-k$.
\end{itemize}
\end{remark}

\begin{example}
Here is an example of decoration map for the coupling $C$ of Example \ref{ex:coupling}: 
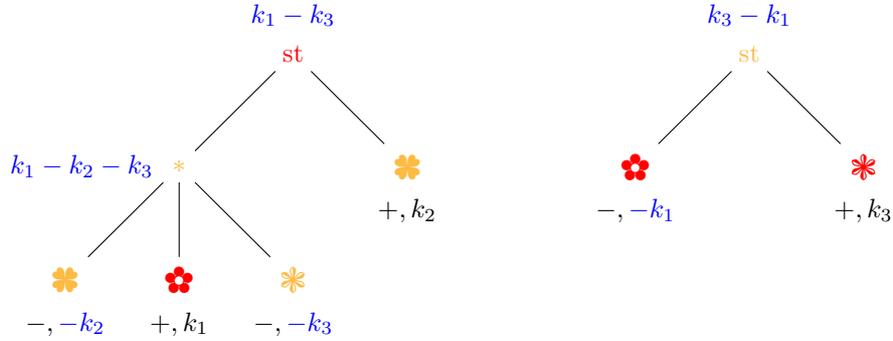
\begin{center} 
\begin{tikzpicture}[level 1/.style={sibling distance=3cm},
                    level 2/.style={sibling distance=1.5cm}]
\node[label=above:{\textcolor{blue}{$k_1-k_3$}}]{\textcolor{red}{$\st$}}
    child{node[label=left:{\textcolor{blue}{$k_1-k_2-k_3$}}]{\textcolor{Dandelion}{$*$}}
        child{node[label=below:{$-,\textcolor{blue}{-k_2}$}]{\textcolor{Dandelion}{\FourClowerSolid}}}
        child{node[label=below:{$+,k_1$}]{\textcolor{red}{\FiveFlowerPetal}}}
        child{node[label=below:{$-,\textcolor{blue}{-k_3}$}]{\textcolor{Dandelion}{\SixFlowerAltPetal}}}
    }
    child{node[label=below:{$+,k_2$}]{\textcolor{Dandelion}{\FourClowerSolid}}}
;
\node at (6,0) [label=above:{\textcolor{blue}{$k_3-k_1$}}]{\textcolor{Dandelion}{$\st$}}
    child{node[label=below:{$-,\textcolor{blue}{-k_1}$}]{\textcolor{red}{\FiveFlowerPetal}}}
    child{node[label=below:{$+,k_3$}]{\textcolor{red}{\SixFlowerAltPetal}}}
;
\end{tikzpicture}
\end{center}
\end{example}


The following proposition makes fully explicit the role played by couplings in the computation of correlations, as introduced at the beginning of this section.

\begin{prop}\label{prop:correlation}
For all $\eta_1,\eta_2\in\{0,1\}$, $\iota_1,\iota_2\in\{\pm\}$, $n_1,n_2\in\mathbb{N}$ and $k\in\Z_L^d$ we have
\begin{align}\label{correlations to coupling}
\E\pth{ \widehat{X^{\eta_1,\iota_1}_{n_1}}(\e^{-2}t,k)\widehat{X^{\eta_2,\iota_2}_{n_2}}(\e^{-2}t,-k)  }& = \sum_{C \in \mathcal C_{n_1,n_2}^{\eta_1,\eta_2,\iota_1,\iota_2}} \widehat{F_C}(\e^{-2}t,k),
\end{align}
where $F_C$ is defined in Fourier space by
\begin{equation}\label{fourier of FC}
\begin{aligned}
\widehat{F_C}(\e^{-2}t,k) & \vcentcolon =  \frac{\e^{n_B(C)}}{L^{\frac{n(C)d}{2}}}   (-i)^{ n_T(C)}  \sum_{\ka\in\mathcal{D}_k(C)} e^{i\e^{-2}\Om_{-1}(C)t} \prod_{j\in \mathcal N(C)} q_{j}(C) 
\\&\hspace{2cm}\quad \times\int_{I_{C}(t)} \prod_{j\in  \mathcal N_T(C)}  e^{i \Om_{j} \e^{-2} t_{j}} \d t_{j} \prod_{\ell\in\mathcal{L}(C)_-}  M^{c(\ell),c(\si(\ell))}(\ffi(\ell)\ka(\ell))^{\ffi(\ell)}.
\end{aligned}
\end{equation}
where if $C=\big( (A_1,\ffi_1,c_1),(A_2,\ffi_2,c_2),\si \big)$ is a coupling then
\begin{align*}
I_C(t) \vcentcolon = I_{A_1}(t)\times I_{A_2}(t)
\end{align*}
and where we use the convention that the integral over $I_C(t)$ equals 1 if $I_C(t)$ is empty.
\end{prop}

\begin{proof}
We first use \eqref{XtoF} to go from the sequence $X^{\eta,\iota}_n$ to trees:
\begin{align*}
\E \pth{ \widehat{X^{\eta_1,\iota_1}_{n_1}}(\e^{-2}t,k)\widehat{X^{\eta_2,\iota_2}_{n_2}}(\e^{-2}t,-k) } = \sum_{\substack{ (A_1,\ffi_1,c_1)\in\mathcal{A}_{n_1}^{\eta_1,\iota_1} \\ (A_2,\ffi_2,c_2)\in\mathcal{A}_{n_2}^{\eta_2,\iota_2} }}\E\pth{\widehat{F_{\pth{A_1,\ffi_1,c_1}}}(\e^{-2}t,k)\widehat{F_{\pth{A_2,\ffi_2,c_2}}}(\e^{-2}t,-k)}.
\end{align*}
Then we use \eqref{fourier F e-2} to express the Fourier transforms on the RHS:
\begin{equation}\label{proof EX2}
\begin{aligned}
&\E \pth{ \widehat{X^{\eta_1,\iota_1}_{n_1}}(\e^{-2}t,k)\widehat{X^{\eta_2,\iota_2}_{n_2}}(\e^{-2}t,-k) } 
\\&\hspace{1cm} = \sum_{\substack{ (A_1,\ffi_1,c_1)\in\mathcal{A}_{n_1}^{\eta_1,\iota_1} \\ (A_2,\ffi_2,c_2)\in\mathcal{A}_{n_2}^{\eta_2,\iota_2} }}\frac{\e^{n_B(A_1)+n_B(A_2)}}{L^{\frac{(n_1+n_2)d}{2}}}   (-i)^{ n_T(A_1)+n_T(A_2)}  \sum_{\substack{\ka_1\in\mathcal D_k(A_1)\\\ka_2\in\mathcal D_{-k}(A_2)}}  e^{i\e^{-2}\pth{\Om_{-1}(A_1)+\Om_{-1}(A_2)}t}
\\&\hspace{1cm}\quad \times \prod_{\substack{j_1\in \mathcal N(A_1)\\j_2\in \mathcal N(A_2)}} q_{j_1}(A_1)q_{j_2}(A_2) \int_{I_{A_1}(t)\times I_{A_2}(t)} \prod_{\substack{j_1\in  \mathcal N_T(A_1)\\ j_2\in  \mathcal N_T(A_2)}}  e^{i \Om_{j_1}(A_1) \e^{-2} t_{j_1}}e^{i \Om_{j_2}(A_2) \e^{-2} t_{j_2}} \d t_{j_1}\d t_{j_2}
\\&\hspace{1cm}\quad \times \E\pth{\prod_{\substack{\ell_1\in \mathcal L(A_1)\\ \ell_2\in \mathcal L(A_2)}} \reallywidehat{X^{c_1(\ell_1), \ffi_1(\ell_1)}_\init} (\ka_1(\ell_1)) \reallywidehat{X^{c_2(\ell_2), \ffi_2(\ell_2)}_\init} (\ka_2(\ell_2))},
\end{aligned}
\end{equation}
where the various $\Om_{-1}(A_i)$, $q_j(A_i)$ and $\Om_j(A_i)$ stand for $\Om_{-1}(A_i,\ffi_i,\ka_i)$, $q_j(A_i,\ffi_i,c_i,\ka_i)$ and $\Om_j(A_i,\ffi_i,\ka_i)$ for $i=1,2$. In order to unify the notations, we define maps $\ffi$, $c$, $\ka$ from $\mathcal{L}(A_1)\sqcup\mathcal{N}(A_1)\sqcup\mathcal{L}(A_2)\sqcup\mathcal{N}(A_2)$ to $\{\pm\}$, $\{0,1\}$, $\R^d$ respectively so that their restriction to $\mathcal{L}(A_i)\sqcup\mathcal{N}(A_i)$ are given by $\ffi_i$, $c_i$, $\ka_i$ respectively, for $i=1,2$. This allows us to write
\begin{align*}
\E\pth{\prod_{\substack{\ell_1\in \mathcal L(A_1)\\ \ell_2\in \mathcal L(A_2)}} \reallywidehat{X^{c_1(\ell_1), \ffi_1(\ell_1)}_\init} (\ka_1(\ell_1)) \reallywidehat{X^{c_2(\ell_2), \ffi_2(\ell_2)}_\init} (\ka_2(\ell_2))} & = \E \pth{ \prod_{\ell\in \mathcal L(A_1)\sqcup \mathcal L(A_2)} \widehat{X^{c(\ell), \ffi(\ell)}_\init} (\ka(\ell)) }.
\end{align*}
Now Wick's formula implies
\begin{align}\label{eq:Wick}
\E \pth{ \prod_{\ell\in \mathcal L(A_1)\sqcup \mathcal L(A_2)} \widehat{X^{c(\ell), \ffi(\ell)}_\init} (\ka(\ell)) } & = \sum_{\substack{\si \text{ pairing of}\\ \mathcal{L}(A_1)\sqcup\mathcal{L}(A_2)}} \prod_{\ell\in S_\si^-} \E\pth{  \widehat{X^{c(\ell), \ffi(\ell)}_\init} (\ka(\ell)) \reallywidehat{X^{c(\si(\ell)), \ffi(\si(\ell))}_\init} (\ka(\si(\ell)))}
\\ \non & = \sum_{\substack{\si \text{ pairing of}\\ \mathcal{L}(A_1)\sqcup\mathcal{L}(A_2)}} \prod_{\ell\in S_\si^-}   \de_{\ffi(\ell)+\ffi(\si(\ell))} \de_{\ka(\ell)+\ka(\si(\ell))} M^{c(\ell),c(\si(\ell))}(\ffi(\ell)\ka(\ell))^{\ffi(\ell)},
\end{align}
where a pairing of a finite set is an involution without fixed point, where $S_\si^-$ is a subset of $\mathcal{L}(A_1)\sqcup\mathcal{L}(A_2)$ such that $S_\si^-\sqcup \si\pth{S_\si^-}=\mathcal{L}(A_1)\sqcup\mathcal{L}(A_2)$ and where we also used our assumptions on the initial data. This shows that in order to contribute, a pairing $\si$ of $\mathcal{L}(A_1)\sqcup\mathcal{L}(A_2)$ has to be such that 
\begin{align}\label{condition pairing}
\ffi(\ell)+\ffi(\si(\ell))=0 \quad \text{and} \quad \ka(\ell)+\ka(\si(\ell))=0
\end{align}
for all $\ell\in S_\si^-$. Together with the fact that each pairing is an involution, the first condition in \eqref{condition pairing} implies that $\si\pth{\mathcal{L}(C)_-}=\mathcal{L}(C)_+$, where $\mathcal{L}(C)_-$ and $\mathcal{L}(C)_+$ are defined in \eqref{def LC-}-\eqref{def LC+}. This shows that we are allowed to consider $S_\si^-=\mathcal{L}(C)_-$ and that $\si:\mathcal{L}(C)_-\longrightarrow\mathcal{L}(C)_+$ is a bijection. In particular, $C=\big( (A_1,\ffi_1,c_1),(A_2,\ffi_2,c_2),\si \big)$ is a coupling as defined in Definition \ref{def:coupling}. Therefore, after having switched the sum over the coupling maps $\si$ and the decoration map $\ka_1$ and $\ka_2$ we obtain
\begin{align*}
&\E \pth{ \widehat{X^{\eta_1,\iota_1}_{n_1}}(\e^{-2}t,k)\widehat{X^{\eta_2,\iota_2}_{n_2}}(\e^{-2}t,-k) } 
\\&\hspace{2cm} = \sum_{ C=\big( (A_1,\ffi_1,c_1),(A_2,\ffi_2,c_2),\si \big) \in \mathcal{C}^{\eta_1,\eta_2,\iota_1,\iota_2}_{n_1,n_2} }\frac{\e^{n_B(C)}}{L^{\frac{(n_1+n_2)d}{2}}}   (-i)^{ n_T(C)}
\\&\hspace{2cm}\quad \times \sum_{\substack{\ka_1\in\mathcal D_k(A_1)\\\ka_2\in\mathcal D_{-k}(A_2)\\ \text{$\ka(\ell)+\ka(\si(\ell))=0$, $\forall \ell\in\mathcal{L}(C)_-$} }} e^{i\e^{-2}\Om_{-1}(C)t} \prod_{j\in \mathcal N(C)} q_{j}(C) 
\\&\hspace{2cm}\quad \times\int_{I_{C}(t)} \prod_{j\in  \mathcal N_T(C)}  e^{i \Om_{j} \e^{-2} t_{j}} \d t_{j} \prod_{\ell\in\mathcal{L}(C)_-}  M^{c(\ell),c(\si(\ell))}(\ffi(\ell)\ka(\ell))^{\ffi(\ell)},
\end{align*}
where $\ffi$, $c$ and $\ka$ are defined as above based on $\ffi_1$, $\ffi_2$, $c_1$, $c_2$, $\ka_1$ and $\ka_2$. Finally, note that Definition \ref{def:decocoup} implies that there is a bijection between $\mathcal{D}_k(C)$ and the set 
\begin{align*}
\enstq{ (\ka_1,\ka_2)\in \mathcal{D}_k(A_1)\times\mathcal{D}_{-k}(A_2)}{ \forall \ell\in\mathcal{L}(C)_-, \; \ka(\ell)+\ka(\si(\ell))=0}.
\end{align*}
Therefore we finally obtain
\begin{align*}
&\E \pth{ \widehat{X^{\eta_1,\iota_1}_{n_1}}(\e^{-2}t,k)\widehat{X^{\eta_2,\iota_2}_{n_2}}(\e^{-2}t,-k) } 
\\&\hspace{1cm} = \sum_{ C=\big( (A_1,\ffi_1,c_1),(A_2,\ffi_2,c_2),\si \big) \in \mathcal{C}^{\eta_1,\eta_2,\iota_1,\iota_2}_{n_1,n_2} }\frac{\e^{n_B(C)}}{L^{\frac{(n_1+n_2)d}{2}}}   (-i)^{ n_T(C)}
\\&\hspace{1cm}\quad \times \sum_{\ka\in\mathcal{D}_k(C)} e^{i\e^{-2}\Om_{-1}(C)t} \prod_{j\in \mathcal N(C)} q_{j}(C) \int_{I_{C}(t)} \prod_{j\in  \mathcal N_T(C)}  e^{i \Om_{j} \e^{-2} t_{j}} \d t_{j} \prod_{\ell\in\mathcal{L}(C)_-}  M^{c(\ell),c(\si(\ell))}(\ffi(\ell)\ka(\ell))^{\ffi(\ell)},
\end{align*}
which concludes the proof of the proposition.
\end{proof}

\begin{remark}\label{remark couplings}
We make two comments on Proposition \ref{prop:correlation}:
\begin{itemize}
\item[(i)] As for the first comment in Remark \ref{remark prop Fourier F}, the assumption \eqref{assumption Q} implies that
\begin{align}\label{FC(0)=0}
\widehat{F_C}(\e^{-2}t,0) & = 0.
\end{align}
\item[(ii)] The result of Proposition \ref{prop:correlation} will be particularly useful to prove the convergence of the series $\sum_{n\in\mathbb{N}}X^\eta_n$ over a large probability set. Indeed, \eqref{correlations to coupling} implies that
\begin{align}\label{correlation in terms of couplings}
\E\pth{ \left| \widehat{X^\eta_n}(\e^{-2}t,k) \right|^2 } & = \sum_{C\in\mathcal{C}^{\eta,\eta,+,-}_{n,n}} \widehat{F_C}(\e^{-2}t,k),
\end{align}
so that estimating $\widehat{F_C}(\e^{-2}t,k)$ will allow us to estimate $\widehat{X^\eta_n}(\e^{-2}t,k)$.
\end{itemize}
\end{remark}

\subsection{Bushes}\label{section bushes}

In this section we introduce the notion of prebush and bush associated to particular nodes of trees and couplings. The definitions below only depend on the nature of nodes themselves, not on their signs/colours, nor on the involution $\si$ in the case of a coupling. They will be mainly useful in the case of couplings.

\saut
Definition \ref{def marked} below allows us to distinguish between the children of low ternary nodes that participate to the single low frequency interaction.

\begin{definition}[Marked and unmarked] \label{def marked}
Let $C$ be a coupling or a tree. 
\begin{itemize}
\item We say that $j\in \mathcal{N}(C)\sqcup \mathcal{L}(C)$ is \textbf{marked} if one of the following conditions hold:
\begin{itemize}
\item[(i)] $j$ is the middle or right child of a $\lowl$ node,
\item[(ii)] $j$ is the left or right child of a $\lowm$ node,
\item[(iii)] $j$ is the left or middle child of a $\lowr$ node, 
\end{itemize}
We say that $j$ is \textbf{unmarked} if it is not marked.
\item If $j\in\mathcal{N}_\low(C)$, we define
\begin{align}\label{def offsprings}
\offspring(j) \vcentcolon = \bigsqcup_{\substack{j'\in\mathtt{children}(j)\\\text{$j'$ marked }}} \enstq{\ell\in\mathcal{L}(C)}{\ell\leq j' }.
\end{align}
\end{itemize}
\end{definition}

It follows from Definition \ref{def marked} that the children of a binary node or a high ternary node are unmarked and that all branching nodes have at least one unmarked child. Moreover, the offspring of a low ternary node are precisely the leaves ultimately contributing to the low-frequency interaction of this node.

\begin{definition}[Prebush]\label{def prebush}
Let $C$ be a coupling or a tree. We define the \textbf{prebush} of any $j\in\mathcal{N}(C)\sqcup\mathcal{L}(C)$ recursively:
\begin{itemize}
\item If $\ell\in\mathcal{L}(C)$, then
\begin{align}\label{def prebush leaf}
\mathtt{prebush}(\ell) \vcentcolon = \{\ell\}.
\end{align}
\item If $j\in\mathcal{N}(C)$, then
\begin{align}\label{def prebush node}
\mathtt{prebush}(j) \vcentcolon = \bigsqcup_{\substack{j'\in\mathtt{children}(j)\\\text{$j'$ unmarked }}}\mathtt{prebush}(j').
\end{align}
\end{itemize}
\end{definition}

The following lemma contains basic properties of prebushes.

\begin{lemma}\label{lem prebush}
Let $C$ be a coupling or a tree.
\begin{itemize}
\item If $j\in\mathcal{N}(C)\sqcup\mathcal{L}(C)$, then
\begin{align}\label{prebush caracterisation}
\mathtt{prebush}(j) = \enstq{\ell\in\mathcal{L}(C)}{ \text{$\ell \leq j$ and there exists no marked node or leaf $j'$ such that $\ell \leq j' < j$} }.
\end{align}
In particular, $\mathtt{prebush}(j)\neq \emptyset$.
\item If $j,j'\in \mathcal{N}(C)\sqcup\mathcal{L}(C)$ are both marked, then 
\begin{align}\label{disjoint prebush}
j\neq j'\Longrightarrow \mathtt{prebush}(j)\cap\mathtt{prebush}(j')=\emptyset.
\end{align}
\end{itemize}
\end{lemma}

\begin{proof}
We start by the first part of the lemma.
Let $\ell$ be an element of the set on the RHS of \eqref{prebush caracterisation}. By definition, $\ell$ is unmarked and there exists a family $(j_i)_{i\in\llbracket 1,n\rrbracket}$ of unmarked nodes such that $\ell\in\mathtt{children}(j_1)$, $j_i\in\mathtt{children}(j_{i+1})$ and $j_n\in\mathtt{children}(j)$. Thanks to \eqref{def prebush leaf} and \eqref{def prebush node}, this implies that 
\begin{align*}
\ell \in \mathtt{prebush}(j_1) \subset \cdots \subset \mathtt{prebush}(j_n) \subset \mathtt{prebush}(j),
\end{align*}
which shows the inclusion from right to left in \eqref{prebush caracterisation}. 

The inclusion from left to right is proved by induction. If $j\in\mathcal{L}(C)$, \eqref{prebush caracterisation} is a rewriting of \eqref{def prebush leaf}. Now, let $j\in\mathcal{N}(C)$ and assume that the inclusion from left to right in \eqref{prebush caracterisation} holds for all $j'\in\mathtt{children}(j)$, and let us prove that it also holds for $j$. Let $\ell\in\mathtt{prebush}(j)$, thanks to \eqref{def prebush node} there exists $j'\in\mathtt{children}(j)$ such that $j'$ is unmarked and $\ell\in\mathtt{prebush}(j')$. Therefore, our induction assumption implies that $(a)$ $\ell\leq j'$ and that $(b$) there is no marked node or leaf $j''$ such that $\ell\leq j''<j'$. Since $j'\in\mathtt{children}(j)$, the first point $(a)$ already implies $\ell\leq j$. Moreover, if there exists $j'''$ marked such that $\ell\leq j'''<j$ then by uniqueness of the path linking a leaf to a node we either have $j'''=j'$, which is impossible since $j'$ is unmarked, or $\ell\leq j'''<j'$, which is impossible by the second point $(b)$. This proves the inclusion from left to right for $j$, and concludes the proof of \eqref{prebush caracterisation}. Since each branching nodes have at least one unmarked child, \eqref{prebush caracterisation} proves that $\mathtt{prebush}(j)$ is never empty.

We now prove the second part. By contradiction, if $\ell\in \mathtt{prebush}(j)\cap\mathtt{prebush}(j')$, then \eqref{prebush caracterisation} implies $\ell\leq j$ and $\ell\leq j'$. This shows that $j$ and $j'$ are comparable and that $\ell\leq j < j'$ or $\ell\leq j'<j$, which contradicts the fact that both $j$ and $j'$ are marked. This concludes the proof of \eqref{disjoint prebush} and of the lemma.
\end{proof}

\begin{definition}[Bush]\label{def bush}
Let $C$ be a coupling or a tree.
\begin{itemize}
\item We define the \textbf{bush} of any $j\in\mathcal{N}_\low(C)$ by
\begin{align}\label{bush}
\mathtt{bush}(j) \vcentcolon = \bigsqcup_{\substack{j'\in\mathtt{children}(j)\\\text{$j'$ marked }}} \mathtt{prebush}(j').
\end{align}
\item We define
\begin{align*}
\mathtt{Bushes}(C)  \vcentcolon = \enstq{\mathtt{bush}(j) }{ j\in \mathcal N_\low(C)}.
\end{align*}
\end{itemize}
\end{definition}

\begin{example}\label{ex:bushes}
We illustrate here Definitions \ref{def marked}, \ref{def  prebush} and \ref{def bush}. Consider the following tree.
\begin{center}
\begin{tikzpicture}[level 1/.style={sibling distance=3cm},
                   level 2/.style={sibling distance=1.5cm}]
\node[label=above:{\textcolor{blue}{$1$}}]{$\lowl$}
    child{ node[label=left:{\textcolor{blue}{$2$}}]{$\st$}
        child{node{\textcolor{magenta}{$2$}}}
        child{node[label=right:{\textcolor{blue}{$4$}}]{$\lowm$}
            child{node{\fbox{\textcolor{magenta}{$5$}}}}
            child{node{\textcolor{magenta}{$6$}}}
            child{node[label=right:{\textcolor{blue}{$6$}}]{\fbox{$\init$}}
                child{node{\textcolor{magenta}{$8$}}}
                child{node{\textcolor{magenta}{$9$}}}
            }
        }
    }
    child{node{\fbox{\textcolor{magenta}{$1$}}}}
    child{node[label=right:{\textcolor{blue}{$3$}}]{\fbox{$\high$}}
        child{node{\textcolor{magenta}{$3$}}}
        child{node{\textcolor{magenta}{$4$}}}
        child{node[label=right:{\textcolor{blue}{$5$}}]{$\lowr$}
            child{node[label=left:{\textcolor{blue}{$7$}}]{\fbox{$\init$}}
                child{node{\textcolor{magenta}{$10$}}}
                child{node{\textcolor{magenta}{$11$}}}
            }
            child{node{\fbox{\textcolor{magenta}{$7$}}}}
            child{node[label=right:{\textcolor{blue}{$8$}}]{$\high$}
                child{node{\textcolor{magenta}{$12$}}}
                child{node{\textcolor{magenta}{$13$}}}
                child{node{\textcolor{magenta}{$14$}}}
            }
        }
    }
;
\end{tikzpicture}
\end{center}
We have labelled arbitrarily the nodes by \textcolor{blue}{blue} numbers and the leaves by \textcolor{magenta}{magenta} numbers. The marked children of low nodes are put into \fbox{boxes}. We have
\begin{align*}
    \mathtt{prebush}(\textcolor{blue}{1}) & = \mathtt{prebush}(\textcolor{blue}{2}) = \{\textcolor{magenta}{2},\textcolor{magenta}{6}\}, &  \mathtt{prebush}(\textcolor{blue}{3}) & =  \{\textcolor{magenta}{3},\textcolor{magenta}{4},\textcolor{magenta}{12},\textcolor{magenta}{13},\textcolor{magenta}{14}\}, &  \mathtt{prebush}(\textcolor{blue}{4}) & =  \{\textcolor{magenta}{6}\}, \\
    \mathtt{prebush}(\textcolor{blue}{5}) & =  \{\textcolor{magenta}{12},\textcolor{magenta}{13},\textcolor{magenta}{14}\}, & \mathtt{prebush}(\textcolor{blue}{6}) & =  \{\textcolor{magenta}{8},\textcolor{magenta}{9}\}, & \mathtt{prebush}(\textcolor{blue}{7}) &  =  \{\textcolor{magenta}{10},\textcolor{magenta}{11}\}.
\end{align*}
Note that $\mathtt{prebush}(\textcolor{blue}{3})$ does not intersect $\mathtt{prebush}(\textcolor{blue}{6})$ or $\mathtt{prebush}(\textcolor{blue}{7})$. Moreover we have
\begin{align*}
     \mathtt{bush}(\textcolor{blue}{1}) =  \{\textcolor{magenta}{1},\textcolor{magenta}{3},\textcolor{magenta}{4},\textcolor{magenta}{12},\textcolor{magenta}{13},\textcolor{magenta}{14}\},\quad     \mathtt{bush}(\textcolor{blue}{4}) =  \{\textcolor{magenta}{5},\textcolor{magenta}{8},\textcolor{magenta}{9}\}, \quad \mathtt{bush}(\textcolor{blue}{5}) =   \{\textcolor{magenta}{7},\textcolor{magenta}{10},\textcolor{magenta}{11}\}.
\end{align*}
\end{example}

The following proposition gathers important properties of bushes.

\begin{prop}\label{prop bush}
Let $C$ be a coupling or a tree.
\begin{itemize}
\item[(i)] The following properties hold:
\begin{align}
\forall j \in\mathcal{N}_\low(C),& \; \mathtt{bush}(j)\neq \emptyset, \label{bush non empty}
\\ \forall j,j'\in\mathcal{N}_\low(C),& \; j\neq j'\Longrightarrow \mathtt{bush}(j)\cap\mathtt{bush}(j')=\emptyset.\label{bush disjoint}
\end{align}
\item[(ii)] Moreover, for $j\in\mathcal{N}(C)\sqcup\mathcal{L}(C)$ we have
\begin{align}\label{leaf below}
\enstq{ \ell\in\mathcal{L}(C) }{\ell\leq j} =  \bigsqcup_{\substack{ j'\in\mathcal{N}_\low(C) \\ j'\leq j}}\mathtt{bush}(j') \sqcup \mathtt{prebush}(j) ,
\end{align}
and for $j\in\mathcal{N}_\low(C)$ we have
\begin{align}\label{lien offsprings bushes}
\offspring (j) & = \bigsqcup_{\substack{ j'\in\mathcal{N}_\low(C) \\ \text{$j'\leq j$ and $j'\nleq j_{\mathrm{unmarked}}$} }} \mathtt{bush}(j'),
\end{align}
where $j_{\mathrm{unmarked}}$ is the unmarked child of $j$.
\end{itemize}
\end{prop}

\begin{proof}
We start with the first part of the proposition:
\begin{itemize}
\item The first part of Lemma \ref{lem prebush} says that all prebushes are non-empty, which thanks to \eqref{bush} implies in turn that all bushes are non-empty, thus proving \eqref{bush non empty}.
\item Consider now $j,j'\in\mathcal{N}_\low(C)$ with $j\neq j'$ and assume that there is $\ell\in\mathtt{bush}(j)\cap\mathtt{bush}(j')$. According to \eqref{bush}, there exists $j_1\in\mathtt{children}(j)$ and $j_1'\in\mathtt{children}(j')$ both marked such that $\ell\in\mathtt{prebush}(j_1)\cap\mathtt{prebush}(j'_1)$ which contradicts the second part of Lemma \ref{lem prebush}. This proves \eqref{bush disjoint}.
\end{itemize}
We now look at the second part of the proposition:
\begin{itemize}
\item If $\ell\leq j$, then either there exists no marked node or leaf in between $\ell$ and $j$ or there exists at least one. In the first situation, \eqref{prebush caracterisation} implies that $\ell\in\mathtt{prebush}(j)$. In the second situation, let $j'$ the minimal such marked node or leaf such that $\ell\leq j'<j$. By minimality, \eqref{prebush caracterisation} again implies that $\ell\in\mathtt{prebush}(j')$. We have thus proved
\begin{align*}
\enstq{ \ell\in\mathcal{L}(C) }{\ell\leq j} & = \mathtt{prebush}(j)\sqcup \bigsqcup_{\substack{j'<j\\\text{$j'$ marked}}} \mathtt{prebush}(j')
\\&= \mathtt{prebush}(j)\sqcup \bigsqcup_{\substack{ j''\in\mathcal{N}_\low(C) \\ j''\leq j}}\mathtt{bush}(j''),
\end{align*}
where we have used the definition of bushes \eqref{bush}. This proves \eqref{leaf below}.
\item Finally, if $j\in\mathcal{N}_\low(C)$ then \eqref{def offsprings} implies
\begin{align*}
\offspring (j) & = \bigsqcup_{j'\in\mathtt{children}(j)\setminus \{j_{\mathrm{unmarked}}\}} \enstq{\ell\in\mathcal{L}(C)}{\ell\leq j' }
\\& = \bigsqcup_{j'\in\mathtt{children}(j)\setminus \{j_{\mathrm{unmarked}}\}} \pth{ \mathtt{prebush}(j')\sqcup \bigsqcup_{\substack{ j''\in\mathcal{N}_\low(C) \\ j''\leq j'}}\mathtt{bush}(j'') }
\end{align*}
where we applied \eqref{leaf below} to each $j'\in\mathtt{children}(j)\setminus \{j_{\mathrm{unmarked}}\}$. We obtain
\begin{align*}
\offspring (j) & = \bigsqcup_{j'\in\mathtt{children}(j)\setminus \{j_{\mathrm{unmarked}}\}} \mathtt{prebush}(j') \sqcup \bigsqcup_{\substack{ j'\in\mathcal{N}_\low(C) \\ \text{$j'< j$ and $j'\nleq j_{\mathrm{unmarked}}$} }}\mathtt{bush}(j'') 
\\& = \bigsqcup_{\substack{ j'\in\mathcal{N}_\low(C) \\ \text{$j'\leq j$ and $j'\nleq j_{\mathrm{unmarked}}$} }}\mathtt{bush}(j'') ,
\end{align*}
where we have used \eqref{bush}.
\end{itemize}
This concludes the proof of the proposition.
\end{proof}

\begin{example} We remark that in Example \ref{ex:bushes}, the bushes do not intersect each other and match
\begin{align*}
    \offspring(\textcolor{blue}{1}) = \mathtt{bush}(\textcolor{blue}{1}) \sqcup \mathtt{bush}(\textcolor{blue}{5}), \quad \offspring(\textcolor{blue}{4}) = \mathtt{bush}(\textcolor{blue}{4}),\quad \offspring(\textcolor{blue}{5}) = \mathtt{bush}(\textcolor{blue}{5}).
\end{align*}
\end{example}

As opposed to the above definitions, the next one does depend on the involution $\sigma$ and therefore cannot be applied to a tree but only to a coupling.

\begin{definition}[Self-coupled bushes] \label{def self-coupled bushes}
Let $C$ be a coupling and $\sigma$ its coupling map.
\begin{itemize}
\item We say that $B\in\mathtt{Bushes}(C)$ is \textbf{self-coupled} if the set
\begin{align*}
B_+ \vcentcolon = B\cap\mathcal{L}(C)_+ 
\end{align*}
has the same cardinal as the set 
\begin{align*}
B_- \vcentcolon = B\cap\mathcal{L}(C)_-
\end{align*}
and if $\sigma (B_+)= B_-$.
\item We define
\begin{align*}
n_{\scb}(C) & \vcentcolon = \# \enstq{B\in\mathtt{Bushes}(C)}{\text{$B$ is self-coupled}}.
\end{align*}
\end{itemize}
\end{definition}

The following lemma gives a natural bound on the number of couplings with a given number of self-coupled bushes. In its proof, we will need a simple fact about factorials: if $(a_i)_{i\in\llbracket 1,p\rrbracket}$ are some non-zero integers, then
\begin{align}\label{fact factorial}
\prod_{i=1}^p a_i! \leq \pth{ \sum_{i=1}^p a_i - p + 1}!
\end{align}

\begin{lemma} \label{counting self-coupled bushes}
Let $n,q\in\mathbb{N}$ and let
\begin{align*}
c_{n,q} := \# \enstq{C\in\mathcal{C}_{n_1,n_2}}{ n_1+n_2=n,\; n_{\scb}(C)=q}.
\end{align*}
There exists $\La>0$ such that
\begin{align*}
c_{n,q} \leq \La^{n+1} \pth{\frac{n}{2}+2-q}!.
\end{align*}
\end{lemma}

\begin{proof}  
Let $C=((A_1,\ffi_1,c_1)),(A_2,\ffi_2,c_2), \sigma) \in \mathcal{C}_{n_1, n_2}$ be a coupling with $n_1+n_2=n$ and $q$ self-coupled bushes. Define
\begin{align*}
M\vcentcolon = \#\enstq{B\in\mathtt{Bushes}(A_1)\sqcup\mathtt{Bushes}(A_2)}{\#B_+=\#B_-},
\end{align*}
where $\#B_+$ and $\#B_-$ are introduced in Definition \ref{def self-coupled bushes}, and assume $M>0$. Recall that for a self-coupled bush $B$, we have that $\#B_+ = \#B_-$ and the pairing $\si$ is a bijection from $B_+$ to $B_-$. Therefore $q\le M$. 

Let $\pth{\mathtt{bush}(j_i)}_{i\in\llbracket 1,q\rrbracket}$ be an enumeration of the self-coupled bushes of $C$ and $m_i=\#\mathtt{bush}(j_i)_+$ their number of positive leaves. The number of different possible pairings $\si$ for which $\mathtt{bush}(j_i)$ is self-coupled for all $i\in\llbracket 1,q\rrbracket$ is equal to
\begin{align}\label{si pairing q bushes}
\prod_{i=1}^q m_i ! \times \pth{\frac{n+2}{2} - \sum_{i=1}^q m_i}!
\end{align}
where $\frac{n+2}{2} - \sum_{i=1}^q m_i$ corresponds to the number of positive leaves in $A_1\sqcup A_2$ that do not belong to any $\mathtt{bush}(j_i)$. We now apply \eqref{fact factorial} to \eqref{si pairing q bushes}, in particular noting that $\frac{n+2}{2} - \sum_{i=1}^q m_i$ might vanish, and obtain
\begin{align*}
\#\left\{ \text{$\si$ pairing such that $\pth{\mathtt{bush}(j_i)}_{i\in\llbracket 1,q\rrbracket}$ are self-coupled} \right\} \leq \pth{ \frac{n+2}{2} - q + 1 }!.
\end{align*}
Since the choice of the $q$ self-coupled bushes among the $M$ bushes satisfying $\#B_+=\#B_-$ is arbitrary, we obtain
\begin{align*}
\#\left\{ \text{$\si$ pairing such that $n_\scb(C)=q$} \right\} & \leq \binom{M}{q} \pth{ \frac{n+2}{2} - q + 1 }! \leq \La_1^n \pth{ \frac{n+2}{2} - q + 1 }!
\end{align*}
with $\La_1 \geq \sqrt 2$. In the above inequality we have used that $M\leq n_T(C) \leq \frac{n}{2}$. Finally, we have
\begin{align*}
c_{n,q} & \leq \La_1^n \pth{ \frac{n+2}{2} - q + 1 }! 
\\&\quad \times \# \enstq{((A_1,\ffi_1,c_1)),(A_2,\ffi_2,c_2)) \in \mathcal{A}_{n_1}^{\eta_1,\iota_1} \times \mathcal{A}_{n_2}^{\eta_2,\iota_2}}{\eta_1,\eta_2\in\{0,1\}, \; \iota_1,\iota_2\in\{\pm\}, \; n_1+n_2=n}.
\end{align*}
According to Remark \ref{rm:cardinal}, there exists $\La_2$ such that we have $\#\mathcal{A}_{n_i}^{\eta_i,\iota_i} \leq \La_2^{n_i}$. Thus we get
\begin{align*}
c_{n,q} & \leq \La_1^n \pth{ \frac{n+2}{2} - q + 1 }! \sum_{n_1+n_2=n}\La_2^{n_1}\La_2^{n_2} \leq (2\La_1 \La_2)^n \pth{ \frac{n+2}{2} - q + 1 }! ,
\end{align*}
which concludes the proof of the lemma.
\end{proof}


\section{Convergence of the series}\label{section convergence}

The aim of this section is to prove the following proposition.

\begin{prop}\label{prop |FC|}
There exists $\La>0$ such that for all $C\in\mathcal{C}^{\eta_1,\eta_2,\iota_1,\iota_2}_{n_1,n_2}$ we have
\begin{align}\label{estimate good}
 \sup_{t\in[0,\de]}\sup_{k\in\Z_L^d}\left|\widehat{F_C}(\e^{-2}t,k) \right| \leq \La (\La\delta)^{\frac{n(C)}{2}}.
\end{align}
Moreover, there exists $\a>0$ such that if $C$ satisfies 
\begin{align*}
    n(C)-2n_\scb(C)& \geq 80, \qquad n(C)\leq |\ln\e|^{3},
\end{align*}
then for all $\e\leq \de$ we have
\begin{align}\label{estimate FC}
 \sup_{t\in[0,\de]}\sup_{k\in\Z_L^d}\left|\widehat{F_C}(\e^{-2}t,k) \right| \leq \La (\La\delta)^{\frac{n(C)}{2}} \varepsilon^{\alpha \pth{\frac{n(C)}{2}-n_\scb(C)}}.
\end{align}
\end{prop}

We will deduce the following key corollary, whose proof is detailed in Section \ref{section coro Xn}. 

\begin{corollary}\label{coro Xn}
There exists $\Lambda>0$ such that, for any fixed $A>0$, with probability $\ge 1-L^{-A}$ we have 
\begin{equation} \label{est X_n}
    \l X_n \r_{\mathcal{C}([0,\e^{-2}\de],H^s(L\T^d))}\le (\Lambda \delta)^{\frac{n}{4}+\frac12}L^{\frac{d}{2}+\frac1\beta + A}
\end{equation}
for any $0\le n \le |\ln \e |^{3}$ and any $s\ge 0$.
\end{corollary}

\begin{remark}
The proof of corollary \ref{coro Xn} will also require us to study the time derivative of $X_n$. Its tree decomposition will be performed in Section \ref{section time derivative Xn}, where the equivalent of Proposition \ref{prop |FC|} is also stated, see Proposition \ref{prop estimate FCjmax}.
\end{remark}

Proving Proposition \ref{prop |FC|} will be done by distinguishing two regimes: one regime where low ternary nodes are dominating and one where high ternary nodes are dominating. The first one corresponds to the case where 
\begin{align*}
n_B(C) + 2(n_\low(C) - n_\scb(C)) \geq \frac{n(C)-2n_\scb(C)}{10},
\end{align*}
and will be treated in Section \ref{section low nodes regime}, while the second one corresponds to the case where
\begin{align*}
n_B(C) + 2(n_\low(C) - n_\scb(C)) < \frac{n(C)-2n_\scb(C)}{10},
\end{align*}
and will be treated in Section \ref{section high nodes regime}. 

The study of these two regimes will involve two combinatorial objects associated to couplings: bushes for the low nodes regime and spanning trees for the high nodes regime. However, in both cases we will need a sort of dual point of view on the decoration maps $\kappa$ introduced for trees in Definition \ref{def decoration tree} and for couplings in Definition \ref{def:decocoup}. Recall the first part of Remark \ref{remark post def decocoup}, which says that a decoration map of a given coupling $C$ can be uniquely defined from its restriction to $\mathcal{L}(C)_+$. In other words, there exists a natural bijection from the set of decoration maps of $C$ and $(\R^d)^{\mathcal{L}(C)_+}$, which is the set of maps from $\mathcal{L}(C)_+$ to $\R^d$.

\begin{definition}\label{def linear map}
Let $C$ be a coupling.
\begin{itemize}
\item If $\ell\in\mathcal{L}(C)_+$, we consider the linear map $K(\ell) : (\R^d)^{\mathcal{L}(C)_+} \longrightarrow \R^d$ defined by
\begin{align}\label{def K(ell)}
K(\ell)\pth{(\ka(\ell'))_{\ell'\in\mathcal{L}(C)_+}}  = \ka(\ell).
\end{align}
\item If $\ell\in\mathcal{L}(C)_-$, we define
\begin{align}
K(\ell)\vcentcolon = - K(\si(\ell)),
\end{align}
where $\si$ is the coupling map of $C$, and if $j\in\mathcal{N}(C)$, we define
\begin{align}\label{def K(j)}
K(j) \vcentcolon = \sum_{\substack{\ell\in\mathcal{L}(C)\\ \ell < j}} K(\ell).
\end{align}
\item We define a linear subspace of $L\pth{  (\R^d)^{\mathcal{L}(C)_+} , \R^d }$ by
\begin{align}\label{def V(C)}
V(C) \vcentcolon = \mathrm{Vect}\enstq{ K(\ell) }{\ell \in\mathcal{L}(C)_+}.
\end{align}
\end{itemize}
\end{definition}

\begin{remark}
We note two simple facts:
\begin{itemize}
\item By definition we have $K(j)\in V(C)$ for all $j\in\mathcal{N}(C)\sqcup\mathcal{L}(C)$.
\item If $\ka$ is a decoration of $C$ as in Definition \ref{def:decocoup}, then $\ka$ can be seen as an element of $(\R^d)^{\mathcal{L}(C)_+}$ and \eqref{def K(ell)}-\eqref{def K(j)} imply that for all $j\in\mathcal{N}(C)\sqcup\mathcal{L}(C)$ we have
\begin{align*}
K(j)(\ka)=\ka(j).
\end{align*}
\end{itemize}
\end{remark}

\subsection{Low nodes regime}\label{section low nodes regime}

The goal of this section is to estimate $\widehat{F_C}$ in the low nodes regime.

\begin{prop}\label{prop low nodes}
There exists $\alpha_\low,\La>0$ such that if $C\in\mathcal{C}^{\eta_1,\eta_2,\iota_1,\iota_2}_{n_1,n_2}$ is a coupling satisfying 
\begin{align}\label{low nodes regime}
n_B(C) + 2(n_\low(C) - n_\scb(C)) \geq \frac{n(C)-2n_\scb(C)}{10},
\end{align}
and 
\begin{align}\label{n-2nscb>80}
n(C)-2n_\scb(C) \geq 80,
\end{align}
then for any $\e\leq \de $ and any $k\in \Z_L^d$, we have 
\begin{align}\label{main estim FC low nodes}
\sup_{t\in[0,\de]}\left|\widehat{F_C}(\e^{-2}t,k) \right| \lesssim (\La\delta)^{\frac{n(C)}{2}} \varepsilon^{\alpha_\low \pth{\frac{n(C)}{2}-n_\scb(C)}}.
\end{align}
\end{prop}

\begin{remark}
Thanks to \eqref{FC(0)=0}, the estimate \eqref{main estim FC low nodes} is obviously true if $k=0$. Therefore, in the rest of Section \ref{section low nodes regime} we assume that $k\neq 0$.
\end{remark}

The following lemma gives a first rough estimate for $\widehat{F_C}(\e^{-2}t,k)$.

\begin{lemma}\label{lem:récriturelow}
We have
\begin{equation}\label{estimée intermédiaire FC low nodes}
\begin{aligned}
&\sup_{t\in[0,\de]}\left| \widehat{F_C}(\e^{-2}t,k) \right|
\\& \leq \La^{n+1}\varepsilon^{n_B(C)} \delta^{n_T(C)} L^{-\frac{nd}{2}} \sum_{\substack{\ka\in\mathcal{D}_k(C)\\ \ka(\mathcal{L}(C)_+)\subset B(0,R)}} \prod_{j\in \mathcal N_\low(C)} \chi\pth{\left|\sum_{\ell\in\offspring(j)} \ka(\ell)\right|\varepsilon^{-\gamma}}.
\end{aligned}
\end{equation}
\end{lemma}

\begin{proof}
We recall \eqref{fourier of FC}:
\begin{align*}
\widehat{F_C}(\e^{-2}t,k) & = \varepsilon^{n_B(C)} (-i)^{n_T(C)}L^{-\frac{nd}{2}}
\\&\quad \times \sum_{\ka\in\mathcal D_k(C)} e^{-i\Omega_{-1}\varepsilon^{-2}t}\prod_{j\in \mathcal N(C)} q_j \int_{I_C(t)} \prod_{j\in \mathcal N_T(C)}e^{-i\Omega_j\varepsilon^{-2} t} \d t_j \prod_{\ell\in\mathcal{L}(C)_-}  M^{c(\ell),c(\si(\ell))}(\ffi(\ell)\ka(\ell))^{\ffi(\ell)}. 
\end{align*}
We use that the oscillating integral is less than $\delta^{n_T(C)}$ if $t\in[0,\de]$, that the maps $M^{\eta,\eta'}$ are continuous maps with compact support included in $B(0,R)$ and that the $q_j$ are uniformly bounded (see Lemma \ref{lem qj}) to get that there exists a constant $\La$ independant of the coupling such that
\begin{align*}
\sup_{t\in[0,\de]}\left| \widehat{F_C}(\e^{-2}t,k) \right| \leq \La^{n+1}\varepsilon^{n_B(C)} \delta^{n_T(C)} L^{-\frac{nd}{2}} \sum_{\substack{\ka\in\mathcal{D}_k(C)\\ \ka(\mathcal{L}(C)_+)\subset B(0,R)}} \prod_{j\in \mathcal N_\low(C)} \chi_j ,
\end{align*}
where we also used the simple fact that $\ka(\mathcal{L}(C)_+)\subset B(0,R)$ implies $\ka(\mathcal{L}(C))\subset B(0,R)$ and where $\chi_j$ satisfies $\chi_j \leq \chi\pth{ \left| \kappa(j_1) + \kappa(j_2)\right|\e^{-\ga}}$ with $j_1$ and $j_2$ the marked children of $j$. Therefore, \eqref{def offsprings} implies that
\begin{align*}
\chi_j & \leq \chi\pth{\left|\sum_{\ell\in\offspring(j)} \ka(\ell)\right|\varepsilon^{-\gamma}},
\end{align*}
which concludes the proof of the lemma.
\end{proof}

Following this lemma, Proposition \ref{prop low nodes} is proved in the following sections. In Section \ref{section non-self-coupled-bushes}, we study the properties of non-self-coupled bushes and in particular obtain a triangular structure for the product in \eqref{estimée intermédiaire FC low nodes}. In Section \ref{section basis of V(C) and change of variable}, we define a basis of $V(C)$ and use it to perform a change of variables in the sum over decorations. Benefiting from both the triangular structure of the product and the change of variables, we finally estimate $\widehat{F_C}(\e^{-2}t,k)$ in Section \ref{section successive summations}.

\subsubsection{Non-self-coupled bushes}\label{section non-self-coupled-bushes}

For $j\in\mathcal{N}_\low(C)$, we define
\begin{align}\label{def k bush}
\mathfrak{K}_\bush(j) \vcentcolon = \sum_{\ell\in \mathtt{bush}(j)}K(\ell).
\end{align}
By definition of a self-coupled bush, $\mathfrak{K}_\bush(j)$ vanishes as a linear map if and only if $\mathtt{bush}(j)$ is self-coupled. To make things more precise, we define
\begin{align*}
\mathcal{N}'_\low(C)\vcentcolon = \left\{  \text{low nodes whose bush is not self-coupled} \right\}, \qquad n'_\low \vcentcolon=\#\mathcal{N}'_\low (C),
\end{align*}
so that in particular we have $\mathfrak{K}_\bush(j)\neq 0$ for all $j\in\mathcal{N}'_\low (C)$. We enumerate the elements in $\mathcal{N}'_\low(C)$ in a way compatible with the order in $C$, that is we write
\begin{align*}
\mathcal{N}'_\low(C) = \enstq{j_i}{i \in \llbracket 1,n'_\low \rrbracket},
\end{align*}
with $j_{p_1}< j_{p_2}$ implies $p_1 < p_2$. By using the decomposition with the $\mathfrak{K}_\bush(j)$'s over positive leaves, the following lemma constructs a linearly independent subfamily of large cardinal, making crucial use of the non-self-coupledness.

\begin{lemma}\label{lem rank}
The rank of the family $(\mathfrak K_\bush(j))_{j\in \mathcal N'_\low(C)}$ is grater or equal to $\frac{n'_\low}{2}$.
\end{lemma}

\begin{proof}
If $\mathcal{N}'_\low(C)=\emptyset$ the lemma is trivial. If not, for each $j\in \mathcal N'_\low(C)$ we define the following subset of $\mathcal L(C)_+$:
\begin{align*}
B_j \vcentcolon = \enstq{ \ell \in \mathcal{L}(C)_+ }{ \pth{\ell\in \mathtt{bush}(j)\, \mathrm{and} \, \sigma(\ell) \notin \mathtt{bush}(j)} \,\mathrm{or}\, \pth{\sigma(\ell) \in \mathtt{bush}(j) \,\mathrm{and}\, \ell\notin \mathtt{bush}(j)} }.
\end{align*}
The set $B_j$ is non-empty because $\mathtt{bush}(j)$ is non-empty (recall \eqref{bush non empty}) and non-self-coupled by definition. Since bushes are disjoint (recall \eqref{bush disjoint}), we also observe that for any $\ell \in \mathcal L(C)_+$ there exist at most two nodes $j_1, j_2\in \mathcal N'_\low(C)$ such that $\ell \in B_{j_1}\cap B_{j_2}$. Furthermore,
\begin{align}\label{k_bush avec B_j}
\mathfrak{K}_\bush (j) = \sum_{\ell\in B_j} \iota_j(\ell) K(\ell),
\end{align}
where $\iota_j(\ell) = 1$ if $\ell\in \mathtt{bush}(j)$ and $\iota_j(\ell) = -1$ if $\sigma(\ell) \in \mathtt{bush}(j)$.

We set $b_j:=\#B_j$ and choose an enumeration of the elements of $B_j$, $B_j = \enstq{\ell^j_i}{i\in\llbracket1,b_j\rrbracket}$. We build the sets $\mathcal N_i \subseteq \mathcal N'_\low(C)$ and $\mathcal R_i \subseteq \llbracket 1,n'_\low \rrbracket$ by induction as follows. For $i=0$, we set $\mathcal N_0\vcentcolon = \emptyset$ and $\mathcal R_0\vcentcolon = \llbracket 1, n'_\low \rrbracket$. If $(\mathcal{N}_0,\dots,\mathcal{N}_i)$ and $(\mathcal{R}_0,\dots,\mathcal{R}_i)$ are constructed, then define $\mathcal{N}_{i+1}$ and $\mathcal{R}_{i+1}$ as follows:
\begin{itemize}
\item if $\mathcal{R}_i=\emptyset$, then $\mathcal{R}_{i+1}\vcentcolon=\emptyset$ and $\mathcal{N}_{i+1}\vcentcolon=\mathcal{N}_i$,
\item otherwise, define $R_i=\min\mathcal{R}_i$ and set
\begin{align*}
\mathcal{N}_{i+1} & \vcentcolon = \mathcal{N}_i \cup \{j_{R_i}\},
\\ \mathcal{R}_{i+1} & \vcentcolon = \mathcal{R}_i \setminus \enstq{ r \in \mathcal{R}_i}{\ell^{j_{R_i}}_1\in B_{j_r} }.
\end{align*}
\end{itemize} 
As observed above, given $\ell\in\mathcal{L}(C)_+$ we have $\#\enstq{j\in\mathcal{N}'_\low(C)}{\ell\in B_j}\leq 2$. Therefore, as long as $\mathcal{R}_i\neq\emptyset$ we have
\begin{align*}
1\leq \enstq{ r \in \mathcal{R}_i}{\ell^{j_{R_i}}_1\in B_{j_r} } \leq 2,
\end{align*}
which implies
\begin{align*}
\#\mathcal{R}_i - 2 \leq \#\mathcal{R}_{i+1} \leq \#\mathcal{R}_i -1.
\end{align*}
The inequality on the right implies that $\mathcal{R}_{n'_\low}=\emptyset$ and the one on the left that $\#\mathcal{R}_i\geq n'_\low - 2i$. We define $i_0\vcentcolon = \min\enstq{i\in\llbracket 1,n'_\low\rrbracket}{\mathcal{R}_i=\emptyset}$ and thus have $i_0\geq \frac{n'_\low}{2}$. Moreover, $\mathcal{N}_i\cap \enstq{j_r}{r\in\mathcal{R}_i}=\emptyset$ for any $i\in \N$ which implies that $\#\mathcal{N}_{i+1}=\#\mathcal{N}_i+1$ as long as $\mathcal{R}_i\neq \emptyset$. Hence 
\begin{align*}
\#\mathcal{N}_{i_0}= i_0  \geq \frac{n'_\low}{2}.
\end{align*}
Now we prove that the family $\pth{ \mathfrak{K}_\bush(j) }_{j\in\mathcal{N}_{i_0}}$ is linearly independent. We fix an enumeration $\mathcal{N}_{i_0}=\enstq{J_i}{i\in\llbracket 1,i_0\rrbracket}$ and let $(r_i)_{i\in\llbracket 1,i_0\rrbracket} \in \R^{i_0}$ be such that
\begin{align}\label{famille libre lol}
\sum_{i=1}^{i_0}r_i \mathfrak{K}_\bush(J_i) = 0.
\end{align}
We assume that the $r_i$s are not all zeros and set $i_1:=\min\enstq{i\in\llbracket 1,i_0\rrbracket}{r_i\neq 0}$. If $i_1=i_0$, then \eqref{famille libre lol} simply becomes $r_{i_1}\mathfrak{K}_\bush(J_{i_1})=0$, implying $r_{i_1}=0$ since $\mathfrak{K}_\bush(J_{i_1})\neq 0$. Then necessarily $i_1<i_0$, and we note that by construction we have $\ell^{J_{i_1}}_1\notin B_{J_i}$ for all $i\in\llbracket i_1+1,i_0\rrbracket$. Since \eqref{famille libre lol} can be rewritten
\begin{align*}
\sum_{i=i_1}^{i_0}r_i \sum_{\ell\in B_{J_i}} \iota_{J_i}(\ell) K(\ell) = 0
\end{align*}
with the help of \eqref{k_bush avec B_j}, we see that $K\pth{ \ell^{J_{i_1}}_1 }$ only appears once in the sum with $r_{i_1}\iota_{J_{i_1}}\pth{\ell^{J_{i_1}}_1}$ as coefficient. Since the family $(K(\ell))_{\ell\in \mathcal{L}(C)_+}$ is linearly independent, we again deduce that $r_{i_1}=0$ and reach a contradiction. Therefore $r_i=0$ for all $i$. This concludes the proof of the lemma since we have constructed a linearly independent subfamily of  $(\mathfrak K_\bush(j))_{j\in \mathcal N'_\low(C)}$ of cardinal grater or equal to $\frac{n'_\low}{2}$. 
\end{proof}

The following lemma constructs a basis of $\Vect\left\{\kfr_\bush(j) \, |\, j\in\mathcal{N}'_\low(C)\right\}$ compatible, in some way, with the parentality order.

\begin{lemma}\label{lem basis of V(C)}
Assume that $\mathcal{N}'_\low(C)\neq\emptyset$. There exists a subset $\mathcal{N}''_\low(C)\subseteq \mathcal{N}'_\low(C)$ such that
\begin{itemize}
\item[(i)] its cardinal $n''_\low\vcentcolon = \#\mathcal{N}''_\low(C)$ satisfies $n''_\low \geq \frac{n'_\low}{2}$,
\item[(ii)] the family $\pth{ \mathfrak{K}_\bush(j) }_{j\in\mathcal{N}''_\low(C)}$ is linearly independent in $V(C)$,
\item[(iii)] there exists an enumeration of $\mathcal{N}''_\low(C)=\enstq{J_i}{i\in\llbracket 1, n''_\low \rrbracket}$ such that
\begin{itemize}
    \item[a.] there exists no $j\in\mathcal{N}'_\low(C)$ satisfying $j<J_1$;
    \item[b.] if $n''_\low\geq 2$, for $i\in\llbracket 2,n''_\low\rrbracket$ and $j\in\mathcal{N}'_\low(C)$ we have
\begin{align}\label{flag 2 énoncé}
j<J_{i}\Longrightarrow\,  \mathfrak{K}_\bush(j)\in \mathrm{Vect}\enstq{\mathfrak{K}_\bush(J_{i'})}{i'\in\llbracket 1,i-1\rrbracket}.
\end{align}
\end{itemize}
 
\end{itemize}
\end{lemma}

\begin{proof}
We first build a subset $P\subseteq \llbracket1,n'_\low \rrbracket$ by induction in the following way: $P_1 \vcentcolon = \{1\}$ and if $P_1 \subseteq \dots \subseteq P_m$ for $m\geq 1$ are defined then 
\begin{itemize}
\item if $\mathfrak{K}_\bush(j_{m+1})$ is a linear combination of $(\mathfrak{K}_\bush(j_p))_{p\in P_m}$, we set $P_{m+1} \vcentcolon = P_m$,
\item otherwise, we set $P_{m+1} \vcentcolon = P_m\sqcup \{m+1\}$.
\end{itemize}
We define $P\vcentcolon = P_{n'_\low}$ and 
\begin{align*}
\mathcal{N}''_\low(C) \vcentcolon = \enstq{j_m}{m\in P}.
\end{align*}
Statement $(ii)$ is satisfied since the family of linear maps $\pth{\mathfrak{K}_\bush(j)}_{j\in\mathcal{N}''_\low(C)}$ is linearly independent by construction. We also have that for any $i\in\llbracket 1,n'_\low\rrbracket$ 
\begin{align}\label{flag 1}
\mathfrak{K}_\bush(j_i) \in \mathrm{Vect}\enstq{ \mathfrak{K}_\bush(j_{i'}) }{i'\in P \cap \llbracket 1,i\rrbracket}.
\end{align}
This shows that $\pth{\mathfrak{K}_\bush(j)}_{j\in\mathcal{N}''_\low(C)}$ is a basis of $\Vect\left\{\kfr_\bush(j) \, |\, j\in\mathcal{N}'_\low(C)\right\}$ and proves statement $(i)$ thanks to Lemma \ref{lem rank}. 

It only remains to prove $(iii)$. Note that it is always possible to enumerate $\mathcal{N}''_\low(C)$ 
\begin{align*}
\mathcal{N}''_\low(C) = \enstq{J_i}{ i\in \llbracket1,n''_\low\rrbracket},
\end{align*}
in such a way that for all $I,I' \in \llbracket 1,n''_\low \rrbracket$, writing $J_I = j_i$ and $J_{I'} = j_{i'}$ with $i,i'\in \llbracket 1,n'_\low \rrbracket$, we have that $i<i'$ is equivalent to $I<I'$. 
Since $j_1\in\mathcal{N}''_\low(C)$ by construction, this enumeration is such that $J_1=j_1$ and there exists no $j\in\mathcal{N}'_\low(C)$ satisfying $j<J_1$, which is point $a.$ of the third statement. Furthermore, if $n''_\low\geq 2$ we observe the following. Let $i\in\llbracket 2,n''_\low \rrbracket$ and $j_m\in \mathcal{N}'_\low(C)$ for some $m\in\llbracket 1,n'_\low\rrbracket$. Thanks to \eqref{flag 1}, there exists $M\in\llbracket 1,n''_\low \rrbracket$ such that $J_M=j_{\max P\cap\llbracket 1,m\rrbracket}$ and
\begin{align}\label{flag 3}
\mathfrak{K}_\bush(j_m) \in \mathrm{Vect} \enstq{ \mathfrak{K}_\bush(J_{i'})}{i'\in \llbracket 1,M \rrbracket}. 
\end{align} 
Assume that $j_m<J_i$ and write $J_i=j_{m'}$ for $m'\in\llbracket 1,n'_\low\rrbracket$. Since $j_m<j_{m'}$ we have that $m<m'$, hence $\max P\cap\llbracket 1,m\rrbracket<m'$, which in turn implies $M<i$. This implies $M\leq i-1$ and together with \eqref{flag 3} this concludes the proof of \eqref{flag 2 énoncé}.
\end{proof}

The third point of the previous lemma allows us to rewrite \eqref{estimée intermédiaire FC low nodes} in a more suitable form.

\begin{lemma}\label{lem:summationslow}
There exist functions $\Xi_{J_i}$ for $i\in\llbracket 1,n''_\low\rrbracket$ such that
\begin{align}
\sup_{t\in[0,\de]}\left| \widehat{F_C}(\e^{-2}t,k) \right| & \leq \La^{n+1}\varepsilon^{n_B(C)} \delta^{n_T(C)} L^{-\frac{nd}{2}}\label{estimée intermédiaire FC low nodes 2}
\\&\quad \times \sum_{\substack{\ka\in\mathcal{D}_k(C)\\ \ka(\mathcal{L}(C)_+)\subset B(0,R)}} \prod_{i=1}^{n''_\low}  \chi\pth{ \left| \mathfrak{K}_\bush(J_{i})(\ka) + \Xi_{J_{i}}\pth{ \pth{ \mathfrak{K}_\bush(J_{i'})(\ka) }_{i'\in\llbracket 1,i-1\rrbracket} } \right| \e^{-\ga} }.\non
\end{align}
\end{lemma}

\begin{proof}
We start from \eqref{estimée intermédiaire FC low nodes} and first bound by $1$ the $\chi_j$ for $j\in\mathcal{N}_\low(C)\setminus\mathcal{N}''_\low(C)$. Using in addition \eqref{lien offsprings bushes} we obtain
\begin{align*}
&\sup_{t\in[0,\de]}\left| \widehat{F_C}(\e^{-2}t,k) \right| 
\\& \leq \La^{n+1}\varepsilon^{n_B(C)} \delta^{n_T(C)} L^{-\frac{nd}{2}} 
\\&\quad \times \sum_{\substack{\ka\in\mathcal{D}_k(C)\\ \ka(\mathcal{L}(C)_+)\subset B(0,R)}} \prod_{i=1}^{n''_\low} \chi\pth{ \left| \mathfrak{K}_\bush(J_{i})(\ka) + \sum_{\substack{ j\in\mathcal{N}'_\low(C) \\ \text{$j< J_{i}$ and $j\nleq J_{i,\mathrm{unmarked}}$} }} \mathfrak{K}_\bush(j)(\ka) \right| \e^{-\ga} },
\end{align*}
where $J_{i,\mathrm{unmarked}}$ is the unmarked children of $J_{i}$ and where we also used the fact that $\mathfrak{K}_\bush(j)=0$ if $j\in\mathcal{N}_\low(C)\setminus \mathcal{N}'_\low(C)$. The third point of the previous lemma shows that there exists a function $\Xi_{J_i}$ such that 
\begin{align*}
\sum_{\substack{ j\in\mathcal{N}'_\low(C) \\ \text{$j< J_{i}$ and $j\nleq J_{i,\mathrm{unmarked}}$} }} \mathfrak{K}_\bush(j)(\ka) = \Xi_{J_i}\pth{ \pth{\mathfrak{K}_\bush(J_{i'})(\ka)}_{i'\in\llbracket1,i-1\rrbracket} },
\end{align*}
with the convention that $\Xi_{J_1}=0$ since the sum on the left is a sum over the emptyset in the case $i=1$. This concludes the proof of the lemma.
\end{proof}

\subsubsection{A basis of $V(C)$ and a change of variable}\label{section basis of V(C) and change of variable}

Since the family of linear maps $\pth{ \mathfrak{K}_\bush(j) }_{j\in\mathcal{N}''_\low(C)}$ is linearly independent in $V(C)$, we can complete it into a basis. Given that $V(C)$ has dimension $\#\mathcal{L}(C)_+=\frac{n}{2}+1$, we set $p:= \frac{n}{2}+1 - n''_\low$ and let  
\begin{align*}
\mathfrak{L}\vcentcolon = \enstq{\ell_i}{i\in\llbracket 1,p \rrbracket}\subseteq  \mathcal{L}(C)_+
\end{align*}
be such that
\begin{align}\label{basis of V(C)}
\pth{\mathfrak{K}_\bush(j), K(\ell) \, | \, j\in \mathcal N''_\low(C),  \ell\in \mathfrak{L} }
\end{align}
forms a basis of $V(C)$. We will use this basis to make a change of variables in \eqref{estimée intermédiaire FC low nodes 2}.

We first benefit from the natural bijection between the set of decoration maps of $C$ and $(\R^d)^{\mathcal{L}(C)_+}$, given by the restriction to $\mathcal{L}(C)_+$ of all decorations, and rewrite the sum on $\ka$ in \eqref{estimée intermédiaire FC low nodes 2} as a sum over $B(0,R)^{\mathcal{L}(C)_+}$ (we still denote the elements of this set by $\ka$). This gives
\begin{align*}
&\sup_{t\in[0,\de]}\left| \widehat{F_C}(\e^{-2}t,k) \right| 
\\& \leq \La^{n+1}\varepsilon^{n_B(C)} \delta^{n_T(C)} L^{-\frac{nd}{2}}
\\&\quad \times \sum_{\ka\in B(0,R)^{\mathcal{L}(C)_+}} \prod_{i=1}^{n''_\low}  \chi\pth{ \left| \mathfrak{K}_\bush(J_{i})(\ka) + \Xi_{J_{i}}\pth{ \pth{ \mathfrak{K}_\bush(J_{i'})(\ka) }_{i'\in\llbracket 1,i-1\rrbracket} } \right| \e^{-\ga} } \mathbbm{1}_{K(r)(\ka)=k},
\end{align*}
where $r$ is the root of the first tree in the coupling $C$. The fact that \eqref{basis of V(C)} is a basis of $V(C)$ implies that the map
\begin{align*}
\Phi : \ka \in (\R^d)^{\mathcal{L}(C)_+} \longmapsto \pth{\mathfrak{K}_\bush(J_i)(\ka), K(\ell_{i'})(\ka) \, | \, i\in\llbracket 1,n''_\low \rrbracket, i'\in \llbracket 1,p\rrbracket  }\in(\R^d)^{\frac{n}{2}+1}
\end{align*}
is a bijection and moreover satisfies $\Phi\pth{ \pth{\Z_L^d}^{\mathcal{L}(C)_+} } \subset \pth{\Z_L^d}^{\frac{n}{2}+1}$ after the definition \eqref{def k bush} of $\mathfrak{K}_\bush$. This allows us to make the change of variable $\pth{x_1,\dots,x_{n''_\low},y_1,\dots,y_{p}}=\Phi(\ka)$ and obtain
\begin{equation}\label{estim FC intermediate}
\begin{aligned}
&\sup_{t\in[0,\de]}\left| \widehat{F_C}(\e^{-2}t,k) \right| 
\\& \leq \La^{n+1} \varepsilon^{n_B(C)} \delta^{n_T(C)} L^{-\frac{nd}{2}}
\\&\quad \times \sum_{\substack{\pth{x_1,\dots,x_{n''_\low}}\in\pth{\Z_L^d}^{n''_\low}\\\pth{y_1,\dots,y_{p}}\in \pth{\Z_L^d\cap B(0,R)}^{ p }}} \prod_{i=1}^{n''_\low}  \chi\pth{ \left| x_i + \Xi_{J_{i}}\pth{\pth{ x_{i'} }_{i'\in\llbracket 1,i-1\rrbracket}}  \right| \e^{-\ga} }  \mathbbm{1}_{\tilde{K}(r)\pth{x_1,\dots,x_{n''_\low},y_1,\dots,y_{p}}=k}
\end{aligned}
\end{equation}
where $\tilde{K}(r)$ is a map defined on $(\R^d)^{\frac{n}{2}+1}$ by 
\begin{align*}
\tilde{K}(r)\pth{x_1,\dots,x_{n''_\low},y_1,\dots,y_{p}} = K(r)\pth{ \Phi^{-1}\pth{x_1,\dots,x_{n''_\low},y_1,\dots,y_{p}}}.
\end{align*}
Since $K(r)\in V(C)$, there exist some real numbers $(a_i)_{i\in\llbracket 1,n''_\low\rrbracket}$ et $(b_i)_{i\in\llbracket 1,p\rrbracket}$ such that
\begin{align*}
K(r) = \sum_{i=1}^{n''_\low}a_i\mathfrak{K}_\bush(J_i) + \sum_{i=1}^p b_i K(\ell_i).
\end{align*}
This implies the following expression for $\tilde{K}(r)$:
\begin{align*}
\tilde{K}(r)\pth{x_1,\dots,x_{n''_\low},y_1,\dots,y_{p}} & = \sum_{i=1}^{n''_\low}a_ix_i + \sum_{i=1}^p b_i y_i,
\end{align*}
where we also used the definition of the change of variable $\Phi$. Therefore \eqref{estim FC intermediate} becomes
\begin{equation}\label{estim FC intermediate bis}
\begin{aligned}
&\sup_{t\in[0,\de]}\left| \widehat{F_C}(\e^{-2}t,k) \right| 
\\& \leq \La^{n+1} \varepsilon^{n_B(C)} \delta^{n_T(C)} L^{-\frac{nd}{2}}
\\&\quad \times \sum_{\substack{\pth{x_1,\dots,x_{n''_\low}}\in\pth{\Z_L^d}^{n''_\low}\\\pth{y_1,\dots,y_{p}}\in \pth{\Z_L^d\cap B(0,R)}^{ p }}}  \prod_{i=1}^{n''_\low}  \chi\pth{ \left| x_i + \Xi_{J_{i}}\pth{\pth{ x_{i'} }_{i'\in\llbracket 1,i-1\rrbracket}}  \right| \e^{-\ga} }  \mathbbm{1}_{\sum_{i=1}^{n''_\low}a_ix_i + \sum_{i=1}^p b_i y_i=k}.
\end{aligned}
\end{equation}

\subsubsection{Proof of Proposition \ref{prop low nodes}}\label{section successive summations}

We are now ready to conclude the proof of Proposition \ref{prop low nodes}.

\begin{proof}[Proof of Proposition \ref{prop low nodes}] Starting from \eqref{estim FC intermediate bis}, we distinguish two cases. The first case is when $a_i=0$ for all $i\in\llbracket 1,n''_\low \rrbracket$. If this holds, then the situation where all the $b_i$'s vanish, i.e when $\tilde{K}(r)=0$, gives a zero contribution because of the characteristic function in \eqref{estim FC intermediate bis} and $k\neq 0$. Therefore we can assume that there exists $i_0\in\llbracket 1,p\rrbracket$ such that $b_{i_0}\neq0$. This implies that
\begin{align*}
&\sum_{\pth{y_1,\dots,y_{p}}\in \pth{\Z_L^d\cap B(0,R)}^{ p }} \mathbbm{1}_{\tilde{K}(r)\pth{x_1,\dots,x_{n''_\low},y_1,\dots,y_{p}}=k} 
\\&\qquad = \sum_{(y_i)_{i\in\llbracket 1,p\rrbracket\setminus\{i_0\}}\in \pth{\Z_L^d\cap B(0,R)}^{ p-1 }} \mathbbm{1}_{b_{i_0}y_{i_0}=k-\sum_{i\in\llbracket 1,p\rrbracket\setminus\{i_0\}} b_i y_i}.
\end{align*}
We thus obtain
\begin{align}\label{estim somme tilde K =k}
\sum_{\pth{y_1,\dots,y_{p}}\in \pth{\Z_L^d\cap B(0,R)}^{ p }} \mathbbm{1}_{\tilde{K}(r)\pth{x_1,\dots,x_{n''_\low},y_1,\dots,y_{p}}=k} \lesssim \La^{(p-1)d}L^{(p-1)d}R^{(p-1)d}.
\end{align}
Plugging \eqref{estim somme tilde K =k} into \eqref{estim FC intermediate} gives
\begin{align}\label{estim FC intermediate 2}
\sup_{t\in[0,\de]}\left| \widehat{F_C}(\e^{-2}t,k) \right| & \leq \La^{n} \varepsilon^{n_B(C)} \delta^{n_T(C)} L^{\pth{p-1-\frac{n}{2}}d} 
\\&\quad \times \sum_{\pth{x_1,\dots,x_{n''_\low}}\in\pth{\Z_L^d}^{n''_\low}}\prod_{i=1}^{n''_\low}  \chi\pth{ \left| x_i + \Xi_{J_{i}}\pth{\pth{ x_{i'} }_{i'\in\llbracket 1,i-1\rrbracket}}  \right| \e^{-\ga} } ,\non
\end{align}
where we also absorbed $R$ into the universal constant $\La$. In order to estimate the sum over $\pth{\Z_L^d}^{n''_\low}$, we benefit from the triangular structure of the product and write
\begin{align*}
&\sum_{\pth{x_1,\dots,x_{n''_\low}}\in\pth{\Z_L^d}^{n''_\low}}\prod_{i=1}^{n''_\low}  \chi\pth{ \left| x_i + \Xi_{J_{i}}\pth{\pth{ x_{i'} }_{i'\in\llbracket 1,i-1\rrbracket}}  \right| \e^{-\ga} }
\\&\hspace{1cm} = \sum_{x_1\in \Z_L^d}\chi\pth{ \left| x_1 \right| \e^{-\ga} } \pth{ \cdots \pth{ \sum_{x_{n''_\low}\in\Z_L^d}\chi\pth{ \left| x_{n''_\low} + \Xi_{J_{n''_\low}}\pth{\pth{ x_{i'} }_{i'\in\llbracket 1,n''_\low-1\rrbracket}}  \right| \e^{-\ga} } }\cdots}.
\end{align*}
Since $L\e^{\ga}=\e^{\ga-\b}$ tends to infinity if $\e$ tends to 0 (since $\b>2$ and $\ga$ is small), for any $\xi\in\R^d$ we have 
\begin{align}\label{counting}
\#\enstq{ x \in \Z_L^d}{|x+\xi|\lesssim 2\e^\ga} \leq \La^d L^d \e^{\ga d},
\end{align}
we gain a factor $\La^d L^d \e^{\ga d}$ for each sum over $x_i\in \Z_L^d$ for $i\in\llbracket 1,n''_\low\rrbracket$. This implies
\begin{align*}
\sum_{\pth{x_1,\dots,x_{n''_\low}}\in\pth{\Z_L^d}^{n''_\low}}\prod_{i=1}^{n''_\low}  \chi\pth{ \left| x_i + \Xi_{J_{i}}\pth{\pth{ x_{i'} }_{i'\in\llbracket 1,i-1\rrbracket}}  \right| \e^{-\ga} } \leq \La^n L^{n''_\low d} \e^{\ga n''_\low d }.
\end{align*}
By combining this estimate with \eqref{estim FC intermediate 2} we obtain 
\begin{align}\label{estim FC low nodes A}
\sup_{t\in[0,\de]}\left| \widehat{F_C}(\e^{-2}t,k) \right| & \leq \La^{n+1} \e^{n_B(C)+\ga n''_\low d } \delta^{n_T(C)}.
\end{align}
where we crucially used that $p=\frac{n}{2}+1-n''_\low$ so that $L$ disappears from the estimate. 

\saut
The second case is when there exists at least one non-zero $a_i$. We set $i_0=\max\enstq{i\in\llbracket 1,n''_\low \rrbracket}{a_i\neq 0}$. Again, we start from \eqref{estim FC intermediate} and benefit from the triangular structure of the product to rewrite \eqref{estim FC intermediate} as
\begin{equation*}
\begin{aligned}
&\sup_{t\in[0,\de]}\left| \widehat{F_C}(\e^{-2}t,k) \right| 
\\& \leq \La^{n+1} \varepsilon^{n_B(C)} \delta^{n_T(C)} L^{-\frac{nd}{2}}
\\&\quad \times \sum_{\substack{\pth{y_1,\dots,y_{p}}\in \pth{\Z_L^d\cap B(0,R)}^{ p }\\ \pth{x_1,\dots,x_{i_0}}\in\pth{\Z_L^d}^{i_0} }}  \prod_{i=1}^{i_0}  \chi\pth{ \left| x_i + \Xi_{J_{i}}\pth{\pth{ x_{i'} }_{i'\in\llbracket 1,i-1\rrbracket}}  \right| \e^{-\ga} }  \mathbbm{1}_{a_{i_0}x_{i_0}  =k-f(y)-\sum_{i=1}^{i_0-1}a_ix_i}
\\&\quad \times \sum_{x_{i_0+1}\in\Z_L^d}\chi\pth{ \left| x_{i_0+1} + \Xi_{J_{i_0+1}}\pth{\pth{ x_{i'} }_{i'\in\llbracket 1,i_0\rrbracket}}  \right| \e^{-\ga} }
\\&\hspace{1cm}\times \pth{ \cdots \pth{ \sum_{x_{n''_\low}\in\Z_L^d}\chi\pth{\left| x_{n''_\low} + \Xi_{J_{n''_\low}}\pth{\pth{ x_{i'} }_{i'\in\llbracket 1,n''_\low-1\rrbracket}}  \right| \e^{-\ga}} } \cdots }
\end{aligned}
\end{equation*}
where for clarity we used the notation $f(y)=\sum_{i=1}^p b_i y_i$. As above, each sum over the $x_i$'s for $i\in\llbracket i_0+1,n''_\low\rrbracket$ gives a factor $\La^d L^d \e^{\ga d}$ so that we obtain
\begin{equation*}
\begin{aligned}
&\sup_{t\in[0,\de]}\left| \widehat{F_C}(\e^{-2}t,k) \right| 
\\& \leq \La^{n+1} \delta^{n_T(C)}   L^{\pth{n''_\low - i_0-\frac{n}{2}}d} \e^{n_B(C)+\ga (n''_\low - i_0)d}
\\&\quad \times  \sum_{\substack{\pth{y_1,\dots,y_{p}}\in \pth{\Z_L^d\cap B(0,R)}^{ p }\\ \pth{x_1,\dots,x_{i_0}}\in\pth{\Z_L^d}^{i_0} }} \prod_{i=1}^{i_0}  \chi\pth{ \left| x_i + \Xi_{J_{i}}\pth{\pth{ x_{i'} }_{i'\in\llbracket 1,i-1\rrbracket}}  \right| \e^{-\ga} }  \mathbbm{1}_{a_{i_0}x_{i_0}  =k-f(y)-\sum_{i=1}^{i_0-1}a_ix_i}.
\end{aligned}
\end{equation*}
Since $a_{i_0}\neq 0$, the sum over $x_{i_0}$ does not contribute and if we simply bound the corresponding cutoff function and the characteristic function by $1$  we obtain
\begin{equation*}
\begin{aligned}
\sup_{t\in[0,\de]}\left| \widehat{F_C}(\e^{-2}t,k) \right| & \leq \La^{n+1} \delta^{n_T(C)}   L^{\pth{n''_\low - i_0-\frac{n}{2}}d} \e^{n_B(C)+\ga (n''_\low - i_0)d}
\\&\quad \times \sum_{\substack{\pth{y_1,\dots,y_{p}}\in \pth{\Z_L^d\cap B(0,R)}^{ p }\\ \pth{x_1,\dots,x_{i_0-1}}\in\pth{\Z_L^d}^{i_0-1} }} \prod_{i=1}^{i_0-1}  \chi\pth{ \left| x_i + \Xi_{J_{i}}\pth{\pth{ x_{i'} }_{i'\in\llbracket 1,i-1\rrbracket}}  \right| \e^{-\ga} } .
\end{aligned}
\end{equation*}
Using again the triangular structure for the product as in the first case, we finally obtain
\begin{align}
\sup_{t\in[0,\de]}\left| \widehat{F_C}(\e^{-2}t,k) \right| & \leq \La^{n+1} \delta^{n_T(C)}   L^{\pth{n''_\low - i_0-\frac{n}{2}+p + i_0-1}d} \e^{n_B(C)+\ga (n''_\low - i_0)d+\ga(i_0-1)d}\non
\\&\leq  \La^{n+1} \delta^{n_T(C)}  \e^{n_B(C)+\ga (n''_\low - 1)d},\label{estim FC low nodes B}
\end{align}
where we again used $p=\frac{n}{2}+1-n''_\low$.

\saut
We conclude the estimate of $\widehat{F_C}(\e^{-2}t,k)$, by distinguishing between $n'_\low$ vanishing or not. If $n'_\low=0$, then \eqref{estim FC low nodes A} holds with $n''_\low=0$. Then, using first $\e\leq \de$ and $n(C)=n_B(C)+2n_T(C)$ and then \eqref{low nodes regime} together with $n_\low-n_\scb=n'_\low=0$ we obtain
\begin{align*}
\sup_{t\in[0,\de]}\left| \widehat{F_C}(\e^{-2}t,k) \right| & \leq \La^{n+1} \e^{n_B(C) } \delta^{n_T(C)}
\\&\leq \La^{n+1} \e^{\frac{n_B(C)}{2} } \delta^{\frac{n(C)}{2}}
\\&\leq \La^{n+1} \e^{\frac{1}{10}\pth{\frac{n(C)}{2}-n_\scb(C)} } \delta^{\frac{n(C)}{2}}.
\end{align*}
If $n'_\low \geq 1$, then the estimates \eqref{estim FC low nodes A} and \eqref{estim FC low nodes B} show that we have
\begin{align*}
\sup_{t\in[0,\de]}\left| \widehat{F_C}(\e^{-2}t,k) \right| \leq  \La^{n+1} \delta^{\frac{n(C)}{2}}  \e^{\frac{n_B(C)}{2}+\ga (n''_\low - 1)d},
\end{align*}
where we already used $\e\leq \de$ and $n(C)=n_B(C)+2n_T(C)$. Using now the first point of Lemma \ref{lem basis of V(C)} we obtain
\begin{align*}
\frac{n_B(C)}{2}+\ga (n''_\low - 1)d & \geq \half \min\pth{ 1, \frac{\ga d}{2}}\pth{ n_B(C)+2n'_\low -4}
\\& \geq \half \min\pth{ 1, \frac{\ga d}{2}}\pth{ \frac{n(C)-2n_\scb(C)}{10} -4}.
\end{align*}
Using \eqref{n-2nscb>80} we obtain $\frac{n(C)-2n_\scb(C)}{10} -4\geq \frac{1}{10}\pth{ \frac{n(C)}{2}-n_\scb(C)}$ and therefore
\begin{align*}
\sup_{t\in[0,\de]}\left| \widehat{F_C}(\e^{-2}t,k) \right| \leq  \La^n \delta^{\frac{n(C)}{2}}  \e^{\frac{1}{20} \min\pth{ 1, \frac{\ga d}{2}}\pth{ \frac{n(C)}{2}-n_\scb(C)}}.
\end{align*}
Setting $\a_\low\vcentcolon = \min\pth{ \frac{1}{10}, \frac{1}{20} \min\pth{ 1, \frac{\ga d}{2}} }$ concludes the proof of Proposition \ref{prop low nodes}.
\end{proof}

\begin{remark}
As a byproduct of this proof, we can easily obtain the weak bound \eqref{estimate good}. Indeed, the estimates \eqref{estim FC low nodes A}-\eqref{estim FC low nodes B} (which don't yet make any use of the assumptions \eqref{low nodes regime}-\eqref{n-2nscb>80}) imply in particular that
\begin{align*}
\sup_{t\in[0,\de]}\sup_{k\in\Z_L^d}\left| \widehat{F_C}(\e^{-2}t,k) \right| \leq \La^{n+1} \e^{n_B(C)} \de^{n_T(C)}.
\end{align*}
Since $\e\leq \de$ and $n(C)=n_B(C)+2n_T(C)$ we clearly have $\e^{n_B(C)} \de^{n_T(C)}\leq \de^{\frac{n(C)}{2}}$.
\end{remark}

\subsection{High nodes regime}\label{section high nodes regime}

The goal of this section is to estimate $\widehat{F_C}$ in the high nodes regime.

\begin{prop}\label{prop high nodes}
There exists $\a_\high,\La>0$ such that if $C\in\mathcal{C}^{\eta_1,\eta_2,\iota_1,\iota_2}_{n_1,n_2}$ is a coupling satisfying 
\begin{align}
n_B(C) + 2(n_\low(C) - n_\scb(C)) & < \frac{n(C)-2n_\scb(C)}{10},\label{high nodes regime}
\\ n(C) - 2 n_\scb(C) & \geq 10,\label{n-2nscb>10}
\\ n(C) & \leq |\ln\e|^{3},\label{n<logepsilon}
\end{align}
then for any $\e\leq \de $ and any $k\in \Z_L^d$ we have
\begin{align}\label{main estim FC high nodes}
\sup_{t\in[0,\de]}\left|\widehat{F_C}(\e^{-2}t,k) \right| \leq \La (\La\delta)^{\frac{n(C)}{2}} \varepsilon^{\alpha \pth{\frac{n(C)}{2}-n_\scb(C)}}.
\end{align}
\end{prop}

As in the low nodes regime of Proposition \ref{prop low nodes}, \eqref{FC(0)=0} implies that \eqref{main estim FC high nodes} holds if $k=0$, so that we can already assume $k\neq 0$. Moreover, note that the high nodes regimes implies $n_T(C)\geq 1$, since otherwise \eqref{high nodes regime} would imply $n_B(C)<\frac{n_B(C)}{10}$.

\saut
Proposition \ref{prop high nodes} is proved in the following sections. In Section \ref{section oscillating integrals}, we express the oscillating integrals in terms of total orders of ternary nodes. In Section \ref{section induced orders}, we define total orders on the whole coupling based on total orders on ternary nodes. In Sections \ref{section N''} and \ref{subsubsec:changevar}, we give properties of a special family $\mathcal{N}''(C)$ and use it to perform a change of variable. In Section \ref{section individual sums} we estimate the sums originating in the oscillating integrals and finally in Section \ref{section proof high nodes} we conclude the proof of Proposition \ref{prop high nodes}.

\subsubsection{Representation of oscillating integrals}\label{section oscillating integrals}

In this subsubsection we explain how to pass from an oscillating integral to a manageable function of the dispersion relation. We first introduce notions related to order relations on finite sets. Recall that if $<_1$ and $<_2$ are two order relations (partial or not) on a given set, we say that $<_2$ is compatible with $<_1$ if $x<_1 y$ implies $x<_2 y$.

\begin{definition}\label{def mathfrak(M)}
Let $A$ be a nonempty finite set with a partial order relation $<_A$. 
\begin{itemize}
\item If $\rho:\llbracket 1,\#A\rrbracket\longrightarrow A$ is a bijection, we define the total order relation $<_\rho$ on $A$ by
\begin{align*}
a_1 <_\rho a_2 \iff \rho^{-1} (a_1)<\rho^{-1}(a_2).
\end{align*}
\item We denote by $\mathfrak{M}(A,<_A)$ the set of bijections $\rho : \llbracket 1, \#A \rrbracket \longrightarrow A$ such that $<_\rho$ is compatible with $<_A$.
\end{itemize}
\end{definition}

\begin{remark}\label{remark total order}
If $<$ is a total order relation on $A$, then $A$ can be written under the form 
\begin{align*}
A=\left\{ a_1 < \cdots < a_{\#A} \right\}
\end{align*}
If we define $\rho(i)=a_i$ for $i\in\llbracket 1,\#A\rrbracket$ then we have $<_\rho=<$. This shows that  
\begin{align*}
\left\{ \text{bijection from $\llbracket 1,\#A\rrbracket$ to $A$} \right\}  & \longrightarrow \left\{ \text{total order relation on $A$} \right\}
\\ \rho \qquad & \longmapsto \qquad <_\rho
\end{align*}
is a bijection. 
\end{remark}

In the sequel, the notation $\mathfrak{M}(A,<_A)$ will always be used with $A$ a subset of leaves and branching nodes of a coupling and with $<_A$ the natural parentality partial order. Therefore, we will use the shorter notation $\mathfrak{M}(A)$ instead of $\mathfrak{M}(A, <_A)$.

\begin{lemma}\label{lemma:oscillations} 
For any coupling $C$, any decoration $\kappa$ of $C$ and any real positive number $\delta>0$, we have 
\begin{align*}
&\int_{I_C(t)} \prod_{j\in \mathcal N_T(C)} e^{i\Omega_j \varepsilon^{-2}t_j} \d t_j 
\\&\hspace{2cm} = \sum_{\rho\in\mathfrak{M}\pth{\mathcal{N}_T(C)}} \frac{e^{n_T(C)\de^{-1} t}}{2\pi}  \int_\R \frac{ e^{ - i\xi t} }{ n_T(C)\de^{-1} - i\xi }\prod_{j\in\mathcal{N}_T(C)} \frac{1}{n_T(C)\delta^{-1}-i\pth{\xi + \omega^\rho_{j}}}  \d\xi,
\end{align*}
where 
\begin{align}\label{def omega rho j}
\omega^\rho_{j} \vcentcolon = \e^{-2}\sum_{k=\rho^{-1}(j)}^{n_T(C)} \Omega_{\rho(k)}.
\end{align}
\end{lemma}

\begin{proof}
To make notations lighter, we will simply denote $n_T(C)$ by $n_T$ in what follows. First, note that there exists a negligible set $\mathcal{Z}_C\subset [0,t]^{n_T}$ such that
\begin{align}\label{I_C(t) up to negligible}
I_C(t)  = \mathcal{Z}_C \sqcup \bigsqcup_{\rho\in\mathfrak{M}\pth{\mathcal{N}_T(C)}} \enstq{(t_j)_{j\in \mathcal N_T(C)}\in[0,t]^{n_T} }{ 0\leq t_{\rho(1)} < \cdots < t_{\rho(n_T)} \leq t}.
\end{align}
We thus get that
\begin{align*}
\int_{I_C(t)} \prod_{j\in \mathcal N_T(C)} e^{-i\Omega_j \varepsilon^{-2}t_j} \d t_j & = \sum_{\rho\in\mathfrak{M}\pth{\mathcal{N}_T(C)}} \int_{0\leq t_{\rho(1)} < \cdots < t_{\rho(n_T)} \leq t} \prod_{k=1}^{n_T} e^{i\Omega_{\rho(k)} \varepsilon^{-2}t_{\rho(k)}} \d t_{\rho(k)}.
\end{align*}
We make the change of variables $s_1=t_{\rho(1)}$ and $s_k=t_{\rho(k)}-t_{\rho(k-1)}$ for $k\in\llbracket 2 , n_T \rrbracket$. Using $\sum_{\ell=1}^ks_\ell = t_{\rho(k)}$ and the fact that the Jacobian determinant of this change of variables is 1 we obtain, for a given $\rho\in\mathfrak{M}\pth{\mathcal{N}_T(C)}$,:
\begin{align*}
\int_{0\leq t_{\rho(1)} < \cdots < t_{\rho(n_T)} \leq t} \prod_{k=1}^{n_T} e^{i\Omega_{\rho(k)} \varepsilon^{-2}t_{\rho(k)}} \d t_{\rho(k)} & = \int_{(\R_+)^{n_T}}\mathbbm{1}_{[0,t]}\pth{ \sum_{k=1}^{n_T}s_k}\prod_{k=1}^{n_T} e^{i\omega^\rho_{\rho(k)} s_k} \d s_k.
\end{align*}
where $\omega^\rho_{\rho(k)}$ is defined in \eqref{def omega rho j}.
For any $x>0$ and some fixed $\eta>0$, we now use the identity 
\[
\mathbbm{1}_{[0,t]}(x)=\int_\R \de(t'-x)\mathbbm{1}_{[-\infty,t]}(t')e^{\eta (t'-x)}\d t'
\]
(here $\de$ is the Dirac distribution) and Fubini's theorem to obtain
\begin{align*}
\int_{0\leq t_{\rho(1)} < \cdots < t_{\rho(n_T)} \leq t} &\prod_{k=1}^{n_T} e^{i\Omega_{\rho(k)} \varepsilon^{-2}t_{\rho(k)}} \d t_{\rho(k)} 
\\& = \int_{\R\times (\R_+)^{n_T}} \mathbbm{1}_{[-\infty,t]}(t')e^{\eta (t'-\sum_{k=1}^{n_T}s_k)}  \de\pth{t'-\sum_{k=1}^{n_T}s_k} \prod_{k=1}^{n_T} e^{i\omega^\rho_{\rho(k)} s_k} \d s_k \d t'.
\end{align*}
The change of variables $(t', s_1, \dots, s_{n_T})\mapsto (t'-\sum_{k=1}^{n_T}s_k, s_1, \dots, s_{n_T})$, followed by another application of Fubini's theorem, yields that what above
\[
= \int_\R \delta(t')  \underbrace{e^{\eta t'}\pth{\int_{(\R_+)^{n_T}} \mathbbm{1}_{[-\infty,t]}\pth{ t'+\sum_{k=1}^{n_T}s_k }\prod_{k=1}^{n_T} e^{i\omega^\rho_{\rho(k)} s_k} \d s_k }}_{F(t')} dt' = \langle \de, F\rangle_{L^2}
\]
Observe that $F, \hat{F}\in L^1$ since 
\[
\begin{aligned}
\hat{F}(\xi) & = \frac{1}{\sqrt{2\pi}} \int_\R e^{-(i\xi - \eta)t'}\pth{\int_{(\R_+)^{n_T}} \mathbbm{1}_{[-\infty,t]}\pth{ t'+\sum_{k=1}^{n_T}s_k }\prod_{k=1}^{n_T} e^{i\omega^\rho_{\rho(k)} s_k} \d s_k }dt' \\
& = \frac{1}{\sqrt{2\pi}} \int_{\R \times (\R_+)^{n_T}} e^{-(i\xi -\eta)(t'-\sum_{k=1}^{n_T} s_k )} \mathbbm{1}_{[-\infty, t]}(t') \prod_{k=1}^{n_T}e^{i\omega^\rho_{\rho(k)} s_k} \d s_k dt' \\
& = \frac{1}{\sqrt{2\pi}} \int_\R e^{-(i\xi -\eta)t'}\mathbbm{1}_{[-\infty, t]}(t') dt'\,  \prod_{k=1}^{n_T} \int_{\R_+} e^{(i\xi -\eta + i\omega^\rho_{\rho(k)})s_k} ds_k \\
& = \frac{1}{\sqrt{2\pi}} \frac{e^{(\eta-i\xi)t}}{\eta-i\xi }\prod_{k=1}^{n_T} \frac{1}{\eta- i(\xi+ \omega^\rho_{\rho(k)})}
\end{aligned}
\]
so recalling that $\sqrt{2\pi}\hat{\delta}=1$ we finally get
\[
\langle \de, F\rangle_{L^2} = \frac{1}{2\pi}\int_\R \frac{e^{(\eta -i\xi )t}}{\eta- i\xi }\prod_{k=1}^{n_T} \frac{1}{ \eta-i(\xi+ \omega^\rho_{\rho(k)})}\, d\xi.
\]
The choice $\eta = n_T \delta^{-1}$ concludes the proof.
\end{proof}

\begin{remark}
Note that $\om^\rho_j$ defined in \eqref{def omega rho j} can alternatively be expressed as
\begin{align}\label{alternative omega rho j}
\omega^\rho_j = \e^{-2} \sum_{\substack{j'\in\mathcal{N}_T(C)\\ j\leq_\rho j'}}\Om_{j'}.
\end{align}
\end{remark}

This representation formula for oscillating integrals already allows us to start the estimate of $\widehat{F_C}(\e^{-2}t,k)$.

\begin{lemma}\label{lem:holder}
The following decomposition holds
\begin{align}\label{decomposition FC}
\widehat{F_C}(\e^{-2}t,k)  =  \sum_{\rho\in\mathfrak{M}\pth{\mathcal{N}_T(C)}} \widehat{F^\rho_C}(\e^{-2}t,k) 
\end{align}
where $\widehat{F^\rho_C}(\e^{-2}t,k)$ satisfies
\begin{align}\label{estim FC en fonction A}
\sup_{t\in[0,\de]}\left| \widehat{F^\rho_C}(\e^{-2}t,k) \right| & \lesssim \La^{n+1} \varepsilon^{n_B(C)} \int_\R \frac{ A_C^\rho(k,\xi)^{\frac{1}{p}} }{ \left| n_T(C)\de^{-1} - i\xi \right|}   \d\xi
\end{align}
for some $p>2d$ and where we defined
\begin{align}\label{def ACrho}
A_C^\rho(k,\xi) & \vcentcolon = L^{-\frac{nd}{2}} \sum_{\substack{\ka\in\mathcal D_k(C)\\ \ka(\mathcal{L}(C)_+)\subset B(0,R) }} \prod_{j\in (\mathcal{N}(C)\sqcup\mathcal{L}(C))\setminus \mathcal{R}(C)} \ps{\ka(j)}^{-\frac{p}{2}}  \prod_{j\in\mathcal{N}_T(C)} \frac{\chi_j^p}{\left|n_T(C)\delta^{-1}-i(\xi + \omega^\rho_{j})\right|^p} .
\end{align}
\end{lemma}

\begin{proof}
Applying Lemma \ref{lemma:oscillations} to \eqref{fourier of FC} we obtain the decomposition \eqref{decomposition FC} with
\begin{align*}
\widehat{F^\rho_C}(\e^{-2}t,k)  & \vcentcolon =  \varepsilon^{n_B(C)} (-i)^{n_T(C)}L^{-\frac{nd}{2}} \sum_{\ka\in\mathcal D_k(C)} e^{-i\Omega_{-1}\varepsilon^{-2}t}\prod_{j\in \mathcal N(C)} q_j 
\\&\quad  \times \frac{e^{n_T(C)\de^{-1} t}}{2\pi}  \int_\R \frac{ e^{ - i\xi t} }{ n_T(C)\de^{-1} - i\xi }\prod_{j\in\mathcal{N}_T(C)} \frac{1}{n_T(C)\delta^{-1}-i(\xi + \omega^\rho_{j})}   \d\xi
\\&\quad \times \prod_{\ell\in\mathcal{L}(C)_-}  M^{c(\ell),c(\si(\ell))}(\ffi(\ell)\ka(\ell))^{\ffi(\ell)}.
\end{align*}
We use the fact that the maps $M^{\eta,\eta'}$ are continuous with compact support included in $B(0,R)$ and the results of Lemma \ref{lem qj} to obtain 
\begin{align*}
\sup_{t\in[0,\de]}\left| \widehat{F^\rho_C}(\e^{-2}t,k) \right| & \leq \La^{n+1} \varepsilon^{n_B(C)} L^{-\frac{nd}{2}} \int_\R \frac{ f_C^\rho(\xi) }{ \left| n_T(C)\de^{-1} - i\xi \right|}  \d\xi,
\end{align*}
where we used $t\leq \de$ to absorb the term $e^{n_T(C)\de^{-1} t}$ in the universal constant $\La$ and where we used the notations
\begin{align*}
f_C^\rho(\xi) & \vcentcolon = \sum_{\substack{\ka\in\mathcal D_k(C)\\ \ka(\mathcal{L}(C)_+)\subset B(0,R) }} f_C^\rho(\xi,\ka) ,
\\ f_C^\rho(\xi,\ka) & \vcentcolon = \prod_{j\in (\mathcal{N}(C)\sqcup\mathcal{L}(C))\setminus \mathcal{R}(C)} \ps{\ka(j)}^{-\half}  \prod_{j\in\mathcal{N}_T(C)} \frac{\chi_j}{\left|n_T(C)\delta^{-1}-i(\xi + \omega^\rho_{j})\right|}.
\end{align*}
We now wish to apply Hölder's inequality on the sum over $\ka$ defining $f_C^\rho(\xi)$. For this, we proceed similarly to what done in Section \ref{section basis of V(C) and change of variable} and first rewrite $f_C^\rho(\xi)$ as 
\begin{align*}
f_C^\rho(\xi) & = \sum_{\substack{\ka\in B(0,R)^{\mathcal{L}(C)_+}\\K(r)(\ka)=k}}f_C^\rho(\xi,\ka),
\end{align*} 
where $r$ denotes the root of the first tree in the coupling $C$. Since $k\neq 0$, we can assume that there exist numbers $b_\ell$ non all vanishing such that $K(r)  = \sum_{\ell\in \mathcal{L}(C)_+}b_\ell K(\ell)$. By performing a change of variable similar to the one of Section \ref{section basis of V(C) and change of variable} and a counting argument similar to the one at the beginning of Section \ref{section successive summations} we can prove that
\begin{align*}
\#\enstq{\ka\in B(0,R)^{\mathcal{L}(C)_+}}{K(r)(\ka)=k} \leq \La^{n+1} L^{\frac{nd}{2}},
\end{align*}
where we also used $\#\mathcal{L}(C)_+=\frac{n}{2}+1$. We can thus apply Hölder's inequality with $\frac{1}{p}+\frac{1}{p'}=1$:
\begin{align*}
f_C^\rho(\xi) & \lesssim \La^{n+1} L^{\frac{nd}{2}}  \pth{ L^{-\frac{nd}{2}} \sum_{\substack{\ka\in\mathcal D_k(C)\\ \ka(\mathcal{L}(C)_+)\subset B(0,R) }} f_C^\rho(\xi,\ka)^p }^{\frac{1}{p}}.
\end{align*}
Plugging the expression of $f_C^\rho(\xi,\ka)$ into this inequality concludes the proof of the lemma.
\end{proof}

Proving Proposition \ref{prop high nodes} thus reduces to estimating $A_C^\rho(k,\xi)$ defined by \eqref{def ACrho}.

\subsubsection{Induced orders from $\mathcal{N}_T(C)$}\label{section induced orders}

The previous section shows that we need to estimate $A_C^\rho(k,\xi)$ for each $\rho\in\mathfrak{M}(\mathcal{N}_T(C))$. In this section, we define a map from $\mathfrak{M}(\mathcal{N}_T(C))$ to $\mathfrak{M}(\mathcal{N}(C) )$ and $\mathfrak{M}(\mathcal{N}(C) \sqcup \mathcal{L}(C))$ with particular properties. Following Remark \ref{remark total order}, it is equivalent to define a total order relation on a given set or a bijection enumerating this set. Moreover, in \eqref{def rho hat 1} and \eqref{def rho tilde 1} below we use the following abuse of notation: if $A$ is a set with a total order relation $<$, $A'\subset A$ is a subset of $A$ and $a\in A\setminus A'$, we write $A'<a$ if $a'<a$ for all $a'\in A$.

\begin{definition}\label{def rho hat}
Let $C$ be a coupling and $<_\rho$ a total order relation on  $\mathcal{N}_T(C)$, which we enumerate as
\begin{align*}
\mathcal{N}_T(C) = \{ j_1 <_\rho \cdots <_\rho j_{n_T(C)} \}.
\end{align*}
We define a total order relation $<_{\hat{\rho}}$ on $\mathcal{N}(C)$ in the following way:
\begin{align}\label{def rho hat 1}
\mathcal{N}(C) = \left\{ \mathtt{int}^{-1}(j_1)  <_{\hat{\rho}} j_1 <_{\hat{\rho}}  \cdots <_{\hat{\rho}} \mathtt{int}^{-1}(j_{n_T(C)}) <_{\hat{\rho}} j_{n_T(C)} <_{\hat{\rho}} \mathtt{int}^{-1}(-1) \right\}, 
\end{align}
where $\mathtt{int}^{-1}(j)$ for $j\in\mathcal{N}_T(C)\sqcup\{-1\}$ is defined in \eqref{def int-1} and where $<_{\hat{\rho}}$ restricted to each $\mathtt{int}^{-1}(j)$ is any total order relation compatible with the parentality order.
\end{definition}

The following lemma shows in particular that the previous definition defines a map from $\mathfrak{M}(\mathcal{N}_T(C))$ to $\mathfrak{M}(\mathcal{N}(C) )$.

\begin{lemma}\label{lem rho hat}
Let $C$ be a coupling, $\rho\in\mathfrak{M}(\mathcal{N}_T(C))$ and $\hat{\rho}$ defined from $\rho$ as in Definition \ref{def rho hat}.
\begin{itemize}
\item[(i)] For all $j,j'\in \mathcal{N}(C)$ we have
\begin{align}
 \mathtt{int}(j') <_\rho \mathtt{int}(j) & \Longrightarrow j'<_{\hat{\rho}} j, \label{three cases induced on nodes 1}
\\  \mathtt{int}(j) = -1 \;\text{ and }\;  \mathtt{int}(j')\neq -1 & \Longrightarrow j'<_{\hat{\rho}} j, \label{three cases induced on nodes 2}
\\  \mathtt{int}(j)=\mathtt{int}(j')  \;\text{ and }\; j'<j & \Longrightarrow j'<_{\hat{\rho}} j. \label{three cases induced on nodes 3}
\end{align}
\item[(ii)] The relation $<_{\hat{\rho}}$ is compatible with the parentality order. 
\end{itemize}
\end{lemma}

\begin{proof}
Let $j,j'\in \mathcal{N}(C)$ such that $\mathtt{int}(j') <_\rho \mathtt{int}(j)$. We distinguish three cases:
\begin{itemize}
\item If $j,j'\in\mathcal{N}_T(C)$, the assumption rewrites $j'<_\rho j$ and by \eqref{def rho hat 1} we get $j'<_{\hat{\rho}} j$.
\item If $j\in\mathcal{N}_B(C)$ and $j'\in\mathcal{N}_T(C)$ then the assumption rewrites $j' <_\rho \mathtt{int}(j)$. By definition we have $j\in \mathtt{int}^{-1}(\mathtt{int}(j))$ so that the assumption and \eqref{def rho hat 1} also implies $j'<_{\hat{\rho}}j$.
\item If $j\in\mathcal{N}_T(C)$ and $j'\in\mathcal{N}_B(C)$ then the assumption rewrites $\mathtt{int}(j') <_\rho j$ and since $j'\in \mathtt{int}^{-1}(\mathtt{int}(j'))$, \eqref{def rho hat 1} again implies $j'<_{\hat{\rho}}j$.
\end{itemize}
In conclusion, this proves \eqref{three cases induced on nodes 1}. Moreover, if $j, j'\in\mathcal{N}(C)$ are such that $\mathtt{int}(j) = -1$ and $\mathtt{int}(j')\neq -1$, then $j\in \mathtt{int}^{-1}(-1)$ and $j'\notin \mathtt{int}^{-1}(-1)$ so that \eqref{def rho hat 1} implies $j'<_{\hat{\rho}} j$, thus proving \eqref{three cases induced on nodes 2}. Finally, \eqref{three cases induced on nodes 3} follows from the compatibility of $<_{\hat{\rho}}$ restricted to each $\mathtt{int}^{-1}(j)$ for $j\in\mathcal{N}_T(C)\sqcup\{-1\}$ with the parentality order. 

We now turn to the compatibility of $<_{\hat{\rho}}$ with the parentality order. Let $j,j'\in\mathcal{N}(C)$ such that $j'<j$. This already implies $\mathtt{int}(j')\leq \mathtt{int}(j)$ and we distinguish three cases:
\begin{itemize}
\item If $\mathtt{int}(j')< \mathtt{int}(j')$, then the compatibility of $<_{\rho}$ with the parentality order implies that $\mathtt{int}(j')<_\rho \mathtt{int}(j)$, which in turn implies $j'<_{\hat{\rho}}j$ according to \eqref{def rho hat 1}.
\item If $\mathtt{int}(j)=\mathtt{int}(j')$ and $j\in\mathcal{N}_T(C)$ and $j'\in\mathcal{N}_B(C)$, then $j'\in\mathtt{int}^{-1}(j)$ and \eqref{def rho hat 1} implies $j' <_{\hat{\rho}} j$.
\item If $\mathtt{int}(j)=\mathtt{int}(j')=\vcentcolon j''$ and $j,j'\in\mathcal{N}_B(C)$, then both $j,j'\in\mathtt{int}^{-1}(j'')$. The fact that $\mathtt{int}^{-1}(j'')$ has been ordered for $<_{\hat{\rho}}$ in a way compatible with the parentality order gives $j'<_{\hat{\rho}}j$.
\end{itemize}
This concludes the proof of the lemma.
\end{proof}

We now define a total order on $\mathcal{N}(C) \sqcup \mathcal{L}(C)$.

\begin{definition}\label{def rho tilde}
Let $C$ be a coupling and $<_{\hat{\rho}}$ a total order relation on $\mathcal{N}(C)$, which we enumerate as
\begin{align*}
\mathcal{N}(C) = \{ j_1 <_{\hat{\rho}} \cdots  <_{\hat{\rho}} j_{n(C)} \}.
\end{align*}
We define a total order relation $<_{\tilde{\rho}}$ on $\mathcal{N}(C)\sqcup \mathcal{L}(C)$ in the following way:
\begin{align}\label{def rho tilde 1}
\mathcal{N}(C)\sqcup \mathcal{L}(C) = \left\{ \children(j_1) <_{\tilde{\rho}} \cdots <_{\tilde{\rho}} \children(j_{n(C)}) <_{\tilde{\rho}} r_2 <_{\tilde{\rho}} r_1 \right\}, 
\end{align}
where $r_1$ and $r_2$ denote the roots of the first and second tree in $C$ respectively, and where $<_{\tilde{\rho}}$ restricted to each $\children(j)$ is arbitrary if $j\notin\mathcal{N}_\low(C)$ and defined by
\begin{align*}
\children(j) = \{ j_{\mathrm{unmarked}} <_{\tilde{\rho}} \{ \text{marked children of $j$} \} \}, 
\end{align*}
if $j\in\mathcal{N}_\low(C)$, where $j_{\mathrm{unmarked}}$ is the unmarked children of $j$ and $<_{\tilde{\rho}}$ restricted to $\{ \text{marked children of $j$} \}$ is arbitrary.
\end{definition}

The following lemma shows in particular that the combination of the two previous definitions defines a map from $\mathfrak{M}(\mathcal{N}_T(C))$ to $\mathfrak{M}(\mathcal{N}(C) \sqcup \mathcal{L}(C))$. 

\begin{lemma}\label{lem rho tilde}
Let $C$ be a coupling, $\rho\in\mathfrak{M}(\mathcal{N}_T(C))$, $\hat{\rho}$ defined from $\rho$ as in Definition \ref{def rho hat} and $\tilde{\rho}$ defined from $\hat{\rho}$ as in Definition \ref{def rho tilde}.
\begin{itemize}
\item[(i)] For all $j,j'\in(\mathcal{N}(C)\sqcup \mathcal{L}(C))\setminus \mathcal{R}(C)$ we have 
\begin{align}\label{j<roots}
j<_{\tilde{\rho}} r_2 <_{\tilde{\rho}} r_1,
\end{align}
where $r_1$ and $r_2$ denote the roots of the first and second tree in $C$ respectively, and
\begin{align}\label{parent rho 1 tilde rho}
\mathtt{parent}(j')<_{\hat{\rho}}\mathtt{parent}(j) \Longrightarrow j' <_{\tilde \rho} j,
\end{align}
and
\begin{equation}\label{last property induced orders}
\left.
\begin{aligned}
&\mathtt{parent}(j)=\mathtt{parent}(j')\in \mathcal N_\low(C)
\\ & \text{$j'$ is unmarked}
\\ & \text{$j$ is marked}
\end{aligned}
\right\}
\Longrightarrow j'<_{\tilde \rho} j.
\end{equation}
\item[(ii)] The relation $<_{\tilde{\rho}}$ is compatible with the parentality order.
\end{itemize}
\end{lemma}

\begin{proof}
The properties \eqref{j<roots}-\eqref{last property induced orders} directly follow from \eqref{def rho tilde 1} and the way that each $\children(j)$ for $j\in\mathcal{N}_\low(C)$ has been ordered for $<_{\tilde{\rho}}$. It remains to prove the compatibility of $<_{\tilde{\rho}}$ with the parentality order. For this, consider $j_1,j_2\in\mathcal{N}(C)\sqcup \mathcal{L}(C)$ and assume that $j_1 > j_2$. If $j_1\in\mathcal{R}(C)$, then $j_2\notin\mathcal{R}(C)$ and \eqref{def rho tilde 1} implies $j_2<_{\tilde{\rho}} j_1$. Now, if $j_1$ is not a root, $j_2$ is not a root either and $j_2 < j_1$ implies $\mathtt{parent}(j_2)< \mathtt{parent}(j_1)$. Thanks to the second point of Lemma \ref{lem rho hat}, we get $\mathtt{parent}(j_2) <_{\hat{\rho}} \mathtt{parent}(j_1)$ and thus $j_2 <_{\tilde{\rho}} j_1$ thanks to \eqref{def rho tilde 1}. This concludes the proof of the lemma.
\end{proof}

\begin{remark}
Note that the properties \eqref{three cases induced on nodes 1}-\eqref{three cases induced on nodes 3} are equivalent to the definition of $\hat{\rho}$ from $\rho$ in Definition \ref{def rho hat}. Similarly the properties \eqref{j<roots}-\eqref{last property induced orders} are equivalent to the definition of $\tilde{\rho}$ from $\hat{\rho}$ in Definition \ref{def rho tilde}. However, the compatibility of $\hat{\rho}$ and $\tilde{\rho}$ with the parentality order is a consequence of how they are defined but also of the compatibility of $\rho$ itself with the parentality order.
\end{remark}

Now that we can induce a total order $<_{\tilde{\rho}}$ on $\mathcal{N}(C) \sqcup \mathcal{L}(C)$ from any total order $<_{\rho}$ on $\mathcal{N}_T(C)$, we need to be able to extract a basis of $V(C)$ compatible with any such $<_{\tilde{\rho}}$. This is the content of Proposition \ref{prop LukSpo} below, which can be found in Section 5 of \cite{LS11}. Since the proof of Proposition \ref{prop LukSpo} requires some minor adaptations from \cite{LS11}, we postpone its proof to Appendix \ref{appendix LukSpo}. First a definition.

\begin{definition}
Let $C$ be a coupling. We say that $K\in V(C)$ is a \textbf{signed combination} of the finite family $(K(j))_{j\in J}$ of $V(C)$ if
\begin{align*}
K = \sum_{j\in J} s_j K(j)
\end{align*}
where $s_j \in \{-1,0,1\}$.
\end{definition}

\begin{remark}
Thanks to Definition \ref{def linear map}, for any $j\in\mathcal{N}(C)\sqcup \mathcal{L}(C)$, the linear map $K(j)$ is a signed combination of $(K(\ell))_{\ell\in\mathcal{L}(C)_+}$.
\end{remark}

\begin{prop} \label{prop LukSpo}
Let $C\in \mathcal{C}_{n_1,n_2}$ be a coupling and let
\begin{align*}
\tilde{\rho} : \llbracket 1,n(C) + n_\ell(C) \rrbracket \longrightarrow \mathcal{N}(C) \sqcup \mathcal{L}(C)
\end{align*}
be a bijection with $<_{\tilde{\rho}}$ its order relation. There exists a subset $\mathcal{N}'(C) \subset \mathcal{N}(C) \sqcup \mathcal{L}(C)$ such that
\begin{itemize}
\item $\#\mathcal{N}'(C)= \frac{n}{2}+1$,
\item the linear maps $(K(j))_{j\in \mathcal{N}'(C)}$ are linearly independent,
\item for all $j\in \pth{\mathcal{N}(C) \sqcup \mathcal{L}(C)}\setminus \mathcal{N}'(C)$, $K(j)$ is a signed combination of 
\begin{align*}
(K(j'))_{j'\in \mathcal{N}'(C), j <_{\tilde{\rho}} j'}.
\end{align*}
\end{itemize}
\end{prop}

\subsubsection{The family $\mathcal{N}''(C)$}\label{section N''}

The goal of this section is to make an appropriate change of variables in $A^\rho_C(k,\xi)$, whose expression we recall from \eqref{def ACrho}: 
\begin{align*}
A_C^\rho(k,\xi) & = L^{-\frac{nd}{2}} \sum_{\substack{\ka\in\mathcal D_k(C)\\ \ka(\mathcal{L}(C)_+)\subset B(0,R) }} \prod_{j\in (\mathcal{N}(C)\sqcup\mathcal{L}(C))\setminus \mathcal{R}(C)} \ps{\ka(j)}^{-\frac{p}{2}}  \prod_{j\in\mathcal{N}_T(C)} \frac{\chi_j^p}{\left|n_T\delta^{-1}-i\pth{\xi + \omega^\rho_{j}}\right|^p} ,
\end{align*}
where $\rho$ is a fixed element of $\mathfrak{M}(\mathcal{N}_T(C))$, and where from now on we write $n_T$ instead of $n_T(C)$ for brevity (the same applies for similar notations).

\saut
We define $\hat{\rho}$ from $\rho$ as in Definition \ref{def rho hat} and $\tilde{\rho}$ from $\hat{\rho}$ as in Definition \ref{def rho tilde}. We apply Proposition \ref{prop LukSpo} to $\tilde{\rho}$ and get a family $\mathcal{N}'(C)\subset \mathcal{N}(C)\sqcup\mathcal{L}(C)$ of cardinal $\frac{n}{2}+1$ such that the family $(K(j))_{j\in \mathcal{N}'(C)}$ is linearly independent and any $K(j)$ is a signed combination of $(K(j'))_{j'\in \mathcal{N}'(C),j \leq_{\tilde \rho}j'} $. From $\mathcal{N}'(C)$, we can define the important family $\mathcal{N}''(C)$.

\begin{definition}\label{def N'' Pi}
Let $\tilde{\rho}$ and $\mathcal{N}'(C)$ as above. We define
\begin{align}\label{def N''}
\mathcal{N}''(C) \vcentcolon = \enstq{\min_{<_{\tilde{\rho}}}\pth{ \mathcal{N}'(C)\cap\siblings(j) }}{j\in\mathcal{N}'(C)\setminus\mathcal{R}(C)}.
\end{align}
For $i=1,2,3$ we also define
\begin{align*}
\mathcal{P}_i \vcentcolon = \enstq{j\in \mathcal{N}(C)}{\#\pth{\mathtt{children}(j)\cap\mathcal{N}'(C)}=i}.
\end{align*}
\end{definition}

Roughly speaking, we go from $\mathcal{N}'(C)$ to its subset $\mathcal{N}''(C)$ by selecting the smallest (for the order defined by $\tilde{\rho}$) element of $\mathcal{N}'(C)$ in each $\children$ sets. We enumerate $\mathcal{N}''(C)$ as follows:
\begin{align*}
\mathcal{N}''(C)= \{ j_1 <_{\tilde{\rho}} \cdots <_{\tilde{\rho}} j_{n''} \}, \qquad n''\vcentcolon = \#\mathcal{N}''(C).
\end{align*}
The following lemma contains important properties concerning $\mathcal{N}''(C)$ and $\mathcal{P}_i$, which will be used throughout the next sections.

\begin{lemma}\label{lem N''}
The following holds:
\begin{itemize}
\item[(i)] If $r_1$ and $r_2$ denote the roots of the first and second tree respectively in the coupling $C$, then $r_1\in\mathcal{N}'(C)\setminus\mathcal{N}''(C)$ and $r_2\notin\mathcal{N}'(C)$;
\item[(ii)] The $\mathcal{P}_i$s satisfy 
\begin{align}
\mathcal{P}_3 & = \emptyset, \label{P3vide}
\\ \mathcal{P}_2 & \subset \mathcal{N}_T(C),\label{P2ternaire}
\\ \#\mathcal{P}_1+2\#\mathcal{P}_2 & = \frac{n}{2} \label{cardP1P2}
\end{align}
\item[(iii)] The set $\parent(\mathcal{N}''(C))$ can be ordered as follows
\begin{align}\label{parent(N'')}
\parent(j_1) <_{\hat{\rho}} \cdots <_{\hat{\rho}} \parent(j_{n''}),
\end{align}
and we have
\begin{align}
n'' & \geq \frac{n}{4} \label{card N''}
\end{align}
\item[(iv)] There exists at least one $j\in\mathcal{N}''(C)$ such that $\parent(j)\in\mathcal{N}_T(C)$ and $\mathtt{int}^{-1}(\parent(j))=\emptyset$.
\end{itemize}
\end{lemma}

\begin{proof}
We prove each point of the lemma successively:
\begin{itemize}
\item Since $k\neq 0$, we can assume $K(r_1)\neq 0$. Therefore, if $r_1\notin\mathcal{N}'(C)$ then $K(r_1)$ is a signed combination of $(K(j'))_{j'\in \mathcal{N}'(C),r_1< _{\tilde \rho}j'} $. However, Lemma \ref{lem rho tilde} implies that the set $\enstq{ j \in \mathcal{N}(C)\sqcup\mathcal{L}(C)}{r_1<_{\tilde{\rho}} j}$ is empty. Therefore, we must have $r_1\in\mathcal{N}'(C)$. Moreover, since $K(r_2)=-K(r_1)$, we must also have $r_2\notin\mathcal{N}'(C)$. The fact that $r_1\notin \mathcal{N}''(C)$ follows from the definition of $\mathcal{N}''(C)$ itself.
\item Assume that there exists $j\in\mathcal{N}(C)$ such that $\mathtt{children}(j)\subset\mathcal{N}'(C)$. On the one hand, by definition of the decoration of a coupling we have 
\begin{align}\label{interA}
K(j)=\sum_{j'\in \mathtt{children}(j)}K(j'),
\end{align} 
which already implies that $j\notin\mathcal{N}'(C)$ since otherwise \eqref{interA} would contradict the linear independence of $(K(j'))_{j'\in \mathcal{N}'(C)}$. On the other hand, from the definition of $\mathcal{N}'(C)$ and $j\notin\mathcal{N}'(C)$, we also have
\begin{align}\label{interB}
K(j)=\sum_{\substack{j''\in\mathcal{N}'(C)\\ j<_{\tilde{\rho}}j''}} s_{j''}K(j''),
\end{align} 
for some $s_{j''}\in\{-1,0,1\}$. Equating the RHS of \eqref{interA} and \eqref{interB} we obtain
\begin{align}\label{interC}
\sum_{j'\in \mathtt{children}(j)}K(j') - \sum_{\substack{j''\in\mathcal{N}'(C)\\ j<_{\tilde{\rho}}j''}} s_{j''}K(j'') & = 0.
\end{align}
Since $\tilde{\rho}$ is compatible with the parentality order (recall Lemma \ref{lem rho tilde}), we have $\mathtt{children}(j)\cap \enstq{j''\in\mathcal{N}'(C)}{j<_{\tilde{\rho}}j''}=\emptyset$. Therefore \eqref{interC} contradicts again the linear independence of the family $(K(j'))_{j'\in \mathcal{N}'(C)}$. Therefore there exists no $j\in\mathcal{N}(C)$ such that $\mathtt{children}(j)\subset\mathcal{N}'(C)$, which proves simultaneously \eqref{P3vide} and \eqref{P2ternaire} since $j\in\mathcal{P}_2$ cannot belong to $\mathcal{N}_B(C)$. For \eqref{cardP1P2}, we apply Remark \ref{remark parent} to $\mathcal{S}=\mathcal{N}'(C)\setminus\{r_1\}$: thanks to \eqref{P3vide} and $\#\mathcal{N}'(C)=\frac{n}{2}+1$ we directly obtain \eqref{cardP1P2}.
\item According to the construction of $\tilde{\rho}$ from $\hat{\rho}$ and in particular the property \eqref{parent rho 1 tilde rho}, we have
\begin{align*}
\parent(j_1) \leq_{\hat{\rho}} \parent(j_2) \leq_{\hat{\rho}} \cdots \leq_{\hat{\rho}} \parent(j_{n''}).
\end{align*}
Moreover, since each $j\in\parent(\mathcal{N}''(C))$ has exactly one children in $\mathcal{N}''(C)$, the equality case in the above string is impossible so that \eqref{parent(N'')} holds. For \eqref{card N''}, note first that \eqref{parent(N'')} already implies that $n''=\#\parent(\mathcal{N}''(C))$. Now, since $\parent(\mathcal{N}'(C))=\mathcal{P}_1\sqcup\mathcal{P}_2$, \eqref{cardP1P2} implies 
\begin{align*}
n'' & = \#\mathcal{P}_1 + \#\mathcal{P}_2 = \frac{n}{4} + \half \#\mathcal{P}_1 \geq \frac{n}{4}.
\end{align*}
\item For the last point of the lemma, assume by contradiction that for all $j\in\mathcal{N}''(C)$ we have $j\in I\sqcup II$ where
\begin{align*}
I & \vcentcolon = \parent^{-1}(\mathcal{N}_B(C)), 
\\ II & \vcentcolon = \enstq{j\in\mathcal{N}(C)\sqcup\mathcal{L}(C)}{\parent(j)\in\mathcal{N}_T(C) \text{ and }\mathtt{int}^{-1}(\parent(j))\neq\emptyset}.
\end{align*}
We define a map $f$ from $\mathcal{N}''(C)$ to $\mathcal{N}_B(C)$ in the following way: if $j\in I$, set $f(j)=\parent(j)$, and if $j\in II$, set $f(j)$ any element of $\mathtt{int}^{-1}(\parent(j))$. Since $\parent_{|_{\mathcal{N}''(C)}}$ is an injection (which follows for instance from \eqref{parent(N'')}), we have $\#f^{-1}(j)\leq 2$ for every $j\in\mathcal{N}_B(C)$. This implies $n''\leq 2n_B$, and also $2n_B \geq \frac{n}{4}$ thanks to \eqref{card N''}. However, the high nodes regime assumption \eqref{high nodes regime} implies $n_B<\frac{n}{10}$, thus leading to a contradiction.
\end{itemize}
This concludes the proof of the lemma.
\end{proof}

In the next section, the sets $\mathcal{N}'(C)$ and $\mathcal{N}''(C)$ will allow us to define a well-chosen change of variables in $A_C^\rho(k,\xi)$. First, we rearrange the two products appearing in $A_C^\rho(k,\xi)$. For the first product, we use $\ps{k}\geq \sqrt{m}$ and $r_2\notin\mathcal{N}'(C)$ to obtain
\begin{align*}
\prod_{j\in (\mathcal{N}(C)\sqcup\mathcal{L}(C))\setminus \mathcal{R}(C)} \ps{\ka(j)}^{-\frac{p}{2}} & \leq \La^{n+1} \prod_{j\in \mathcal{N}'(C)\setminus ( \mathcal{N}''(C)\sqcup \{r_1\}) } \ps{\ka(j)}^{-\frac{p}{2}} \times \prod_{j\in \mathcal{N}''(C) } \ps{\ka(j)}^{-\frac{p}{2}}. 
\end{align*}
For the second product, we use $\left|n_T\delta^{-1}-i\pth{\xi + \omega^\rho_{j}}\right|\geq n_T\delta^{-1}$ and $\chi_j\leq 1$ to obtain
\begin{align*}
\prod_{j\in\mathcal{N}_T(C)} \frac{\chi_j^p}{\left|n_T\delta^{-1}-i\pth{\xi + \omega^\rho_{j}}\right|^p} & \leq  \pth{\frac{\de}{n_T}}^{p\pth{n_T-n''_T}} \prod_{j\in\mathcal{N}_T(C)\cap\parent(\mathcal{N}''(C))} \frac{\chi_j^p}{\left|n_T\delta^{-1}-i\pth{\xi + \omega^\rho_{j}}\right|^p} 
\\&\leq \pth{\frac{\de}{n_T}}^{p\pth{n_T-n''_T}} \prod_{\substack{j\in\mathcal{N}''(C)\\ \parent(j)\in\mathcal{N}_T(C)}} \frac{\chi_{\parent(j)}^p}{\left|n_T\delta^{-1}-i\pth{\xi + \omega^\rho_{\parent(j)}}\right|^p},
\end{align*}
where $n''_T\vcentcolon = \#\pth{\mathcal{N}_T(C)\cap\parent(\mathcal{N}''(C))}$ and where we changed variables in the product over $\mathcal{N}_T(C)\cap\parent(\mathcal{N}''(C))$ (using the fact that each $j\in\parent(\mathcal{N}''(C))$ has exactly one child in $\mathcal{N}''(C)$). Putting everything together, we get
\begin{align}\label{ACrho avec Aj}
A_C^\rho(k,\xi) \leq \La^{n+1}  L^{-\frac{nd}{2}} \pth{\frac{\de}{n_T}}^{p\pth{n_T-n''_T}} \sum_{\substack{\ka\in\mathcal D_k(C)\\ \ka(\mathcal{L}(C)_+)\subset B(0,R) }}  \prod_{j\in \mathcal{N}'(C)\setminus ( \mathcal{N}''(C)\sqcup \{r_1\}) } \ps{\ka(j)}^{-\frac{p}{2}}  \prod_{j\in\mathcal{N}''(C)} A_j(\ka,\xi) ,
\end{align}
where $A_j(\ka,\xi)$ is defined by
\begin{equation}\label{def Aj}
A_j(\ka,\xi) \vcentcolon =
\left\{
\begin{aligned}
\ps{\kappa(j)}^{-\frac{p}{2}} &, \quad \text{if $\mathtt{parent}(j)\in\mathcal{N}_B(C)$,}
\\ \ps{\kappa(j)}^{-\frac{p}{2}} \chi_{\parent(j)}^p \pth{\frac{\delta}{n_T}}^p & , \quad \text{if $\parent(j)\in\mathcal{N}_T(C)$ and $\mathtt{int}^{-1}(\parent(j))\neq\emptyset$,}
\\ \frac{\chi_{\parent(j)}^p \ps{\ka(j)}^{-\frac{p}{2}}}{\left|n_T\delta^{-1}-i\pth{\xi + \omega^\rho_{\parent(j)}}\right|^p} & , \quad \text{if $\parent(j)\in\mathcal{N}_T(C)$ and $\mathtt{int}^{-1}(\parent(j))=\emptyset$.}
\end{aligned}
\right.
\end{equation}

\subsubsection{A change of variable in $A_C^\rho(k,\xi) $}\label{subsubsec:changevar}

The following lemma exhibits an important structure in the quantities $A_{j_i}(\ka,\xi)$ for $j_i\in\mathcal{N}''(C)$, which is mainly a consequence of the definition of $\hat{\rho}$ and $\tilde{\rho}$.

\begin{lemma}\label{triangle}
For $i\in\llbracket 1,n''\rrbracket$, $A_{j_i}(\ka,\xi)$ is a function of $(K(j_{i'})(\kappa))_{i\leq i'\leq n''}$ and $(K(j)(\kappa))_{j \in \mathcal{N}'(C)\setminus \mathcal{N}''(C)}$.
\end{lemma}

\begin{proof}
We fix $i\in\llbracket 1,n''\rrbracket$ and first prove that each $K(j)$ for $j\in\siblings(j_i)$ is a function of 
\begin{align*}
\pth{ \pth{K(j_{i'})}_{i\leq i'\leq n''}, \pth{K(j)}_{j\in\mathcal{N}'(C)\setminus\mathcal{N}''(C)} }.
\end{align*}
This is obviously true for $j_i$, so consider $j'\in\siblings(j_i)\setminus\{j_i\}$. If $j'\in\mathcal{N}'(C)$, then $j'\notin\mathcal{N}''(C)$ since $j'\neq j_i$ and $j_i$ cannot have another sibling in $\mathcal{N}''(C)$. The claim thus holds in the case $j'\in\mathcal{N}'(C)$. If $j'\notin\mathcal{N}'(C)$, then by definition of $\mathcal{N}'(C)$, $K(j')$ is a signed combination of $\pth{ K(j'') }_{j''\in \mathcal{N}'(C), j'' >_{\tilde{\rho}} j'}$. However, if $j''\in\mathcal{N}''(C)$ is such that $j'' >_{\tilde{\rho}} j'$, then either $j''\in\siblings(j')\cap\mathcal{N}''(C)=\{j_i\}$ or $j''>_{\tilde{\rho}} j_i$ (recall \eqref{def rho tilde 1}) so that $j=j_{i'}$ with $i'>i$. The claim thus also holds in the case $j'\notin\mathcal{N}'(C)$.

In order to prove the lemma, it remains to prove that $A_{j_i}(\ka)$ only depends on $K(j)$ for $j\in\siblings(j_i)$. This is obvious if $\parent(j_i)\in\mathcal{N}_B(C)$. In the case where $\parent(j_i)\in\mathcal{N}_T(C)$ and $\mathtt{int}^{-1}(\parent(j_i))\neq\emptyset$, then $A_{j_i}(\ka)$ involves also $\chi_{\parent(j_i)}$, which thanks to \eqref{chi j} depends on $K(j)$ for $j\in\siblings(j_i)$. Finally, if $\parent(j_i)\in\mathcal{N}_T(C)$ and $\mathtt{int}^{-1}(\parent(j_i))=\emptyset$, then $A_{j_i}(\ka)$ involves also $\omega^\rho_{\parent(j_i)}$. From \eqref{alternative omega rho j} we obtain
\begin{align*}
\e^2\omega^\rho_{\parent(j_i)} = \De_{\parent(j_i)} + \sum_{\substack{j'\in\mathcal{N}_T(C)\\ j'>_\rho \parent(j_i)}}\Om_{j'},
\end{align*}
where we also used \eqref{Omegaj} and $\mathtt{int}^{-1}(\parent(j_i))=\emptyset$ to write $\Om_{\parent(j_i)}=\De_{\parent(j_i)}$. Using \eqref{Omegaj} again we get
\begin{align*}
\e^2\omega^\rho_{\parent(j_i)} = \De_{\parent(j_i)} + \sum_{\substack{j'\in\mathcal{N}(C)\\\mathtt{int}(j')>_{\rho} \parent(j_i)}}\De_{j'}.
\end{align*}
By definition of $\hat{\rho}$ (recall \eqref{def rho hat 1}), $\mathtt{int}(j')>_{\rho} \parent(j_i)$ is equivalent to $j'>_{\hat{\rho}} \parent(j_i)$ so that
\begin{align*}
\e^2\omega^\rho_{\parent(j_i)} = \De_{\parent(j_i)} + \sum_{\substack{j'\in\mathcal{N}(C)\\ j'>_{\hat{\rho}} \parent(j_i)}}\De_{j'}.
\end{align*}
First, note that \eqref{def De j} and $K(\parent(j_i))=\sum_{j\in\siblings(j_i)}K(j)$ imply that $\De_{\parent(j_i)}$ depends on $K(j)$ for $j\in\siblings(j_i)$. For the same reason, if $j'\in\mathcal{N}(C)$ then $\De_{j'}$ is a function of $K(j'')$ for $j''\in\children(j')$. Moreover, if $j'>_{\hat{\rho}} \parent(j_i)$, then \eqref{parent rho 1 tilde rho} implies that $j''>_{\tilde{\rho}}j_i$ for all $j''\in\children(j')$. Therefore, also using the definition of $\mathcal{N}'(C)$, $\De_{j'}$ is a function of $\pth{ \pth{K(j''')}_{j'''\in\mathcal{N}'(C), j'''>_{\tilde{\rho}}j_i} }$, which concludes the proof of the lemma.
\end{proof}

As in Section \ref{section basis of V(C) and change of variable}, we will now perform a change of variables in the sum over $\ka$ in \eqref{ACrho avec Aj}, chosen so that we benefit from the structure exhibited in the previous lemma. Recall that a decoration map of a coupling can be seen as an element of $(\R^d)^{\mathcal{L}(C)_+}$. Therefore the sum over $\ka\in \mathcal{D}_k(C)$ with the constraint $\ka(\mathcal{L}(C)_+)\subset B(0,R)$ can be written as a sum over $\ka\in B(0,R)^{\mathcal{L}(C)_+}$ with the constraint that $K(r_1)(\ka)=k$. Now consider the map
\begin{align*}
\Psi : \ka\in (\R^d)^{\mathcal{L}(C)_+} \longmapsto \pth{ \pth{K(j_i)(\ka) }_{i\in\llbracket 1,n''\rrbracket} , \pth{K(j)(\ka)}_{j\in\mathcal{N}'(C)\setminus(\mathcal{N}''(C)\sqcup\{r_1\})}, K(r_1)(\ka) } \in (\R^d)^{\frac{n}{2}+1}.
\end{align*}
Since $r_1\in\mathcal{N}'(C)$ and $\pth{ K(j)}_{j\in\mathcal{N}'(C)}$ is a basis of $V(C)$, $\Psi$ is a bijection. Moreover it satisfies $\Psi\pth{ \pth{\Z_L^d}^{\mathcal{L}(C)_+} } \subset \pth{\Z_L^d}^{\frac{n}{2}+1}$. This allows us to make the change of variables
\begin{align*}
\pth{ x_1,\dots, x_{n''}, (k_j)_{j\in\mathcal{N}'(C)\setminus(\mathcal{N}''(C)\sqcup\{r_1\})}, k_{r_1}}=\Psi(\ka)
\end{align*}
and obtain
\begin{align*}
&A_C^\rho(k,\xi) 
\\& \leq \La^{n+1}  L^{-\frac{nd}{2}} \pth{\frac{\de}{n_T}}^{p\pth{n_T-n''_T}}  \sum_{\substack{(k_j)_{j\in\mathcal{N}'(C)\setminus(\mathcal{N}''(C)\sqcup\{r_1\})}\in\pth{\Z_L^d}^{\frac{n}{2}-n''}\\ \pth{x_1,\dots,x_{n''}} \in\pth{\Z_L^d}^{n''}}}   \prod_{j\in \mathcal{N}'(C)\setminus ( \mathcal{N}''(C)\sqcup \{r_1\}) } \ps{k_j}^{-\frac{p}{2}}  \prod_{i=1}^{n''} \tilde{A}_{j_i} ,
\end{align*}
where we defined the quantities $\tilde{A}_{j_i}$ by
\begin{equation}\label{def A tilde ji}
\begin{aligned}
&\tilde{A}_{j_i}\pth{(x_{i'})_{i\leq i'\leq n''},  (k_j)_{j\in\mathcal{N}'(C)\setminus(\mathcal{N}''(C)\sqcup\{r_1\})}, k,\xi} 
\\&\qquad \vcentcolon = A_{j_i}\pth{\Psi^{-1}\pth{ (x_{i'})_{1\leq i'\leq n''}, (k_j)_{j\in\mathcal{N}'(C)\setminus(\mathcal{N}''(C)\sqcup\{r_1\})}, k},\xi}.
\end{aligned}
\end{equation}
Note that in \eqref{def A tilde ji} we already took into account the constraint $k_{r_1}=k$ and also already used the result of Lemma \ref{triangle}, since $\tilde{A}_{j_i}$ does not depend on $x_{i'}$ if $i'<i$. This last fact allows us to rewrite the above estimate of $A_C^\rho(k,\xi)$ in the following way
\begin{equation}\label{ACrho successive sums}
\begin{aligned}
&A_C^\rho(k,\xi) 
\\& \leq \La^{n+1}  L^{-\frac{nd}{2}} \pth{\frac{\de}{n_T}}^{p\pth{n_T-n''_T}} 
\\&\quad  \sum_{(k_j)_{j\in\mathcal{N}'(C)\setminus(\mathcal{N}''(C)\sqcup\{r_1\})}\in\pth{\Z_L^d}^{\frac{n}{2}-n''}}   \prod_{j\in \mathcal{N}'(C)\setminus ( \mathcal{N}''(C)\sqcup \{r_1\}) } \ps{k_j}^{-\frac{p}{2}}
\\&\hspace{6cm} \times \pth{\sum_{x_{n''}\in\Z_L^d}\tilde{A}_{j_{n''}}\pth{  \cdots \pth{ \sum_{x_1\in\Z_L^d}\tilde{A}_{j_1} } \cdots    } }  .
\end{aligned}
\end{equation}

\subsubsection{Estimating individual sums}\label{section individual sums}

In this section, we estimate individually each sum $\sum_{x_i\in \Z_L^d}\tilde{A}_{j_i}$ appearing in \eqref{ACrho successive sums}. This will be done by highlighting the dependence of $\tilde{A}_{j_i}$ on the variable $x_i$. The estimates will crucially depend on the nature of $\parent(j_i)$ and we distinguish five cases:
\begin{align*}
&\text{$\mathbf{C1}$: $\parent(j_i)\in\mathcal{N}_B(C)$,}
\\&\text{$\mathbf{C2}$: $\parent(j_i)\in\mathcal{N}_\low(C)\cap\mathcal{P}_2$,}
\\&\text{$\mathbf{C3}$: $\parent(j_i)\in\mathcal{N}_\low(C)\cap\mathcal{P}_1$,}
\\&\text{$\mathbf{C4}$: $\parent(j_i)\in\mathcal{N}_\high(C)$ and $\mathtt{int}^{-1}(\parent(j_i))\neq \emptyset$,}
\\&\text{$\mathbf{C5}$: $\parent(j_i)\in\mathcal{N}_\high(C)$ and $\mathtt{int}^{-1}(\parent(j_i)) = \emptyset$.}
\end{align*}
Since $\parent(\mathcal{N}''(C))=\mathcal{P}_1\sqcup\mathcal{P}_2$, the five above cases cover all possibilities.

\saut
In Lemma \ref{lem individual sums} below, we will need to bound and approximate sums over $\Z_L^d$, this will be done with the following technical lemma.

\begin{lemma}\label{lem:RiemanntoInt}
Let $f:\R^d\longrightarrow \R_+$ be a smooth function and let $q>d$. 
\begin{itemize}
\item[(i)] If $f(x)\lesssim \ps{x}^{-q}$, then
\begin{align}\label{riemann 1}
    \sum_{x\in\Z_L^d}f(x) \lesssim L^d,
\end{align}
where the implicit constant only depends on $q$;
\item[(ii)] If $\int_{\R^d} f$ is finite and $C_{x, L} = B(x,\frac{1}{2L})$ is the $d$-dimensional ball for the $\infty$ norm, then
\begin{equation} \label{riemann2.1}
    \left| \frac{1}{L^d}\sum_{x\in\Z_L^d}f(x) - \int_{\R^d} f\right|  \lesssim \frac{1}{L^{d+1}} \sum_{x\in\Z_L^d}  \sup_{C_{x,L}} \left| \nabla f\right|.
\end{equation}
If in addition $|\nabla f(x)|\lesssim \ps{x}^{-q}$, then
\begin{align}\label{riemann 2}
\sum_{x\in\Z_L^d}f(x) \lesssim L^{d-1} + L^d \int_{\R^d} f,
\end{align}
where the implicit constant only depends on $q$.
\end{itemize}
\end{lemma}

\begin{proof}
The bound \eqref{riemann 1} is a consequence of the Poisson summation formula which implies
\begin{align*}
\sum_{x\in\Z_L^d}f(x) & \lesssim L^d \sum_{x\in L\mathbb{Z}^d} |h(x)|,
\end{align*}
where $h$ is the Fourier transform of $x\longmapsto \ps{x}^{-q}$. Since $q>d$, the function $x\longmapsto \ps{x}^{-q}$ and all its derivatives are in $L^1$ and thus $h$ decays faster than any polynomial, so that the sum $\sum_{x\in L\mathbb{Z}^d} |h(x)|$ can be bounded independently of $L$. Moreover, if $\int_{\R^d} f$ finite and denoting $C_{x,L}\vcentcolon=B\pth{ x,\frac{1}{2L}}$ we have
\begin{align*}
\left| \frac{1}{L^d}\sum_{x\in\Z_L^d}f(x) - \int_{\R^d} f\right| & \leq \sum_{x\in\Z_L^d}  \int_{C_{x,L}} \left| f(y)- f(x)\right|\d y 
\\&\lesssim \frac{1}{L^{d+1}} \sum_{x\in\Z_L^d}  \sup_{C_{x,L}} \left| \nabla f\right|
\end{align*}
If $|\nabla f(x)|\lesssim \ps{x}^{-q}$ this gives
\begin{align*}
\left| \frac{1}{L^d}\sum_{x\in\Z_L^d}f(x) - \int_{\R^d} f\right| & \lesssim  \frac{1}{L^{d+1}} \sum_{x\in\Z_L^d}  \ps{x}^{-q}
\\& \lesssim \frac{1}{L},
\end{align*}
where we applied \eqref{riemann 1} to $\ps{x}^{-q}$. This estimate imples \eqref{riemann 2} and concludes the proof of the lemma.
\end{proof}

\begin{lemma}\label{lem individual sums}
Let $i\in\llbracket 1,n''\rrbracket$ and fix $(x_{i'})_{i+1\leq i'\leq n''}\in\pth{\Z_L^d}^{n''-i}$, $(k_j)_{j\in \mathcal{N}'(C)\setminus(\mathcal{N}''(C)\sqcup\{r_1\})}\in\pth{\Z_L^d}^{\frac{n}{2}-n''}$ and $\xi\in\R$. The following estimate holds:
\begin{equation}\label{estimate sum Atildeji}
\sum_{x_i\in\Z_L^d}\tilde{A}_{j_i}\pth{x_i,(x_{i'})_{i+1\leq i'\leq n''},  (k_j)_{j\in\mathcal{N}'(C)\setminus(\mathcal{N}''(C)\sqcup\{r_1\})}, k,\xi}  \lesssim
\left\{
\begin{aligned}
L^d &, \quad \text{in case $\mathbf{C1}$,} 
\\ L^{d} \pth{\frac{\delta}{n_T}}^p &, \quad \text{in cases $\mathbf{C3}$ and $\mathbf{C4}$,} 
\\ L^{d} \pth{\frac{\delta}{n_T}}^p \e^{\a_0} &, \quad \text{in cases $\mathbf{C2}$ and $\mathbf{C5}$,} 
\end{aligned}
\right.
\end{equation}
for some $\a_0>0$ depending only on $\b$, $p$, $\ga$, $d$ and for $\e$ small enough compared to $\de$.
\end{lemma}

\begin{proof}
In this proof, we write $\tilde{A}_{j_i}$ instead of $\tilde{A}_{j_i}\pth{x_i,(x_{i'})_{i+1\leq i'\leq n''},  (k_j)_{j\in\mathcal{N}'(C)\setminus(\mathcal{N}''(C)\sqcup\{r_1\})}, k,\xi}  $. We start with the case $\mathbf{C1}$. According to \eqref{def A tilde ji} and \eqref{def Aj} we have $\tilde{A}_{j_i}=\ps{x_i}^{-\frac{p}{2}}$. Since $p>2d$ and thanks to \eqref{riemann 1} we obtain
\begin{align*}
\sum_{x_i\in\Z_L^d}\tilde{A}_{j_i} \lesssim L^d,
\end{align*}
which proves \eqref{estimate sum Atildeji} in the case $\mathbf{C1}$. We now consider the cases $\mathbf{C3}$ and $\mathbf{C4}$. From \eqref{def A tilde ji} and \eqref{def Aj}, and also using $\chi^p_{\parent(j_i)}\leq 1$ we simply obtain
\begin{align*}
\tilde{A}_{j_i} \lesssim \ps{x_i}^{-\frac{p}{2}} \pth{\frac{\de}{n_T}}^p.
\end{align*}
We proceed similarly as for the case $\mathbf{C1}$ and obtain
\begin{align*}
\sum_{x_i\in\Z_L^d}\tilde{A}_{j_i} \lesssim L^d \pth{\frac{\de}{n_T}}^p,
\end{align*}
which proves \eqref{estimate sum Atildeji} in the case $\mathbf{C3}$ and $\mathbf{C4}$.

\saut
We now consider the case $\mathbf{C2}$. First, \eqref{def A tilde ji} and \eqref{def Aj} implies that
\begin{align*}
\tilde{A}_{j_i} \lesssim \ps{x_i}^{-\frac{p}{2}} \chi^p_{\parent(j_i)} \pth{\frac{\de}{n_T}}^p,
\end{align*}
where we used $\left|n_T\delta^{-1}-i\pth{\xi - \omega^\rho_{\parent(j)}}\right|\leq n_T\de^{-1}$ if $\mathtt{int}^{-1}(\parent(j_i))=\emptyset$. By definition of the case $\mathbf{C2}$, we have $\siblings(j_i)=\{j_i,j_{i,1},j_{i,2}\}$ and we can assume that $j_{i,1}\in\mathcal{N}'(C)\setminus\mathcal{N}''(C)$ and $j_{i,2}\notin\mathcal{N}'(C)$. By definition of $\mathcal{N}''(C)$ we have $j_i<_{\tilde{\rho}}j_{i,1}$. We now distinguish two cases:
\begin{itemize}
\item If $j_i$ is marked, then according to \eqref{last property induced orders}, $j_{i,1}$ is marked. Thanks to Definitions \ref{def cutoff} and \ref{def marked} we have $\chi^p_{\parent(j_i)} \leq \mathbbm{1}_{|x_i + k_{j_{i,1}}|\leq 2 \e^\ga}$ and therefore
\begin{align*}
\sum_{x_i\in\Z_L^d}\tilde{A}_{j_i} & \lesssim \pth{\frac{\de}{n_T}}^p \#\enstq{ x_i \in \Z_L^d}{|x_i+k_{j_{i,1}}|\leq 2\e^\ga} 
\\& \lesssim \pth{\frac{\de}{n_T}}^p  L^d \e^{\ga d},
\end{align*}
where we used \eqref{counting}.
\item If $j_i$ is unmarked, then $j_{i,2}$ and $j_{i,2}$ must be the two marked children of $\parent(j_i)$ and thus Definitions \ref{def cutoff} and \ref{def marked} again implies $\chi^p_{\parent(j_i)} \leq \mathbbm{1}_{|K(j_{i,2})(\ka) + k_{j_{i,1}}|\leq 2 \e^\ga}$. Since $K(\parent(j_i))=K(j_i)+K(j_{i,1})+K(j_{i,2})$ this becomes $\chi^p_{\parent(j_i)} \leq \mathbbm{1}_{|K(\parent(j_i))(\ka) - x_i|\leq 2 \e^\ga}$. Since $\parent(j_i)>j_i$, the second point of Lemma \ref{lem rho tilde} implies that $\parent(j_i)>_{\tilde{\rho}}j_i$ and thus $K(\parent(j_i))(\ka)$ is a signed combination of $(x_{i'})_{i+1\leq i'\leq n''}$ and $(k_j)_{j\in \mathcal{N}'(C)\setminus(\mathcal{N}''(C))}$ and thus we again obtain
\begin{align*}
\sum_{x_i\in\Z_L^d}\tilde{A}_{j_i} & \lesssim  \pth{\frac{\de}{n_T}}^p\#\enstq{ x_i \in \Z_L^d}{|K(\parent(j_i))(\ka) - x_i|\lesssim 2\e^\ga} 
\\& \lesssim \pth{\frac{\de}{n_T}}^p  L^d \e^{\ga d}.
\end{align*}
\end{itemize}
By taking $0<\a_0<\ga d$, we have proved \eqref{estimate sum Atildeji} in the case $\mathbf{C2}$.

\saut
We now consider the case $\mathbf{C5}$. First, \eqref{def A tilde ji} and \eqref{def Aj} implies that
\begin{align*}
\tilde{A}_{j_i} \lesssim \frac{\chi_{\parent(j_i)}^p \ps{x_i}^{-\frac{p}{2}}}{\left|n_T\delta^{-1}-i\pth{\xi + \omega^\rho_{\parent(j_i)}}\right|^p}.
\end{align*}
According to Definition \ref{def cutoff} and the fact that $\parent(j_i)\in\mathcal{N}_{\high}(C)$, $\chi_{\parent(j_i)}^p$ is a sum of four terms: three terms corresponding to the situation with at least two low-frequency interactions among the three possible ones, and one term corresponding to three high-frequency interactions. More precisely, writing as above $\siblings(j_i)=\{j_i,j_{i,1},j_{i,2}\}$, we have
\begin{align*}
\sum_{x_i\in\frac{1}{L}\mathbb{Z}^d}\tilde{A}_{j_i} \lesssim B_1 + B_2 + B_3 + B_4,
\end{align*}
with
\begin{align*}
B_1 & \vcentcolon = \pth{\frac{\de}{n_T}}^p \sum_{x_i\in\Z_L^d}  \ps{x_i}^{-\frac{p}{2}} \mathbbm{1}_{| x_i + K(j_{i,1})(\ka) |\leq 2\e^\ga} \mathbbm{1}_{| x_i + K(j_{i,2})(\ka) |\leq 2\e^\ga} ,
\\ B_2 & \vcentcolon = \pth{\frac{\de}{n_T}}^p \sum_{x_i\in\Z_L^d}  \ps{x_i}^{-\frac{p}{2}} \mathbbm{1}_{| K(j_{i,1})(\ka) + x_i |\leq 2\e^\ga} \mathbbm{1}_{| K(j_{i,1})(\ka) + K(j_{i,2})(\ka) |\leq 2\e^\ga} ,
\\ B_3 & \vcentcolon = \pth{\frac{\de}{n_T}}^p \sum_{x_i\in\Z_L^d}  \ps{x_i}^{-\frac{p}{2}} \mathbbm{1}_{| K(j_{i,2})(\ka)+ x_i |\leq 2\e^\ga} \mathbbm{1}_{| K(j_{i,2})(\ka) + K(j_{i,1})(\ka) |\leq 2\e^\ga} ,
\\ B_4 & \vcentcolon = \sum_{x_i\in\Z_L^d} \frac{ \mathbbm{1}_{ | x_i + K(j_{i,1})(\ka) | \geq \e^\ga } \mathbbm{1}_{ | x_i + K(j_{i,2})(\ka)  | \geq \e^\ga }\mathbbm{1}_{ | K(j_{i,1})(\ka) + K(j_{i,2})(\ka)| \geq \e^\ga }}{\ps{x_i}^{\frac{p}{2}}\left|n_T\delta^{-1}-i\pth{\xi + \omega^\rho_{\parent(j_i)}}\right|^p} .
\end{align*}
We start by estimating $B_2$ and $B_3$. Since $K(j_{i,1})(\ka) + K(j_{i,2})(\ka)=K(\parent(j_i))(\ka) - x_i$ we have
\begin{align*}
B_2 + B_3 \lesssim \pth{\frac{\de}{n_T}}^p \sum_{x_i\in\Z_L^d}  \ps{x_i}^{-\frac{p}{2}} \mathbbm{1}_{| K(\parent(j_i))(\ka) - x_i |\leq 2\e^\ga} .
\end{align*}
As above, one can show that $K(\parent(j_i))(\ka)$ is a function of $(x_{i'})_{i+1\leq i'\leq n''}$ and $(k_j)_{j\in \mathcal{N}'(C)\setminus(\mathcal{N}''(C)\sqcup\{r_1\})}$. Therefore we obtain
\begin{align}\label{estim B2 B3}
B_2 + B_3 \lesssim \pth{\frac{\de}{n_T}}^p  L^d \e^{\a_0},
\end{align}
if $0<\a_0<\ga d$. For $B_1$, we claim that at least one of the $K(j_{i,q})(\ka)$ for $q=1,2$ only depends on $(x_{i'})_{i+1\leq i'\leq n''}$ and $(k_j)_{j\in \mathcal{N}'(C)\setminus\mathcal{N}''(C)}$. First, note that $j_{i,1}$ and $j_{i,2}$ cannot both belong to $\mathcal{N}'(C)$ since otherwise we would have $\parent(j_i)\in\mathcal{P}_3$, thus contradicting \eqref{P3vide}. Therefore, we can assume that $j_{i,1}\notin\mathcal{N}'(C)$. We now distinguish two cases:
\begin{itemize}
\item If $j_{i,2}\in\mathcal{N}'(C)$, then it cannot belong to $\mathcal{N}''(C)$ (since it already has a sibling in $\mathcal{N}''(C)$) so that $K(j_{i,2})(\ka)=k_{j_{i,2}}$.
\item If $j_{i,2}\notin\mathcal{N}'(C)$, then for $q=1,2$ and by definition of $\mathcal{N}'(C)$, $K(j_{i,q})$ is a signed combination of $\pth{(K(j))_{j\in\mathcal{N}'(C)\setminus\mathcal{N}''(C)},(K(j_{i'}))_{i\leq i' \leq n''}}$ (we used that $j_{i'}>_{\tilde{\rho}}j_{i,q}$ implies $i'\geq i$ to exclude the lower element of $\mathcal{N}''(C)$). More precisely, there exists numbers $s_q\in\{-1,0,1\}$ for $q=1,2$ such that
\begin{align}\label{K-sqK}
K(j_{i,q}) - s_q K(j_i) \in \mathrm{Vect}\pth{(K(j))_{j\in\mathcal{N}'(C)\setminus\mathcal{N}''(C)},(K(j_{i'}))_{i+1\leq i' \leq n''}}.
\end{align}
The relation $K(\parent(j_i))=K(j_i)+K(j_{i,1})+K(j_{i,2})$ and the fact that $K(\parent(j_i))\in\mathrm{Vect}\pth{(K(j))_{j\in\mathcal{N}'(C)\setminus\mathcal{N}''(C)},(K(j_{i'}))_{i+1\leq i' \leq n''}}$, we obtain
\begin{align*}
(1+s_1+s_2)K(j_i) \in \mathrm{Vect}\pth{(K(j))_{j\in\mathcal{N}'(C)\setminus\{j_i\}}}.
\end{align*}
The linear independence of $(K(j))_{j\in\mathcal{N}'(C)}$ implies $1+s_1+s_2=0$, so that there must exist $q_0\in\{1,2\}$ with $s_{q_0}=0$. Together with \eqref{K-sqK}, this implies that $K(j_{i,q_0})(\ka)$ only depends on $(x_{i'})_{i+1\leq i'\leq n''}$ and $(k_j)_{j\in \mathcal{N}'(C)\setminus\mathcal{N}''(C)}$.
\end{itemize}
We have proved the claim, and if we assume without loss of generality that $K(j_{i,1})(\ka)$ only depends on $(x_{i'})_{i+1\leq i'\leq n''}$ and $(k_j)_{j\in \mathcal{N}'(C)\setminus\mathcal{N}''(C)}$ we proceed as above and use the constraint on $x_i + K(j_{i,1})(\ka)$ to obtain
\begin{align}
B_1 & \lesssim \pth{\frac{\de}{n_T}}^p \sum_{x_i\in\Z_L^d}  \ps{x_i}^{-\frac{p}{2}} \mathbbm{1}_{| x_i + K(j_{i,1})(\ka) |\leq 2\e^\ga} \lesssim \pth{\frac{\de}{n_T}}^p  L^d \e^{\a_0}. \label{estim B1}
\end{align}
We now turn to the estimate for $B_4$. We use the claim we have proved to estimate $B_1$ and without loss of generality we assume that $K(j_{i,1})(\ka)$ only depends on $(x_{i'})_{i+1\leq i'\leq n''}$ and $(k_j)_{j\in \mathcal{N}'(C)\setminus\mathcal{N}''(C)}$. Since this also holds for $K(\parent(j_i))(\ka)$ and since $K(\parent(j_i))=K(j_i)+K(j_{i,1})+K(j_{i,2})$, the quantity $\mathcal{K}\vcentcolon = x_i + K(j_{i,2})(\ka)$ only depends on $(x_{i'})_{i+1\leq i'\leq n''}$ and $(k_j)_{j\in \mathcal{N}'(C)\setminus\mathcal{N}''(C)}$. We now have
\begin{align*}
B_4 & \lesssim \sum_{x_i\in\Z_L^d} \frac{  \mathbbm{1}_{ | \mathcal{K}  | \geq \e^\ga }}{\ps{x_i}^{\frac{p}{2}}\left|n_T\delta^{-1}-i\pth{\xi + \omega^\rho_{\parent(j_i)}}\right|^p}.
\end{align*}
As proved in the proof of Lemma \ref{triangle}, if $\mathtt{int}^{-1}(\parent(j_i))=\emptyset$ then $\omega^\rho_{\parent(j_i)}$ is of the form
\begin{align*}
\omega^\rho_{\parent(j_i)} = \e^{-2} \De_{\parent(j_i)} + F_i, 
\end{align*}
where $F_i$ is a function of $(x_{i'})_{i+1\leq i'\leq n''}$ and $(k_j)_{j\in \mathcal{N}'(C)\setminus\mathcal{N}''(C)}$. Moreover, \eqref{def De j} implies that
\begin{align*}
\De_{\parent(j_i)} & = \iota_0 \ps{K(\parent(j_i))(\ka)} - \iota_1 \ps{x_i} - \iota_2 \ps{K(j_{i,1})(\ka)} - \iota_3 \ps{K(j_{i,2})(\ka)},
\end{align*}
for some $\iota_0,\iota_1,\iota_2,\iota_3\in\{\pm\}$. Since both $K(\parent(j_i))(\ka)$ and $K(j_{i,1})(\ka)$ only depend on $(x_{i'})_{i+1\leq i'\leq n''}$ and $(k_j)_{j\in \mathcal{N}'(C)\setminus\mathcal{N}''(C)}$, we obtain
\begin{align*}
\omega^\rho_{\parent(j_i)} = \e^{-2} \pth{ - \iota_1 \ps{x_i} - \iota_3 \ps{\mathcal{K}-x_i}} + \tilde{F}_i, 
\end{align*}
where $\tilde{F}_i$ is a function of $(x_{i'})_{i+1\leq i'\leq n''}$ and $(k_j)_{j\in \mathcal{N}'(C)\setminus\mathcal{N}''(C)}$. Therefore we have $B_4 \lesssim \sum_{x\in \Z_L^d} \Theta_\e(x)$ where the function $\Theta_\e$ is defined on $\R^d$ by
\begin{align*}
\Theta_\e(x) \vcentcolon = \frac{  \mathbbm{1}_{ | \mathcal{K}  | \geq \e^\ga }}{\ps{x}^{\frac{p}{2}}\left|n_T\delta^{-1} - i\pth{  \e^{-2} \pth{ - \iota_1 \ps{x} - \iota_3 \ps{\mathcal{K}-x}} + \xi + \tilde{F}_i}\right|^p} .
\end{align*}
Since $|\nabla \Theta_\e(x)| \lesssim \e^{-2}\ps{x}^{-{\frac{p}{2}}}$ and $p>2d$ we obtain from \eqref{riemann 2} the bound
\begin{align*}
B_4 \lesssim L^{d-1} \e^{-2} + L^d \int_{\R^d} \Theta_\e.
\end{align*} 
Using $L=\e^{-\b}$ for the first term and Lemma \ref{lem:intTheta} below for the integral of $\Theta_\e$ we finally obtain
\begin{align*}
B_4 \lesssim L^{d} \e^{\b-2} + L^d \ps{\K}^{\frac{p}{2}-d+2} \pth{\frac{\de}{n_T}}^{p-1} \e^{2-\ga} .
\end{align*}
Recall that $\K=K(\parent(j_i))(\ka) - K(j_{i,2})(\ka)$, which together with $\ka(\mathcal{L}(C)_+)\subset B(0,R)$ implies that $|\K|\leq(1+ n(C)) R$. Together with \eqref{n<logepsilon} this implies $\ps{\K}^{\frac{p}{2}-d+2}\lesssim |\ln\e|^{3\pth{\frac{p}{2}-d+2}}$ and thus
\begin{align*}
B_4 \lesssim L^{d} \e^{\b-2} + L^d |\ln\e|^{3\pth{\frac{p}{2}-d+2}}\pth{\frac{\de}{n_T}}^{p-1} \e^{2-\ga} .
\end{align*}
Together with \eqref{estim B2 B3} and \eqref{estim B1} we have proved that in the case $\mathbf{C5}$ we have
\begin{align}\label{estim finale cas C5}
\sum_{x_i\in\frac{1}{L}\mathbb{Z}^d}\tilde{A}_{j_i} \lesssim L^{d} \e^{\b-2} + L^d |\ln\e|^{3\pth{\frac{p}{2}-d+2}}\pth{\frac{\de}{n_T}}^{p-1} \e^{2-\ga}+ \pth{\frac{\de}{n_T}}^p  L^d \e^{\a_0}.
\end{align}
Using again \eqref{n<logepsilon} to bound $n_T$ and and $\e\leq \de$ we have
\begin{align*}
L^d |\ln\e|^{3\pth{\frac{p}{2}-d+2}}\pth{\frac{\de}{n_T}}^{p-1}  \e^{2-\ga} & \lesssim \pth{\frac{\de}{n_T}}^{p} L^d |\ln\e|^{3\pth{\frac{p}{2}-d+3}} \e^{1-\ga}
\\& \lesssim \pth{\frac{\de}{n_T}}^{p} L^d \e^{\a_0},
\end{align*}
if $0<\a_0<1-\ga$ and where the implicit constant depends only on $p$ and $d$. Similarly, if $0<\a_0<\b-2$ and if $\e$ is small enough compared to $\de$ then $|\ln\e|^{3p}\e^{\b-2-\a_0}\leq \de^p$ which shows that $L^{d} \e^{\b-2}\leq \pth{\frac{\de}{n_T}}^p  L^d \e^{\a_0}$. Therefore the three terms on the RHS of \eqref{estim finale cas C5} are bounded by the third one, which thus proves \eqref{estimate sum Atildeji} in the case $\mathbf{C5}$ and concludes the proof of the lemma.
\end{proof}

We conclude this section by proving the technical lemma used in the previous proof.

\begin{lemma}\label{lem:intTheta} 
Let $p>2d$, $A\in\R^*$, $B\in\R$, $\e>0$, $\iota\in\{\pm\}$ and $\mathcal{K}\in\R^d\setminus\{0\}$. If $\iota=-$, assume that $ |\mathcal{K}|\geq \e^\ga$. We have  
\begin{align*}
\int_{\R^d} \frac{\d x}{\an{x}^{\frac{p}{2}} \left|A + i \varepsilon^{-2} ( \an{x} + \iota \an{x-\K} + B)\right|^p} \lesssim \ps{\K}^{\frac{p}{2}-d+2} |A|^{1-p}\e^{2-\ga},
\end{align*}
where the implicit constant only depends on $p$ and $d$.
\end{lemma}

\begin{proof} 
We set
\begin{align*}
\Upsilon(x) \vcentcolon = \frac{1}{\an{x}^{\frac{p}{2}}\left|A + i \varepsilon^{-2} ( \an{x} + \iota \an{x-\K} + B)\right|^p}.
\end{align*}
Since $\mathcal{K}\neq 0$ we have the orthogonal decomposition $\R^d=\mathrm{Vect}(\mathcal{K}) \overset{\perp}{\oplus} \R^{d-1}$ and we can consider the change of variables defined as the inverse of
\begin{align*}
(x_\K, r , \omega)\in \R\times \R_+ \times \mathbb{S}^{d-2} \longmapsto x_\K \frac{\K}{|\K|} + r \omega \in \mathrm{Vect}(\mathcal{K})\overset{\perp}{\oplus} \R^{d-1},
\end{align*}
where we also used $d\geq 3$ to define the polar change of variables in $\R^{d-1}$. We obtain
\begin{align*}
\int_{\R^d}\Upsilon(x)\d x = \int_\R\pth{ \int_{\R_+\times\mathbb{S}^{d-2}} \Upsilon\pth{ x_\K \frac{\K}{|\K|} + r \omega } r^{d-2}\d r \d\omega }\d x_\K.
\end{align*}
We recall that $\ps{x}=\sqrt{m+|x|^2}$, so that
\begin{align*}
&\Upsilon\pth{ x_\K \frac{\K}{|\K|} + r \omega } 
\\&\qquad = \pth{ m + x_\K^2 + r^2 }^{-\frac{p}{4}} \left| A + i \varepsilon^{-2} \pth{ \pth{ m + x_\K^2 + r^2 }^{\half} + \iota \pth{ m + (x_\K - |\K|)^2 + r^2}^{\half} + B} \right|^{-p}.
\end{align*}
Note that this expression does not depend on $\om$ anymore, so that the integral over $\mathbb{S}^{d-2}$ gives a constant depending on $d$. Moreover, for $x_\K\in\R$ fixed we define the function $\ffi$ by 
\begin{align*}
\ffi(r) = \pth{ m + x_\K^2 + r^2 }^{\half} + \pth{ m + (x_\K - |\K|)^2 + r^2}^{\half}.
\end{align*}
Since 
\begin{align*}
\ffi'(r) = \frac{r\ffi(r)}{\pth{ m + x_\K^2 + r^2 }^{\half} \pth{ m + (x_\K - |\K|)^2 + r^2}^{\half}},
\end{align*}
we can perform the change of variables $y=\ffi(r)$ in the integral over $\R_+$. We obtain
\begin{equation}\label{integrale de upsilon}
\begin{aligned}
& \int_{\R_+\times\mathbb{S}^{d-2}} \Upsilon\pth{ x_\K \frac{\K}{|\K|} + r \omega } r^{d-2}\d r \d\omega 
\\& \lesssim \int_{\ffi(0)}^{+\infty}  \left| A + i \varepsilon^{-2} \pth{ \pth{ m + x_\K^2 + r^2 }^{\half} + \iota \pth{ m + (x_\K - |\K|)^2 + r^2}^{\half} + B} \right|^{-p} 
\\&\hspace{3cm} \times r^{d-3}  \pth{ m + x_\K^2 + r^2 }^{\half-\frac{p}{4}}  \pth{ m + (x_\K - |\K|)^2 + r^2}^{\half} y^{-1} \d y.
\end{aligned}
\end{equation}
In the case $\iota=+$, \eqref{integrale de upsilon} becomes
\begin{align*}
& \int_{\R_+\times\mathbb{S}^{d-2}} \Upsilon\pth{ x_\K \frac{\K}{|\K|} + r \omega } r^{d-2}\d r \d\omega 
\\& \lesssim \int_{\ffi(0)}^{+\infty}  \left| A + i \varepsilon^{-2} \pth{ y + B} \right|^{-p}  r^{d-3}  \pth{ m + x_\K^2 + r^2 }^{\half-\frac{p}{4}}  \pth{ m + (x_\K - |\K|)^2 + r^2}^{\half} y^{-1} \d y.
\end{align*}
We use the obvious facts
\begin{align*}
r \leq \pth{ m + x_\K^2 + r^2 }^{\half}, \qquad \pth{ m + (x_\K - |\K|)^2 + r^2}^{\half} \leq y
\end{align*}
to obtain
\begin{align*}
\int_{\R_+\times\mathbb{S}^{d-2}} \Upsilon\pth{ x_\K \frac{\K}{|\K|} + r \omega } r^{d-2}\d r \d\omega & \lesssim \int_{\ffi(0)}^{+\infty}  \left| A + i \varepsilon^{-2} \pth{ y + B} \right|^{-p}  \pth{ m + x_\K^2 + r^2 }^{\frac{-\frac{p}{2}+d-2}{2}}  \d y.
\end{align*}
Since $-\frac{p}{2}+d-2<0$ we have $ \pth{ m + x_\K^2 + r^2 }^{\frac{-\frac{p}{2}+d-2}{2}} \leq  \pth{ m + x_\K^2 }^{\frac{-\frac{p}{2}+d-2}{2}} $ and therefore we obtain
\begin{align*}
\int_{\R^d}\Upsilon(x)\d x \lesssim \pth{ \int_\R \pth{ m + x_\K^2}^{\frac{-p+d-2}{2}} \d x_\K } \pth{ \int_{\ffi(0)}^{+\infty}  \left| A + i \varepsilon^{-2} \pth{ y + B} \right|^{-p}    \d y }.
\end{align*}
Since $-\frac{p}{2}+d-2<-2$ the integral over $x_\K$ is finite and by performing the change of variable $z=\frac{y}{\e^2 A}$ in the second integral we finally obtain
\begin{align}\label{intupsilon iota=+}
\int_{\R^d}\Upsilon(x)\d x \lesssim \e^2 |A|^{1-p}.
\end{align}
In the case $\iota=-$, we use the key identity
\begin{align}\label{key identity}
\ps{a} - \ps{b}  = \frac{|a|^2-|b|^2}{\ps{a} + \ps{b}}
\end{align}
which in our context becomes
\begin{align*}
\pth{ m + x_\K^2 + r^2 }^{\half} - \pth{ m + (x_\K - |\K|)^2 + r^2}^{\half} & =  \frac{2|\K|x_\K - |\K|^2}{y}. 
\end{align*}
Therefore, in the case $\iota=-$ \eqref{integrale de upsilon} becomes
\begin{align*}
& \int_{\R_+\times\mathbb{S}^{d-2}} \Upsilon\pth{ x_\K \frac{\K}{|\K|} + r \omega } r^{d-2}\d r \d\omega 
\\& \lesssim \int_{\ffi(0)}^{+\infty}  \left| A + i \varepsilon^{-2} \pth{ \frac{2|\K|x_\K - |\K|^2}{y} + B} \right|^{-p}  r^{d-3}  \pth{ m + x_\K^2 + r^2 }^{\half-\frac{p}{4}}  \pth{ m + (x_\K - |\K|)^2 + r^2}^{\half} y^{-1} \d y.
\end{align*}
Since $r\leq  \pth{ m + x_\K^2 + r^2 }^{\half}$ we have
\begin{align*}
&r^{d-3}  \pth{ m + x_\K^2 + r^2 }^{\frac{1-\frac{p}{2}}{2}}  \pth{ m + (x_\K - |\K|)^2 + r^2}^{\half} y^{-1}
\\&\qquad \lesssim y^{-\frac{p}{2}+d-2} \pth{ m + x_\K^2 + r^2 }^{\frac{-\frac{p}{2}+d-2}{2}}  \pth{ m + (x_\K - |\K|)^2 + r^2}^{\half} y^{\frac{p}{2}-d+1},
\end{align*}
where we also introduced artificially the factor $y^{-\frac{p}{2}+d-2}$. Since the triangle inequality $\ps{a+b}\lesssim \ps{a}+\ps{b}$ implies 
\begin{align*}
\pth{ m + (x_\K - |\K|)^2 + r^2}^{\half} & \lesssim \pth{ m + x_\K^2 + r^2 }^\half + \ps{\K} \quad \text{and} \quad y  \lesssim \pth{ m + x_\K^2 + r^2 }^\half + \ps{\K},
\end{align*}
we obtain
\begin{align*}
\pth{ m + x_\K^2 + r^2 }^{\frac{-\frac{p}{2}+d-2}{2}}  \pth{ m + (x_\K - |\K|)^2 + r^2}^{\half} y^{\frac{p}{2}-d+1} \lesssim \ps{\K}^{\frac{p}{2}-d+2}.
\end{align*}
Combining the above estimates and integrating first in $x_\K$ and then in $y$ we obtain
\begin{align*}
\int_{\R^d}\Upsilon(x)\d x & \lesssim  \ps{\K}^{\frac{p}{2}-d+2} \int_{\ffi(0)}^{+\infty} \pth{ \int_\R \left| A + i \varepsilon^{-2} \pth{ \frac{2|\K|x_\K - |\K|^2}{y} + B} \right|^{-p} \d x_\K}  y^{-\frac{p}{2}+d-2} \d y 
\\&\lesssim   \ps{\K}^{\frac{p}{2}-d+2} |A|^{1-p}\e^2 |\K|^{-1} \int_{\ffi(0)}^{+\infty} y^{-\frac{p}{2}+d-1}\d y 
\end{align*}
where the inner integral was estimated thanks to the change of variables $z=\frac{2|\K|}{\e^2 A y}x_\K$. Using now $-\frac{p}{2}+d-1<-1$ and $\ffi(0)>0$ we finally obtain
\begin{align}\label{intupsilon iota=-}
\int_{\R^d}\Upsilon(x)\d x \lesssim  \ps{\K}^{\frac{p}{2}-d+2} |A|^{1-p}\e^2 |\K|^{-1}.
\end{align}
Since $1\lesssim  |\K|^{-1}\ps{\K}^{\frac{p}{2}-d+2}$, we can combine \eqref{intupsilon iota=+} and \eqref{intupsilon iota=-} and obtain
\begin{align*}
\int_{\R^d}\Upsilon(x)\d x \lesssim  \ps{\K}^{\frac{p}{2}-d+2} |A|^{1-p}\e^{2-\ga} ,
\end{align*}
where we also used $|K|^{-1}\leq \e^{-\ga}$. This concludes the proof of the lemma.
\end{proof}

\begin{remark}
Note that Lemmas \ref{lem individual sums} and \ref{lem:intTheta} are the only place where we actually use the exact expression of the Klein-Gordon dispersion relation, in particular for the identity \eqref{key identity}.
\end{remark}

\subsubsection{Proof of Proposition \ref{prop high nodes}}\label{section proof high nodes}

We are now ready to conclude the proof of Proposition \ref{prop high nodes}.

\begin{proof}[Proof of Proposition \ref{prop high nodes}]
We start from the estimate \eqref{ACrho successive sums} for $A_C^\rho(k,\xi)$. Thanks to the last point of Lemma \ref{lem N''}, we can set
\begin{align*}
i_0 \vcentcolon = \max \enstq{ i\in\llbracket 1,n''\rrbracket }{\text{$\parent(j_i)\in\mathcal{N}_T(C)$ and $\mathtt{int}^{-1}(\parent(j_i))=\emptyset$} },
\end{align*}
and rewrite \eqref{ACrho successive sums} as
\begin{equation}\label{ACrho successive sums 2}
\begin{aligned}
&A_C^\rho(k,\xi) 
\\& \leq \La^{n+1}  L^{-\frac{nd}{2}} \pth{\frac{\de}{n_T}}^{p\pth{n_T-n''_T}} 
\\&\quad \times \sum_{(k_j)_{j\in\mathcal{N}'(C)\setminus(\mathcal{N}''(C)\sqcup\{r_1\})}\in\pth{\Z_L^d}^{\frac{n}{2}-n''}} \prod_{j\in \mathcal{N}'(C)\setminus ( \mathcal{N}''(C)\sqcup \{r_1\}) } \ps{k_j}^{-\frac{p}{2}}
\\& \hspace{6cm} \times\pth{\sum_{x_{n''}\in\Z_L^d}\tilde{A}_{j_{n''}}\pth{  \cdots \pth{ \sum_{x_{i_0}\in\Z_L^d}\tilde{A}_{j_{i_0}}R_{i_0} } \cdots    } }  .
\end{aligned}
\end{equation}
where we used the notation
\begin{align*}
R_{i_0} = \sum_{x_{i_0-1}\in\Z_L^d}\tilde{A}_{j_{i_0-1}}\pth{  \cdots \pth{ \sum_{x_{1}\in \Z_L^d}\tilde{A}_{j_{1}} } \cdots    } .
\end{align*}
We estimate successively each sum in $R_{i_0}$ using \eqref{estimate sum Atildeji} and obtain
\begin{align*}
| R_{i_0} | \leq \La^{n+1} L^{d(i_0-1)}\pth{\frac{\de}{n_T}}^{pn''_{T,i_0}} \e^{\a_0 n''_{\mathrm{gain},i_0}},
\end{align*} 
where
\begin{align*}
n''_{T,i_0} & \vcentcolon = \#\enstq{i\in\llbracket 1,i_0-1\rrbracket}{\text{$j_i$ in cases $\mathbf{C2}$, $\mathbf{C3}$, $\mathbf{C4}$ or $\mathbf{C5}$}},
\\ n''_{\mathrm{gain},i_0} & \vcentcolon = \#\enstq{i\in\llbracket 1,i_0-1\rrbracket}{\text{$j_i$ in cases $\mathbf{C2}$ or $\mathbf{C5}$}}.
\end{align*}
By plugging this into \eqref{ACrho successive sums 2} we obtain
\begin{equation}\label{ACrho successive sums 3}
\begin{aligned}
&A_C^\rho(k,\xi) 
\\& \leq \La^{n+1}  L^{d\pth{i_0-1 - \frac{n}{2}}} \pth{\frac{\de}{n_T}}^{p\pth{n_T-n''_T+n''_{T,i_0}}} \e^{\a_0 n''_{\mathrm{gain},i_0}} 
\\&\quad \times \sum_{(k_j)_{j\in\mathcal{N}'(C)\setminus(\mathcal{N}''(C)\sqcup\{r_1\})}\in\pth{\Z_L^d}^{\frac{n}{2}-n''}} \prod_{j\in \mathcal{N}'(C)\setminus ( \mathcal{N}''(C)\sqcup \{r_1\}) } \ps{k_j}^{-\frac{p}{2}}
\\&\hspace{6cm}\times\pth{\sum_{x_{n''}\in\Z_L^d}\tilde{A}_{j_{n''}}\pth{  \cdots \pth{ \sum_{x_{i_0}\in\Z_L^d}\tilde{A}_{j_{i_0}} } \cdots    } }  .
\end{aligned}
\end{equation}
We now consider the integral over $\R$ in \eqref{estim FC en fonction A} and use Hölder's inequality:
\begin{align}
\int_\R \frac{ A_C^\rho(k,\xi)^{\frac{1}{p}} }{ \left| n_T\de^{-1} - i\xi \right|}   \d\xi & \leq     \l \left| n_T\de^{-1} - i \bullet \right|^{-p^\star} \r_{L^{p'}(\R)}      \l \left| n_T\de^{-1} - i\bullet \right|^{p^\star - 1} A_C^\rho(k,\bullet)^{\frac{1}{p}} \r_{L^p(\R)} \non
\\&\lesssim \pth{\frac{\de}{n_T}}^{\frac{1}{pp'}}  \l \left| n_T\de^{-1} - i\bullet \right|^{p^\star - 1} A_C^\rho(k,\bullet)^{\frac{1}{p}} \r_{L^p(\R)} , \label{estim inter integrale sur xi}
\end{align}
where the bullet $\bullet$ is the $\xi$ variable, and where $p^\star\vcentcolon=\frac1{p'}\pth{1+\frac1{p}}$. Using \eqref{ACrho successive sums 3} and the fact that the $\tilde{A}_{j_i}$ for $i> i_0$ don't depend on $\xi$ (recall \eqref{def Aj} and the definition of $i_0$) we obtain
\begin{align*}
&\l \left| n_T\de^{-1} - i\bullet \right|^{p^\star - 1} A_C^\rho(k,\bullet)^{\frac{1}{p}} \r_{L^p(\R)}^p  
\\& \leq \La^{n+1}  L^{d\pth{i_0-1 - \frac{n}{2}}} \pth{\frac{\de}{n_T}}^{p\pth{n_T-n''_T+n''_{T,i_0}}} \e^{\a_0 n''_{\mathrm{gain},i_0}}
\\&\quad \times  \sum_{(k_j)_{j\in\mathcal{N}'(C)\setminus(\mathcal{N}''(C)\sqcup\{r_1\})}\in\pth{\Z_L^d}^{\frac{n}{2}-n''}}  \prod_{j\in \mathcal{N}'(C)\setminus ( \mathcal{N}''(C)\sqcup \{r_1\}) } \ps{k_j}^{-\frac{p}{2}}
\\&\quad \times \pth{\sum_{x_{n''}\in\Z_L^d}\tilde{A}_{j_{n''}}\pth{  \cdots\pth{\sum_{x_{i_0+1}\in\Z_L^d}\tilde{A}_{j_{i_0+1}} \pth{ \sum_{x_{i_0}\in\Z_L^d}\int_\R \left| n_T\de^{-1} - i\xi \right|^{p(p^\star - 1)}\tilde{A}_{j_{i_0}}\d\xi  }} \cdots    } } .
\end{align*}
The inner integral is estimated using \eqref{def Aj}, the definition of $i_0$ and Hölder's inequality:
\begin{align*}
& \sum_{x_{i_0}\in\Z_L^d} \int_\R \left| n_T\de^{-1} - i\xi \right|^{p(p^\star - 1)}\tilde{A}_{j_{i_0}}\d\xi 
\\&\qquad \lesssim \sum_{x_{i_0}\in\Z_L^d} \ps{x_{i_0}}^{-\frac{p}{2}} \l \left| n_T\de^{-1} - i\bullet \right|^{p(p^\star - 1)} \r_{L^{2p}(\R)} \l \left|n_T\delta^{-1}-i\pth{\bullet - \omega^\rho_{\parent(j)}}\right|^{-p} \r_{L^{\frac{2p}{2p-1}}(\R)}
\\&\qquad \lesssim L^d \pth{\frac{\de}{n_T}}^{p-\frac{1}{p'}},
\end{align*}
where we also used \eqref{riemann 1} and $p>2d$. Moreover, the remaining successive sums above (i.e over $x_{i}$ for $i_0+1\leq i \leq n''$) are again estimated using \eqref{estimate sum Atildeji}. This implies
\begin{align*}
&\l \left| n_T\de^{-1} - i\bullet \right|^{p^\star - 1} A_C^\rho(k,\bullet)^{\frac{1}{p}} \r_{L^p(\R)}^p  
\\& \leq \La^{n+1}  L^{d\pth{n'' - \frac{n}{2}}} \pth{\frac{\de}{n_T}}^{p n_T-\frac{1}{p'}} \e^{\a_0 n''_{\mathrm{gain}}}  \sum_{(k_j)_{j\in\mathcal{N}'(C)\setminus(\mathcal{N}''(C)\sqcup\{r_1\})}\in\pth{\Z_L^d}^{\frac{n}{2}-n''}}  \prod_{j\in \mathcal{N}'(C)\setminus ( \mathcal{N}''(C)\sqcup \{r_1\}) } \ps{k_j}^{-\frac{p}{2}}.
\end{align*}
where 
\begin{align*}
n''_{\mathrm{gain}} \vcentcolon = \#\enstq{i\in\llbracket 1,n''\rrbracket\setminus\{i_0\}}{\text{$j_i$ in cases $\mathbf{C2}$ or $\mathbf{C5}$}},
\end{align*}
where we used that 
\begin{align*}
\#\enstq{i\in\llbracket 1,n''\rrbracket\setminus\{i_0\}}{\text{$j_i$ in cases $\mathbf{C2}$, $\mathbf{C3}$, $\mathbf{C4}$ or $\mathbf{C5}$}} = n_T''-1,
\end{align*}
since the cases $\mathbf{C2}$, $\mathbf{C3}$, $\mathbf{C4}$ or $\mathbf{C5}$ cover $\parent(j_i)\in\mathcal{N}_T(C)$ and since $\parent(j_{i_0})\in\mathcal{N}_T(C)$ by definition of $i_0$. We use \eqref{riemann 1} and $p>2d$ to finally bound the sum over the $k_j$'s and we obtain
\begin{align*}
&\l \left| n_T\de^{-1} - i\bullet \right|^{p^\star - 1} A_C^\rho(k,\bullet)^{\frac{1}{p}} \r_{L^p(\R)}   \leq \La^{n+1}  \pth{\frac{\de}{n_T}}^{ n_T-\frac{1}{pp'}} \e^{\frac{\a_0}{p} n''_{\mathrm{gain}}} .
\end{align*}
By plugging this into \eqref{estim inter integrale sur xi} we obtain
\begin{align*}
\int_\R \frac{ A_C^\rho(k,\xi)^{\frac{1}{p}} }{ \left| n_T\de^{-1} - i\xi \right|}   \d\xi \leq  \La^{n+1}   \pth{\frac{\de}{n_T}}^{ n_T} \e^{\frac{\a_0}{p} n''_{\mathrm{gain}}} .
\end{align*}
We plug this into \eqref{estim FC en fonction A} and obtain
\begin{align*}
\sup_{t\in[0,\de]}\left| \widehat{F^\rho_C}(\e^{-2}t,k) \right| & \leq \La^{n+1}  \pth{\frac{\de}{n_T}}^{ n_T} \e^{\frac{\a_0}{p} n''_{\mathrm{gain}}+n_B}.
\end{align*}
By using $\e\leq\de$ and $n=n_B+2n_T$ we obtain
\begin{align*}
\sup_{t\in[0,\de]}\left| \widehat{F^\rho_C}(\e^{-2}t,k) \right| & \leq \La^{n+1}  \pth{\frac{\de}{n_T}}^{ \frac{n}{2}} \e^{\a_1\pth{n''_{\mathrm{gain}}+n_B}},
\end{align*}
where we defined $\alpha_1\vcentcolon = \min\pth{\half,\frac{\alpha_0}{p}}$. Recalling \eqref{decomposition FC} and using $\#\mathfrak{M}(\mathcal{N}_T(C))\leq n_T! \lesssim n_T^{n/2}$ we obtain
\begin{align}
\sup_{t\in[0,\de]}\left| \widehat{F_C}(\e^{-2}t,k) \right| & \leq \La^{n+1} \#\mathfrak{M}(\mathcal{N}_T(C)) \pth{\frac{\de}{n_T}}^{ \frac{n}{2}} \e^{\a_1\pth{n''_{\mathrm{gain}}+n_B}}\non
\\&\leq \La^{n+1} \de^{ \frac{n}{2}} \e^{\a_1\pth{n''_{\mathrm{gain}}+n_B}}.\label{high nodes estim FC inter}
\end{align}
In order to deduce \eqref{main estim FC high nodes} from \eqref{high nodes estim FC inter}, we manipulate the exponent $n''_{\mathrm{gain}}+n_B$. For that we define
\begin{align*}
n^1 & \vcentcolon = \# \enstq{ i\in\llbracket 1,n''\rrbracket }{\text{$j_i$ in case $\mathbf{C5}$ and $\parent(j_i)\in\mathcal{P}_1$}},
\\ n^2 & \vcentcolon = \# \enstq{ i\in\llbracket 1,n''\rrbracket }{\text{$j_i$ in case $\mathbf{C5}$ and $\parent(j_i)\in\mathcal{P}_2$}} +  \# \enstq{ i\in\llbracket 1,n''\rrbracket }{\text{$j_i$ in case $\mathbf{C2}$}}.
\end{align*}
We first estimate $n^1$. Since a node in $\mathcal{P}_1$ belongs either to $\mathcal{N}_B(C)$, $\mathcal{N}_\low(C)$ or $\mathcal{N}_\high(C)$ we have
\begin{align*}
\#\mathcal{P}_1 & \leq n_B + n_\low + \#\pth{ \mathcal{P}_1\cap\mathcal{N}_\high(C) }
\\& \leq n_B + n_\low + \#\enstq{i\in\llbracket 1,n''\rrbracket }{\parent(j_i)\in \mathcal{P}_1\cap\mathcal{N}_\high(C) },
\end{align*}
where we have also used the natural bijection between elements of $\mathcal{P}_1$ and their unique children in $\mathcal{N}''(C)$. Since to any $j\in\mathcal{N}_T(C)$ such that $\mathtt{int}^{-1}(j)\neq 0$ we can injectively associate a unique binary node, we have
\begin{align*}
 \#\enstq{i\in\llbracket 1,n''\rrbracket }{\parent(j_i)\in \mathcal{P}_1\cap\mathcal{N}_\high(C) } \leq n^1 + n_B.
\end{align*}
We have proven that
\begin{align}\label{estim n1}
n^1 \geq \#\mathcal{P}_1 - 2n_B - n_\low .
\end{align}
Together with \eqref{P2ternaire}, the exact same arguments implies that $\#\mathcal{P}_2 \leq n^2 + n_B$ so that
\begin{align}\label{estim n2}
n^2 \geq \#\mathcal{P}_2 - n_B.
\end{align}
Since we don't know if $j_{i_0}$ is in one of the cases $\mathbf{C2}$ or $\mathbf{C5}$ we have $n''_{\mathrm{gain}}\geq n^1 + n^2 - 1$ which in particular implies $2n''_{\mathrm{gain}}\geq n^1 + 2n^2 - 2$ (simply using $n^1\leq 2n^1$). Together with \eqref{estim n1} and \eqref{estim n2} this implies
\begin{align*}
2(n''_{\mathrm{gain}} + n_B) & \geq \#\mathcal{P}_1 + 2\#\mathcal{P}_2 - 2n_B - n_\low    - 2 
\\& \geq \frac{n}{2} - 2n_B - n_\low    - 2 ,
\end{align*}
where we used \eqref{cardP1P2}. Introducing artificially $n_\scb$ we obtain
\begin{align*}
2(n''_{\mathrm{gain}} + n_B) & \geq \frac{n-2n_\scb}{2} - 2n_B - (n_\low - n_\scb)    - 2.
\end{align*}
We now use the high nodes regime assumption \eqref{high nodes regime} to bound $-2n_B$ from below, we obtain:
\begin{align*}
2(n''_{\mathrm{gain}} + n_B) & \geq \frac{3}{10} \pth{n-2n_\scb} + 3(n_\low-n_\scb)  - 2
\\& \geq \frac{3}{10} \pth{n-2n_\scb}   - 2
\\& \geq \frac{n-2n_\scb}{10},
\end{align*}
where we used $n_\low-n_\scb\geq 0$ and \eqref{n-2nscb>10}. Recalling \eqref{high nodes estim FC inter}, this bound on $n''_{\mathrm{gain}} + n_B$ implies that
\begin{align*}
\sup_{t\in[0,\de]}\left|\widehat{F_C}(\e^{-2}t,k) \right| & \leq \La^{n+1} \de^{ \frac{n}{2}} \e^{\a_\high\pth{\frac{n}{2}-n_\scb}},
\end{align*}
where we have set $\a_\high\vcentcolon=\frac{\a_1}{10}$. This concludes the proof of Proposition \ref{prop high nodes}. 
\end{proof}

We conclude Section \ref{section high nodes regime} by combining the results of Propositions \ref{prop low nodes} and \ref{prop high nodes}, which together leads to Proposition \ref{prop |FC|} (after having set $\a\vcentcolon=\min\pth{\a_\low,\a_\high}$).

\subsection{Time derivative of \texorpdfstring{$X_n$}{Xn}}\label{section time derivative Xn}

In order to bound in a high probability setting the $L^\infty$ norm in time of $X_n$, we need to control in law its derivative in time. For this we state the following propositions.

\begin{prop}\label{prop dt X}
For all $n\in \N$, $\eta\in \{0,1\}$ and $\iota \in \{\pm\}$, we have 
\begin{align*}
\reallywidehat{\dr_tX^{\eta,\iota}_n} = \sum_{(A,\ffi,c) \in \mathcal A_n^{\eta,\iota}} \reallywidehat{\dr_tF_{(A,\ffi,c)}} ,
\end{align*}
together with the decomposition
\begin{align*}
\reallywidehat{\dr_tF_{(A,\ffi,c)}} & = \sum_{j_{\mathrm{max}}\in\{-1\}\sqcup \mathcal N_T^{\mathrm{max}}(A) } \reallywidehat{F_{(A,\ffi,c),j_{ \mathrm{max} }}},
\end{align*}
where 
\begin{align}\label{def NTmax}
\mathcal N_T^{\mathrm{max}}(A) \vcentcolon = \enstq{ j \in \mathcal N_T(A) }{ \nexists\, j' \in \mathcal N_T(A),\; j'>j }
\end{align}
and where $\reallywidehat{F_{(A,\ffi,c),j_{ \mathrm{max} }}}(\e^{-2}t,k)$ is defined for all $t\in\R$ and $k\in\Z_L^d$ by
\begin{itemize}
\item if $j_{\mathrm{max}}=-1$, then
\begin{align*}
\reallywidehat{F_{(A,\ffi,c), -1}}(\e^{-2}t,k) & \vcentcolon = -\frac{\e^{n_B}}{L^{\frac{nd}{2}}}   (-i)^{ n_T(A)+1} \sum_{\ka\in\mathcal D_k(A)} \Om_{-1} e^{i\Om_{-1}\e^{-2}t} \prod_{j\in \mathcal N(A)} q_j
\\&\hspace{2cm} \times \int_{I_A(t)} \prod_{j\in  \mathcal N_T(A)} e^{i \Om_j \e^{-2}t_j} \d t_j \prod_{\ell\in \mathcal L(A)} \reallywidehat{X^{c(\ell), \ffi(\ell)}_\init} (\ka(\ell)),
\end{align*}
\item if $j_{\mathrm{max}}\in  \mathcal N_T^{\mathrm{max}}(A) $, then
\begin{align*}
\reallywidehat{F_{(A,\ffi,c), j_{ \mathrm{max} }}}(\e^{-2}t,k) & \vcentcolon = \frac{\e^{n_B+2}}{L^{\frac{nd}{2}}}   (-i)^{ n_T(A)} \sum_{\ka\in\mathcal D_k(A)}  e^{i \pth{ \Om_{-1} + \Om_{ j_{\mathrm{max}}} } \e^{-2}t} \prod_{j\in \mathcal N(A)} q_j
\\&\hspace{2cm}\times    \int_{I^{ j_{\mathrm{max}} }_A(t)}  \prod_{j\in\mathcal{N}_T(A)\setminus\{ j_{\mathrm{max}} \}}  e^{i\Om_j \e^{-2}t_j}\d t_j  \prod_{\ell\in \mathcal L(A)} \reallywidehat{X^{c(\ell), \ffi(\ell)}_\init} (\ka(\ell)),
\end{align*}
where
\begin{align}\label{def I_A^jmax}
I_A^{j_{\mathrm{max}}}(t) \vcentcolon = \enstq{(t_j)_{j\in \mathcal N_T(A)\setminus \{j_{\mathrm{max}}\}} \in [0,t]^{n_T-1} }{ j\leq j' \Rightarrow t_j\leq t_{j'} }.
\end{align}
\end{itemize}
\end{prop}

\begin{proof} 
By directly differentiating \eqref{XtoF} we obtain
\begin{align*}
\dr_tX^{\eta,\iota}_n = \sum_{(A,\ffi,c) \in \mathcal A_n^{\eta,\iota}} \dr_tF_{(A,\ffi,c)} . 
\end{align*}
Since $\reallywidehat{\partial_t F_{(A,\ffi,c)} } = \partial_t \reallywidehat{F_{(A,\ffi,c)} }$ we want to differentiate with respect to time the formula \eqref{fourier of F}. Note that $t$ appears in $e^{i\Om_{-1} t}$ and in the time domain $I_A(t)$. In order to obtain the time derivative of the integral over $I_A(t)$, we can make use of \eqref{I_C(t) up to negligible} which implies
\begin{align*}
\int_{I_A(t)}  \prod_{j\in\mathcal{N}_T(A)} f_j(t_j)\d t_j & = \sum_{\rho\in\mathfrak{M}(\mathcal{N}_T(A))} \int_{0\leq t_{\rho(1)} < \cdots < t_{\rho(n_T)} \leq t}  \prod_{i=1}^{n_T} f_{\rho(i)}\pth{t_{\rho(i)}}  \d t_{\rho(i)},
\end{align*}
for some functions $f_j$. Rewriting the integrals on the RHS as successive time integrals we thus obtain
\begin{align*}
\dr_t &\pth{ \int_{I_A(t)}  \prod_{j\in\mathcal{N}_T(A)} f_j(t_j)\d t_j } 
 = \sum_{\rho\in\mathfrak{M}(\mathcal{N}_T(A))} f_{\rho(n_T)}(t) \int_{0\leq t_{\rho(1)} < \cdots < t_{\rho(n_T-1)} \leq t} \prod_{i=1}^{n_T-1} f_{\rho(i)}\pth{t_{\rho(i)}}  \d t_{\rho(i)}.
\end{align*}
By definition of the set $\mathcal N_T^{\mathrm{max}}(A)$ (see \eqref{def NTmax}), the following map is a bijection
\begin{align*}
\mathfrak{M}(\mathcal{N}_T(A)) & \longrightarrow \bigsqcup_{j_{\mathrm{max}}\in \mathcal{N}_T^{\mathrm{max}}(A)} \Big( \{ j_{\mathrm{max}}\}\times \mathfrak{M}\pth{\mathcal{N}_T(A)\setminus\{j_{\mathrm{max}}\}} \Big)
\\ \rho \qquad & \longmapsto \hspace{2.4cm} \pth{ \rho(n_T),\rho_{|_{ \left\llbracket 1, n_T-1\right\rrbracket }} }
\end{align*}
and therefore we can perform a change of variables in the sum over $\mathfrak{M}(\mathcal{N}_T(A))$ and get
\begin{align*}
\dr_t &\pth{ \int_{I_A(t)}  \prod_{j\in\mathcal{N}_T(A)} f_j(t_j)\d t_j } 
\\\qquad& = \sum_{j_{\mathrm{max}}\in \mathcal{N}_T^{\mathrm{max}}(A)} f_{j_{\mathrm{max}}}(t)  \sum_{\bar{\rho}\in  \mathfrak{M}(\mathcal{N}_T(A)\setminus\{ j_{\mathrm{max}} \})} \int_{0\leq t_{\bar{\rho}(1)} < \cdots < t_{\bar{\rho}(n_T-1)} \leq t} \prod_{i=1}^{n_T-1} f_{\bar{\rho}(i)}\pth{t_{\bar{\rho}(i)}}  \d t_{\bar{\rho}(i)}.
\end{align*}
Using again the formula \eqref{I_C(t) up to negligible} but now applied to the sets $\mathfrak{M}(\mathcal{N}_T(A)\setminus\{ j_{\mathrm{max}} \})$ we finally obtain
\begin{align*}
\dr_t &\pth{\int_{I_A(t)}  \prod_{j\in\mathcal{N}_T(A)} f_j(t_j)\d t_j }  = \sum_{j_{\mathrm{max}}\in \mathcal{N}_T^{\mathrm{max}}(A)} f_{j_{\mathrm{max}}}(t)  \int_{I^{ j_{\mathrm{max}} }_A(t)}  \prod_{j\in\mathcal{N}_T(A)\setminus\{ j_{\mathrm{max}} \}} f_j(t_j)\d t_j,
\end{align*}
where $I^{ j_{\mathrm{max}} }_A(t)$ is defined in \eqref{def I_A^jmax}. Applying this formula with $f_j(t_j)=e^{i\Om_jt_j}$ allows us to differentiate \eqref{fourier of F} with respect to time and conclude the proof of the proposition, after the usual dilation in time.
\end{proof}

From this proposition and the usual Wick formula, we obtain the following decomposition
\begin{align}\label{correlation dtX}
\E\pth{ \left| \reallywidehat{\partial_t  X_n^{\eta}}(\e^{-2}t,k)\right|^2} & = \sum_{C\in \mathcal C_{n,n}^{\eta,\eta,+,-}} \sum_{\substack{j_{\mathrm{max}} \in \{-1\} \sqcup \mathcal N_T^{\mathrm{max}}(A) \\ j'_{\mathrm{max}} \in \{-1\} \sqcup \mathcal N_T^{\mathrm{max}}(A')}} \reallywidehat{F_{C,j_{\mathrm{max}},j'_{\mathrm{max}}}}(\e^{-2}t,k),
\end{align}
where $\reallywidehat{F_{C,j_{\mathrm{max}},j'_{\mathrm{max}}}}(\e^{-2}t,k)$ is defined by
\begin{itemize}
\item if $j_{\mathrm{max}}=j'_{\mathrm{max}}=-1$, then
\begin{align*}
 \reallywidehat{F_{C,-1,-1}} (\e^{-2}t,k) & = -  \frac{\e^{n_B}}{L^{\frac{nd}{2}}}   (-i)^{ n_T}  \sum_{\ka\in\mathcal{D}_k(C)}\Om_{-1}(A) \Om_{-1}(A') e^{i\Om_{-1}(C)\e^{-2}t} \prod_{j\in \mathcal N(C)} q_{j}
\\&\hspace{2cm} \times\int_{I_{C}(t)} \prod_{j\in  \mathcal N_T(C)}  e^{i \Om_{j} \e^{-2} t_{j}} \d t_{j} \prod_{\ell\in\mathcal{L}(C)_-}  M^{c(\ell),c(\si(\ell))}(\ffi(\ell)\ka(\ell))^{\ffi(\ell)},
\end{align*}
\item if $j_{\mathrm{max}}=-1$ and $j'_{\mathrm{max}} \neq  -1$, then
\begin{align*}
 \reallywidehat{F_{C,-1,j'_{\mathrm{max}}} }(\e^{-2}t,k) & = i \frac{\e^{n_B+2}}{L^{\frac{nd}{2}}}   (-i)^{ n_T}  \sum_{\ka\in\mathcal{D}_k(C)}\Om_{-1}(A) e^{i(\Om_{-1}(C) + \Om_{j'_{\mathrm{max}}}(A'))t} \prod_{j\in \mathcal N(C)} q_{j}
\\&\hspace{1cm} \times\int_{I^{-1,j'_{\mathrm{max}}}_{C}(t)} \prod_{j\in  \mathcal N_T(C)\setminus\{j'_{\mathrm{max}}\}}  e^{i \Om_{j} \e^{-2} t_{j}} \d t_{j} \prod_{\ell\in\mathcal{L}(C)_-}  M^{c(\ell),c(\si(\ell))}(\ffi(\ell)\ka(\ell))^{\ffi(\ell)},
\end{align*}
where $I^{-1,j'_{\mathrm{max}}}_{C}(t) = I_A(t) \times I_{A'}^{j'_{\mathrm{max}}}(t)$,
\item if $j_{\mathrm{max}}\neq -1$ and $j'_{\mathrm{max}} =  -1$, then
\begin{align*}
 \reallywidehat{F_{C,j_{\mathrm{max}},-1} }(\e^{-2}t,k) & = i \frac{\e^{n_B+2}}{L^{\frac{nd}{2}}}   (-i)^{ n_T}  \sum_{\ka\in\mathcal{D}_k(C)}\Om_{-1}(A') e^{i(\Om_{-1}(C) + \Om_{j_{\mathrm{max}}}(A))t} \prod_{j\in \mathcal N(C)} q_{j}
\\&\hspace{1cm} \times\int_{I^{j_{\mathrm{max}},-1}_{C}(t)} \prod_{j\in  \mathcal N_T(C)\setminus\{j_{\mathrm{max}}\}}  e^{i \Om_{j} \e^{-2} t_{j}} \d t_{j} \prod_{\ell\in\mathcal{L}(C)_-}  M^{c(\ell),c(\si(\ell))}(\ffi(\ell)\ka(\ell))^{\ffi(\ell)},
\end{align*}
where $I^{j_{\mathrm{max}},-1}_{C}(t) =I_{A}^{j_{\mathrm{max}}}(t)\times I_{A'}(t) $,
\item if $j_{\mathrm{max}}\neq -1$ and $j'_{\mathrm{max}} \neq  -1$, then
\begin{align*}
 \reallywidehat{F_{C,j_{\mathrm{max}},j'_{\mathrm{max}}}} (\e^{-2}t,k) & =  \frac{\e^{n_B+4}}{L^{\frac{nd}{2}}}   (-i)^{ n_T}  \sum_{\ka\in\mathcal{D}_k(C)} e^{i(\Om_{-1}(C) + \Om_{j_{\mathrm{max}}}(A) + \Om_{j'_{\mathrm{max}}}(A'))t} \prod_{j\in \mathcal N(C)} q_{j}
\\& \times\int_{I^{j_{\mathrm{max}},j'_{\mathrm{max}}}_{C}(t)} \prod_{j\in  \mathcal N_T(C)\setminus\{j_{\mathrm{max}}, j'_{\mathrm{max}}\}}  e^{i \Om_{j} \e^{-2} t_{j}} \d t_{j} \prod_{\ell\in\mathcal{L}(C)_-}  M^{c(\ell),c(\si(\ell))}(\ffi(\ell)\ka(\ell))^{\ffi(\ell)}.
\end{align*}
where $I^{j_{\mathrm{max}},j'_{\mathrm{max}}}_{C}(t) = I_A^{j_{\mathrm{max}}}(t) \times I^{j'_{\mathrm{max}}}_{A'}(t)$.
\end{itemize}

\begin{remark}\label{remark counting time derivative}
Based on the fact that the cardinal of $N_T^{\mathrm{max}}(A)$ is bounded by $\frac{n(A)}{2}$ and thus grows linearly in $n(A)$ the result of Lemma \ref{counting self-coupled bushes} is unchanged. More precisely we have
\begin{align*}
\tilde{c}_{n,q} \leq \La^{n+1} \pth{\frac{n}{2}+2-q}!.
\end{align*}
where $\tilde c_{n,q}$ be the cardinal of the set
\begin{align*}
 \enstq{(C,j_1,j_2)}{C\in\mathcal{C}_{n_1,n_2},\; n_1+n_2=n,\; n_{\scb}(C)=q,\; j_1\in \{-1\}\sqcup \mathcal N_T^{\mathrm{max}}(A_1),\; j_2\in \{-1\} \mathcal N_T^{\mathrm{max}}(A_2)}
\end{align*}
where $A_1$ and $A_2$ are the two trees in $C$.
\end{remark}

The following proposition contains the key estimate satisfied by $\reallywidehat{F_{C,j_{\mathrm{max}},j'_{\mathrm{max}}}}$.

\begin{prop}\label{prop estimate FCjmax}
There exists $\La>0$ such that for all $C\in\mathcal{C}^{\eta,\eta,+,-}_{n,n}$ and for all $(j_{\mathrm{max}},j'_{\mathrm{max}})\in (\{-1\}\sqcup \mathcal N_T^{\mathrm{max}}(A))\times (\{-1\}\sqcup \mathcal N_T^{\mathrm{max}}(A'))$ (where $A$ and $A'$ are the two trees in $C$) we have
\begin{align*}
\sup_{t\in[0,\de]}\sup_{k\in\Z_L^d}\left|\reallywidehat{F_{C,j_{\mathrm{max}},j'_{\mathrm{max}}}}(\e^{-2}t,k) \right| \leq \La (\La\delta)^{\frac{n(C)}{2}} .
\end{align*}
Moreover, there exists $\a>0$ such that if $C$ satisfies 
\begin{align*}
n(C)-2n_\scb(C) \geq 80, \qquad n(c) \leq |\ln \varepsilon|^c
\end{align*}
then for any $\e\leq \de $ we have 
\begin{align}\label{estimate FCjmax}
\sup_{t\in[0,\de]}\sup_{k\in\Z_L^d}\left|\reallywidehat{F_{C,j_{\mathrm{max}},j'_{\mathrm{max}}}}(\e^{-2}t,k) \right| \leq \La (\La\delta)^{\frac{n(C)}{2}} \varepsilon^{\alpha_0 \pth{\frac{n(C)}{2}-n_\scb(C)}}.
\end{align}
\end{prop}

\begin{proof} 
The proof of this proposition is identical to the one of Proposition \ref{prop |FC|}, i.e we consider the low nodes regime and the high nodes regime. The only two differences between the formulas for $\reallywidehat{F_{C,j_{\mathrm{max}},j'_{\mathrm{max}}}}$ and the formula for $\reallywidehat{F_{C}}$ are the presence of the terms $\Omega_{-1}(A)$ or $\Omega_{-1}(A')$ and the fact that there might be one or two oscillating integrals less. However since the factors $\Omega_{-1}(A)$ and $\Omega_{-1}(A')$ grow linearly in $n(C)$ (thanks to the compact support property) and since the missing oscillating integrals are automatically compensated by a better $\e$ power, the proofs of Section \ref{section low nodes regime} and \ref{section high nodes regime} apply without any change (and produce even better estimates since $\e\leq \de$).
\end{proof}

\subsection{Proof of Corollary \ref{coro Xn}}\label{section coro Xn}

We conclude Section \ref{section convergence} by proving its main result, i.e. Corollary \ref{coro Xn}.

\begin{proof}
First, we begin with proving the existence of $\Lambda>0$ such that for any $n\le |\ln \e|^{3}$, $t\in [0, \delta]$ and $k\in \Z_L^d$ we have
\begin{equation} \label{expectation Xn(k)}
\E\pth{ \left| \widehat{X^\eta_n}(\e^{-2}t, k) \right|^2} + \E\pth{ \left| \widehat{\dr_tX^\eta_n}(\e^{-2}t, k) \right|^2} \le \Lambda (\Lambda \delta)^\frac{n}{2}.
\end{equation}
We explain how to get the estimate for the first expectation, which is deduced from identity \eqref{correlation in terms of couplings}, Proposition \ref{prop |FC|} and Lemma \ref{counting self-coupled bushes}. The estimate for the latter follows analogously using the identity \eqref{correlation dtX}, Proposition \ref{prop estimate FCjmax} and Remark \ref{remark counting time derivative} instead.

For $n\le 80$, the statement follows straightforwardly. For $n>80$, we observe that
\begin{multline*}
\sum_{C\in\mathcal{C}^{\eta,\eta,+,-}_{n,n}} \left| \widehat{F_C}(\e^{-2}t,k)\right| = \sum_{0\le q\le \floor{\frac{n}{2}}-40} \sum_{\substack{C\in\mathcal{C}^{\eta,\eta,+,-}_{n,n} \\ n_\scb(C)=q}} \left|\widehat{F_C}(\e^{-2}t,k)\right|  + \sum_{ \floor{\frac{n}{2}}-40< q\le \floor{\frac{n}{2}}} \sum_{\substack{C\in\mathcal{C}^{\eta,\eta,+,-}_{n,n} \\ n_\scb(C)=q}} \left|\widehat{F_C}(\e^{-2}t,k)\right| 
\end{multline*}
The estimate of the latter sum in the above right hand side is again straightforward using \eqref{estimate good} and Lemma \ref{counting self-coupled bushes}:
\begin{align*}
\sum_{ \floor{\frac{n}{2}}-40< q\le \floor{\frac{n}{2}}} \sum_{\substack{C\in\mathcal{C}^{\eta,\eta,+,-}_{n,n} \\ n_\scb(C)=q}} \left|\widehat{F_C}(\e^{-2}t,k)\right| & \lesssim \La^{\frac{n}{2}+1} \de^{\frac{n}{2}} \sum_{\floor{\frac{n}{2}}-q<40}\Big(\frac{n}{2}+2-q\Big)!
\\& \lesssim \La^{\frac{n}{2}+1}\de^{\frac{n}{2}}\sum_{m\leq 40}(m+2)!
\\& \lesssim \La^{\frac{n}{2}+1}\de^{\frac{n}{2}}.
\end{align*} 
For the first sum, we use the improved bound \eqref{estimate FC} and get
\[
\aligned
\sum_{0\le q\le \floor{\frac{n}{2}}-40}\sum_{\substack{C\in\mathcal{C}^{\eta,\eta,+,-}_{n,n} \\ n_\scb(C)=q}} \left|\widehat{F_C}(\e^{-2}t,k)\right| & \lesssim \La^{\frac{n}{2}+1} \de^{\frac{n}{2}} \sum_{0\le q\le \floor{\frac{n}{2}}-40} \e^{\alpha(\frac{n}{2}-q)}\Big(\frac{n}{2}+2-q\Big)! \\
& \lesssim \La^{\frac{n}{2}+1} \de^{\frac{n}{2}} \sum_{40\le m\le \floor{\frac{n}{2}}}\e^{\alpha m}(m+3)!.
\endaligned
\]
Stirling's approximation of factorials and the restriction $n\le |\ln \e|^{3}$ then shows that the last sum is bounded by $(n\e^\a)^n\leq (|\ln\e|^{3}\e^\a)^n$ which is indeed uniformly bounded when $\e$ goes to 0. Therefore we have proved \eqref{expectation Xn(k)}.

Second, we observe that the identities \eqref{correlation in terms of couplings} and \eqref{correlation dtX}, the formula \eqref{fourier of FC} and the equivalent ones for $\reallywidehat{F_{C,j_{\mathrm{max}},j'_{\mathrm{max}}}}$, together with the assumption on the functions $M^{\eta, \eta'}$ to be supported in $B(0, R)$ imply that the functions
\begin{align*}
k\longmapsto\E\pth{ \left| \widehat{X^\eta_n}(\e^{-2}t, k) \right|^2} \quad \text{and} \quad k\longmapsto  \E\pth{ \left| \widehat{\dr_tX^\eta_n}(\e^{-2}t, k) \right|^2}
\end{align*}
are actually supported in a ball of $\Z_L^d$ of radius $\mathcal O(n)$. This implies that for any $s\ge 0$, we can get from \eqref{expectation Xn(k)} the following bound, for some new positive constant $\Lambda$,
\[
\E\pth{\l X^\eta_n(t,\cdot)\r^2_{H^s(L\T^d)} } + \E \pth{\l\partial_t X^\eta_n(t,\cdot)\r^2_{H^s(L\T^d)} }  \leq \Lambda^{\frac{n}{2}+1} \delta^{\frac{n}{2}}n^{2s+d}L^d.
\]
We successively apply Sobolev injection on $[0, \e^{-2}\delta]$, Fubini's theorem and integrate the above estimate in time to obtain that
\begin{align*}
\E \pth{ \sup_{t\in [0, \e^{-2}\delta]} \l X^\eta_n(t,\cdot)\r^2_{H^s(L\T^d)} } & \lesssim \E \left(\|X_n\|^2_{H^1([0, \e^{-2}\delta], H^s(L\T^d))}\right)
\\& \lesssim \int_0^{\e^{-2}\delta} \pth{\E\pth{\l X^\eta_n(t,\cdot)\r^2_{H^s(L\T^d)} } + \E \pth{\l\partial_t X^\eta_n(t,\cdot)\r^2_{H^s(L\T^d)} }   }\d t 
\\&\lesssim(\Lambda\de)^{\frac{n}{2}+1} n^{2s+d}L^{d+\frac{2}{\b}},
\end{align*}
where we used $\e = L^{-1/\beta}$. Finally, Markov's inequality and the restriction $n\le |\ln \e|^{3}$, allows us to conclude that for $a \vcentcolon= (\Lambda \delta)^{\frac{n}{2}+1}L^{d+\frac{2}{\beta}+2A}$
\[
\mathbb P\pth{\sup_{t\in [0, \e^{-2}\delta]} \l X^\eta_n(t,\cdot)\r^2_{H^s(L\T^d)} \ge a }\le L^{-A},
\]
which concludes the proof of the corollary.
\end{proof}


\section{Fixed point argument}\label{section remainder}

In this section we show how to construct a true solution to the fixed point \eqref{fixed point after normal form}. More precisely, we will prove the following.

\begin{theorem}\label{thm X as series}
Let $s>d/2$ and $A>0$ be fixed. There exists a set $\mathcal{E}_{L, A}$ of probability larger than $1-L^{-A}$ on which there exists a unique solution $(X^0, X^1)\in \mathcal{C}([0,\e^{-2}\delta]; H^s(L\T^d))$ to \eqref{fixed point after normal form} that admits a decomposition of the form 
\begin{equation} \label{full solution}
 X^\eta = \sum_{n\le N(L)} X^\eta_{n} + v^\eta   
\end{equation}
where $N(L)\vcentcolon =\floor{\log L}$, $X^\eta_n$ is defined in Definition \ref{def X_n} and satisfies \eqref{est X_n}  and $v$ satisfies
\[
\l v\r_{\mathcal{C}([0,\e^{-2}\delta]; H^s(L\T^d))}\lesssim L^{-\pth{\frac32d+\frac14+\frac1\beta+3A}}. 
\]
\end{theorem}

We use for brevity the notation
\begin{align}\label{def XleqN}
X^\eta_{\leq N} \vcentcolon = \sum_{n=0}^N X_n^\eta.
\end{align}
We want $X^\eta$ defined by \eqref{full solution} to solve
\begin{align*}
X^\eta(t) & = X^\eta_0 +  \e \sum_{\substack{\iota_1,\iota_2\in\{\pm\}}} \pth{B^\eta_{(\iota_1,\iota_2)}\pth{t, X^{\bar{\eta},\iota_1}(t),X^{\bar{\eta},\iota_2}(t) } - B^\eta_{(\iota_1,\iota_2)}\pth{0, X^{\bar{\eta},\iota_1}_0,X^{\bar{\eta},\iota_2}_0 }}
\\&\quad - i \e^2 \int_0^t \sum_*\sum_{\iota_1,\iota_2,\iota_3\in\{\pm\}}  C^\eta_{*,(\iota_1,\iota_2,\iota_3)}\pth{ \tau , X^{\eta,\iota_1}(\tau),X^{\bar{\eta},\iota_2}(\tau),X^{\eta,\iota_3}(\tau) } \d\tau,
\end{align*}
where we already replaced $X^\eta_\init$ by $X^\eta_0$ (cf \eqref{def X0}). By plugging the ansatz \eqref{full solution} into this equation, we obtain the following fixed point formulation for the remainder $v$:
\begin{align}\label{eq remainder}
v^\eta & = \mathfrak{S}^\eta + \mathfrak{L}(v)^\eta + \mathfrak{N}(v)^\eta,
\end{align}
where the source term, linear term and nonlinear term are defined by
\begin{equation}\label{source term}
\begin{aligned}
\mathfrak{S}^\eta(t) & \vcentcolon = - X^\eta_{\leq N(L)}(t) + X^\eta_0 +  \e \sum_{\substack{\iota_1,\iota_2\in\{\pm\}}} \pth{ B^\eta_{(\iota_1,\iota_2)}\pth{t, X^{\bar{\eta},\iota_1}_{\leq N(L)}(t) ,X^{\bar{\eta},\iota_2}_{\leq N(L)}(t) } - B^\eta_{(\iota_1,\iota_2)}\pth{0, X^{\bar{\eta},\iota_1}_0,X^{\bar{\eta},\iota_2}_0 }}
\\&\quad  - i \e^2 \int_0^t \sum_*\sum_{\iota_1,\iota_2,\iota_3\in\{\pm\}}  C^\eta_{*,(\iota_1,\iota_2,\iota_3)}\pth{ \tau , X^{\eta,\iota_1}_{\leq N(L)}(\tau) ,X^{\bar{\eta},\iota_2}_{\leq N(L)}(\tau) ,X^{\eta,\iota_3}_{\leq N(L)}(\tau)  } \d\tau,
\end{aligned}
\end{equation}
\begin{equation}\label{linear term}
\begin{aligned}
\mathfrak{L}(v)^\eta(t) & \vcentcolon =  \e \sum_{\substack{\iota_1,\iota_2\in\{\pm\}}} \pth{B^\eta_{(\iota_1,\iota_2)}\pth{t, X^{\bar{\eta},\iota_1}_{\leq N(L)}(t) , v^{\bar{\eta},\iota_2}(t) } + B^\eta_{(\iota_1,\iota_2)}\pth{t, v^{\bar{\eta},\iota_1}(t),X^{\bar{\eta},\iota_2}_{\leq N(L)}(t) }}
\\&\quad - i \e^2 \int_0^t \sum_*\sum_{\iota_1,\iota_2,\iota_3\in\{\pm\}} \bigg(  C^\eta_{*,(\iota_1,\iota_2,\iota_3)}\pth{ \tau , X^{\eta,\iota_1}_{\leq N(L)}(\tau) ,X^{\bar{\eta},\iota_2}_{\leq N(L)}(\tau) , v^{\eta,\iota_3}(\tau) }
\\&\hspace{5cm} + C^\eta_{*,(\iota_1,\iota_2,\iota_3)}\pth{ \tau , X^{\eta,\iota_1}_{\leq N(L)}(\tau) , v^{\bar{\eta},\iota_2}(\tau),X^{\eta,\iota_3}_{\leq N(L)}(\tau)  }
\\&\hspace{5cm} + C^\eta_{*,(\iota_1,\iota_2,\iota_3)}\pth{ \tau , v^{\eta,\iota_1}(\tau),X^{\bar{\eta},\iota_2}_{\leq N(L)}(\tau) ,X^{\eta,\iota_3}_{\leq N(L)}(\tau)  } \bigg) \d\tau,
\end{aligned}
\end{equation}
\begin{equation}\label{nonlinear term}
\begin{aligned}
\mathfrak{N}(v)^\eta(t) & \vcentcolon =  \e \sum_{\substack{\iota_1,\iota_2\in\{\pm\}}} B^\eta_{(\iota_1,\iota_2)}\pth{t, v^{\bar{\eta},\iota_1}(t), v^{\bar{\eta},\iota_2}(t) }
\\&\quad - i \e^2 \int_0^t \sum_*\sum_{\iota_1,\iota_2,\iota_3\in\{\pm\}} \bigg( C^\eta_{*,(\iota_1,\iota_2,\iota_3)}\pth{ \tau , X^{\eta,\iota_1}_{\leq N(L)}(\tau) , v^{\bar{\eta},\iota_2}(\tau), v^{\eta,\iota_3}(\tau) }
\\&\hspace{5cm} + C^\eta_{*,(\iota_1,\iota_2,\iota_3)}\pth{ \tau , v^{\eta,\iota_1}(\tau),X^{\bar{\eta},\iota_2}_{\leq N(L)}(\tau) , v^{\eta,\iota_3}(\tau) }
\\&\hspace{5cm} + C^\eta_{*,(\iota_1,\iota_2,\iota_3)}\pth{ \tau , v^{\eta,\iota_1}(\tau), v^{\bar{\eta},\iota_2}(\tau),X^{\eta,\iota_3}_{\leq N(L)}(\tau)  } 
\\&\hspace{6cm} + C^\eta_{*,(\iota_1,\iota_2,\iota_3)}\pth{ \tau , v^{\eta,\iota_1}(\tau), v^{\bar{\eta},\iota_2}(\tau), v^{\eta,\iota_3}(\tau) } \bigg) \d\tau.
\end{aligned}
\end{equation}

As in \cite{denghani2021}, the main obstructions to the resolution of \eqref{eq remainder} are that the source term must be small (and not diverging with $L$ as $L^{d/2}$) and that the linear operator $\mathfrak{L}$ must be such that $\mathrm{Id}-\mathfrak{L}$ is invertible (where $\mathrm{Id}$ denotes the identity operator of $\mathcal C([0,\e^{-2}\delta],H^s(L\T^d))$). A naive use of \eqref{est X_n} shows that $\mathfrak{L}$ has operator norm of order $L^{d+2A}$, thus implying that $\mathrm{Id}-\mathfrak{L}$ is \textit{a priori} not invertible. However, by relying once again on tree expansions, we will prove that its $N(L)$-th iterate $\mathfrak{L}^{N(L)}$ does have a small operator norm. This will imply that $\mathrm{Id}-\mathfrak{L}^{N(L)}$ is invertible and thus that $\mathrm{Id}-\mathfrak{L}$ is in fact invertible, thanks to the formula
\begin{align}\label{I-L inverse}
\pth{ \mathrm{Id}-\mathfrak{L}}^{-1} =\pth{\mathrm{Id}-\mathfrak{L}^{N(L)}}^{-1} \sum_{n=0}^{N(L)-1}\mathfrak{L}^n.
\end{align}
Once this is done, one can solve \eqref{eq remainder} by showing that the mapping $v\longmapsto\pth{\mathrm{Id}-\mathfrak{L}}^{-1}\pth{ \mathfrak{S} +  \mathfrak{N}(v) }$ is a contraction from a ball of small radius to itself.
This is the content of the following proposition, which would then implies Theorem \ref{thm X as series}.

\begin{prop}\label{prop fixed point}
Let $s>d/2$ be fixed. There exists $L_0 = L_0(\beta,d)$ such that for any $L\geq L_0$ and any $A>0$, there exists a set $\mathcal E_{L,A}$ of probability $\ge 1-L^{-A}$ such that the mapping
\[
v\longmapsto\pth{\mathrm{Id}-\mathfrak{L}}^{-1}\pth{ \mathfrak{S} +  \mathfrak{N}(v) }
\]
is a contraction from the set 
\begin{align*}
\enstq{v\in \mathcal{C}([0,\e^{-2}\delta]; H^s(L\T^d)) }{ \l v\r_{\mathcal{C}([0,\e^{-2}\delta]; H^s(L\T^d))}\leq \La_0 L^{-\pth{\frac32d+\frac14+\frac1\beta+3A}} }
\end{align*}
to itself, where $\La_0\ge 1$ is some fixed constant independent of $R, L$.
\end{prop}

This section is organized as follows: in Section \ref{section linear term} we estimate the operator norm of $\mathfrak{L}^N$ for any positive integer $N$ by introducing an adapted diagrammatic extending the one of Section \ref{subsec:diag}; Section \ref{proof fixed point} will be dedicated to the proof of Proposition \ref{prop fixed point}.

\subsection{Linear estimates}\label{section linear term}

In this section, we estimate the operator norm of the $N$-th iterate $\mathfrak{L}^N$. For this, we represent the action of the iterates $\mathfrak{L}^k$ with trees expansions.

\subsubsection{Representation of the iterates with linear trees}\label{section representation iterates}

Compared to the objects introduced in Section \ref{section trees}, we introduce a new type of leaf $\bullet$, a new type of binary nodes $B^\lin_\st$ and new types of ternary nodes $T_*^\lin$ for $*\in \{\lowl,\lowm,\lowr,\high\}$. They play roles similar to the ones they play in the diagrammatic used to bound $X_n$. We define
\begin{align*}
\tilde{\mathcal{A}}_{\leq N(L)} \vcentcolon = \bigsqcup_{0\leq n \leq N(L)} \tilde{\mathcal{A}}_n,
\end{align*}
where the sets of trees $\tilde{\mathcal{A}}_n$ are defined in Section \ref{section trees}. We now define what we call linear trees:
\begin{align*}
\tilde{\mathcal{A}}^\lin_0 & \vcentcolon = \{ \bullet \},
\\ \tilde{\mathcal{A}}^\lin_{n+1} & \vcentcolon = \bigsqcup_* \enstq{ T_*^\lin\pth{ A,A',A^\lin} , T_*^\lin\pth{ A,A^\lin,A'} , T_*^\lin\pth{ A^\lin,A,A'} }{\pth{A,A',A^\lin}\in\pth{\tilde{\mathcal{A}}_{\leq N(L)}}^2\times \tilde{\mathcal{A}}^\lin_n}
\\&\quad \sqcup \enstq{ B_\st^\lin\pth{ A,A^\lin},B_\st^\lin\pth{ A^\lin,A}  }{ \pth{A,A^\lin}\in \tilde{\mathcal{A}}_{\leq N(L)}\times \tilde{\mathcal{A}}^\lin_n },
\end{align*}
for $n\in\mathbb{N}$. As opposed to the sets $\tilde{\mathcal{A}}_n$, here the $n$ in $\tilde{\mathcal{A}}^\lin_{n}$ directly counts the number of linear nodes in a given linear tree. Moreover, each linear tree only has one $\bullet$ leaf. Note also that the ``linear" structure of a linear tree is a comb to which we attach standard trees from $\tilde{\mathcal{A}}_{\leq N(L)}$, $A_1$ to $A_4$ in the example below:
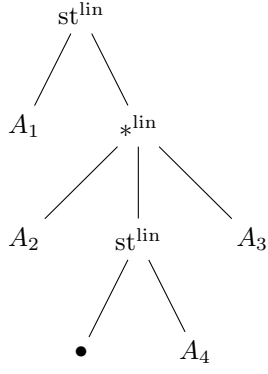
\begin{center}
\begin{tikzpicture}
\node{$\st^\lin$}
    child{node{$A_1$}}
    child{node{$*^\lin$}
        child{node{$A_2$}}
        child{node{$\st^\lin$}
        	    child{node{$\bullet$}}
            child{node{$A_4$}}}
        child{node{$A_3$}}
            }
;
\end{tikzpicture}
\end{center}
As done in Definition \ref{def ensembles liés à un arbre}, we need to define the relevant sets of branching nodes and leaves of a given linear tree $A^\lin\in\tilde{\mathcal{A}}^\lin_n$. The difference between a linear tree $A^\lin$ and a normal tree only lies in the presence of the new types of nodes and leaves, so that the first six sets from Definition \ref{def ensembles liés à un arbre} can be also defined for linear trees. In addition to these, we define
\begin{align*}
\mathcal{N}_\st^\lin(A^\lin) & \vcentcolon = \left\{ \text{linear standard binary nodes of $A^\lin$} \right\}, 
\\ \mathcal{N}_\st^\tot(A^\lin) & \vcentcolon = \mathcal{N}_\st^\lin(A^\lin) \sqcup \mathcal{N}_\st(A^\lin),
\\ \mathcal{N}_B^\tot(A^\lin) & \vcentcolon = \mathcal{N}_B(A^\lin)  \sqcup \mathcal{N}_\st^\lin(A^\lin),
\\ \mathcal{N}_*^\lin(A^\lin) & \vcentcolon = \left\{ \text{linear ternary nodes of $A^\lin$ of type $*$} \right\},
\\ \mathcal{N}_*^\tot(A^\lin) & \vcentcolon = \mathcal{N}_*^\lin(A^\lin) \sqcup \mathcal{N}_*(A^\lin),
\\ \mathcal{N}_\low^\lin(A^\lin) & \vcentcolon = \mathcal{N}_\lowl^\lin(A^\lin) \sqcup \mathcal{N}_\lowm^\lin(A^\lin) \sqcup \mathcal{N}_\lowr^\lin(A^\lin),
\\ \mathcal{N}_\low^\tot(A^\lin) & \vcentcolon = \mathcal{N}_\low^\lin(A^\lin) \sqcup \mathcal{N}_\low(A^\lin),
\\ \mathcal{N}_T^\lin(A^\lin) & \vcentcolon = \mathcal{N}_\low^\lin(A^\lin) \sqcup \mathcal{N}_\high^\lin(A^\lin),
\\ \mathcal{N}_T^\tot(A^\lin) & \vcentcolon = \mathcal{N}_T^\lin(A^\lin) \sqcup \mathcal{N}_T(A^\lin).
\end{align*}
The cardinals of these sets are denoted as in Definition \ref{def ensembles liés à un arbre}, for instance $n^\lin_\st(A^\lin)$ is the cardinal of $\mathcal{N}_\st^\lin(A^\lin)$ and $n_\low^\tot(A^\lin)$ is the cardinal of $\mathcal{N}_\low^\tot(A^\lin)$. We also define
\begin{align*}
\mathcal{N}^\lin(A^\lin) & \vcentcolon = \mathcal{N}_\st^\lin(A^\lin) \sqcup \mathcal{N}_T^\lin(A^\lin),
\\ \mathcal{N}^\tot(A^\lin) & \vcentcolon = \mathcal{N}^\lin(A^\lin) \sqcup \mathcal{N}(A^\lin),
\\ \mathcal{L}(A^\lin) & \vcentcolon = \left\{ \text{leaves of $A^\lin$ that are not $\bullet$} \right\}.
\end{align*}
The size of $A^\lin$ is given by $n(A^\lin)\vcentcolon=n_B^\tot(A^\lin) + 2 n_T^\tot(A^\lin)$, this matches with the cardinal of $\mathcal{L}(A^\lin)$ as it can be shown by induction. Finally, we can easily extend the definition of the sign and colour maps (recall Definition \ref{def sign map colour map}) to elements of $\tilde{\mathcal{A}}_n^\lin$. As before, a sign map 
\begin{align*}
\ffi : \mathcal{L}(A^\lin) \sqcup \{\bullet\} \sqcup \mathcal{N}^\tot(A^\lin) \longrightarrow \{\pm\}
\end{align*}
is arbitrary while a colour map
\begin{align*}
c : \mathcal{L}(A^\lin) \sqcup \{\bullet\} \sqcup \mathcal{N}^\tot(A^\lin) \longrightarrow \{\pm\}
\end{align*}
can be entirely deduced from the knowledge of the root's colour by asking that the rules defining the colour of a child with respect to the colour of its parent also apply to linear nodes. As in Definition \ref{def sign map colour map}, we define
\begin{align*}
\mathcal{A}^{\lin}_{n,\eta,\iota} \vcentcolon = \enstq{\text{$(A,\ffi,c)$ signed and coloured linear tree}}{A\in \tilde{\mathcal{A}}^\lin_n, \; \ffi(\text{root of $A$})=\iota, \; c(\text{root of $A$})=\eta},
\end{align*} 
for $n\in\mathbb{N}$, $\eta\in\{0,1\}$ and $\iota\in\{\pm\}$. Mimicking the recursive definition of the $F_{(A,\ffi,c)}$ (see Definition \ref{def FA}), we fix $v$ and define $F^\lin_{(A,\ffi,c)}$ by
\begin{itemize}
\item First, if $(A,\ffi,c)\in \bigsqcup_{0\leq n \leq N(L)}\mathcal{A}_{n}^{\eta,\iota}$, then set $F^\lin_{(A,\ffi,c)}=F_{(A,\ffi,c)}$ where $F_{(A,\ffi,c)}$ is defined in Definition \ref{def FA}.
\item Second, if $(A,\ffi,c)\in \mathcal{A}^{\lin}_{n,\eta,\iota}$, then $F^\lin_{(A,\ffi,c)}$ is defined recursively as follows:
\begin{itemize}
\item if $n=0$, set $F^\lin_{(\bullet,\iota,\eta)} = v^{\eta,\iota}$,
\item if $A=B^\lin_\st(A_1,A_2)$, set
\begin{align*}
F^\lin_{(A,\ffi,c)}(t) \vcentcolon = \e B^{\eta,\iota_r}_{\iota_r(\iota_1,\iota_2)} \pth{ t , F^\lin_{\pth{A_1,\iota_r \ffi_{|_{A_1}}, c_{|_{A_1}}}}(t), F^\lin_{\pth{A_2,\iota_r \ffi_{|_{A_2}}, c_{|_{A_2}}}}(t)},
\end{align*}
where $\iota_r$ and $\iota_1$, $\iota_2$ are the signs of the root of $A$ and the roots of $A_1$, $A_2$ in $A$ and $\eta$ is the colour of the root of $A$,
\item if $A=T_*^\lin(A_1,A_2,A_3)$, set
\begin{align*}
&F^\lin_{\pth{A,\ffi, c}} (t) 
\\& \vcentcolon = -i \iota_r \varepsilon^2 \int_0^t C_{*,\iota_r(\iota_1,\iota_2,\iota_3)}^{\eta,\iota_r}\pth{\tau, F^\lin_{\pth{A_1,\iota_r \ffi_{|_{A_1}},c_{|_{A_1}}}} (\tau), F^\lin_{\pth{A_2,\iota_r \ffi_{|_{A_2}},c_{|_{A_2}}} }(\tau), F^\lin_{\pth{A_3,\iota_r \ffi_{|_{A_3}}, c_{|_{A_3}}}} (\tau)}\d\tau ,
\end{align*}
where $\iota_r$ and $\iota_1$, $\iota_2$, $\iota_3$ are the signs of the root of $A$ and the roots of $A_1$, $A_2$, $A_3$ in $A$, and $\eta$ is the colour of the root of $A$.
\end{itemize}
\end{itemize}
The following proposition shows how we represent the iterates of the linear operator $\mathfrak{L}$ in terms of the $F^\lin_{(A,\ffi,c)}$.

\begin{prop}\label{prop representation iterates}
For $n\in\mathbb{N}$, $\eta\in\{0,1\}$ and $\iota\in\{\pm\}$ we have
\begin{align*}
\mathfrak{L}^{n}(v)^{\eta,\iota} & = \sum_{(A,\ffi,c)\in \mathcal{A}^{\lin}_{n,\eta,\iota}} F^{\lin}_{(A,\ffi,c)}.
\end{align*}
\end{prop}

\begin{proof}
We prove the proposition by induction over $n\in\mathbb{N}$. For $n=0$, the proposition holds according to the definition of $F^\lin_{(\bullet,\iota,\eta)}$ above. If it holds for some $n\in\mathbb{N}$, then we compute $\mathfrak{L}^{n+1}(v)$ according to the definition \eqref{linear term} of $\mathfrak{L}$:
\begin{align*}
&\mathfrak{L}^{n+1}(v)^{\eta,\iota}(t) 
\\& =  \e \sum_{\substack{\iota_1,\iota_2\in\{\pm\}}} \pth{B^{\eta,\iota}_{(\iota_1,\iota_2)}\pth{t, X^{\bar{\eta},\iota_1}_{\leq N(L)}(t) , \mathfrak{L}^{n}(v)^{\bar{\eta},\iota_2}(t) } + B^{\eta,\iota}_{(\iota_1,\iota_2)}\pth{t, \mathfrak{L}^{n}(v)^{\bar{\eta},\iota_1}(t),X^{\bar{\eta},\iota_2}_{\leq N(L)}(t) }}
\\&\quad - i \iota \e^2 \int_0^t \sum_*\sum_{\iota_1,\iota_2,\iota_3\in\{\pm\}} \bigg(  C^{\eta,\iota}_{*,(\iota_1,\iota_2,\iota_3)}\pth{ \tau , X^{\eta,\iota_1}_{\leq N(L)}(\tau) ,X^{\bar{\eta},\iota_2}_{\leq N(L)}(\tau) , \mathfrak{L}^{n}(v)^{\eta,\iota_3}(\tau) }
\\&\hspace{5cm} + C^{\eta,\iota}_{*,(\iota_1,\iota_2,\iota_3)}\pth{ \tau , X^{\eta,\iota_1}_{\leq N(L)}(\tau) , \mathfrak{L}^{n}(v)^{\bar{\eta},\iota_2}(\tau),X^{\eta,\iota_3}_{\leq N(L)}(\tau)  }
\\&\hspace{5cm} + C^{\eta,\iota}_{*,(\iota_1,\iota_2,\iota_3)}\pth{ \tau , \mathfrak{L}^{n}(v)^{\eta,\iota_1}(\tau),X^{\bar{\eta},\iota_2}_{\leq N(L)}(\tau) ,X^{\eta,\iota_3}_{\leq N(L)}(\tau)  } \bigg) \d\tau.
\end{align*}
Thanks to Proposition \ref{prop:XtoF}, \eqref{fact SM 2} and the definition of $F^\lin_{(A,\ffi,c)}$ for $(A,\ffi,c)\in \bigsqcup_{0\leq n \leq N(L)}\mathcal{A}_{n}^{\eta,\iota}$ we have
\begin{align}\label{lolilol1}
X^{\eta',\iota'}_{\leq N(L)} = \sum_{A\in \tilde{\mathcal{A}}_{\leq N(L)}}\sum_{\ffi\in \mathrm{SM}(A,\iota')}F^\lin_{(A,\ffi,c_{\eta'})},
\end{align}
where the notations $\mathrm{SM}(A,\iota)$ and $c_\eta$ are defined in the proof of Proposition \ref{prop:XtoF}. Moreover, thanks to the induction hypothesis and \eqref{fact SM 2} we also have
\begin{align}\label{lolilol2}
\mathfrak{L}^{n}(v)^{\eta',\iota'} & =  \sum_{A\in\tilde{\mathcal{A}}^\lin_n}\sum_{\ffi\in \mathrm{SM}(A,\iota')} F^\lin_{(A,\ffi,c_{\eta'})}. 
\end{align}
Therefore, by inserting \eqref{lolilol1} and \eqref{lolilol2} into the above expression of $\mathfrak{L}^{n+1}(v)^{\eta,\iota}(t)$ and after having changed variables in the sums over $\{\pm\}$ from $\iota_i$ to $\iota \iota_i$ (which implies a similar change of variables in the sums over sign maps) we obtain
\begin{align*}
&\mathfrak{L}^{n+1}(v)^{\eta,\iota}(t)
\\& =  \e \sum_{\substack{(A',A)\in \tilde{\mathcal{A}}_{\leq N(L)} \times\tilde{\mathcal{A}}^\lin_n}} \sum_{\substack{\iota_1,\iota_2\in\{\pm\}\\ \ffi' \in \mathrm{SM}(A',\iota_1) \\\ffi\in\mathrm{SM}(A,\iota_2) }}
\\&\qquad\qquad \times  \pth{  B^{\eta,\iota}_{\iota(\iota_1,\iota_2)}\pth{t, F^\lin_{(A',\iota\ffi',c_{\bar{\eta}})}(t), F^\lin_{(A,\iota\ffi,c_{\bar{\eta}})}(t) } + B^{\eta,\iota}_{\iota(\iota_2,\iota_1)}\pth{t,  F^\lin_{(A,\iota\ffi,c_{\bar{\eta}})}(t),F^\lin_{(A',\iota\ffi',c_{\bar{\eta}})}(t) }  }
\\&\quad - i \iota\e^2 \int_0^t \sum_{\substack{*\\ A'\in \tilde{\mathcal{A}}_{\leq N(L)}\\ A''\in \tilde{\mathcal{A}}_{\leq N(L)} \\ A\in \tilde{\mathcal{A}}^\lin_n }} \sum_{\substack{\iota_1,\iota_2,\iota_3\in\{\pm\}\\ \ffi \in \mathrm{SM}(A,\iota_1) \\ \ffi' \in \mathrm{SM}(A',\iota_2) \\ \ffi'' \in \mathrm{SM}(A'',\iota_3) }}
\\&\qquad \hspace{1cm} \times  \bigg( C^{\eta,\iota}_{*,\iota(\iota_1,\iota_2,\iota_3)}\pth{ \tau , F^\lin_{(A,\iota\ffi,c_\eta)}(\tau), F^\lin_{(A',\iota\ffi',c_{\bar{\eta}})}(\tau) , F^\lin_{(A'',\iota\ffi'',c_\eta)}(\tau)  } 
\\&\qquad \hspace{4cm} + C^{\eta,\iota}_{*,\iota(\iota_2,\iota_3,\iota_1)}\pth{ \tau ,F^\lin_{(A',\iota\ffi',c_\eta)}(\tau) ,F^\lin_{(A'',\iota\ffi'',c_{\bar{\eta}})}(\tau) , F^\lin_{(A,\iota\ffi,c_\eta)}(\tau) } 
\\&\qquad \hspace{4cm}  +  C^{\eta,\iota}_{*,\iota(\iota_3,\iota_1,\iota_2)}\pth{ \tau ,F^\lin_{(A'',\iota\ffi'',c_\eta)}(\tau) , F^\lin_{(A,\iota\ffi,c_{\bar{\eta}})}(\tau),F^\lin_{(A',\iota\ffi',c_\eta)}(\tau) }\bigg) \d\tau.
\end{align*}
Using now \eqref{fact SM bi}-\eqref{fact SM ter} and their natural equivalent for linear nodes, the recursive definition of $F^\lin$ and the recursive rules defining a colour map we obtain
\begin{align*}
&\mathfrak{L}^{n+1}(v)^{\eta,\iota}(t)
\\& =  \sum_{\substack{(A',A)\in \tilde{\mathcal{A}}_{\leq N(L)} \times\tilde{\mathcal{A}}^\lin_n}} \pth{\sum_{\ffi\in \mathrm{SM}(B^\lin_\st(A',A),\iota)} F^\lin_{\pth{ B^\lin_\st(A',A),\ffi, c_\eta }}(t) + \sum_{\ffi\in \mathrm{SM}(B^\lin_\st(A,A'),\iota)} F^\lin_{\pth{ B^\lin_\st(A,A'),\ffi, c_\eta }}(t) } 
\\&\quad + \sum_{\substack{*\\ A'\in \tilde{\mathcal{A}}_{\leq N(L)}\\ A''\in \tilde{\mathcal{A}}_{\leq N(L)} \\ A\in \tilde{\mathcal{A}}^\lin_n }}  \left( \sum_{\ffi\in\mathrm{SM}(T^\lin_*(A,A',A''),\iota)}F^\lin_{\pth{T^\lin_*(A,A',A''),\ffi,c_\eta}}(t) \right.
\\&\hspace{2cm} +\left. \sum_{\ffi\in\mathrm{SM}(T^\lin_*(A',A'',A),\iota)}F^\lin_{\pth{T^\lin_*(A',A'',A),\ffi,c_\eta}}(t) + \sum_{\ffi\in\mathrm{SM}(T^\lin_*(A'',A,A'),\iota)}F^\lin_{\pth{T^\lin_*(A'',A,A'),\ffi,c_\eta}}(t) \right).
\end{align*}
Using now again \eqref{fact SM 2} and the definition of $\tilde{\mathcal{A}}^\lin_{n+1}$ we finally obtain
\begin{align*}
&\mathfrak{L}^{n+1}(v)^{\eta,\iota}(t) =  \sum_{(A,\ffi,c)\in \mathcal{A}^{\lin}_{n+1,\eta,\iota}} F^{\lin}_{(A,\ffi,c)},
\end{align*}
which concludes the proof of the proposition.
\end{proof}

\subsubsection{Fourier transform of the iterates}
In this section, we give the Fourier transform of $F^{\lin}_{(A,\ffi,c)}$. First, we extend the definition of decoration map and $k$-decoration from Definition \ref{def decoration tree} to linear trees, the only difference being the domain of definition and the recursive relation, which now takes into account the different definition of $\mathcal{L}(A^\lin)$. More precisely a decoration map $\ka$ of $A^\lin\in\tilde{\mathcal{A}}^\lin_n$ is now defined on $\mathcal{L}(A^\lin) \sqcup \{\bullet\} \sqcup \mathcal{N}^\tot(A^\lin)$ and the recursive relation is
\begin{align*}
\ka(j) = \sum_{\ell \in \mathcal{L}(A^\lin) \sqcup \{\bullet\}, \ell\leq j}\ka(\ell).
\end{align*}
We define $\mathcal{D}_{k,k_\bullet}(A^\lin)$ to be the set of $k$-decoration of $A^\lin$ such that $\ka(\bullet)=k_\bullet$. We also extend Definitions \ref{def:deltaj}, \ref{def int}, \ref{def:Omegaj} and \ref{def:qj} to the linear nodes of a linear tree. Finally, if $\mathcal{N}^\lin_T(A^\lin)\neq\emptyset$ then we define
\begin{align*}
\mathrm{int}(\bullet) \vcentcolon = \min \mathcal{N}^\lin_T(A^\lin),
\end{align*}
where the minimum is defined with respect to parentality order (note that the linear structure defined in Section \ref{section representation iterates} in particular implies that $\mathcal{N}^\lin_T(A^\lin)$ is totally ordered for the parentality order).

\begin{prop}\label{prop fourier transform of Flin}
Let $A$ be a linear tree, $\ffi$ and $c$ sign and colour maps on $A$, $t\in\R$ and $k\in\Z_L^d$.
\begin{itemize}
\item If $\mathcal{N}^\lin_T(A)=\emptyset$, then
\begin{align*}
\reallywidehat{F^\lin_{(A,\ffi,c)}}(t,k) = \frac{1}{L^{\frac{d}{2}}}\sum_{k_\bullet \in \Z_L^d} \reallywidehat{ v^{c(\bullet),\ffi(\bullet)}}(t,k_\bullet)\reallywidehat{G_{(A,\ffi,c)}}(t,k,k_\bullet)
\end{align*}
with $G_{(A,\ffi,c)}(t,\cdot,k_\bullet)$ defined in Fourier space by
\begin{align}\label{G_A bi}
\reallywidehat{G_{(A,\ffi,c)}}(t,k,k_\bullet) & = (-i)^{n_T(A)}\frac{\e^{n(A)}}{L^{(n(A)-1)\frac{d}{2}}} \sum_{\ka\in\mathcal{D}_{k,k_\bullet}(A)} e^{i\Om_{-1}t}
\\&\hspace{2cm} \times \prod_{j\in\mathcal{N}^\tot(A)}q_j \int_{I_A(t)}\prod_{j\in\mathcal{N}_T(A)}e^{i\Om_j t_j}\d t_j \prod_{\ell\in\mathcal{L}(A)} \reallywidehat{X^{c(\ell), \ffi(\ell)}_\init} (\ka(\ell)),\non
\end{align}
where $I_A(t)$ is defined in Proposition \ref{prop:FFourier}.
\item If $\mathcal{N}^\lin_T(A)\neq\emptyset$, then
\begin{align*}
\reallywidehat{F^\lin_{(A,\ffi,c)}}(t,k) = \frac{\e^2}{L^{\frac{d}{2}}} \sum_{k_\bullet \in \Z_L^d} \int_0^t \reallywidehat{ v^{c(\bullet),\ffi(\bullet)}}(t_\bullet,k_\bullet) \reallywidehat{G_{(A,\ffi,c)}}(t,t_\bullet,k,k_\bullet)\d t_\bullet
\end{align*}
with $G_{(A,\ffi,c)}(t,t_\bullet,\cdot,k_\bullet)$ defined in Fourier space by
\begin{align}
\reallywidehat{G_{(A,\ffi,c)}}(t,t_\bullet,k,k_\bullet) & =  (-i)^{n^\tot_T(A)} \frac{\e^{n(A)-2}}{L^{(n(A)-1)\frac{d}{2}}}  \sum_{\ka\in\mathcal{D}_{k,k_\bullet}(A)} e^{i(\Om_{-1}t} e^{i\Om_\bullet t_\bullet}\label{G_A ter}
\\&\quad \times \prod_{j\in\mathcal{N}^\tot(A)}q_j \int_{I_A(t,t_\bullet)} \prod_{j\in \mathcal{N}^\tot_T(A)\setminus\{\mathrm{int}(\bullet)\}} e^{i\Om_j t_j}\d t_j  \prod_{\ell\in\mathcal{L}(A)} \reallywidehat{X^{c(\ell), \ffi(\ell)}_\init} (\ka(\ell)),\non
\end{align}
where we use the notations $\Om_\bullet=\Om_{\mathrm{int}(\bullet)}$ and
\begin{equation*}
I_{A}(t,t_\bullet) \vcentcolon = \Bigg\{ (t_j)_{j\in \mathcal{N}^\tot_T(A)\setminus\{\mathrm{int}(\bullet)\}}\in[0,t]^{n_T^\tot(A)-1}\; \Bigg|\; j < j'\Rightarrow t_j<t_{j'} \;\;\text{and}\;\;
\begin{aligned}
&j>\mathrm{int}(\bullet)\Rightarrow t_j>t_\bullet 
\\& j<\mathrm{int}(\bullet)\Rightarrow t_j<t_\bullet 
\end{aligned}
\Bigg\}.
\end{equation*}
\end{itemize}
\end{prop}

\begin{proof}
We again follow the same steps as in Section \ref{subsec:diag} and compute the Fourier transform of $F^\lin_{(A,\ffi,c)}$ for $(A,\ffi,c)\in \mathcal{A}^\lin_{n,\eta,\iota}$. We crucially note that the recursive definition of the $F^\lin_{(A,\ffi,c)}$ is identical to the one of the $F_{(A,\ffi,c)}$ and therefore we can easily obtain the following formula:
\begin{align*}
&\reallywidehat{F^\lin_{(A,\ffi,c)}}(t,k) 
\\& = (-i)^{n_T(A)}\frac{\e^{n(A)}}{L^{ n(A)\frac{d}{2} }} \sum_{\ka\in\mathcal{D}_k(A)} e^{i\Om_{-1}t}
\\&\quad \times \prod_{j\in\mathcal{N}^\tot(A)} q_j \int_{I_A(t)} \reallywidehat{F^\lin_{(\bullet,\ffi(\bullet),c(\bullet))}}(t_{\mathrm{int}(\bullet)},\ka(\bullet)) \prod_{\ell\in\mathcal{L}(A)}\reallywidehat{F^\lin_{(\perp,\ffi(\ell),c(\ell))}}(t_{\mathrm{int}(\ell)},\ka(\ell))\prod_{j\in\mathcal{N}_T^\tot(A)}e^{i\Om_j t_j} \d t_j,
\end{align*}
where by convention $t_{-1}=t$. This formula is the generalization of \eqref{fourier of F} to the case where the leaves of a tree are allowed to depend on time and it can be proved identically. According to the definition of $F^\lin_{(A,\ffi,c)}$ and \eqref{def Fbot} we have $\reallywidehat{F^\lin_{(\perp,\ffi(\ell),c(\ell))}}(t_{\mathrm{int}(\ell)},\ka(\ell)) = \reallywidehat{X^{c(\ell), \ffi(\ell)}_\init} (\ka(\ell))$ and $F^\lin_{(\bullet,\ffi(\bullet),c(\bullet))}=v^{c(\bullet),\ffi(\bullet)}$ and thus this formula becomes
\begin{align*}
&\reallywidehat{F^\lin_{(A,\ffi,c)}}(t,k) 
\\& = (-i)^{n_T(A)}\frac{\e^{n(A)}}{L^{ n(A)\frac{d}{2} }} \sum_{\ka\in\mathcal{D}_k(A)} e^{i\Om_{-1}t}
\\&\quad \times \prod_{j\in\mathcal{N}^\tot(A)} q_j \prod_{\ell\in\mathcal{L}(A)}\reallywidehat{X^{c(\ell), \ffi(\ell)}_\init} (\ka(\ell)) \int_{I_A(t)} \reallywidehat{v^{c(\bullet),\ffi(\bullet)}}(t_{\mathrm{int}(\bullet)},\ka(\bullet)) \prod_{j\in\mathcal{N}_T^\tot(A)}e^{i\Om_j t_j} \d t_j,
\end{align*}
Now we distinguish cases between $\mathrm{int}(\bullet)=-1$ or $\mathrm{int}(\bullet)\in\mathcal{N}_T^\tot(A)$. Note that due to the structure of linear trees, the first case is equivalent to $\mathcal{N}^\lin_T(A)=\emptyset$ and the second to $\mathcal{N}^\lin_T(A)\neq\emptyset$, since the only nodes comparable to $\bullet$ are linear nodes. 
\begin{itemize}
\item In the case $\mathcal{N}^\lin_T(A)=\emptyset$, we have $t_{\mathrm{int}(\bullet)}=t$ and thus $\reallywidehat{v^{c(\bullet),\ffi(\bullet)}}(t_{\mathrm{int}(\bullet)},\ka(\bullet))$ is factorized outside the integral over $I_A(t)$ and using $\mathcal{D}_k(A)=\bigsqcup_{k_\bullet\in \Z_L^d}\mathcal{D}_{k,k_\bullet}(A)$ we directly obtain \eqref{G_A bi}.
\item In the case $\mathcal{N}^\lin_T(A)\neq\emptyset$, we use the fact that $(t_j)_{j\in\mathcal{N}_T^\tot(A)}\in I_A(t)$ if and only if $t_{\mathrm{int}(\bullet)}\in[0,t]$ and $(t_j)_{j\in\mathcal{N}_T^\tot(A)\setminus\{\mathrm{int}(\bullet)\}}\in I_A\pth{t,t_{\mathrm{int}(\bullet)}}$ (where $I_A\pth{t,t_{\mathrm{int}(\bullet)}}$ is defined in the proposition) to rewrite the time integral as
\begin{align*}
& \int_{I_A(t)} \reallywidehat{v^{c(\bullet),\ffi(\bullet)}}(t_{\mathrm{int}(\bullet)},\ka(\bullet)) \prod_{j\in\mathcal{N}_T^\tot(A)}e^{i\Om_j t_j} \d t_j 
\\&\qquad= \int_0^t \reallywidehat{v^{c(\bullet),\ffi(\bullet)}}(t_{\mathrm{int}(\bullet)},\ka(\bullet)) \pth{ \int_{I_A\pth{t,t_{\mathrm{int}(\bullet)}}} \prod_{j\in\mathcal{N}_T^\tot(A)\setminus\{\mathrm{int}(\bullet)\}}e^{i\Om_j t_j} \d t_j } \d t_{\mathrm{int}(\bullet)}.
\end{align*}
Changing variables from $t_{\mathrm{int}(\bullet)}$ to $t_\bullet$ and using $\mathcal{D}_k(A)=\bigsqcup_{k_\bullet\in \Z_L^d}\mathcal{D}_{k,k_\bullet}(A)$ we obtain \eqref{G_A ter}.
\end{itemize}
This concludes the proof of the proposition.
\end{proof}

In the estimates that we will provide it will be useful to discriminate the linear trees according to their number of leaves that are not $\bullet$. To that purpose we define
\begin{align*}
\tilde{\mathcal{A}}^\lin_{m,n} & \vcentcolon = \enstq{ A^\lin \in \tilde{\mathcal{A}}^\lin_{m}}{\#\mathcal{L}(A^\lin)=n}.
\end{align*}
Since the number of leaves that are not $\bullet$ corresponding to each linear node ranges from $1$ (in the case of a linear standard binary node with $\perp$ attached) to $2(N(L)+1)$ (in the case of a linear ternary node with two elements of $\tilde{\mathcal{A}}_{N(L)}$ attached) we have
\begin{align*}
\tilde{\mathcal{A}}^\lin_{m} & = \bigsqcup_{n=m}^{2m(N(L)+1)}\tilde{\mathcal{A}}^\lin_{m,n}.
\end{align*}
As Proposition \ref{prop fourier transform of Flin} suggests, we distinguish linear trees with respect to the presence or absence of ternary linear nodes and thus define:
\begin{align*}
\mathcal{A}^\bi_{m,n}(\eta,\iota,\eta_\bullet,\iota_\bullet) & \vcentcolon = \enstq{(A,\ffi,c)\in\tilde{\mathcal{A}}^\lin_{m,\eta,\iota}}{A \in \tilde{\mathcal{A}}^\lin_{m,n},\,\mathcal{N}^\lin_T(A)=\emptyset ,\, \ffi(\bullet)=\iota_\bullet,\,c(\bullet)=\eta_\bullet } ,
\\ \mathcal{A}^\ter_{m,n}(\eta,\iota,\eta_\bullet,\iota_\bullet) & \vcentcolon = \enstq{(A,\ffi,c)\in\tilde{\mathcal{A}}^\lin_{m,\eta,\iota}}{A \in \tilde{\mathcal{A}}^\lin_{m,n},\,\mathcal{N}^\lin_T(A)\neq\emptyset,\, \ffi(\bullet)=\iota_\bullet,\,c(\bullet)=\eta_\bullet } ,
\end{align*}
for $n\in\llbracket m,2m(N(L)+1)\rrbracket$. By combining the results of Propositions \ref{prop representation iterates} and \ref{prop fourier transform of Flin} we obtain the following kernel representation in Fourier space of the iterates of $\mathfrak{L}$.

\begin{corollary}
If $t\in[0,\de]$ and $k\in\Z_L^d$ we have
\begin{align}
&\reallywidehat{\mathfrak{L}^m(v)^\eta}(\e^{-2}t,k) \non
\\&\quad = \sum_{\substack{\eta_\bullet\in\{0,1\}\\\iota_\bullet\in\{\pm\} \\k_\bullet\in\Z_L^d}} \sum_{n=m}^{2m(N(L)+1)} \Bigg( \frac{1}{L^{\frac{d}{2}}} \reallywidehat{v^{\eta_\bullet,\iota_\bullet}}(\e^{-2}t,k_\bullet) \reallywidehat{Y^\bi_{m,n}}(\eta,+,\eta_\bullet,\iota_\bullet)(\e^{-2}t,k,k_\bullet)\label{itérées in terms Y}
\\&\quad\hspace{4.3cm} + \frac{1}{L^{\frac{d}{2}}} \int_0^t \reallywidehat{v^{\eta_\bullet,\iota_\bullet}}(\e^{-2}t_\bullet,k_\bullet) \reallywidehat{Y^\ter_{m,n}}(\eta,+,\eta_\bullet,\iota_\bullet)(\e^{-2}t,\e^{-2}t_\bullet,k,k_\bullet)\d t_\bullet \Bigg),\non
\end{align}
where
\begin{align}
Y^\bi_{m,n}(\eta,\iota,\eta_\bullet,\iota_\bullet) & \vcentcolon = \sum_{(A,\ffi,c)\in  \mathcal{A}^\bi_{m,n}(\eta,\iota,\eta_\bullet,\iota_\bullet)} G_{(A,\ffi,c)},\label{def Y bi}
\\ Y^\ter_{m,n}(\eta,\iota,\eta_\bullet,\iota_\bullet) & \vcentcolon = \sum_{(A,\ffi,c)\in  \mathcal{A}^\ter_{m,n}(\eta,\iota,\eta_\bullet,\iota_\bullet)} G_{(A,\ffi,c)}.\label{def Y ter}
\end{align}
\end{corollary}

Therefore, we have reduced estimating the operator norm of the iterates of $\mathfrak{L}$ to estimating $Y^\bi_{m,n}(\eta,\iota,\eta_\bullet,\iota_\bullet)$ and $Y^\ter_{m,n}(\eta,\iota,\eta_\bullet,\iota_\bullet)$. This will be done using coupling of linear trees.

\subsubsection{Coupling between linear trees}

In order to estimate $\reallywidehat{Y^\bi_{m,n}}(\eta,\iota,\eta_\bullet,\iota_\bullet)$ and $\reallywidehat{Y^\ter_{m,n}}(\eta,\iota,\eta_\bullet,\iota_\bullet)$ we will express the quantities
\begin{align*}
\E\pth{ \left|\reallywidehat{Y^\bi_{m,n}}(\eta,\iota,\eta_\bullet,\iota_\bullet)(\e^{-2}t,k,k_\bullet)\right|^2 } \quad \text{and} \quad \E\pth{ \left|\reallywidehat{Y^\ter_{m,n}}(\eta,\iota,\eta_\bullet,\iota_\bullet)(\e^{-2}t,\e^{-2}t_\bullet,k,k_\bullet)\right|^2 } 
\end{align*}
as a sum over couplings of linear trees, similarly to Section \ref{section coupling between trees}. We first adapt the notions of this section:
\begin{itemize}
\item First, the notion of coupling between trees from Definition \ref{def:coupling} can be easily extended to linear trees: if $(A_1,\ffi_1,c_1)$ and $(A_2,\ffi_2,c_2)$ are two signed and coloured linear trees with $A_1,A_2\in \tilde{\mathcal{A}}_{m,n}^\lin$ for common $m,n\in\mathbb{N}$ then $C=\big( (A_1,\ffi_1,c_1),(A_2,\ffi_2,c_2),\si \big)$ is a coupling if $\si$ is an involution of $\mathcal{L}(C)\vcentcolon = \mathcal{L}(A_1)\sqcup\mathcal{L}(A_2)$ and a bijection from $\mathcal{L}(C)_-$ to $\mathcal{L}(C)_+$ (defined again by \eqref{def LC-}-\eqref{def LC+}). The definition is identical, except that each linear tree $A$ have the extra leaf $\bullet$ which we recall is not included in $\mathcal{L}(A)$. Finally, for $\star=\bi,\ter$ we define
\begin{equation*}
\mathcal{C}^\star_{m,n}(\eta,\eta_\bullet,\iota_\bullet)  \vcentcolon = \Bigg\{\big( (A_1,\ffi_1,c_1),(A_2,\ffi_2,c_2),\si \big) \:\text{coupling} \, \Bigg| \,  
\begin{aligned}
&(A_1,\ffi_1,c_1)\in \mathcal{A}^\star_{m,n}(\eta,+,\eta_\bullet,\iota_\bullet)
\\& (A_2,\ffi_2,c_2)\in \mathcal{A}^\star_{m,n}(\eta,-,\eta_\bullet,-\iota_\bullet) 
\end{aligned}
\Bigg\}.
\end{equation*}
\item The notion of a decoration of a coupling from Definition \ref{def:decocoup} extends directly to maps defined on $\mathcal{L}(C)\sqcup\mathcal{N}^\tot(C)\sqcup\{\bullet_1,\bullet_2\}$ (where $\bullet_i$ is the extra leaf of $A_i$) with the novelty that, in addition to the constraint \eqref{decoration coupling rule}, we ask that $\ka(\bullet_1)+\ka(\bullet_2)=0$. The set $\mathcal{D}_{k,k_\bullet}(C)$ for a given coupling $C=\big( (A_1,\ffi_1,c_1),(A_2,\ffi_2,c_2),\si \big)$ is then the set of decoration such that $\ka(r_1)=k$ and $\ka(\bullet_1)=k_\bullet$ where $r_1$ and $\bullet_1$ are respectively the root and the extra leaf of $A_1$.
\end{itemize}
The following proposition is the equivalent of Proposition \ref{prop:correlation} in the context of linear trees.

\begin{prop}\label{prop norme Y}
For all $\eta,\eta_\bullet\in\{0,1\}$, $\iota_\bullet\in\{\pm\}$ and $k,k_\bullet\in\mathbb{Z}_L^d$ we have
\begin{align*}
\E\pth{ \left|\reallywidehat{Y^\bi_{m,n}}(\eta,+,\eta_\bullet,\iota_\bullet)(\e^{-2}t,k,k_\bullet)\right|^2 } & = \sum_{C\in \mathcal{C}^\bi_{m,n}(\eta,\eta_\bullet,\iota_\bullet)} \reallywidehat{G_C}(\e^{-2}t,k,k_\bullet),
\\ \E\pth{ \left|\reallywidehat{Y^\ter_{m,n}}(\eta,+,\eta_\bullet,\iota_\bullet)(\e^{-2}t,\e^{-2}t_\bullet,k,k_\bullet)\right|^2 } & = \sum_{C\in \mathcal{C}^\ter_{m,n}(\eta,\eta_\bullet,\iota_\bullet)} \reallywidehat{G_C}(\e^{-2}t,\e^{-2}t_\bullet,k,k_\bullet),
\end{align*}
where if $C\in \mathcal{C}^\bi_{m,n}(\eta,\eta_\bullet,\iota_\bullet)$ then
\begin{align*}
\reallywidehat{G_C}(\e^{-2}t,k,k_\bullet) & = (-1)^{n_T(A')} (-i)^{n_T(C)} \frac{\e^{n_B^\tot(C)}}{L^{(n(C)-1)d}} \sum_{\ka\in\mathcal{D}_{k,k_\bullet}(C)} e^{i\e^{-2}\Om_{-1}t} \prod_{j\in\mathcal{N}^\tot(C)} q_j
\\&\hspace{3cm} \times \int_{I_C(t)} \prod_{j\in\mathcal{N}_T(C)} e^{i\e^{-2}\Om_j t_j} \d t_j \prod_{\ell\in\mathcal{L}(C)_+} M^{c(\si(\ell)),c(\ell)}(\ka(\ell))^{\ffi(\ell)},
\end{align*}
and if $C\in \mathcal{C}^\ter_{m,n}(\eta,\eta_\bullet,\iota_\bullet)$ then
\begin{align*}
\reallywidehat{G_C}(\e^{-2}t,\e^{-2}t_\bullet,k,k_\bullet) & = (-1)^{n_T(A')} (-i)^{n_T(C)} \frac{\e^{n_B^\tot(C)}}{L^{(n(C)-1)d}} \sum_{\ka\in\mathcal{D}_{k,k_\bullet}(C)} e^{i\e^{-2}\Om_{-1}t}e^{i\e^{-2}(\Om_\bullet +\Om_{\bullet'})t_\bullet} \prod_{j\in\mathcal{N}^\tot(C)} q_j
\\&\hspace{1cm} \times \int_{I_C(t,t_\bullet)} \prod_{j\in\mathcal{N}^\tot_T(C)\setminus\{\mathrm{int}(\bullet),\mathrm{int}(\bullet')\}} e^{i\e^{-2}\Om_j t_j} \d t_j \prod_{\ell\in\mathcal{L}(C)_+} M^{c(\si(\ell)),c(\ell)}(\ka(\ell))^{\ffi(\ell)},
\end{align*}
where in this last expression $\bullet$ and $\bullet'$ are the two extra leaves in the two linear trees $A$ and $A'$ composing the coupling, and where $I_C(t,t_\bullet) = I_A(t,t_\bullet) \times I_{A'}(t,t_\bullet)$.
\end{prop}

\begin{proof}
The proof follows exactly the same lines as the one of Proposition \ref{prop:correlation}. We first expand the expectancies using \eqref{def Y bi}-\eqref{def Y ter} and obtain in the case of $Y^\bi_{m,n}$
\begin{align*}
&\E\pth{ \left|\reallywidehat{Y^\bi_{m,n}}(\eta,+,\eta_\bullet,\iota_\bullet)(\e^{-2}t,k,k_\bullet)\right|^2 } 
\\&\quad = \sum_{\substack{(A,\ffi,c)\in \mathcal{A}^\bi_{m,n}(\eta,+,\eta_\bullet,\iota_\bullet) \\ (A',\ffi',c')\in \mathcal{A}^\bi_{m,n}(\eta,-,\eta_\bullet,-\iota_\bullet)}}\E\pth{G_{(A,\ffi,c)}(\e^{-2}t,k,k_\bullet) G_{(A',\ffi',c')}(\e^{-2}t,-k,-k_\bullet)},
\end{align*}
where we have used $G^\iota_{(A,\ffi,c)}=G_{(A,\iota\ffi,c)}$ which can be proved directly using the Fourier expression \eqref{G_A bi}. Using then this expression again and the Wick formula to extract the nontrivial terms in the expectancy leads to the formula in the proposition. The proof is identical for $Y^\ter_{m,n}$.
\end{proof}

Proposition \ref{prop norme Y} reduces estimating $Y^\bi_{m,n}(\eta,+,\eta_\bullet,\iota_\bullet)$ and $Y^\ter_{m,n}(\eta,+,\eta_\bullet,\iota_\bullet)$ to estimating each $\reallywidehat{G_C}$, which will be done in a similar way as for $\widehat{F_C}$.

\subsubsection{Bushes and dual decorations}

In order to estimate the $G_C$ as defined in Proposition \ref{prop norme Y}, we first need to slightly adapt the notion of bushes introduced in Section \ref{section bushes} to coupling of linear trees:
\begin{itemize}
\item The notions of marked and unmarked children from Definition \ref{def marked} is extended to $j\in \mathcal{N}^\tot(C)\sqcup \mathcal{L}(C)\sqcup\{\bullet,\bullet'\}$ where $\bullet$ and $\bullet'$ are the two extra leaves in the coupling $C$, and $\offspring(j)$ is now defined for $j\in\mathcal{N}^\tot_\low(C)$ by the same formula \eqref{def offsprings}.
\item The notion of prebush from Definition \ref{def prebush} is extended to $j\in\mathcal{N}^\tot(C)\sqcup\mathcal{L}(C)\sqcup\{\bullet,\bullet'\}$ by defining $\mathtt{prebush}(j)$ for $j\in\mathcal{N}^\lin(C)$ by \eqref{def prebush node} and $\mathtt{prebush}(\bullet)=\mathtt{prebush}(\bullet')=\emptyset$. Note that if $\min\enstq{j>\bullet}{\text{$j$ marked}}$ exists, its prebush might be reduced to $\mathtt{prebush}(\bullet)$ and thus empty, but in this case the prebush of its other marked sibling is not empty.
\item The notion of bush from Definition \ref{def bush} is directly extended to $j\in\mathcal{N}^\tot_\low(C)$ and the properties of Proposition \ref{prop bush} still holds (with $\mathcal{N}_\low(C)$ replaced by $\mathcal{N}^\tot_\low(C)$).
\item Self-coupled bushes and $n_\scb(C)$ are defined as in Definition \ref{def self-coupled bushes}. Moreover, the result of Lemma \ref{counting self-coupled bushes} becomes (for $\star=\bi,\ter$)
\begin{align}\label{counting nscb linear coupling}
\#\mathcal{C}^\star_{m,n,q}(\eta,\eta_\bullet,\iota_\bullet) \leq \La^{n+1}(n+2-q)!,
\end{align}
where $\mathcal{C}^\star_{m,n,q}(\eta,\eta_\bullet,\iota_\bullet) $ is the set of coupling in $\mathcal{C}^\star_{m,n}(\eta,\eta_\bullet,\iota_\bullet) $ such that $n_\scb=q$. Indeed, to go from Lemma \ref{counting self-coupled bushes} to \eqref{counting nscb linear coupling}, simply note that if $C\in \mathcal{C}^\star_{m,n,q}(\eta,\eta_\bullet,\iota_\bullet)$ is a coupling of two linear trees, then these trees can be seen as regular trees both in $\tilde{\mathcal{A}}_{n}$.
\end{itemize}

Finally, we also need to adapt the dual point of view on decorations from Definition \ref{def linear map}. For a linear coupling, the definition itself is unchanged but recall that $\{\bullet,\bullet'\}$ don't belong to $\mathcal{L}(C)$ and this has several consequences. If $\ka$ is a decoration of a coupling, it is now uniquely determined by $\ka_{|_{\mathcal{L}_+(C)\sqcup\{\bullet,\bullet'\}}}$. Therefore, if $\ka\in\mathcal{D}_{k,k_\bullet}(C)$, it can be seen as an element of $(\R^d)^{\mathcal{L}(C)_+}$ (since $\ka(\bullet)$ and $\ka(\bullet')$ are fixed) and if $j\in\mathcal{L}(C)\sqcup\mathcal{N}^\tot(C)$ then
\begin{equation*}
K(j)(\ka) = \left\{
\begin{aligned}
&\ka(j) \qquad\qquad \text{if $j\notin\mathcal{N}^\lin(C)$,}
\\& \ka(j) - k_\bullet \qquad \text{if $j\in\mathcal{N}^\lin(A)$,}
\\& \ka(j)+k_\bullet \qquad \text{if $j\in\mathcal{N}^\lin(A')$,}
\end{aligned}
\right.
\end{equation*}
where $A$ and $A'$ are the two linear trees in the coupling. In particular since the root of a linear tree is always a linear tree, we have $K(r)(\ka)=k-k_\bullet$ and $K(r')(\ka)=-k+k_\bullet$ where $r$ and $r'$ are the roots of $A$ and $A'$.

\subsubsection{Low nodes regime for the linear response}

In this section, we prove the following proposition, which is the equivalent for $G_C$ of Proposition \ref{prop low nodes}.

\begin{prop}\label{prop low nodes GC}
There exists $\a_\low,\La>0$ such that if $C\in\mathcal{C}^\star_{m,n}(\eta,\eta_\bullet,\iota_\bullet)$ for $\star=\bi,\ter$ is a coupling satisfying
\begin{align*}
n_B^\tot(C) + 2 (n^\tot_\low(C) - n_\scb(C)) & \geq \frac{n(C)-n_\scb(C)}{5},
\\ n(C)-n_\scb(C) & \geq 40,
\end{align*}
then for any $\e\leq \de$ and any $k,k_\bullet\in\Z_L^d$ we have for all $t\in[0,\de]$
\begin{align*}
\left| \widehat{G_C}(\e^{-2}t,[\e^{-2}t_\bullet,]k,k_\bullet) \right| \leq \left\{
\begin{aligned}
&\La^{n+1} \de^{n[-2]} \e^{\a_\low(n(C)-n_\scb(C))}\ps{k_\bullet}^{-\half} \qquad \text{if $k\neq k_\bullet$,}
\\&\La^{n+1} L^d \de^{n[-2]} \e^{\a_\low(n(C)-n_\scb(C))}\ps{k_\bullet}^{-\half} \quad \text{if $k = k_\bullet$,}
\end{aligned}
\right.
\end{align*}
where the content of the brackets only appears when $\star=\ter$, in which case the estimate holds for all $0\leq t_\bullet \leq t \leq \de$.
\end{prop}

\begin{remark}\label{rm:AnalysisZero} 
Compared to the estimate \eqref{estimate FC} for $\widehat{F_C}(\e^{-2}t,k)$, we see in the above proposition an extra factor $L^d$ when $k=k_\bullet$. The analogue situation for $\widehat{F_C}(\e^{-2}t,k)$ is $k=0$ for which the convergence is obvious thanks to the assumption \eqref{assumption Q} and its consequence $\widehat{F_C}(\e^{-2}t,0)=0$. Note that this extra factor in the estimate for $\widehat{G_C}(\e^{-2}t,[\e^{-2}t_\bullet,]k,k_\bullet)$ is merely polynomial with a fixed degree, still allowing to conclude the fixed point argument, see Section \ref{proof fixed point}.
\end{remark}

\begin{proof}
The proof of Proposition \ref{prop low nodes GC} follows the same line as the one of Proposition \ref{prop low nodes} which was spread over all the subsections of Section \ref{section low nodes regime}. Therefore, we don't give as much details here but rather highlight the differences.

First, the result of Lemma \ref{lem:récriturelow} is slightly changed. Taking into account that $q_{\parent(\bullet)}\lesssim \ps{\ka(\bullet)}^{-\half}$ and the fact that by definition if $\ka\in\mathcal{D}_{k,k_\bullet}(C)$ we have $\ka(\bullet)=k_\bullet$ we obtain
\begin{align*}
& \left|  \widehat{G_C}(\e^{-2}t,[\e^{-2}t_\bullet,]k,k_\bullet) \right| \leq  \La^{n+1} \e^{n_B^\tot} \de^{n_T^\bullet} \ps{k_\bullet}^{-\half} L^{-(n-1)d} \sum_{\substack{\ka\in\mathcal{D}_{k,k_\bullet}(C)\\ \ka(\mathcal{L}(C)_+)\subset B(0,R) }} \prod_{j\in\mathcal{N}_\low^\tot(C)}\chi_j
\end{align*}
where $n_T^\bullet\vcentcolon=\#\pth{ \mathcal{N}_T^\tot(C)\setminus \{ \mathrm{int}(\bullet),\mathrm{int}(\bullet')\} }$. Now, if $j\in \mathcal{N}_\low^\tot(C)$ and if $j_1$ and $j_2$ are its marked children we have
\begin{align*}
\ka(j_1)+\ka(j_2) = \sum_{\ell\in\mathtt{offsprings}(j)}\ka(\ell) + \nu(j)k_\bullet,
\end{align*}
where
\begin{equation*}
\nu(j)  = 
\left\{
\begin{aligned}
0& \qquad \text{if $j\notin\mathcal{N}^\lin_\low(C)$,}
\\0& \qquad \text{if $j\in\mathcal{N}^\lin_\low(C)$ but $\bullet$ or $\bullet'$ are comparable to the unmarked child of $j$,}
\\ 1& \qquad \text{if $j\in\mathcal{N}^\lin_\low(C)$ and $\bullet$ is comparable to one of the marked children of $j$,}
\\ -1& \qquad \text{if $j\in\mathcal{N}^\lin_\low(C)$ and $\bullet'$ is comparable to one of the marked children of $j$.}
\end{aligned}
\right.
\end{equation*}
Therefore we have
\begin{align*}
& \left|  \widehat{G_C}(\e^{-2}t,[\e^{-2}t_\bullet,]k,k_\bullet) \right| 
\\&\leq  \La^{n+1} \e^{n_B^\tot} \de^{n_T^\bullet} \ps{k_\bullet}^{-\half} L^{-(n-1)d} \sum_{\substack{\ka\in\mathcal{D}_{k,k_\bullet}(C)\\ \ka(\mathcal{L}(C)_+)\subset B(0,R) }} \prod_{j\in\mathcal{N}_\low^\tot(C)}\chi\pth{\left|  \sum_{\ell\in\mathtt{offsprings}(j)}\ka(\ell) + \nu(j)k_\bullet \right| \e^{-\ga}}.
\end{align*}
The constructions of Section \ref{section non-self-coupled-bushes} apply directly to the present context of linear trees, the only difference being due to the terms $\nu(j) k_\bullet$ in the above formula. The result of Lemma \ref{lem:summationslow} is thus changed into the following: there exists a family $(J_i)_{1\leq i \leq n''_\low}$ of low nodes and functions $\Xi_{J_i}$ such that 
\begin{align*}
& \left|  \widehat{G_C}(\e^{-2}t,[\e^{-2}t_\bullet,]k,k_\bullet) \right| 
\\&\leq  \La^{n+1} \e^{n_B^\tot} \de^{n_T^\bullet} \ps{k_\bullet}^{-\half} L^{-(n-1)d}
\\& \hspace{3cm} \times \sum_{\substack{\ka\in\mathcal{D}_{k,k_\bullet}(C)\\ \ka(\mathcal{L}(C)_+)\subset B(0,R) }} \prod_{i=1}^{n''_\low}\chi\pth{\left| \mathfrak{K}_\bush(J_{i})(\ka) + \Xi_{J_{i}}\pth{ \pth{ \mathfrak{K}_\bush(J_{i'})(\ka) }_{i'\in\llbracket 1,i-1\rrbracket},k_\bullet } \right| \e^{-\ga}},
\end{align*}
where $\mathfrak{K}_\bush(j)$ is still defined by \eqref{def k bush} and $n''_\low\geq \frac{n_\low-n_\scb}{2}$. Since the $J_i$ are such that $\pth{\mathfrak{K}_\bush(J_{i})}_{1\leq i \leq n''_\low}$ is linearly independent, the change of variables from Section \ref{section basis of V(C) and change of variable} can still be performed here and we obtain
\begin{align*}
& \left|  \widehat{G_C}(\e^{-2}t,[\e^{-2}t_\bullet,]k,k_\bullet) \right| 
\\&\leq  \La^{n+1} \e^{n_B^\tot} \de^{n_T^\bullet} \ps{k_\bullet}^{-\half} L^{-(n-1)d}
\\& \hspace{1cm} \times \sum_{\substack{\pth{x_1,\dots,x_{n''_\low}}\in\pth{\Z_L^d}^{n''_\low}\\\pth{y_1,\dots,y_{p}}\in \pth{\Z_L^d\cap B(0,R)}^{ p }}} \prod_{i=1}^{n''_\low}\chi\pth{\left| x_i + \Xi_{J_{i}}\pth{ \pth{ x_{i'} }_{i'\in\llbracket 1,i-1\rrbracket},k_\bullet } \right| \e^{-\ga}}\mathbbm{1}_{\sum_{i=1}^{n''_\low}a_ix_i + \sum_{i=1}^p b_i y_i=k-k_\bullet},
\end{align*}
where $p=n-n''_\low$ and the only difference with \eqref{estim FC intermediate bis} coming from the fact that the condition $\ka\in\mathcal{D}_{k,k_\bullet}(C)$ now translates to $K(r)(\ka)=k-k_\bullet$. 

We now estimate the RHS of the above inequality, as we did in Section \ref{section successive summations}. If one of the $b_i$ is not zero, then the estimates can be performed in the same way: the sum over the $y_i$ gives a factor $L^{(p-1)d}$ and the sum over the $x_i$ gives a factor $(L^d\e^{\ga d})^{n''_\low}$. Using $p=n-n''_\low$ we thus obtain
\begin{align*}
& \left|  \widehat{G_C}(\e^{-2}t,[\e^{-2}t_\bullet,]k,k_\bullet) \right| \leq  \La^{n+1} \de^{n_T^\bullet} \ps{k_\bullet}^{-\half}  \e^{n''_\low \ga d + n_B^\tot}.
\end{align*}
Similarly if one of the $a_i$ is not zero the sums over the $y_i$ and the $x_i$ give a factor $L^{pd}$ and $(L^d\e^{\ga d})^{n''_\low-1}$ respectively so that we obtain
\begin{align*}
& \left|  \widehat{G_C}(\e^{-2}t,[\e^{-2}t_\bullet,]k,k_\bullet) \right| \leq  \La^{n+1} \de^{n_T^\bullet} \ps{k_\bullet}^{-\half}  \e^{(n''_\low-1) \ga d + n_B^\tot}.
\end{align*}
If all the $a_i$ and the $b_i$ vanish, then the sum over the $y_i$ gives a factor $L^{pd}$ and the sum over the $x_i$ gives a factor $(L^d\e^{\ga d})^{n''_\low}$ so that
\begin{align*}
& \left|  \widehat{G_C}(\e^{-2}t,[\e^{-2}t_\bullet,]k,k_\bullet) \right| \leq \La^{n+1}  \de^{n_T^\bullet} \ps{k_\bullet}^{-\half} L^{d}  \e^{n''_\low \ga d + n_B^\tot} \mathbbm{1}_{k=k_\bullet}.
\end{align*}
We can then conclude the proof of the main estimate in Proposition \ref{prop low nodes GC} from these three estimates as we did at the end of Section \ref{section successive summations}, noting that if $\star=\bi$ then $n_T^\bullet=n_T$ and if $\star=\ter$ then $n_T^\bullet=n_T^\tot -2$.
\end{proof}

\subsubsection{High nodes regime for the linear response}

In this section, we prove the following proposition, which is the equivalent for $G_C$ of Proposition \ref{prop high nodes}.

\begin{prop}\label{prop high nodes GC}
There exists $\a_\high,\La>0$ such that if $C\in\mathcal{C}^\star_{m,n}(\eta,\eta_\bullet,\iota_\bullet)$ for $\star=\bi,\ter$ is a coupling satisfying
\begin{align*}
n_B^\tot(C) + 2 (n^\tot_\low(C) - n_\scb(C)) & < \frac{n(C)-n_\scb(C)}{5},
\\ n(C)-n_\scb(C) & \geq 30,
\\ n(C) & \leq |\ln\e|^{3},
\end{align*}
then for any $\e\leq \de$ and any $k,k_\bullet\in\Z_L^d$ we have for all $t\in[0,\de]$

\begin{align*}
\left| \widehat{G_C}(\e^{-2}t,[\e^{-2}t_\bullet,]k,k_\bullet) \right| \leq \left\{
\begin{aligned}
&\La^{n+1} \de^{n[-2]} \e^{\a_\high(n(C)-n_\scb(C))}\ps{k_\bullet}^{-\half} \qquad \text{if $k\neq k_\bullet$,}
\\&\La^{n+1} L^d \de^{n[-2]} \e^{\a_\high(n(C)-n_\scb(C))}\ps{k_\bullet}^{-\half} \quad \text{if $k = k_\bullet$,}
\end{aligned}
\right.
\end{align*}
where the content of the brackets only appears when $\star=\ter$, in which case the estimate holds for all $0\leq t_\bullet \leq t \leq \de$.
\end{prop}

\begin{proof}
The proof of Proposition \ref{prop high nodes GC} follows the same line as the one of Proposition \ref{prop high nodes} which was spread over all the subsections of Section \ref{section high nodes regime}. Therefore, we don't give as much details here but rather highlight the differences. However, as opposed to the low node regime, here the cases $\star=\bi$ and $\star=\ter$ are different and we start by treating them differently before concluding the estimates for the two cases together.

\paragraph{The case $\star=\bi$.} This case requires almost zero modification since it corresponds to the situation where there are no linear ternary nodes and thus no oscillating integral coming from linear nodes. In particular, the results of Section \ref{section oscillating integrals} apply without any difference and we have
\begin{align*}
\widehat{G_C}(\e^{-2}t,k,k_\bullet) & = \sum_{\rho \in \mathfrak{M}(\mathcal{N}_T(C))} \widehat{G_C^\rho}(\e^{-2}t,k,k_\bullet)
\end{align*}
where $\widehat{G^\rho_C}(\e^{-2}t,k,k_\bullet)$ satisfies for all $t\in[0,\de]$
\begin{align*}
\left| \widehat{G^\rho_C}(\e^{-2}t,k,k_\bullet) \right| & \leq \La^{n+1} \ps{k_\bullet}^{-\half} \varepsilon^{n^\tot_B} \int_\R \frac{ A_C^\rho(k,k_\bullet,\xi)^{\frac{1}{p}} }{ \left| n_T\de^{-1} - i\xi \right|}   \d\xi
\end{align*}
for some $p>2d$ and where 
\begin{align*}
A_C^\rho(k,k_\bullet,\xi) & \vcentcolon = L^{-(n-1)d} \sum_{\substack{\ka\in\mathcal D_{k,k_\bullet}(C)\\ \ka(\mathcal{L}(C)_+)\subset B(0,R) }} \prod_{j\in (\mathcal{N}(C)\sqcup\mathcal{L}(C))\setminus \mathcal{R}(C)} \ps{\ka(j)}^{-\frac{p}{2}}  \prod_{j\in\mathcal{N}_T(C)} \frac{\chi_j^p}{\left|n_T\delta^{-1}-i(\xi - \omega^\rho_{j})\right|^p} ,
\end{align*}
where $\omega^\rho_{j}$ is defined as in \eqref{alternative omega rho j}. From a given $\rho \in \mathfrak{M}(\mathcal{N}_T(C))$ we can again construct a total order relation $<_{\tilde{\rho}}$ on $\mathcal{N}^\tot(C)\sqcup \mathcal{L}(C)$ as in Section \ref{section induced orders}, the only difference being that $\bullet$ and $\bullet'$ are not included in $\mathcal{L}(C)$. For this reason, the family $\mathcal{N}'(C)$ from Proposition \ref{prop LukSpo} is now of cardinal $n$.

\paragraph{The case $\star=\ter$.} This case requires a different treatment of oscillating integrals due to mainly to the difference between $I_C(t)$ and $I_C(t,t_\bullet)$, which we recall is defined to be $I_A(t,t_\bullet)\times I_{A'}(t,t_\bullet)$ (where $A$ and $A'$ are the two trees composing the coupling $C$). Therefore we have
\begin{align}
\int_{I_C(t,t_\bullet)}& \prod_{j\in\mathcal{N}^\tot_T(C)\setminus\{\mathrm{int}(\bullet),\mathrm{int}(\bullet')\}} e^{i\e^{-2}\Om_j t_j} \d t_j  \label{int sur I_C(.,.)}
\\& = \int_{I_A(t,t_\bullet)} \prod_{j\in\mathcal{N}^\tot_T(A)\setminus\{\mathrm{int}(\bullet)\}} e^{i\e^{-2}\Om_j t_j} \d t_j  \times \int_{I_{A'}(t,t_\bullet)} \prod_{j\in\mathcal{N}^\tot_T(A')\setminus\{\mathrm{int}(\bullet')\}} e^{i\e^{-2}\Om_j t_j} \d t_j \non
\end{align}
Note that there exists a negligible set $\mathcal{Z}_A$ such that
\begin{align}\label{IA(ttbullet)}
I_A(t,t_\bullet) = \mathcal{Z}_A \sqcup \bigsqcup_{\rho\in\mathfrak{M}(\mathcal{N}_T^\tot(A))} I^\rho_A(t,t_\bullet)
\end{align}
where $I^\rho_A(t,t_\bullet)$ is defined to be the set of $(t_j)_{j\in \mathcal{N}_T^\tot(A)\setminus \{\mathrm{int}(\bullet)\}} \in [0,t]^{n_T^\tot-1}$ such that
\begin{align*}
 0 \leq t_{\rho(1)} < \cdots < t_{\rho(\rho^{-1}(\mathrm{int}(\bullet))-1)} \leq t_\bullet \leq t_{\rho(\rho^{-1}(\mathrm{int}(\bullet))+1)} < \cdots < t_{\rho(n_T^\tot)} \leq t .
\end{align*}
By denoting
\begin{align*}
\mathcal{N}_T^{>_\rho}(A) & = \rho\pth{ \llbracket \rho^{-1}(\mathrm{int}(\bullet)) + 1, n_T^\tot \rrbracket }, \qquad n_T^{>_\rho} = \# \mathcal{N}_T^{>_\rho}(A),
\\ \mathcal{N}_T^{<_\rho}(A) & = \rho\pth{ \llbracket 1, \rho^{-1}(\mathrm{int}(\bullet)) - 1 \rrbracket }, \qquad n_T^{<_\rho} = \# \mathcal{N}_T^{<_\rho}(A),
\end{align*}
and
\begin{align*}
I_A^{\rho,>}(t,t_\bullet) & = \enstq{ (t_j)_{j\in \mathcal{N}_T^{>_\rho}(A)}\in[t_\bullet,t]^{n_T^{>_\rho}} }{ t_\bullet\leq t_{\rho(\rho^{-1}(\mathrm{int}(\bullet))+1)} < \cdots < t_{\rho(n_T^\tot)} \leq t },
\\  I_A^{\rho,<}(t_\bullet) & = \enstq{ (t_j)_{j\in \mathcal{N}_T^{<_\rho}(A)}\in[0,t_\bullet]^{n_T^{<_\rho}} }{  0 \leq t_{\rho(1)} < \cdots < t_{\rho(\rho^{-1}(\mathrm{int}(\bullet))-1)} \leq t_\bullet },
\end{align*}
we find that $I^\rho_A(t,t_\bullet)=I_A^{\rho,<}(t_\bullet) \times I_A^{\rho,>}(t,t_\bullet)$. We also note that
\begin{align*}
I_A^{\rho,>}(t,t_\bullet) & = (t_\bullet,\dots,t_\bullet) + I_A^{\rho,>}(t-t_\bullet,0).
\end{align*}
Therefore we have obtained
\begin{align*}
&\int_{I_A(t,t_\bullet)} \prod_{j\in\mathcal{N}^\tot_T(A)\setminus\{\mathrm{int}(\bullet)\}} e^{i\e^{-2}\Om_j t_j} \d t_j 
\\& = \sum_{\rho\in\mathfrak{M}(\mathcal{N}_T^\tot(A))} \prod_{j\in \mathcal{N}_T^{>_\rho}(A)} e^{i\e^{-2}\Om_j t_\bullet } \int_{I_A^{\rho,>}(t-t_\bullet,0)} \prod_{j\in\mathcal{N}_T^{>_\rho}(A)} e^{i\e^{-2}\Om_j t_j} \d t_j \times \int_{I_A^{\rho,<}(t_\bullet)} \prod_{j\in\mathcal{N}_T^{<_\rho}(A)} e^{i\e^{-2}\Om_j t_j} \d t_j.
\end{align*}
We can now apply the proof of Lemma \ref{lemma:oscillations} to these two integrals and get
\begin{align}
\int_{I_A(t,t_\bullet)} \prod_{j\in\mathcal{N}^\tot_T(A)\setminus\{\mathrm{int}(\bullet)\}} e^{i\e^{-2}\Om_j t_j} \d t_j  & = \sum_{\rho\in\mathfrak{M}(\mathcal{N}_T^\tot(A))} e^{i \om_\bullet^\rho t_\bullet} \frac{e^{\pth{n_T^\tot(A)-1}\de^{-1} t}}{2\pi} \label{int sur I_A(.,.)}
\\&\quad \times \int_{\R} \frac{e^{-i\xi (t-t_\bullet)}}{n_T^{>_\rho}(A)\de^{-1}-i\xi} \prod_{j\in\mathcal{N}_T^{>_\rho}(A)} \frac{1}{n_T^{>_\rho}(A)\de^{-1} - i (\xi-\om_j^\rho)}\; \d\xi \non 
\\&\quad \times \int_{\R} \frac{e^{-i\xi t_\bullet }}{n_T^{<_\rho}(A)\de^{-1}-i\xi} \prod_{j\in\mathcal{N}_T^{<_\rho}(A)} \frac{1}{n_T^{<_\rho}(A)\de^{-1} - i (\xi-\om_j^\rho)}\; \d\xi, \non
\end{align}
where we defined
\begin{equation*}
\om_\bullet^\rho \vcentcolon = \e^{-2} \sum_{j\in \mathcal{N}_T^{>_\rho}(A)}\Om_j \quad\text{and}\quad  \om_j^\rho \vcentcolon = 
\left\{
\begin{aligned}
\e^{-2} \sum_{\substack{j'\in \mathcal{N}_T^{<_\rho}(A)\\ j'\geq_\rho j}} \Om_{j'}, & \quad \text{if $j\in\mathcal{N}_T^{<_\rho}(A)$,}
\\ \e^{-2} \sum_{\substack{j'\in \mathcal{N}_T^\tot(A)\\ j'\geq_\rho j}} \Om_{j'}, & \quad \text{if $j\in\mathcal{N}_T^{>_\rho}(A)$,}
\end{aligned}
\right.
\end{equation*}
and where by convention in the RHS of \eqref{int sur I_A(.,.)} if one of the sets $\mathcal{N}_T^{>_\rho}(A)$ or $\mathcal{N}_T^{<_\rho}(A)$ is empty then we remove entirely the corresponding integral over $\R$. Of course, a formula similar to \eqref{int sur I_A(.,.)} also holds for $A'$. Inserting them into \eqref{int sur I_C(.,.)} and applying the proof of Lemma \ref{lem:holder} we finally obtain the following decomposition of $\widehat{G_C}(\e^{-2}t,\e^{-2}t_\bullet,k,k_\bullet)$ in the case $\star=\ter$:
\begin{align*}
\widehat{G_C}(\e^{-2}t,\e^{-2}t_\bullet,k,k_\bullet) & = \sum_{\substack{ \rho\in\\\rho'\in}} \widehat{G_C^{\rho,\rho'}}(\e^{-2}t,\e^{-2}t_\bullet,k,k_\bullet)
\end{align*}
where $\widehat{G_C^{\rho,\rho'}}(\e^{-2}t,\e^{-2}t_\bullet,k,k_\bullet)$ satisfies for all $0\leq t_\bullet \leq t \leq \de$:
\begin{align*}
&\left| \widehat{G_C^{\rho,\rho'}}(\e^{-2}t,\e^{-2}t_\bullet,k,k_\bullet) \right|
\\&\leq \La^{n+1} \ps{k_\bullet}^{-\half} \e^{n_B^\tot} \int_{\R^4} \frac{A^{\rho,\rho'}_C(k,k_\bullet,\xi_1,\xi_2,\xi_3,\xi_4)^{\frac{1}{p}}\d\xi_1\d\xi_2\d\xi_3\d\xi_4}{ \left| n_T^{>_\rho}(A)\de^{-1}-i\xi_1 \right| \left| n_T^{<_\rho}(A)\de^{-1}-i\xi_2 \right| \left| n_T^{>_{\rho'}}(A')\de^{-1}-i\xi_3 \right| \left| n_T^{<_{\rho'}}(A')\de^{-1}-i\xi_4 \right|} 
\end{align*}
for some $p>2d$ and where
\begin{align*}
A^{\rho,\rho'}_C&(k,k_\bullet,\xi_1,\xi_2,\xi_3,\xi_4)^{\frac{1}{p}} 
\\& \vcentcolon = L^{-(n-1)d} \sum_{\substack{\ka\in\mathcal D_{k,k_\bullet}(C)\\ \ka(\mathcal{L}(C)_+)\subset B(0,R) }}  \prod_{j\in (\mathcal{N}(C)\sqcup\mathcal{L}(C))\setminus \mathcal{R}(C)} \ps{\ka(j)}^{-\frac{p}{2}}
\\&\qquad \times \prod_{j\in \mathcal{N}_T^{>_\rho}(A)} \frac{\chi_j^p}{\left| n_T^{>_\rho}(A)\de^{-1}-i(\xi_1-\om_j^\rho) \right|^p}  \prod_{j\in \mathcal{N}_T^{<_\rho}(A)}\frac{\chi_j^p}{\left| n_T^{<_\rho}(A)\de^{-1}-i(\xi_2-\om_j^\rho) \right|^p}
\\&\qquad \times \prod_{j\in \mathcal{N}_T^{>_\rho}(A)} \frac{\chi_j^p}{\left| n_T^{>_{\rho'}}(A')\de^{-1}-i(\xi_3-\om_j^{\rho'}) \right|^p} \prod_{j\in \mathcal{N}_T^{<_\rho}(A)}\frac{\chi_j^p}{\left| n_T^{<_{\rho'}}(A')\de^{-1}-i(\xi_4-\om_j^{\rho'}) \right|^p}.
\end{align*}
Again, these formulas hold with the convention that if one of the sets $\mathcal{N}_T^{>_\rho}(A)$, $\mathcal{N}_T^{<_\rho}(A)$, $\mathcal{N}_T^{>_{\rho'}}(A')$ or $\mathcal{N}_T^{>_{\rho'}}(A')$ is empty then we remove entirely the corresponding integral over $\R$, and in such case $A^{\rho,\rho'}_C$ would only depend on three (or less) $\xi$'s. 

The content of Section \ref{section induced orders} must be changed. We have first to define an order given two orders $\rho$, $\rho'$ on $\mathcal{N}_T^\tot(A)$, $\mathcal{N}_T^\tot(A')$. Since the nodes in $A$ are not comparable to the nodes in $A'$, we build $\tilde \rho$ from the order $\rho*\rho'$ which is defined by
\[
j'<_{\rho*\rho'} j \; \iff \; \left \lbrace{\begin{array}{cc}
j'<_\rho j & \mathrm{if}\, j,j'\in\mathcal N_T^\tot(A) \\
j'<_{\rho'} j & \mathrm{if}\, j,j'\in \mathcal N_T^\tot(A') \\
j'\in \mathcal N_T^\tot(A'), j\in \mathcal N_T^\tot(A) & \mathrm{otherwise.}
\end{array}} \right.
\]
We define $\mathcal N'(C)$ as before and as in the case $\star=\bi$, it has size $n$.

\paragraph{The common endgame.} The end of the proof holds for both $\star=\bi$ and $\star=\ter$. We define the family $\mathcal{N}''(C)$ by
\begin{equation*}
\mathcal{N}''(C) = \Bigg\{ j\in\mathcal{N}'(C) \;\Bigg| \;
\begin{aligned}
& \mathtt{parent}(j)\notin\{\mathtt{int}(\bullet),\mathtt{int}(\bullet')\}
\\ & j = \min_{\tilde{\rho}}\pth{ \mathtt{siblings}(j)\cap \mathcal{N}'(C)} 
\end{aligned} \;\Bigg\}
\end{equation*}
while the definition of the sets $\mathcal{P}_i$ is the same as in Definition \ref{def N'' Pi}. The results of Lemma \ref{lem N''} must be changed to the following:
\begin{itemize}
\item[(i)] If $K(\mathtt{root}(A)) = 0$, then $K(\mathtt{root}(A')) =0$. In that case, if $k\neq k_\bullet$, then $\reallywidehat{G_C}(\e^{-2}t,[\varepsilon^{-2}t_\bullet,]k,k_\bullet) =0$ and the estimate of the proposition holds. Otherwise we have $\mathtt{root}(A) \in \mathcal N''(C)$ and $\mathtt{root}(A')\notin \mathcal N''(C)$.
\item[(ii)] The $\mathcal P_i$'s satisfy $\mathcal P_3 = \emptyset $, $ \mathcal P_2  \subset  \mathcal N_T(C)\setminus \{\mathtt{int}(\bullet),\mathtt{int}(\bullet')\}$ and $\# \mathcal P_1 +2\#\mathcal P_2  \geq  n-3$.
\item[(iii)] The cardinal $n''$ of $\mathcal{N}''(C)$ satisfy $n''\geq \frac{n-3}{2}$.
\end{itemize}
Moreover, there again exists at least one ternary node that is a parent of a node in $\mathcal N''(C)$, otherwise we would have $n_B^{\tot}(C) \geq \frac{n-1}4$ which would contradict $n_B^{\tot}(C)<\frac{n}5$ and $n\geq 30$ according to the high nodes regime assumption. In the case $K(\mathtt{root}(A))\neq 0$, the results in Section \ref{subsubsec:changevar} remain valid. But in the case $K(\mathtt{root}(A))= 0$ and $k=k_\bullet$ we have no choice but to sum on $n$ variables and we obtain
\begin{align*}
&A_C^\rho(k,k_\bullet,\xi,[\xi_2,\xi_3,\xi_4]) 
\\& \leq \La^{n+1}  L^{-(n-1)d} \pth{\frac{\de}{n_T}}^{p\pth{n_T^\bullet-n''_T}} 
\\&\quad  \sum_{(k_j)_{j\in\mathcal{N}'(C)\setminus\mathcal{N}''(C)}\in\pth{\frac{1}{L}\mathbb{Z}^d}^{n-n''}}   \prod_{j\in \mathcal{N}'(C)\setminus \mathcal{N}''(C) } \ps{k_j}^{-p}
\\&\hspace{6cm} \times \pth{\sum_{x_{n''}\in\mathbb{Z}_L^d}\tilde{A}_{j_{n''}}\pth{  \cdots \pth{ \sum_{x_1\in\mathbb{Z}_L^d}\tilde{A}_{j_1} } \cdots    } }  .
\end{align*}
The rest of the proof follows the same lines as what is done in Section \ref{section proof high nodes}. Regarding the numerology, we now remark that 
\[
2n''_{\mathrm{gain}} \geq n-3 - 4 n_B^\tot -8 -n_\low^\tot
\]
from which we deduce
\[
2n''_{\mathrm{gain}} + 2n_B^\tot \geq \frac35 (n - n_\scb) -11
\]
which is bigger that $\frac{n-n_\scb}{5}$ if $n-n_\scb\geq 30$. This concludes the proof of Proposition \ref{prop high nodes GC}.
\end{proof}

\subsubsection{Time derivative of $Y^\star_{m,n}$}

As we did for the $X_n$ in Section \ref{section time derivative Xn}, we need to estimate the derivative in time of the $Y^\star_{m,n}$ for $\star=\bi,\ter$. Recall that if $\mathcal{N}_T^\lin(A)=\emptyset$, then $\mathrm{int}(\bullet)$ is not defined and thus in \eqref{decomp dt G} below the set $\mathcal N_T^{\tot,\mathrm{max}}(A) \setminus \{ \mathrm{int}(\bullet) \}$ simply denotes $\mathcal N_T^{\tot,\mathrm{max}}(A) $.

\begin{prop}
We have
\begin{align*}
\reallywidehat{\dr_tY^\star_{m,n}}(\eta,\iota,\eta_\bullet,\iota_\bullet) & = \sum_{(A,\ffi,c)\in  \mathcal{A}^\star_{m,n}(\eta,\iota,\eta_\bullet,\iota_\bullet)} \reallywidehat{\dr_tG_{(A,\ffi,c)}}
\end{align*}
together with the decomposition
\begin{align}\label{decomp dt G}
\reallywidehat{\dr_tG_{(A,\ffi,c)}} & = \sum_{j_{ \mathrm{max} }\in\{-1\}\sqcup (\mathcal N_T^{\tot,\mathrm{max}}(A) \setminus \{ \mathrm{int}(\bullet) \} ) } \reallywidehat{ G_{(A,\ffi,c),j_{ \mathrm{max} }} }
\end{align}
where
\begin{align*}
\mathcal N_T^{\tot,\mathrm{max}}(A) \vcentcolon = \enstq{ j \in \mathcal N_T^{\tot}(A) }{ \nexists\, j' \in \mathcal N_T^\tot(A),\; j'>j }
\end{align*}
and where $\reallywidehat{ G_{(A,\ffi,c),j_{ \mathrm{max} }} }$ is defined by:
\begin{itemize}
\item[(i)] If $\star=\bi$, then
\begin{itemize}
\item if $j_{\mathrm{max}}=-1$, then
\begin{align*}
\reallywidehat{ G_{(A,\ffi,c),-1} } (\e^{-2}t,k,k_\bullet)& \vcentcolon = -(-i)^{n_T+1}\frac{\e^{n_B^\tot}}{L^{(n-1)\frac{d}{2}}} \sum_{\ka\in\mathcal{D}_{k,k_\bullet}(A)} \Om_{-1} e^{i\Om_{-1}\e^{-2}t}
\\&\quad \times \prod_{j\in\mathcal{N}^\tot(A)}q_j \int_{I_A(t)}\prod_{j\in\mathcal{N}_T(A)}e^{i\Om_j \e^{-2}t_j}\d t_j \prod_{\ell\in\mathcal{L}(A)} \reallywidehat{X^{c(\ell), \ffi(\ell)}_\init} (\ka(\ell)),
\end{align*}
\item if $j_{\mathrm{max}}\in\mathcal N_T^{\tot,\mathrm{max}}(A) $, then
\begin{align*}
\reallywidehat{ G_{(A,\ffi,c),j_{\mathrm{max}}} }& (\e^{-2}t,k,k_\bullet)
\\& \vcentcolon = (-i)^{n_T}\frac{\e^{n_B^\tot+2}}{L^{(n-1)\frac{d}{2}}} \sum_{\ka\in\mathcal{D}_{k,k_\bullet}(A)} e^{i\pth{ \Om_{-1}+\Om_{j_{\mathrm{max}}}}\e^{-2}t}
\\&\hspace{1cm} \times \prod_{j\in\mathcal{N}^\tot(A)}q_j     \int_{I^{ j_{\mathrm{max}} }_A(t)}  \prod_{j\in\mathcal{N}_T(A)\setminus\{ j_{\mathrm{max}} \}} e^{i\Om_j \e^{-2}t_j}\d t_j \prod_{\ell\in\mathcal{L}(A)} \reallywidehat{X^{c(\ell), \ffi(\ell)}_\init} (\ka(\ell))
\end{align*}
where $I^{ j_{\mathrm{max}} }_A(t)$ is defined as in \eqref{def I_A^jmax}.
\end{itemize}
\item[(ii)] If $\star=\ter$, then
\begin{itemize}
\item if $j_{\mathrm{max}}=-1$, then
\begin{align*}
&\reallywidehat{ G_{(A,\ffi,c),-1} } (\e^{-2}t,\e^{-2}t_\bullet,k,k_\bullet)
\\&\hspace{1.5cm} \vcentcolon = - (-i)^{n^\tot_T+1} \frac{\e^{n_B^\tot}}{L^{(n-1)\frac{d}{2}}}  \sum_{\ka\in\mathcal{D}_{k,k_\bullet}(A)} \Om_{-1} e^{i\Om_{-1}\e^{-2}t} e^{i\Om_\bullet \e^{-2}t_\bullet}
\\&\quad\hspace{1.5cm} \times \prod_{j\in\mathcal{N}^\tot(A)}q_j \int_{I_A(t,t_\bullet)} \prod_{j\in \mathcal{N}^\tot_T(A)\setminus\{\mathrm{int}(\bullet)\}} e^{i\Om_j \e^{-2}t_j}\d t_j  \prod_{\ell\in\mathcal{L}(A)} \reallywidehat{X^{c(\ell), \ffi(\ell)}_\init} (\ka(\ell)),
\end{align*}
\item if $j_{\mathrm{max}}\in\mathcal N_T^{\tot,\mathrm{max}}(A) \setminus \{ \mathrm{int}(\bullet) \}$, then
\begin{align*}
&\reallywidehat{ G_{(A,\ffi,c),j_{\mathrm{max}}} } (\e^{-2}t,\e^{-2}t_\bullet,k,k_\bullet)
\\&\hspace{1cm} \vcentcolon =  (-i)^{n^\tot_T} \frac{\e^{n_B^\tot+2}}{L^{(n-1)\frac{d}{2}}}  \sum_{\ka\in\mathcal{D}_{k,k_\bullet}(A)} e^{i\pth{\Om_{-1}+\Om_{j_{\mathrm{max}}}}\e^{-2}t}e^{i\Om_\bullet \e^{-2}t_\bullet}
\\&\quad\hspace{1cm} \times \prod_{j\in\mathcal{N}^\tot(A)}q_j     \int_{I_A^{j_{\mathrm{max}}  }(t,t_\bullet)} \prod_{j\in\mathcal{N}_T^\tot(A)\setminus\{\mathrm{int}(\bullet), j_{\mathrm{max}} \}} e^{i\Om_j \e^{-2}t_j} \d t_j \prod_{\ell\in\mathcal{L}(A)} \reallywidehat{X^{c(\ell), \ffi(\ell)}_\init} (\ka(\ell)),
\end{align*}
where
\end{itemize}
\end{itemize}
\begin{equation*}
I_A^{j_{\mathrm{max}}}(t,t_\bullet)  \vcentcolon =   
\Bigg\{ (t_j)_{j\in\mathcal{N}_T^\tot(A)\setminus\{\mathrm{int}(\bullet), j_{\mathrm{max}} \}} \in [0,t]^{n^\tot_T-2} \; \Bigg| \; j < j'\Rightarrow t_j<t_{j'} \;\;\text{and}\;\;
\begin{aligned}
&j>\mathrm{int}(\bullet)\Rightarrow t_j>t_\bullet 
\\& j<\mathrm{int}(\bullet)\Rightarrow t_j<t_\bullet 
\end{aligned}
\Bigg\}.
\end{equation*}
\end{prop}

\begin{proof}
If $\star=\bi$ we can repeat the proof of Proposition \ref{prop dt X} now applied to \eqref{G_A bi}. If $\star=\ter$, we need to compute the derivative with respect to $t$ of the integral over $I_A(t,t_\bullet)$. Let $f_j$ be some functions of time, thanks to \eqref{IA(ttbullet)} and $I^\rho_A(t,t_\bullet)=I_A^{\rho,<}(t_\bullet) \times I_A^{\rho,>}(t,t_\bullet)$ we have
\begin{align*}
\int_{I_A(t,t_\bullet)} &\prod_{j\in\mathcal{N}_T^\tot(A)\setminus\{\mathrm{int}(\bullet)\}} f_j(t_j)\d t_j 
\\& = \sum_{\substack{\rho\in\mathfrak{M}(\mathcal{N}_T^\tot(A))\\\rho^{-1}(\mathrm{int}(\bullet))\leq n_T^\tot-1}} \int_{I_A^{\rho,<}(t_\bullet)} \prod_{j\in\mathcal{N}_T^{<_\rho}(A)} f_j(t_j)\d t_j 
\\&\hspace{4cm} \times \int_{t_\bullet\leq t_{\rho(\rho^{-1}(\mathrm{int}(\bullet))+1)} < \cdots < t_{\rho(n_T^\tot)} \leq t  } \prod_{i=\rho(\rho^{-1}(\mathrm{int}(\bullet))+1)}^{n_T^\tot} f_{\rho(i)}(t_{\rho(i)})\d t_{\rho(i)}
\\&\quad + \sum_{\substack{\rho\in\mathfrak{M}(\mathcal{N}_T^\tot(A))\\\rho^{-1}(\mathrm{int}(\bullet)) = n_T^\tot}} \int_{0\leq t_{\rho(1)} < \cdots < t_{\rho(n_T^\tot-1)} \leq t_\bullet}  \prod_{i=1}^{n_T^\tot-1} f_{\rho(i)}(t_{\rho(i)})\d t_{\rho(i)}.
\end{align*}
In this formula we note that the second sum does not depend on $t$ but only on $t_\bullet$, while in the first sum only the second integral depends on $t$. Therefore we have
\begin{align*}
&\dr_t \pth{ \int_{I_A(t,t_\bullet)} \prod_{j\in\mathcal{N}_T^\tot(A)\setminus\{\mathrm{int}(\bullet)\}} f_j(t_j)\d t_j  }
\\& = \sum_{\substack{\rho\in\mathfrak{M}(\mathcal{N}_T^\tot(A))\\\rho^{-1}(\mathrm{int}(\bullet))\leq n_T^\tot-1}} \int_{I_A^{\rho,<}(t_\bullet)} \prod_{j\in\mathcal{N}_T^{<_\rho}(A)} f_j(t_j)\d t_j 
\\&\hspace{3cm} \times \dr_t  \pth{\int_{t_\bullet\leq t_{\rho(\rho^{-1}(\mathrm{int}(\bullet))+1)} < \cdots < t_{\rho(n_T^\tot)} \leq t  } \prod_{i=\rho(\rho^{-1}(\mathrm{int}(\bullet))+1)}^{n_T^\tot} f_{\rho(i)}(t_{\rho(i)})\d t_{\rho(i)}} .
\end{align*}
Rewriting the second integral as successive time integrals as we did in the proof of Proposition \ref{prop dt X} we obtain
\begin{align*}
&\dr_t \pth{ \int_{I_A(t,t_\bullet)} \prod_{j\in\mathcal{N}_T^\tot(A)\setminus\{\mathrm{int}(\bullet)\}} f_j(t_j)\d t_j  }
\\& = \sum_{\substack{\rho\in\mathfrak{M}(\mathcal{N}_T^\tot(A))\\\rho^{-1}(\mathrm{int}(\bullet))\leq n_T^\tot-1}} \int_{I_A^{\rho,<}(t_\bullet)} \prod_{j\in\mathcal{N}_T^{<_\rho}(A)} f_j(t_j)\d t_j 
\\&\hspace{3cm} \times f_{\rho(n_T^\tot)}(t) \int_{t_\bullet\leq t_{\rho(\rho^{-1}(\mathrm{int}(\bullet))+1)} < \cdots < t_{\rho(n_T^\tot-1)} \leq t  } \prod_{i=\rho(\rho^{-1}(\mathrm{int}(\bullet))+1)}^{n_T^\tot-1} f_{\rho(i)}(t_{\rho(i)})\d t_{\rho(i)} .
\end{align*}
Now, by definition of $\mathcal{N}_T^{\tot,\mathrm{max}}(A)$ the following map is a bijection
\begin{align*}
 \enstq{ \rho\in\mathfrak{M}(\mathcal{N}_T^\tot(A))}{ \rho^{-1}(\mathrm{int}(\bullet))\leq n_T^\tot-1 } & \longrightarrow \bigsqcup_{j_{\mathrm{max}}\in \mathcal{N}_T^{\tot,\mathrm{max}}(A)\setminus\{ \mathrm{int}(\bullet) \}} \Big( \{ j_{\mathrm{max}}\}\times \mathfrak{M}\pth{\mathcal{N}_T^\tot(A)\setminus\{j_{\mathrm{max}}\}} \Big)
\\ \rho \hspace{3cm} & \longmapsto \hspace{3cm} \pth{ \rho(n_T^\tot),\rho_{|_{ \left\llbracket 1, n_T^\tot-1\right\rrbracket }} }
\end{align*}
which allows to perform a change of variable in the last expression and obtain
\begin{align*}
&\dr_t \pth{ \int_{I_A(t,t_\bullet)} \prod_{j\in\mathcal{N}_T^\tot(A)\setminus\{\mathrm{int}(\bullet)\}} f_j(t_j)\d t_j  }
\\&\qquad = \sum_{j_{\mathrm{max}}\in\mathcal{N}_T^{\tot,\mathrm{max}}(A)\setminus\{\mathrm{int}(\bullet)\}} f_{j_{\mathrm{max}}}(t) \sum_{\rho\in \mathfrak{M}\pth{\mathcal{N}_T^\tot(A)\setminus\{j_{\mathrm{max}}\}} } \int_{I_A^\rho(t,t_\bullet)} \prod_{j\in \mathcal{N}_T^\tot(A) \setminus\{j_{ \mathrm{max} }\} }f_j(t_j)\d t_j
\\&\qquad = \sum_{j_{\mathrm{max}}\in\mathcal{N}_T^{\tot,\mathrm{max}}(A)\setminus\{\mathrm{int}(\bullet)\}}  f_{j_{\mathrm{max}}}(t) \int_{I_A^{j_{\mathrm{max}}  }(t,t_\bullet)} \prod_{j\in\mathcal{N}_T^\tot(A)\setminus\{\mathrm{int}(\bullet), j_{\mathrm{max}} \}} f_j(t_j)\d t_j.
\end{align*}
where we used \eqref{IA(ttbullet)} now applied to the set $\mathfrak{M}\pth{\mathcal{N}_T^\tot(A)\setminus\{j_{\mathrm{max}}\}}$. We can then use this formula with $f_j(t_j)=e^{i\Om_j t_j}$ to differentiate \eqref{G_A ter} and conclude the proof of the proposition, after the usual dilation in time.
\end{proof}

From this proposition and the usual Wick formula, we obtain the following decomposition
\begin{align}
&\E\pth{ \left| \reallywidehat{\dr_t Y^\star_{m,n}}(\eta,+,\eta_\bullet,\iota_\bullet)(\e^{-2}t,[\e^{-2}t_\bullet],k,k_\bullet) \right|^2}\label{decomp EdtY^2}
\\&\hspace{2cm} = \sum_{C\in\mathcal{C}^\star_{m,n}(\eta,\eta_\bullet,\iota_\bullet)} \sum_{\substack{j_{ \mathrm{max} }\in\{-1\}\sqcup (\mathcal N_T^{\tot,\mathrm{max}}(A) \setminus \{ \mathrm{int}(\bullet) \} )\\ j'_{ \mathrm{max} }\in\{-1\}\sqcup (\mathcal N_T^{\tot,\mathrm{max}}(A') \setminus \{ \mathrm{int}(\bullet') \} )}} \reallywidehat{ G_{C,j_{ \mathrm{max} },j'_{ \mathrm{max} }} }(\e^{-2}t,[\e^{-2}t_\bullet],k,k_\bullet)\non
\end{align}
where $\reallywidehat{ G_{C,j_{ \mathrm{max} },j'_{ \mathrm{max} }} }(\e^{-2}t,[\e^{-2}t_\bullet],k,k_\bullet)$ is defined by
\begin{itemize}
\item[(i)] If $\star=\bi$, then,
\begin{itemize}
\item if $j_{\mathrm{max}}=j'_{\mathrm{max}}=-1$, then
\begin{align*}
&\reallywidehat{ G_{C,-1,-1} }(\e^{-2}t,k,k_\bullet) 
\\& = -(-i)^{n_T}\frac{\e^{n_B^\tot}}{L^{(n-1)d}} \sum_{\ka\in\mathcal{D}_{k,k_\bullet}(C)} \Om_{-1}(A)\Om_{-1}(A') e^{i\Om_{-1}(C)\e^{-2}t}
\\&\quad \times \prod_{j\in\mathcal{N}^\tot(C)}q_j \int_{I_C(t)}\prod_{j\in\mathcal{N}_T(C)}e^{i\Om_j \e^{-2}t_j}\d t_j\prod_{\ell\in\mathcal{L}(C)_-}  M^{c(\ell),c(\si(\ell))}(\ffi(\ell)\ka(\ell))^{\ffi(\ell)},
\end{align*}
\item if $j_{\mathrm{max}}=-1$ and $j'_{\mathrm{max}}\neq-1$, then
\begin{align*}
&\reallywidehat{ G_{C,-1,j'_{ \mathrm{max} }} }(\e^{-2}t,k,k_\bullet)
\\&=-(-i)^{n_T+1}\frac{\e^{n_B^\tot+2}}{L^{(n-1)d}} \sum_{\ka\in\mathcal{D}_{k,k_\bullet}(C)} \Om_{-1}(A) e^{i\pth{ \Om_{-1}(C)+\Om_{j'_{\mathrm{max}}}(A')}\e^{-2}t}
\\&\quad \times \prod_{j\in\mathcal{N}^\tot(C)}q_j \int_{I^{-1,j'_{\mathrm{max}}}_C(t)}\prod_{j\in\mathcal{N}_T(C)\setminus\{ j'_{\mathrm{max}} \}}e^{i\Om_j \e^{-2}t_j}\d t_j \prod_{\ell\in\mathcal{L}(C)_-}  M^{c(\ell),c(\si(\ell))}(\ffi(\ell)\ka(\ell))^{\ffi(\ell)},
\end{align*}
\item if $j_{\mathrm{max}}\neq-1$ and $j'_{\mathrm{max}}=-1$, then
\begin{align*}
&\reallywidehat{ G_{C,j_{ \mathrm{max} },-1} }(\e^{-2}t,k,k_\bullet)
\\&=-(-i)^{n_T+1}\frac{\e^{n_B^\tot+2}}{L^{(n-1)d}} \sum_{\ka\in\mathcal{D}_{k,k_\bullet}(C)} \Om_{-1}(A') e^{i\pth{ \Om_{-1}(C)+\Om_{j_{\mathrm{max}}}(A)}\e^{-2}t}
\\&\quad \times \prod_{j\in\mathcal{N}^\tot(C)}q_j \int_{I^{j_{\mathrm{max}},-1}_C(t)}\prod_{j\in\mathcal{N}_T(C)\setminus\{ j_{\mathrm{max}} \}}e^{i\Om_j \e^{-2}t_j}\d t_j \prod_{\ell\in\mathcal{L}(C)_-}  M^{c(\ell),c(\si(\ell))}(\ffi(\ell)\ka(\ell))^{\ffi(\ell)},
\end{align*}
\item if $j_{\mathrm{max}}\neq-1$ and $j'_{\mathrm{max}}\neq-1$, then
\begin{align*}
&\reallywidehat{ G_{C,j_{ \mathrm{max} },j'_{ \mathrm{max} }} }(\e^{-2}t,k,k_\bullet)
\\&=(-i)^{n_T}\frac{\e^{n_B^\tot+4}}{L^{(n-1)d}} \sum_{\ka\in\mathcal{D}_{k,k_\bullet}(C)} e^{i\pth{ \Om_{-1}(C)+\Om_{j_{\mathrm{max}}}(A) + \Om_{j'_{\mathrm{max}}}(A') }\e^{-2}t}  \prod_{j\in\mathcal{N}^\tot(C)}q_j   
\\&\quad \times  \int_{I^{ j_{\mathrm{max}},j'_{\mathrm{max}} }_C(t)}  \prod_{j\in\mathcal{N}_T(C)\setminus\{ j_{\mathrm{max}},j'_{\mathrm{max}} \}} e^{i\Om_j \e^{-2}t_j}\d t_j \prod_{\ell\in\mathcal{L}(C)_-}  M^{c(\ell),c(\si(\ell))}(\ffi(\ell)\ka(\ell))^{\ffi(\ell)},
\end{align*}
\end{itemize}
\item[(ii)] If $\star=\ter$, then,
\begin{itemize}
\item if $j_{\mathrm{max}}=j'_{\mathrm{max}}=-1$, then
\begin{align*}
&\reallywidehat{ G_{C,-1,-1} }(\e^{-2}t,\e^{-2}t_\bullet,k,k_\bullet)
\\&= -  (-i)^{n^\tot_T} \frac{\e^{n_B^\tot}}{L^{(n-1)d}}  \sum_{\ka\in\mathcal{D}_{k,k_\bullet}(C)} \Om_{-1}(A)\Om_{-1}(A') e^{i\Om_{-1}(C)\e^{-2}t} e^{i\pth{\Om_\bullet + \Om_{\bullet'}} \e^{-2}t_\bullet} \prod_{j\in\mathcal{N}^\tot(C)}q_j
\\&\quad \times  \int_{I_C(t,t_\bullet)} \prod_{j\in \mathcal{N}^\tot_T(A)\setminus\{\mathrm{int}(\bullet),\mathrm{int}(\bullet')\}} e^{i\Om_j \e^{-2}t_j}\d t_j  \prod_{\ell\in\mathcal{L}(C)_-}  M^{c(\ell),c(\si(\ell))}(\ffi(\ell)\ka(\ell))^{\ffi(\ell)},
\end{align*}
\item if $j_{\mathrm{max}}=-1$ and $j'_{\mathrm{max}}\neq-1$, then
\begin{align*}
&\reallywidehat{ G_{C,-1,j'_{ \mathrm{max} }} }(\e^{-2}t,\e^{-2}t_\bullet,k,k_\bullet)
\\&=- (-i)^{n^\tot_T+1} \frac{\e^{n_B^\tot+2}}{L^{(n-1)d}}  \sum_{\ka\in\mathcal{D}_{k,k_\bullet}(C)} \Om_{-1}(A) e^{i\pth{\Om_{-1}(C)+\Om_{j'_{\mathrm{max}}}(A')}\e^{-2}t} e^{i\pth{\Om_\bullet + \Om_{\bullet'}} \e^{-2}t_\bullet}
\\&\quad \times \prod_{j\in\mathcal{N}^\tot(C)}q_j \int_{I_C^{-1,j'_{\mathrm{max}}  }(t,t_\bullet)} \prod_{j\in \mathcal{N}^\tot_T(C)\setminus\{\mathrm{int}(\bullet),\mathrm{int}(\bullet'), j'_{\mathrm{max}}\}} e^{i\Om_j \e^{-2}t_j}\d t_j 
\\&\quad \times \prod_{\ell\in\mathcal{L}(C)_-}  M^{c(\ell),c(\si(\ell))}(\ffi(\ell)\ka(\ell))^{\ffi(\ell)},
\end{align*}
where $I_C^{-1,j'_{\mathrm{max}}  }(t,t_\bullet)=I_A(t,t_\bullet)\times I_{A'}^{j'_{\mathrm{max}}  }(t,t_\bullet)$,
\item if $j_{\mathrm{max}}\neq-1$ and $j'_{\mathrm{max}}=-1$, then
\begin{align*}
&\reallywidehat{ G_{C,j_{ \mathrm{max} },-1} }(\e^{-2}t,\e^{-2}t_\bullet,k,k_\bullet)
\\&=- (-i)^{n^\tot_T+1} \frac{\e^{n_B^\tot+2}}{L^{(n-1)d}}  \sum_{\ka\in\mathcal{D}_{k,k_\bullet}(C)} \Om_{-1}(A') e^{i\pth{\Om_{-1}(C)+\Om_{j_{\mathrm{max}}}(A)}\e^{-2}t} e^{i\pth{\Om_\bullet + \Om_{\bullet'}} \e^{-2}t_\bullet}
\\&\quad \times \prod_{j\in\mathcal{N}^\tot(C)}q_j \int_{I_C^{j_{\mathrm{max}},-1  }(t,t_\bullet)} \prod_{j\in \mathcal{N}^\tot_T(C)\setminus\{\mathrm{int}(\bullet),\mathrm{int}(\bullet'), j_{\mathrm{max}}\}} e^{i\Om_j \e^{-2}t_j}\d t_j 
\\&\quad \times \prod_{\ell\in\mathcal{L}(C)_-}  M^{c(\ell),c(\si(\ell))}(\ffi(\ell)\ka(\ell))^{\ffi(\ell)},
\end{align*}
where $I_C^{j_{\mathrm{max}},-1  }(t,t_\bullet)=I_A^{ j_{\mathrm{max}} }(t,t_\bullet)\times I_{A}(t,t_\bullet)$,
\item if $j_{\mathrm{max}}\neq-1$ and $j'_{\mathrm{max}}\neq-1$, then
\begin{align*}
&\reallywidehat{ G_{C,j_{ \mathrm{max} },j'_{ \mathrm{max} }} }(\e^{-2}t,\e^{-2}t_\bullet,k,k_\bullet)
\\&=  (-i)^{n^\tot_T} \frac{\e^{n_B^\tot+4}}{L^{(n-1)d}}  \sum_{\ka\in\mathcal{D}_{k,k_\bullet}(C)} e^{i\pth{\Om_{-1}(C)+\Om_{j_{\mathrm{max}}}+\Om_{j'_{\mathrm{max}}}}\e^{-2}t}e^{i\pth{\Om_\bullet + \Om_{\bullet'} }\e^{-2}t_\bullet}
\\&\quad \times \prod_{j\in\mathcal{N}^\tot(C)}q_j     \int_{I_C^{j_{\mathrm{max}},j'_{\mathrm{max}}  }(t,t_\bullet)} \prod_{j\in\mathcal{N}_T^\tot(C)\setminus\{\mathrm{int}(\bullet),\mathrm{int}(\bullet'), j_{\mathrm{max}},j'_{\mathrm{max}} \}} e^{i\Om_j \e^{-2}t_j} \d t_j 
\\&\quad \times \prod_{\ell\in\mathcal{L}(C)_-}  M^{c(\ell),c(\si(\ell))}(\ffi(\ell)\ka(\ell))^{\ffi(\ell)},
\end{align*}
where $I_C^{j_{\mathrm{max}},j'_{\mathrm{max}}  }(t,t_\bullet)=I_A^{j_{\mathrm{max}}  }(t,t_\bullet)\times I_{A'}^{j'_{\mathrm{max}}  }(t,t_\bullet)$.
\end{itemize}
\end{itemize}

\begin{remark}
Since the cardinal of $N_T^{\tot,\mathrm{max}}(A)$ grows linearly in $n(A)$, the cardinal of the set
\begin{align*}
 \enstq{(C,j_1,j_2)}{C\in\mathcal{C}^\star_{m,n,q}(\eta,\eta_\bullet,\iota_\bullet),\; j_i\in \{-1\}\sqcup( \mathcal N_T^{\tot,\mathrm{max}}(A_i)\setminus \{ \mathrm{int}(\bullet_i) \}), i=1,2}
\end{align*}
where $A_1$ and $A_2$ are the two trees in $C$, is bounded by $\La^{n+1} \pth{ n + 2 - q }!$ for some $\La>0$.
\end{remark}

The following proposition contains the key estimate satisfied by $\reallywidehat{G_{C,j_{\mathrm{max}},j'_{\mathrm{max}}}}$. It is proved as Propositions \ref{prop low nodes GC} and \ref{prop high nodes GC} (similarly to how Proposition \ref{prop estimate FCjmax} was proved as Proposition \ref{prop |FC|}).

\begin{prop}\label{prop G_Cjmax}
Set $\a_0=\min(\a_\low,\a_\high)$. There exists $\La>0$ such that for all $C\in\mathcal{C}^\star_{m,n}(\eta,\eta_\bullet,\iota_\bullet)$ for $\star=\bi,\ter$ satisfying
\begin{align*}
n(C)-n_\scb(C)\geq 40, \qquad n(C)\leq |\ln \e|^{3},
\end{align*}
and for all 
\begin{align*}
(j_{\mathrm{max}},j'_{\mathrm{max}})\in \Big( \{-1\}\sqcup( \mathcal N_T^{\tot,\mathrm{max}}(A)\setminus \{ \mathrm{int}(\bullet) \}) \Big) \times \Big( \{-1\}\sqcup( \mathcal N_T^{\tot,\mathrm{max}}(A')\setminus \{ \mathrm{int}(\bullet') \}) \Big),
\end{align*}
(where $A$ and $A'$ are the two trees in $C$) then for any $\e\leq \de$ and any $k,k_\bullet\in\Z_L^d$ we have for all $t\in[0,\de]$
\begin{align*}
\left| \reallywidehat{G_{C,j_{\mathrm{max}},j'_{\mathrm{max}}}}(\e^{-2}t,[\e^{-2}t_\bullet,]k,k_\bullet) \right| \leq \left\{
\begin{aligned}
&\La^{n+1} \de^{n[-2]} \e^{\a_\high(n(C)-n_\scb(C))}\ps{k_\bullet}^{-\half} \qquad \text{if $k\neq k_\bullet$,}
\\&\La^{n+1} L^d \de^{n[-2]} \e^{\a_\high(n(C)-n_\scb(C))}\ps{k_\bullet}^{-\half} \quad \text{if $k = k_\bullet$,}
\end{aligned}
\right.
\end{align*}
where the content of the brackets only appears when $\star=\ter$, in which case the estimate holds for all $0\leq t_\bullet \leq t \leq \de$.
\end{prop}

\subsubsection{Estimates for $Y^\star_{m,n}$ and the iterates} 

The previous results imply the following proposition.

\begin{prop} \label{Ybi and Yter estimates}
There exists $\Lambda>0$ such that, for any fixed $A>0$, with probability $\ge 1-L^{-A}$ and for all $n\leq |\ln\e|^{3}$ we have
\begin{align*}
 \l \reallywidehat{Y^{\bi}_{m,n}}(\eta, \iota, \eta_\bullet, \iota_\bullet) \r_{L^\infty_{t, k, k_\bullet}([0, \e^{-2}\delta]\times (\mathbb Z^d_L)^2)} &\leq \La^{n}\delta^{\frac{n}{2}} L^{\frac14 +\frac{d}{2} + A},
\\ \l \reallywidehat{Y^{\ter}_{m,n}}(\eta, \iota, \eta_\bullet, \iota_\bullet)\r_{L^\infty_t L^1_{t_\bullet} L^\infty_{k,k_\bullet}([0, \e^{-2}\delta]^2\times (\mathbb Z^d_L)^2)} & \leq \La^{n}\delta^{\frac{n}{2}} L^{\frac14 +\frac{d}{2} + \frac2\beta + A}.
\end{align*}
\end{prop}

\begin{proof}
The $L^p$ norms in this proof are always taken with respect to the variables $t$ and/or $t_\bullet$ on the interval $[0, \e^{-2}\delta]$. We will omit to write such intervals to keep notations lighter.
Moreover, the estimates that follows are independent of the choice of $\eta$, $\iota$, $\eta_\bullet$ and $\iota_\bullet$, hence we will refer to $Y^\star_{m,n}(\eta, \iota, \eta_\bullet, \iota_\bullet)$ simply as $Y^\star_{m,n}$ for $\star=\bi, \ter$.

We begin by observing that, using Proposition \ref{prop norme Y}, \eqref{counting nscb linear coupling}, Propositions \ref{prop low nodes GC} and \ref{prop high nodes GC} and an argument analogous to the one detailed in the proof of Corollary \ref{coro Xn}, we can prove the existence of some $\Lambda>0$ such that for $\star=\bi, \ter$ and for any $n\in \N$ and $t, t_\bullet\in [0,\delta]$ we have the estimates
\begin{equation}\label{E|Y|2}
\E\pth{ \left| \reallywidehat{Y^{\star}_{m,n}}(\e^{-2}t,[\e^{-2}t_\bullet,]  k, k_\bullet) \right|^2  } \leq \left\{
\begin{aligned}
\an{k_\bullet}^{-1/2}\La^{2n+1}\delta^{n[-2]}  & \qquad  \mathrm{if}\, k\neq k_\bullet,
\\ \an{k_\bullet}^{-1/2}L^d\La^{2n+1}\delta^{n[-2]}  & \qquad \mathrm{if}\, k= k_\bullet.
\end{aligned}
\right.
\end{equation}
Furthermore, given the definition of $G_C$ in Proposition \ref{prop norme Y}, we remark that the expectation in the above left hand side is supported for $k, k_\bullet$ such that $k-k_\bullet$ belongs to a ball of $\mathbb Z^d_L$ centered at the origin, of radius $\mathcal O (n )$. 

Next, recall that Gaussian hypercontractivity (see for instance Theorem 5.1 from \cite{Janson}) gives the existence, for any $p<\infty$, of a constant $c(p)>0$ only depending on $p$, such that
\begin{align*}
\E\pth{ \left|\reallywidehat{Y^{\star}_{m,n}}(\e^{-2}t, [\e^{-2}t_\bullet ,]k,k_\bullet)\right|^p }\leq c(p)^{np} \pth{ \E \pth{ \left|\reallywidehat{Y^{\star}_{m,n}}(\e^{-2}t, [\e^{-2}t_\bullet ,]k,k_\bullet)\right|^2} }^{p/2}.
\end{align*}
Together with \eqref{E|Y|2} this implies
\begin{equation*}
\E\pth{ \left|\reallywidehat{Y^{\star}_{m,n}}(\e^{-2}t, [\e^{-2}t_\bullet ,]k,k_\bullet)\right|^p }\leq \left\{
\begin{aligned}
c(p)^{np} \an{k_\bullet}^{-\frac{p}{4}}\La^{(2n+1)\frac{p}{2}}\delta^{(n[-2])\frac{p}{2}} & \qquad  \mathrm{if}\, k\neq k_\bullet,
\\ c(p)^{np} \an{k_\bullet}^{-\frac{p}{4}}L^{\frac{dp}{2}}\La^{(2n+1)\frac{p}{2}}\delta^{(n[-2])\frac{p}{2}} &  \qquad \mathrm{if}\, k= k_\bullet.
\end{aligned}
\right.
\end{equation*}
The above estimate being uniform in $t$ and $t_\bullet$ we can integrate it in $t$ (and in $t_\bullet$ if $\star=\ter$) and obtain
\begin{multline}\label{expectation Lp norm Y}
    \E\pth{ \left\|\reallywidehat{Y^{\bi}_{m,n}}(\cdot,  k, k_\bullet) \right\|^p_{L^p_t}} + \E\pth{ \left\|\reallywidehat{Y^{\ter}_{m,n}}(\cdot,\cdot,  k, k_\bullet) \right\|^p_{L^p_{t,t_\bullet}}}\\
\le \left \lbrace{\begin{array}{cc}
c(p)^{np} \an{k_\bullet}^{-\frac{p}{4}}\La^{(2n+1)\frac{p}{2}}\delta^{(n[-1])\frac{p}{2}+1} \e^{-2[-2]},  & \mathrm{if}\, k\neq k_\bullet,\\
c(p)^{np} \an{k_\bullet}^{-\frac{p}{4}}L^{\frac{dp}{2}}\La^{(2n+1)\frac{p}{2}}\delta^{(n[-1])\frac{p}{2}+1}\e^{-2[-2]},  & \mathrm{if}\, k= k_\bullet.
\end{array}} \right.
\end{multline}
Looking at the decomposition \eqref{decomp EdtY^2} and Proposition \ref{prop G_Cjmax}, we see that we can treat $\reallywidehat{\dr_tY^{\star}_{m,n}}$ identically and thus obtain
\begin{multline}\label{expectation Lp norm dtY}
    \E\pth{ \left\|\reallywidehat{\dr_tY^{\bi}_{N(L),n}}(\cdot,  k, k_\bullet) \right\|^p_{L^p_t}} + \E\pth{ \left\|\reallywidehat{\dr_tY^{\ter}_{N(L),n}}(\cdot,\cdot,  k, k_\bullet) \right\|^p_{L^p_{t,t_\bullet}}}\\
\le \left \lbrace{\begin{array}{cc}
c(p)^{np} \an{k_\bullet}^{-\frac{p}{4}}\La^{(2n+1)\frac{p}{2}}\delta^{(n[-1])\frac{p}{2}+1} \e^{-2[-2]},  & \mathrm{if}\, k\neq k_\bullet,\\
c(p)^{np} \an{k_\bullet}^{-\frac{p}{4}}L^{\frac{dp}{2}}\La^{(2n+1)\frac{p}{2}}\delta^{(n[-1])\frac{p}{2}+1}\e^{-2[-2]},  & \mathrm{if}\, k= k_\bullet.
\end{array}} \right.
\end{multline}

Now, using successively the $\ell^1\hookrightarrow\ell^\infty$ injection, the Gagliardo-Niremberg inequality on $[0,\e^{-2}\de]$ (which yields an extra factor $(\e^{-2}\de)^{-1}$) and Hölder inequality we deduce the following estimate for $\reallywidehat{Y^{\bi}_{m,n}}$:
\begin{align*}
&\E\pth{ \sup_{t\in[0,\e^{-2}\de]} \sup_{k,k_\bullet\in\Z_L^d}\left| \reallywidehat{Y^{\bi}_{m,n}}(t,k,k_\bullet) \right|^p  } 
\\& \leq  \E\pth{ \sum_{k,k_\bullet\in\Z_L^d}\sup_{t\in[0,\e^{-2}\de]} \left| \reallywidehat{Y^{\bi}_{m,n}}(t,k,k_\bullet) \right|^p  }
\\&\lesssim \sum_{k,k_\bullet\in\Z_L^d} \E \pth{  \l \reallywidehat{\dr_tY^{\bi}_{m,n}} (\cdot,k,k_\bullet)\r_{L^p} \l \reallywidehat{Y^{\bi}_{m,n}}(\cdot,k,k_\bullet)\r_{L^p}^{p\pth{1-\frac{1}{p}}} +  (\e^{-2}\de)^{-1} \l \reallywidehat{Y^{\bi}_{m,n}}(\cdot,k,k_\bullet) \r_{L^p}^p  }
\\&\lesssim \sum_{k,k_\bullet\in\Z_L^d} \Bigg( \E\pth{ \l \reallywidehat{\dr_tY^{\bi}_{m,n}} (\cdot,k,k_\bullet)\r_{L^p}^p}^{\frac{1}{p}}  \E\pth{ \l \reallywidehat{Y^{\bi}_{m,n}} (\cdot,k,k_\bullet)\r_{L^p}^p}^{1-\frac{1}{p}}  
\\&\hspace{8cm}+  (\e^{-2}\de)^{-1} \E\pth{\l \reallywidehat{Y^{\bi}_{m,n}}(\cdot,k,k_\bullet) \r_{L^p}^p} \Bigg).
\end{align*}
For $\reallywidehat{Y^{\ter}_{m,n}}$, we first use the $L^p\hookrightarrow L^1$ injection on $[0,\e^{-2}\de]$ and then follow the same strategy as for $\reallywidehat{Y^{\bi}_{m,n}}$. This gives
\begin{align*}
&\E\pth{ \sup_{t\in[0,\e^{-2}\de]} \sup_{k,k_\bullet\in\Z_L^d}\l \reallywidehat{Y^{\ter}_{m,n}}(t,\cdot,k,k_\bullet) \r_{L^1_{t_\bullet}}^p  } 
\\&\leq (\e^{-2}\de)^{p-1} \E\pth{ \sup_{t\in[0,\e^{-2}\de]} \sup_{k,k_\bullet\in\Z_L^d}\l \reallywidehat{Y^{\ter}_{m,n}}(t,\cdot,k,k_\bullet) \r_{L^p_{t_\bullet}}^p  } 
\\&\lesssim \sum_{k,k_\bullet\in\Z_L^d} \Bigg((\e^{-2}\de)^{p-1}\E\pth{ \l \reallywidehat{\dr_tY^{\ter}_{m,n}} (\cdot,\cdot,k,k_\bullet)\r_{L^p_{t,t_\bullet}}^p}^{\frac{1}{p}}  \E\pth{ \l \reallywidehat{Y^{\ter}_{m,n}} (\cdot,\cdot,k,k_\bullet)\r_{L^p_{t,t_\bullet}}^p}^{1-\frac{1}{p}}  
\\&\hspace{8cm} +  (\e^{-2}\de)^{p-2}\E\pth{ \l \reallywidehat{Y^{\ter}_{m,n}}(\cdot,\cdot,k,k_\bullet) \r_{L^p_{t,t_\bullet}}^p } \Bigg).
\end{align*}
To estimate the expectations in the above estimates we use \eqref{expectation Lp norm Y} and \eqref{expectation Lp norm dtY} with $p>4d$ to ensure summability in $k_\bullet$, while summability follows from the fact that the expectations in \eqref{expectation Lp norm Y} and \eqref{expectation Lp norm dtY} are supported for $k-k_\bullet$ in a ball of radius $\mathcal{O}(n)$. Therefore we obtain
\begin{align}\label{estim Ybi final}
\E\pth{ \sup_{t\in[0,\e^{-2}\de]} \sup_{k,k_\bullet\in\Z_L^d}\left| \reallywidehat{Y^{\bi}_{m,n}}(t,k,k_\bullet) \right|^p  } \leq (nL)^d L^{\frac{dp}{2}} \La^{\pth{n+\half}p}\de^{\frac{np}{2}}
\end{align}
and
\begin{align}\label{estim Yter final}
\E\pth{ \sup_{t\in[0,\e^{-2}\de]} \sup_{k,k_\bullet\in\Z_L^d}\l \reallywidehat{Y^{\ter}_{m,n}}(t,\cdot,k,k_\bullet) \r_{L^1_{t_\bullet}}^p  } \leq (nL)^d L^{\frac{dp}{2}} \La^{\pth{n+\half}p} \de^{ \frac{(n+1)p}{2}-1}\e^{-2p}.
\end{align}
Setting
\begin{align*}
a_{n,\bullet}^\bi \vcentcolon = \La^{n+\half}\de^{\frac{n}{2}}  L^{\frac{d}{p} + \frac{d}{2} + A}, \quad a_{n,\bullet}^\ter \vcentcolon = \La^{n+\half}\de^{\frac{n+1}{2}-\frac{1}{p}}  L^{\frac{d}{p} + \frac{d}{2} + \frac{2}{\b} + A},
\end{align*}
the estimates \eqref{estim Ybi final}-\eqref{estim Yter final}, together with Markov's inequality and the assumptions $n\leq |\ln\e|^{3}$ and $\e=L^{-\frac{1}{\b}}$, finally imply 
\begin{align*}
&\mathbb{P}\pth{ \sup_{t\in[0,\e^{-2}\de]} \sup_{k,k_\bullet\in\Z_L^d}\left| \reallywidehat{Y^{\bi}_{m,n}}(t,k,k_\bullet) \right|^p \geq (a_{n,\bullet}^\bi)^p } 
\\&\hspace{3cm}+ \mathbb{P}\pth{ \sup_{t\in[0,\e^{-2}\de]} \sup_{k,k_\bullet\in\Z_L^d}\l \reallywidehat{Y^{\ter}_{m,n}}(t,\cdot,k,k_\bullet) \r_{L^1_{t_\bullet}}^p \geq (a_{n,\bullet}^\ter)^p } \leq L^{-A}.
\end{align*}
The conclusion of the proof follows using the fact that we have chosen $p>4d$ and from the fact that $\delta<1$.
\end{proof}

From this proposition and the representation formula \eqref{itérées in terms Y} we deduce the following.

\begin{prop}\label{prop norme itérées}
Let $s\geq 0$ and $A>0$. There exists $\La>0$, $\de_0>0$ (depending on $d$ and $A$) and a set $\mathcal{E}_{L,A}$ of probability larger than $1-L^{-A}$ such that if $0<\de\leq \de_0$ and when restricted to $\mathcal{E}_{L,A}$ then we have for all $m\leq N(L)$
\begin{align}\label{norme op L^m}
\l \mathfrak{L}^{m} \r_{\mathrm{op}} \leq \La (\La\sqrt{\de})^{\frac{m}{2}} L^{\frac{1}{4}+ d + A} (\ln L)^{s+d}  .
\end{align}
\end{prop}

\begin{proof}
Let $t\in[0,\e^{-2}\de]$ be fixed. Using for brevity the notation $I_m=\llbracket m,2m(N(L)+1)\rrbracket$, the representation formula \eqref{itérées in terms Y} and the Fourier version of $Y^{\star,\iota'}_{m,n}(\eta,\iota,\eta_\bullet,\iota_\bullet) = Y^{\star}_{m,n}(\eta,\iota'\iota,\eta_\bullet,\iota'\iota_\bullet) $ we obtain the decomposition
\begin{align*}
\l \mathfrak{L}^{m}(v)^\eta(t,\cdot) \r_{H^s(L\T^d)}^2  & = S_\bi + S_{\mathrm{cross}} + S_\ter,
\end{align*}
where
\begin{align*}
S_\bi & = \frac{1}{L^d}  \sum_{\substack{\eta_\bullet,\eta'_\bullet\in\{0,1\}\\\iota_\bullet,\iota'_\bullet\in\{\pm\} \\k,k_\bullet,k'_\bullet\in\Z_L^d\\ n,n'\in I_m\\ }} \ps{k}^{2s} \reallywidehat{v^{\eta_\bullet,\iota_\bullet}}(t,k_\bullet)  \reallywidehat{v^{\eta'_\bullet,-\iota'_\bullet}}(t,-k'_\bullet)
\\&\hspace{3cm} \times  \reallywidehat{Y^\bi_{m,n}}(\eta,+,\eta_\bullet,\iota_\bullet)(t,k,k_\bullet) \reallywidehat{Y^\bi_{m,n'}}(\eta,-,\eta'_\bullet,-\iota'_\bullet)(t,-k,-k'_\bullet),
\\S_\ter & = \frac{\e^4}{L^d}  \sum_{\substack{\eta_\bullet,\eta'_\bullet\in\{0,1\}\\\iota_\bullet,\iota'_\bullet\in\{\pm\} \\k,k_\bullet,k'_\bullet\in\Z_L^d\\ n,n'\in I_m\\ }} \ps{k}^{2s} \int_{[0,t]^2} \Bigg( \reallywidehat{v^{\eta_\bullet,\iota_\bullet}}(t_\bullet,k_\bullet) \reallywidehat{Y^\ter_{m,n}}(\eta,+,\eta_\bullet,\iota_\bullet)(t,t_\bullet,k,k_\bullet)
\\&\hspace{4.5cm} \times  \reallywidehat{v^{\eta'_\bullet,-\iota'_\bullet}}(t'_\bullet,-k'_\bullet) \reallywidehat{Y^\ter_{m,n}}(\eta,-,\eta'_\bullet,-\iota'_\bullet)(t,t'_\bullet,-k,-k'_\bullet) \Bigg) \d t_\bullet \d t'_\bullet ,
\\ S_{\mathrm{cross}} & = \frac{\e^2}{L^d}  \sum_{\substack{\eta_\bullet,\eta'_\bullet\in\{0,1\}\\\iota,\iota_\bullet,\iota'_\bullet\in\{\pm\} \\k,k_\bullet,k'_\bullet\in\Z_L^d\\ n,n'\in I_m\\ }} \ps{k}^{2s}  \reallywidehat{v^{\eta_\bullet,\iota_\bullet}}(t,k_\bullet) \reallywidehat{Y^\bi_{m,n'}}(\eta,\iota,\eta_\bullet,\iota_\bullet)(t,k,k_\bullet)
\\&\hspace{3cm}\times \int_0^t  \reallywidehat{v^{\eta'_\bullet,-\iota'_\bullet}}(t_\bullet,-k'_\bullet) \reallywidehat{Y^\ter_{m,n'}}(\eta,-\iota,\eta'_\bullet,-\iota'_\bullet)(t,t_\bullet,-k,-k'_\bullet) \d t_\bullet.
\end{align*}
In order to estimate $S_\bi$, $S_\ter$ and $S_{\mathrm{cross}}$, we will use the following simple key facts:
\begin{itemize}
\item Thanks to the definition of $N(L)$ (see the very beginning of Section \ref{section remainder}) and the link between $L$ and $\e$ (that is $L=\e^{-\b}$), if $m\leq N(L)$ then we have $2m(N(L)+1)\leq 3\b^{2}|\ln\e|^{2}$ which is less than $|\ln\e|^3$ if $L$ is large enough.
\item For any fixed $k\in\Z_L^d$, $ \reallywidehat{Y^\star_{m,n}}(\eta,\iota,\eta_\bullet,\iota_\bullet)(t,[t_\bullet,]k,\cdot)$ is supported in a ball centered at $k$ and of radius $\mathcal{O}(n)$. 
\end{itemize}
We proceed with estimating $S_\bi$. For any fixed $k\in\Z_L^d$ we use the change of variables $k_\bullet\longmapsto k_\bullet - k$ and $k'_\bullet\longmapsto k'_\bullet - k$ and bound $\ps{k}^{2s}$ by $\ps{k-k_\bullet}^s\ps{k_\bullet}^s\ps{k-k'_\bullet}^s\ps{k'_\bullet}^s$ to obtain the bound
\begin{align*}
S_\bi & \lesssim  \frac{1}{L^d} \sum_{\substack{\eta_\bullet,\eta'_\bullet\in\{0,1\}\\\iota_\bullet,\iota'_\bullet\in\{\pm\} }} \sum_{n,n'\in I_m} \sum_{k\in\Z_L^d}  \sum_{\substack{|k_\bullet| \lesssim n \\ |k'_\bullet|\lesssim n' }} \ps{k-k_\bullet}^s\ps{k-k'_\bullet}^s\left|\reallywidehat{v^{\eta_\bullet,\iota_\bullet}}(t,k_\bullet-k)\right| \left|  \reallywidehat{v^{\eta'_\bullet,-\iota'_\bullet}}(t,k-k'_\bullet) \right|
\\&\hspace{2.5cm} \times \ps{k_\bullet}^s \ps{k'_\bullet}^s \left| \reallywidehat{Y^\bi_{m,n}}(\eta,+,\eta_\bullet,\iota_\bullet)(t,k,k_\bullet-k) \right| \left| \reallywidehat{Y^\bi_{m,n'}}(\eta,-,\eta'_\bullet,-\iota'_\bullet)(t,-k,k-k'_\bullet) \right|
\\ & \lesssim  \frac{1}{L^d} \sum_{\substack{\eta_\bullet,\eta'_\bullet\in\{0,1\}\\\iota_\bullet,\iota'_\bullet\in\{\pm\} }} \sum_{n,n'\in I_m} \ps{n}^{s}\ps{n'}^s
\\&\hspace{0.5cm} \times  \l \reallywidehat{Y^{\bi}_{m,n}}(\eta, +, \eta_\bullet, \iota_\bullet) \r_{L^\infty_{t, k, k_\bullet}([0, \e^{-2}\delta]\times (\mathbb Z^d_L)^2)} \l \reallywidehat{Y^{\bi}_{m,n'}}(\eta, -, \eta_\bullet, -\iota_\bullet) \r_{L^\infty_{t, k, k_\bullet}([0, \e^{-2}\delta]\times (\mathbb Z^d_L)^2)}
\\&\hspace{0.5cm} \times \sum_{k\in\Z_L^d}  \sum_{\substack{|k_\bullet| \lesssim n \\ |k'_\bullet|\lesssim n' }} \ps{k-k_\bullet}^s\ps{k-k'_\bullet}^s\left|\reallywidehat{v^{\eta_\bullet,\iota_\bullet}}(t,k_\bullet-k)\right| \left|  \reallywidehat{v^{\eta'_\bullet,-\iota'_\bullet}}(t,k-k'_\bullet) \right|
\end{align*}
where we bound $\ps{k_\bullet}^s \ps{k'_\bullet}^s$ by $\ps{n}^{s}\ps{n'}^s$ and estimated the two $\reallywidehat{Y^{\bi}_{m,n}}$ and $\reallywidehat{Y^{\bi}_{m,n'}}$ in $L^\infty$. Exchanging the sums in the last line and using Cauchy-Schwartz inequality we obtain
\begin{align*}
\sum_{k\in\Z_L^d}  \sum_{\substack{|k_\bullet| \lesssim n \\ |k'_\bullet|\lesssim n' }} \ps{k-k_\bullet}^s\ps{k-k'_\bullet}^s & \left|\reallywidehat{v^{\eta_\bullet,\iota_\bullet}}(t,k_\bullet-k)\right| \left|  \reallywidehat{v^{\eta'_\bullet,-\iota'_\bullet}}(t,k-k'_\bullet) \right|
\\& \lesssim  \sum_{\substack{|k_\bullet| \lesssim n \\ |k'_\bullet|\lesssim n' }} \l v^{\eta_\bullet}(t,\cdot) \r_{H^s(L\T^d)} \l v^{\eta'_\bullet}(t,\cdot) \r_{H^s(L\T^d)} 
\\&\lesssim n^d (n')^d L^{2d} \l v(t,\cdot) \r_{H^s(L\T^d)}^2.
\end{align*}
Therefore, with probability larger than $1-L^{-A}$ we have
\begin{align*}
S_\bi & \lesssim  \l v(t,\cdot) \r_{H^s(L\T^d)}^2 L^{d} \sum_{\substack{\eta_\bullet,\eta'_\bullet\in\{0,1\}\\\iota_\bullet,\iota'_\bullet\in\{\pm\} }} \sum_{n,n'\in I_L} \ps{n}^{s}\ps{n'}^s n^d (n')^d 
\\&\hspace{0.5cm} \times  \l \reallywidehat{Y^{\bi}_{m,n}}(\eta, +, \eta_\bullet, \iota_\bullet) \r_{L^\infty_{t, k, k_\bullet}([0, \e^{-2}\delta]\times (\mathbb Z^d_L)^2)} \l \reallywidehat{Y^{\bi}_{m,n'}}(\eta, -, \eta_\bullet, -\iota_\bullet) \r_{L^\infty_{t, k, k_\bullet}([0, \e^{-2}\delta]\times (\mathbb Z^d_L)^2)}
\\&\lesssim  \l v \r_{\mathcal{C}\pth{[0,\e^{-2}\de], H^s(L\T^d)}}^2  L^{\half + 2d + 2A} \pth{\sum_{n\in I_m} n^{s+d}  \La^{n}\delta^{\frac{n}{2}} }^2
\end{align*}
where we used the result of Proposition \ref{Ybi and Yter estimates} since in the sums over $I_m$ we indeed have $n\leq |\ln\e|^3$. Choosing $\de$ small enough such that $\La \sqrt{\de}<1$  and recalling that $I_m=\llbracket m,2m(N(L)+1)\rrbracket$ and $m\leq N(L)$ we obtain
\begin{align*}
S_\bi & \lesssim  \l v \r_{\mathcal{C}\pth{[0,\e^{-2}\de], H^s(L\T^d)}}^2  L^{\half + 2d + 2A} N(L)^{2(s+d)} (\La\sqrt{\de})^m .
\end{align*}
The same procedure can be applied to $S_{\mathrm{cross}}$ and $S_\ter$ using also the second estimate in Proposition \ref{Ybi and Yter estimates} (note that $\e^2$ and $\e^4$ $S_{\mathrm{cross}}$ and $S_\ter$ will be compensated by the $L^{\frac{2}{\b}}$ in the $L^\infty_t L^1_{t_\bullet} L^\infty_{k,k_\bullet}$ estimate for $\reallywidehat{Y^{\ter}_{m,n}}$). We have thus proved that
\begin{align*}
\l \mathfrak{L}^{m}(v) \r_{\mathcal{C}\pth{[0,\e^{-2}\de], H^s(L\T^d)}} \lesssim  L^{\frac{1}{4}+ d + A} (\ln L)^{s+d} (\La\sqrt{\de})^{\frac{m}{2}}  \l v \r_{\mathcal{C}\pth{[0,\e^{-2}\de], H^s(L\T^d)}},
\end{align*}
which concludes the proof of the proposition.
\end{proof}

\subsection{Proof of Proposition \ref{prop fixed point}}\label{proof fixed point}

Now that we have estimated the operator norm of the iterates of the linear response, we can conclude the proof of Proposition \ref{prop fixed point}.

\begin{proof}[Proof of Proposition \ref{prop fixed point}]
The estimate \eqref{norme op L^m} applied with $m=N(L)$ shows that
\begin{align*}
    \l \mathfrak{L}^{N(L)}\r_{\mathrm{op}}\leq \La  L^{\frac{1}{4}\ln\pth{\de}+\half\ln\La+\frac{1}{4}+d+A}(\ln L)^{s+d}.
\end{align*}
Therefore, choosing $\de$ small enough depending only on $d$ and $A$ so that 
\begin{align*}
    \frac{1}{4}\ln\pth{\de}+\half\ln\La+\frac{1}{4}+d+A<0
\end{align*}
we can then choose $L_0$ large enough depending only on $s$, $d$ and $A$ so that $\l \mathfrak{L}^{N(L)}\r_{\mathrm{op}}<1$ for all $L\geq L_0$. Therefore, when restricted to the set $\mathcal{E}_{L,A}$, the operator $\mathrm{Id}-\mathfrak{L}^{N(L)}$ is invertible and, thanks to the identity \eqref{I-L inverse}, so is $\mathrm{Id}-\mathfrak{L}$ with
\begin{align*}
\l \pth{ \mathrm{Id}-\mathfrak{L}}^{-1} \r_{\mathrm{op}} &\lesssim \sum_{m=0}^{N(L)-1} \l \mathfrak{L}^m \r_{\mathrm{op}} \lesssim \La L^{\frac{1}{4}+ d + A} (\ln L)^{s+d}  \sum_{m=0}^{N(L)-1}  (\La\sqrt{\de})^{\frac{m}{2}} ,
\end{align*} 
where we again used \eqref{norme op L^m}. Choosing $\de$ small, bounding the sum by a universal constant and bounding $(\ln L)^{s+d}$ by $L^A$ for $L$ large enough we have proved that
\begin{align}\label{norme op 1-L}
    \l \pth{ \mathrm{Id}-\mathfrak{L}}^{-1} \r_{\mathrm{op}}\leq L^{d+\frac{1}{4}+2A}.
\end{align}
We now consider the mapping
\begin{align}\label{mapping}
v\longmapsto\pth{\mathrm{Id}-\mathfrak{L}}^{-1}\pth{ \mathfrak{S} +  \mathfrak{N}(v) }
\end{align}
where $\mathfrak{S}$ and $\mathfrak{N}(v)$ have been defined in \eqref{eq remainder} and \eqref{nonlinear term} respectively. We prove that for a particular choice of $\La_0>0$, the above map sends the set
\begin{align*}
Z_{s,A,\La_0} \vcentcolon = \enstq{v\in \mathcal{C}([0,\e^{-2}\delta]; H^s(L\T^d))}{ \|v\|_{\mathcal{C}([0,\e^{-2}\delta]; H^s(L\T^d))}\le \La_0 L^{-\pth{
\frac32d+\frac14+\frac1\beta+3A}} }
\end{align*}
onto itself. The contraction property follows similarly, we leave the details to the reader.

We first estimate the source term $\mathfrak{S}$. Starting from \eqref{source term} and using \eqref{def XleqN}, we find that for $\eta\in \{0,1\}$
\begin{align*}
\mathfrak{S}^\eta(t) & = - \sum_{n=0}^{N(L)} X_n^\eta(t) + X_0^\eta(t) - \e \sum_{\substack{\iota_1,\iota_2\in\{\pm\}}} B^\eta_{(\iota_1,\iota_2)}\pth{0, X^{\bar{\eta},\iota_1}_0,X^{\bar{\eta},\iota_2}_0 }
\\&\quad +  \e \sum_{n_1,n_2=0}^{N(L)} \sum_{\substack{\iota_1,\iota_2\in\{\pm\}}}  B^\eta_{(\iota_1,\iota_2)}\pth{t, X^{\bar{\eta},\iota_1}_{n_1}(t) , X^{\bar{\eta},\iota_2}_{n_2}(t) }
\\&\quad  - i \e^2 \int_0^t \sum_* \sum_{n_1,n_2,n_3=0}^N\sum_{\iota_1,\iota_2,\iota_3\in\{\pm\}}  C^\eta_{*,(\iota_1,\iota_2,\iota_3)}\pth{ \tau , X^{\eta,\iota_1}_{n_1}(\tau) ,X^{\bar{\eta},\iota_2}_{n_2}(\tau) ,X^{\eta,\iota_3}_{n_3}(\tau)  } \d\tau.
\end{align*}
Regrouping terms with respect to the value of $n_1+n_2$ or $n_1+n_2+n_3$ and using \eqref{def X1} and \eqref{def Xn+2} we obtain
\begin{align*}
\mathfrak{S}^\eta(t) & = \e  \sum_{\substack{0\leq n_1,n_2\leq N(L)\\ n_1+n_2 \geq N(L)\\\iota_1,\iota_2\in\{\pm\}}}   B^\eta_{(\iota_1,\iota_2)}\pth{t, X^{\bar{\eta},\iota_1}_{n_1}(t) , X^{\bar{\eta},\iota_2}_{n_2}(t) }
\\&\quad  - i \e^2 \int_0^t \sum_*  \sum_{\substack{0\leq n_1,n_2,n_3\leq N(L)\\n_1+n_2+n_3\geq N(L)-1\\\iota_1,\iota_2,\iota_3\in\{\pm\}}}  C^\eta_{*,(\iota_1,\iota_2,\iota_3)}\pth{ \tau , X^{\eta,\iota_1}_{n_1}(\tau) ,X^{\bar{\eta},\iota_2}_{n_2}(\tau) ,X^{\eta,\iota_3}_{n_3}(\tau)  } \d\tau.
\end{align*}
The bilinear (resp. trilinear) operators in the above right hand side, defined in
\eqref{def B Fourier} (resp. \eqref{def C Fourier}), are continuous operators from $\pth{\mathcal{C}([0,\e^{-2}\delta]; H^s(L\T^d))}^2$ (resp. $\pth{\mathcal{C}([0,\e^{-2}\delta]; H^s(L\T^d))}^3$) to $\mathcal{C}([0,\e^{-2}\delta]; H^s(L\T^d))$ whenever $s>d/2$. Indeed, although there are not multiplication in the physical space, the facts that the functions $Q^\eta$ are bounded and that the $L^1$ norm of the Fourier transform is bounded by the $H^s$ norm, we deduce the boundedness of the operators in a similar way as to prove that $H^s$ is a Banach algebra. Therefore, from Corollary \ref{coro Xn} we deduce that there exists some $\Lambda>0$ such that for any $t\in [0, \e^{-2}\delta]$, we have that
\[
\aligned
\e \sum_{\substack{0\leq n_1,n_2\leq N(L)\\ n_1+n_2 \geq N(L)\\\iota_1,\iota_2\in\{\pm\}}}   \l B^\eta_{(\iota_1,\iota_2)}\pth{t, X^{\bar{\eta},\iota_1}_{n_1}(t) , X^{\bar{\eta},\iota_2}_{n_2}(t) }\r_{H^s(L\T^d)}  \le \Lambda (\Lambda \delta)^{\frac{N(L)}{2}} L^{d+\frac{1}{\beta}+2A},
\endaligned
\]
recalling that assuming that $\delta$ is small enough ensures the convergence of the series $\sum (\Lambda \delta)^{n/2}$. Similarly, we have that
\begin{multline*}
   \e^2 \int_0^t \sum_* \sum_{\substack{0\leq n_1,n_2,n_3\leq N(L)\\n_1+n_2+n_3\geq N(L)-1\\\iota_1,\iota_2,\iota_3\in\{\pm\}}} \| C^\eta_{*,(\iota_1,\iota_2,\iota_3)}\pth{ \tau , X^{\eta,\iota_1}_{n_1}(\tau) ,X^{\bar{\eta},\iota_2}_{n_2}(\tau) ,X^{\eta,\iota_3}_{n_3}(\tau)  } \|_{H^s(L\T^d)}\d\tau \\
   \le  (\Lambda\delta)^{\frac{3N(L)}{4}+1}L^{3(\frac{d}{2}  +\frac1\beta+A)}.
\end{multline*}
We deduce the following estimates 
\[
\aligned
\l\mathfrak{S}\r_{\mathcal{C}([0,\e^{-2}\delta]; H^s(L\T^d))} & \leq \Lambda (\Lambda\delta)^{\frac{N(L)}{2}}L^{3(\frac{d}{2}  +\frac1\beta+A)}, \\
\l\pth{\mathrm{Id}-\mathfrak{L}}^{-1}\mathfrak{S}\r_{\mathcal{C}([0,\e^{-2}\delta]; H^s(L\T^d))}& \leq \Lambda(\Lambda\delta)^{\frac{N(L)}{2}} L^{\frac52d +\frac14+\frac3\beta+5A},
\endaligned
\]
where we used \eqref{norme op 1-L} for the second one. Therefore, since $N(L)\leq \ln L$, we can choose $\de$ small enough so that
\begin{align*}
    \frac{7}{2}d+ \half + \frac{4}{\b} + 8 A + \half \ln\La + \half \ln\de<0
\end{align*}
and obtain
\[
\l(\mathrm{Id}-\mathfrak{L})^{-1}\mathfrak{S}\r_{\mathcal{C}([0,\e^{-2}\delta]; H^s(L\T^d))}\leq \La L^{-\pth{\frac32d+\frac14+\frac1\beta+3A}}.
\]
As for the nonlinear term defined in \eqref{nonlinear term}, using again Corollary \ref{coro Xn} and $\e=L^{-\frac{1}{\b}}$, we find that for $v\in Z_{s,A,\La_0}$ we have
\begin{align*}
    \l \mathfrak{N}(v) \r_{\mathcal{C}([0,\e^{-2}\delta]; H^s(L\T^d))} & \lesssim \e \l v  \r_{\mathcal{C}([0,\e^{-2}\delta]; H^s(L\T^d))}^2 
    \\&\hspace{1cm} \times \pth{ 1 + \e \l X_{\leq N(L)} \r_{\mathcal{C}([0,\e^{-2}\delta]; H^s(L\T^d))} + \e \l v \r_{\mathcal{C}([0,\e^{-2}\delta]; H^s(L\T^d))}}
    \\&\lesssim  L^{-\pth{\frac{3}{2}d+\frac14+\frac1\beta+3A}} \La_0^2 L^{-\pth{ d+\frac14+\frac{2}{\b}+2A}}.  
\end{align*}
Using now \eqref{norme op 1-L} this gives
\begin{align*}
\l \pth{\mathrm{Id}-\mathfrak{L}}^{-1}(\mathfrak{N}(v))\r_{\mathcal{C}([0,\e^{-2}\delta]; H^s(L\T^d))} \lesssim L^{-\pth{\frac{3}{2}d+\frac14+\frac1\beta+3A}} \La_0^2 L^{-\frac{2}{\b}} .
\end{align*}
We have thus proved that there exists a universal constant $\La>0$ such that if $v\in Z_{s,A,\La_0}$ then
\begin{align*}
\l \pth{\mathrm{Id}-\mathfrak{L}}^{-1}\pth{ \mathfrak{S} +  \mathfrak{N}(v) } \r_{\mathcal{C}([0,\e^{-2}\delta]; H^s(L\T^d))} \leq \La L^{-\pth{ \frac32d+\frac14+\frac1\beta+3A}} \pth{ 1 +  \La_0^2 L^{-\frac{2}{\b}}}.
\end{align*}
We then choose $\La_0=\frac{\La}{2}$ and $L$ large enough so that $\La_0^2 L^{-\frac{2}{\b}}<1$, and the above estimate then shows that the mapping \eqref{mapping} sends $Z_{s,A,\La_0}$ to itself. We can prove similarly that it is in fact a contraction.
\end{proof}

This also concludes the proof of Theorem \ref{thm X as series}, after having applied the Banach fixed point theorem to \eqref{mapping}.


\section{Resonant nodes and resonant system}\label{section resonant system}

In this section, we conclude the proof of Theorem \ref{main theorem} by showing that \eqref{cross system} is well-posed on $[0,\de]$ and approaches well the correlations associated to \eqref{fixed point after normal form} over the time interval $[0,\e^{-2}\de]$. In particular, we identify the resonant couplings (Sections \ref{section resonant nodes} and \ref{section resonant couplings}), i.e couplings whose contribution to the correlations survives in the limit $\e\to 0$. The well-posedness of \eqref{cross system} is proved in Section \ref{section WP cross system}, while the final comparison argument is performed in Section \ref{section comparison}.

\subsection{Resonant nodes}\label{section resonant nodes}

We start by giving the definition of resonant nodes.

\begin{definition}\label{def resonant nodes}
Let $C$ be a coupling with sign map $\ffi$ and let $j\in\mathcal{N}(C)$. We say that $j$ is a \textbf{resonant node} if the following holds:
\begin{itemize}
\item[(i)] $j\in\mathcal{N}_\low(C)$,
\item[(ii)] $\mathtt{bush}(j)$ is self-coupled, in the sense of Definition \ref{def self-coupled bushes},
\item[(iii)] we have
\begin{align*}
\ffi(j_1)=-\ffi(j_2) \quad \text{and} \quad \ffi(j_3)=\ffi(j),
\end{align*}
where $j_1$, $j_2$ are the marked children of $j$ and $j_3$ is its unmarked child.
\end{itemize}
\end{definition}

The goal of this section is to obtain appropriate estimates for $F_C$ defined in \eqref{correlations to coupling} 
according to whether or not all nodes in $C$ are resonant. The main result is the following.

\begin{prop}\label{prop nonresonant}
There exist $\a,\La>0$ such that, for all couplings $C$ with at least one non-resonant node, we have
\begin{align*}
\sup_{t\in[0,\de]}\left| \widehat{F_C}(\e^{-2}t,k) \right| \leq \La^{n(C)+1} \de^{\frac{n(C)}{2}} \e^\a,
\end{align*}
for all $\e\leq \de$ and $k\in \Z_L^d$.
\end{prop}

The proof of the above proposition is detailed in in Section \ref{section proof nonresonant} and it is based on intermediate lemmas contained in Section \ref{section intermediate resonant node}.

\subsubsection{Properties of resonant nodes}\label{section intermediate resonant node}

We gather here some properties of resonant nodes.

\begin{lemma}\label{lemma 1 resonant nodes}
Let $C$ be a coupling and $j\in\mathcal{N}(C)$ be such that 
\begin{equation}\label{assumption of the lemma}
\forall j' \in \mathcal{N}(C), \quad j'\leq j \Longrightarrow 
\left\{
\begin{aligned}
&j'\in\mathcal{N}_\low(C),
\\&\text{$\mathtt{bush}(j')$ is self-coupled.}
\end{aligned}
\right.
\end{equation}
\begin{itemize}
\item[(i)] If $j_1$ and $j_2$ are the marked children of $j$, then $K(j_1)+K(j_2)=0$ and $K(j)=K(\ell)$ for some $\ell\in\mathcal{L}(C)$ with $\ell\leq j$.
\item[(ii)] Moreover, if $j_3$ is the unmarked child of $j$, then $(K(j_1),K(j_3))$ is linearly independent.
\end{itemize}
\end{lemma}

\begin{proof}
We prove statement $(i)$ by induction on the height $h$ of $j$, which is the maximal distance between $j$ and a leaf $\ell\leq j$ in the graph representing the coupling. 
\begin{itemize}
\item If $h=1$, i.e if the children of $j$ are leaves, then $\mathtt{bush}(j)=\{\ell_1,\ell_2\}$ where $\ell_1$ and $\ell_2$ are its two marked children. By assumption $\mathtt{bush}(j)$ is self-coupled so  $\si(\ell_1)=\ell_2$ (where $\si$ denotes the coupling map of $C$) and $K(\ell_1)+K(\ell_2)=0$ (recall Definition \ref{def linear map}). Since $K(j)=K(\ell_1)+K(\ell_2)+K(\ell_3)$ (where $\ell_3$ is the unmarked children of $j$) we also obtain $K(j)=K(\ell_3)$.
\item Assume now $h>1$ and that $(i)$ holds for all nodes satisfying \eqref{assumption of the lemma} and of height strictly less than $h$.
 We write $\mathtt{children}(j)=\{j_1,j_2,j_3\}$ with $j_1$ and $j_2$ its marked children. Recalling \eqref{def offsprings}, \eqref{lien offsprings bushes} and \eqref{def k bush} we have
\begin{align*}
K(j_1)+K(j_2) = \sum_{\ell\in\mathtt{offsprings}(j)} K(\ell) = \mathfrak{K}_\bush(j) + \sum_{\substack{j'\in\mathcal{N}_\low(C)\\j'\leq j_1 \text{ or }j'\leq j_2}}\mathfrak{K}_\bush(j').
\end{align*}
By assumption on $j$, we have $\mathfrak{K}_\bush(j')=0$ for all $j'\leq j$ so that $K(j_1)+K(j_2)=0$. This implies $K(j)=K(j_3)$ and since $j_3$ satisfies \eqref{assumption of the lemma} and its height is strictly less than $h$, by induction there must exist $\ell\in\mathcal{L}(C)$ such that $\ell\leq j_3$, and thus $\ell\leq j$, satisfying $K(\ell)=K(j_3)=K(j)$.
\end{itemize}
In order to prove $(ii)$, we first observe that the children of $j$ satisfy \eqref{assumption of the lemma}. From $(i)$ we thus derive the existence of three distinct leaves $\ell_i$ such that $K(j_i)=K(\ell_i)$ for $i=1,2,3$. Since $K(\ell_1)+K(\ell_2)=K(j_1)+K(j_2)=0$, we deduce that $\ell_1=\si(\ell_2)$ and consequently $\si(\ell_3)\notin\{\ell_1,\ell_2\}$, which indeed shows that $K(\ell_3)$ and $K(\ell_1)$, and thus $K(j_3)$ and $K(j_1)$, are linearly independent.
\end{proof}

\begin{lemma}
Let $C$ be a coupling and $j\in\mathcal{N}(C)$ be such that 
\begin{equation}\label{assumption of the lemma 2}
\forall j' \in \mathcal{N}(C), \quad j'\leq j \Longrightarrow \text{$j'$ is resonant.}
\end{equation}
Then, for any $j'\leq j$ and any decoration map of $C$, we have $\De_{j'}=\Om_{j'}=0$ and
\begin{align*}
\int_{I_j(t)} \prod_{j'\leq j} e^{i\e^{-2}\Om_{j'} t_{j'}} \d t_{j'} = c_j t^{n_j},
\end{align*}
where we use the notations
\begin{align*}
n_j &  = \#  \enstq{j'\in\mathcal{N}(C)}{ j'\leq j},
\\ c_j &  = \frac{1}{n_j!} \# \mathfrak{M}\big( \enstq{j'\in\mathcal{N}(C)}{ j'\leq j} \big),
\\ I_j(t) &  = \enstq{(t_{j'})_{j'\in\mathcal{N}(C),j'\leq j}\in [0,t]^{n_j} }{ j'_1<j'_2 \Longrightarrow t_{j'_1}< t_{j'_2}}.
\end{align*}
\end{lemma}

\begin{proof}
Assume $j$ satisfies \eqref{assumption of the lemma 2} and $j'\leq j$. In particular, $j'$ admits no binary nodes below it and $\De_{j'}=\Om_{j'}$. Moreover, $j'$ satisfies \eqref{assumption of the lemma} and from Lemma \ref{lemma 1 resonant nodes} we deduce that $K(j'_1)+K(j'_2)=0$ and $K(j')=K(j'_3)$, where $j'_1$ and $j'_2$ are the two marked children of $j'$ and $j'_3$ its unmarked child. Therefore, using the relation $K(j)(\ka)=\ka(j)$ and the sign properties of the resonant node $j'$, we deduce
\begin{align*}
\De_{j'} & = \ffi(j') \ps{ \ka(j') } - \ffi(j'_3) \ps{ \ka(j'_3) } - \ffi(j'_1) \ps{ \ka(j'_1) } - \ffi(j'_2) \ps{ \ka(j'_2) } 
\\& = \pth{ \ffi(j') - \ffi(j'_3)} \ps{ \ka(j') } - \pth{ \ffi(j'_1)  + \ffi(j'_2)   } \ps{ \ka(j'_1) }
\\& = 0.
\end{align*}
To evaluate the integral, we use the analogue of \eqref{I_C(t) up to negligible} for $I_j(t)$ and write
\begin{align*}
\int_{I_j(t)} \prod_{j'\leq j} e^{i\e^{-2}\Om_{j'} t_{j'}} \d t_{j'} & = \sum_{\rho\in\mathfrak{M}_j} \int_{0\leq t_{\rho(1)}<\cdots < t_{\rho(n_j)}\leq t } \prod_{i=1}^{n_j} e^{i\e^{-2}\Om_{\rho(i)} t_{\rho(i)}} \d t_{\rho(i)} 
\end{align*}
where $\mathfrak{M}_j$ denotes $\mathfrak{M}\big( \enstq{j'\in\mathcal{N}(C)}{ j'\leq j} \big)$. Since all $\Om_{\rho(i)}=0$, we conclude
\begin{align*}
\int_{I_j(t)} \prod_{j'\leq j} e^{i\e^{-2}\Om_{j'} t_{j'}} \d t_{j'} & = \sum_{\rho\in\mathfrak{M}_j} \int_{0\leq t_{\rho(1)}<\cdots < t_{\rho(n_j)}\leq t }  \d t_{\rho(i)}\cdots  \d t_{\rho(n_j)} 
\\& =  \sum_{\rho\in\mathfrak{M}_j} \frac{t^{n_j}}{n_j!}
\\& =  c_j t^{n_j}.
\end{align*}
\end{proof}

\begin{lemma}\label{lem int osc coupling}
Let $C$ be a coupling and $j\in\mathcal{N}_T(C)$ be such that 
\begin{equation}\label{assumption of the lemma 3}
\forall j' \in \mathcal{N}(C), \quad j'< j \Longrightarrow \text{$j'$ is resonant.}
\end{equation}
For all decoration maps of $C$ and for all $t\in[0,\de]$ we have
\begin{equation*}
\left|\int_{I_j(t)} \prod_{j'\leq j} e^{i\e^{-2}\Om_{j'} t_{j'}} \d t_{j'} \right| \leq \left\{
\begin{aligned}
\de^{n_j} & \quad \text{if $|\Om_j|\leq 3\e^2\de^{-1}$,}
\\ 3\de^{n_j-1} \frac{\e^2}{|\Om_j|} & \quad \text{otherwise.}
\end{aligned}
\right.
\end{equation*}
\end{lemma}

\begin{proof}
By hypothesis, all children ${j_i}$, $i=1,2,3$, of $j$
 satisfy \eqref{assumption of the lemma 2} so the previous lemma implies that
\begin{align*}
\int_{I_{j_i}(t)} \prod_{j'\leq {j_i}} e^{i\e^{-2}\Om_{j'} t_{j'}} \d t_{j'} = c_{j_i} t^{n_{j_i}}
\end{align*}
and
\begin{align*}
\int_{I_j(t)} \prod_{j'\leq j} e^{i\e^{-2}\Om_{j'} t_{j'}} \d t_{j'} & = \int_0^t e^{i\e^{-2}\Om_{j}t_{j}}\pth{ \prod_{i=1,2,3}\int_{I_{j_i}(t_{j})} \prod_{j'\leq {j_i}} e^{i\e^{-2}\Om_{j'} t_{j'}} \d t_{j'} }\d t_j
\\& = c_{j_1}c_{j_2}c_{j_3}  \int_0^t e^{i\e^{-2}\Om_{j}t_{j}} t_{j}^{n_{j_1}+n_{j_2}+n_{j_3}} \d t_j.
\end{align*}
Since $n_{j_1}+n_{j_2}+n_{j_3}=n_j-1$, 
\begin{align*}
\int_{I_j(t)} \prod_{j'\leq j} e^{i\e^{-2}\Om_{j'} t_{j'}} \d t_{j'} & = c_{j_1}c_{j_2}c_{j_3}  \int_0^t e^{i\e^{-2}\Om_{j}\tau} \tau^{n_j-1} \d \tau
\end{align*}
Using $c_i\leq 1$ and $n_j\geq 1$, we obtain the following bound independently of the value of $\Omega_j$
\begin{align*}
\left| \int_{I_j(t)} \prod_{j'\leq j} e^{i\e^{-2}\Om_{j'} t_{j'}} \d t_{j'} \right| \leq \de^{n_j}.
\end{align*}
In addition, if $\Om_{j}\neq 0$ we can perform an integration by parts
\begin{align*}
\int_0^t e^{i\e^{-2}\Om_{j}\tau} \tau^{n_j-1} \d \tau & = \frac{e^{i\e^{-2}\Om_{j} t} - 1}{i\e^{-2}\Om_{j}} t^{n_j-1} - \int_0^t \frac{e^{i\e^{-2}\Om_{j} t}}{i\e^{-2}\Om_{j}} (n_j-1)\tau^{n_j-2}\d\tau,
\end{align*}
then simply bound the imaginary exponentials by 1 in order to obtain the following estimate
\begin{align*}
\left| \int_0^t e^{i\e^{-2}\Om_{j}\tau} \tau^{n_j-1} \d \tau \right| \leq 3\frac{\de^{n_j-1}}{\e^{-2}|\Om_{j}|}  ,
\end{align*}
and the conclusion of the proof.
\end{proof}

\begin{lemma}\label{dernier lem inter}
Let $C$ be a coupling such that $n_B(C)=0$ and such that there exists $j\in\mathcal{N}_T(C)$ satisfying \eqref{assumption of the lemma 3}. For all decoration maps of $C$ and all $t\in[0,\de]$, we have
\begin{equation}\label{borne int osc nB=0}
\left| \int_{I_C(t)} \prod_{j'\in\mathcal{N}_T(C)} e^{i\e^{-2}\Om_{j'} t_{j'}} \d t_{j'} \right| \leq \left\{
\begin{aligned}
\de^{n_T(C)} & \quad \text{if $|\Om_j|\leq 3\e^2\de^{-1}$,}
\\ 3\de^{n_T(C)-1} \frac{\e^2}{|\Om_j|} & \quad \text{otherwise.}
\end{aligned}
\right.
\end{equation}
\end{lemma}

\begin{proof}
Let $A$ and $A'$ denote the two trees in $C$ and assume, without loss of generality, that $j$ satisfying \eqref{assumption of the lemma 3} is a ternary node in $A$. By definition of $I_C(t)$ we have 
\begin{align*}
\int_{I_C(t)} \prod_{j'\in\mathcal{N}_T(C)} e^{i\e^{-2}\Om_{j'} t_{j'}} \d t_{j'} & = \int_{I_A(t)} \prod_{j'\in\mathcal{N}_T(A)} e^{i\e^{-2}\Om_{j'} t_{j'}} \d t_{j'}  \times \int_{I_{A'}(t)} \prod_{j'\in\mathcal{N}_T(A')} e^{i\e^{-2}\Om_{j'} t_{j'}} \d t_{j'} .
\end{align*}
We bound the latter integral in the above right hand side simply by
\begin{align}\label{straightforward bound on oscillating integral}
\left| \int_{I_{A'}(t)} \prod_{j'\in\mathcal{N}_T(A')} e^{i\e^{-2}\Om_{j'} t_{j'}} \d t_{j'}  \right| \leq \de^{n_T(A')}.
\end{align}
In order to conclude the proof, it is therefore enough to prove the following statement by induction on $n\ge 2$: for all $n\ge 2$ and for all signed coloured trees $A$ of size $n$ with $n_B(A)=0$ and such that there exists  $j\in\mathcal{N}_T(A)$ satisfying \eqref{assumption of the lemma 3}, we have the bound \eqref{borne int osc nB=0} with $C$ replaced by $A$.
\begin{itemize}
\item If $n=2$ then $j$ is the root, $I_j(t)=I_A(t)$ and the result of Lemma \ref{lem int osc coupling} applies.
\item Assume the induction hypothesis is true up to some $n\ge 2$ and let $A$ be a signed coloured tree in $\tilde{A}_{n+2}$ with $n_B(A)=0$ and with some $j\in \mathcal{N}_T(A)$ satisfying \eqref{assumption of the lemma 3}. Then $A=T_*(A_1,A_2,A_3)$ for some $A_i\in\tilde{A}_{n_i}$ with $n_i\leq n$ and $n_1+n_2+n_3=n$. If $j=r_A$ coincides with the root of $A$, the result of Lemma \ref{lem int osc coupling} applies again. If $j\neq r_A$, we can assume without loss of generality that $j\in A_1$, then use the induction hypothesis on $A_1$ and \eqref{straightforward bound on oscillating integral} to the integrals over $I_{A_2}, I_{A_3}$ to derive
\begin{align*}
\left| \int_{I_A(t)} \prod_{j'\in\mathcal{N}_T(A)} e^{i\e^{-2}\Om_{j'} t_{j'}} \d t_{j'} \right| & = \left| \int_0^t e^{i\e^{-2}\Om_r \tau} \prod_{i=1}^3 \int_{I_{A_i}(t)} \prod_{j'\in\mathcal{N}_T(A_i)} e^{i\e^{-2}\Om_{j'} t_{j'}} \d t_{j'} \right|
\\& \leq \de^{1+n_T(A_2)+n_T(A_3)} \left| \int_{I_{A_1}(t)} \prod_{j'\in\mathcal{N}_T(A_1)} e^{i\e^{-2}\Om_{j'} t_{j'}} \d t_{j'} \right|
\\& \leq \de^{1+n_T(A_2)+n_T(A_3)} \de^{n_T(A_1)} \min \pth{ 1 , \frac{3\e^2}{\de |\Om_j|}}.
\end{align*}
 Since $n_T(A)= 1 + n_T(A_1)+ n_T(A_2)+n_T(A_3)$, this concludes the proof.
\end{itemize}
\end{proof}

\subsubsection{Proof of Proposition \ref{prop nonresonant}}\label{section proof nonresonant}

We will make use of the following simple lemma.

\begin{lemma}\label{lem span}
Let $C=(A,A',\si)$ be a coupling and $\mathcal{F}\subset\mathcal{L}(A)$. If $K(\mathtt{root}(A)) \in \mathrm{Span}\enstq{ K(\ell) }{\ell\in\mathcal{F}}$ then
\begin{align}\label{Kroot=}
K(\mathtt{root}(A)) = \sum_{\ell\in\mathcal{F},\si(\ell)\notin\mathcal{L}(A)}K(\ell)
\end{align}
and moreover $\#\enstq{\ell\in\mathcal{F}}{\si(\ell)\notin\mathcal{L}(A)}$ has the same parity of $\mathcal{L}(A)$.
\end{lemma}

\begin{proof}
By assumption, there exist numbers $s_\ell$ such that 
\begin{align}\label{identity span K root}
K(\mathtt{root}(A)) = \sum_{\ell\in\mathcal{F}} s_\ell K(\ell)
\end{align}
Since $K(\mathtt{root}(A))=\sum_{\ell\in\mathcal{L}(A)}K(\ell)$, \eqref{identity span K root} rewrites $\sum_{\ell\in\mathcal{L}(A)} \mu_\ell K(\ell)=0$ where $\mu_\ell=1$ if $\ell\in\mathcal{L}(A)\setminus \mathcal{F}$ and $\mu_\ell=1-s_\ell$ if $\ell\in\mathcal{F}$. Using $K(\ell)=-K(\si(\ell))$, we can reformulate this equality as a relation involving only sums over positive distinct leaves
\begin{align*}
\sum_{\substack{\ell\in\mathcal{L}(A)_+ \\ \ell\notin\si(\mathcal{L}(A)_-) }} \mu_\ell K(\ell) + \sum_{\substack{\ell\in\mathcal{L}(A)_+ \\ \ell\in\si(\mathcal{L}(A)_-) }} (\mu_\ell - \mu_{\si(\ell)}) K(\ell) - \sum_{\substack{\ell\notin\mathcal{L}(A)_+ \\ \ell\in\si(\mathcal{L}(A)_-) }} \mu_{\si(\ell)}K(\ell) & = 0.
\end{align*}
But the maps $\{K(\ell)\, |\, \ell\in\mathcal{L}(A)_+\}$ are linearly independent, therefore all coefficients in the above left hand side are zero. Defining $\mathfrak{C}(A)\vcentcolon = \enstq{\ell\in\mathcal{L}(A)}{\si(\ell)\in\mathcal{L}(A)}$, this implies that for all $\ell \in\mathcal{L}(A)$ we have
\begin{align}\label{caractérisation centre}
\ell\in \mathfrak{C}(A) \Longrightarrow \mu_\ell =\mu_{\si(\ell)} \quad \text{and} \quad \ell\notin \mathfrak{C}(A) \Longrightarrow \mu_\ell = 0.
\end{align}
Now let $\ell\in\mathcal{F}$, so that $\mu_\ell= 1-s_\ell$. We distinguish three cases:
\begin{itemize}
\item If $\ell \in  \mathfrak{C}(A)$ and $\si(\ell)\in \mathcal{F}$, then $\mu_\ell=\mu_{\si(\ell)}$ and $\mu_{\si(\ell)}=1-s_{\si(\ell)}$ which imply $s_\ell=s_{\si(\ell)}$. Therefore, in this case we can remove $\ell$ and $\si(\ell)$ from the identity \eqref{identity span K root}. 
\item If $\ell \in  \mathfrak{C}(A)$ and $\si(\ell)\notin \mathcal{F}$, then $\mu_\ell=\mu_{\si(\ell)}$ and $\mu_{\si(\ell)}=1$ which imply $s_\ell=0$. Also in this case,we can remove $\ell$ from \eqref{identity span K root}.
\item If $\ell \notin  \mathfrak{C}(A)$, then $\mu_\ell = 0$ and so $s_\ell=1$.
\end{itemize}
We have thus proved \eqref{Kroot=}. Another consequence of the definition of $\mu_\ell$ and \eqref{caractérisation centre} is that $\mathcal{F}\cup \mathfrak{C}(A) = \mathcal{L}(A)$. This implies that 
\begin{align*}
\#\enstq{\ell\in\mathcal{F}}{\si(\ell)\notin\mathcal{L}(A)} = \#\mathcal{L}(A) - \#\mathfrak{C}(A),
\end{align*}
which concludes the proof of the lemma since $\#\mathfrak{C}(A)$ is even.
\end{proof}

We are now in the position to prove Proposition \ref{prop nonresonant}.

\begin{proof}[Proof of Proposition \ref{prop nonresonant}]
Let us first assume that $C$ has at least one binary node, i.e $n_B(C)\geq 1$. Then, when estimating the right hand side of \eqref{fourier of FC}, we simply bound the time integral by $\de^{n_T(C)}$, the $q_j$s by 1 and the $M^{\eta,\eta'}$ by some universal constant, then make the standard change of variable in the sum over $\mathcal{D}_k(C)$ using factor $L^{-\frac{n(C)d}{2}}$ and obtain that
\begin{align*}
\left| \reallywidehat{F_C}(\e^{-2}t,k) \right| & \leq \La^{n(C)+1} \de^{n_T(C)} \e^{n_B(C)}.
\end{align*}
Since $n(C)=n_B(C)+2n_T(C)$, $\e\leq \de$ and $n_B(C)\geq 1$, this gives
\begin{align*}
\left| \reallywidehat{F_C}(\e^{-2}t,k) \right| & \leq \La^{n(C)+1} \de^{\frac{n(C)}{2}} \e^{\half}.
\end{align*}

\saut
Assume now that $n_B(C)=0$ and that $C$ admits at least one non-resonant ternary node $j\in\mathcal{N}_T(C)$. Without loss of generality, we can assume that $j$ satisfies \eqref{assumption of the lemma 3} and belongs to $A$. Therefore, Lemma \ref{dernier lem inter} applies and \eqref{borne int osc nB=0} holds. Applying the inequality $\min\pth{1,\frac{a}{b}}\leq \frac{\sqrt{2}a}{\sqrt{a^2+b^2}}$ with $a=3\de^{-1}$ and $b=\e^{-2}\Om_j$, we obtain
\begin{align*}
\left|  \int_{I_C(t)} \prod_{j'\in\mathcal{N}_T(C)} e^{i\e^{-2}\Om_{j'} t_{j'}} \d t_{j'} \right| \leq \de^{n_T(C)-1} \frac{3\sqrt{2}}{\left| 3\de^{-1} + i \e^{-2} \Om_j \right|}.
\end{align*}
Plugging this into \eqref{fourier of FC}, using the compact support assumption for the $M^{\eta,\eta'}$ and bounding all the $q_{j'}$ for $j'\neq j$ by 1, we obtain
\begin{align}\label{inter bound FC}
\left|\widehat{F_C}(\e^{-2}t,k) \right| \leq  \La^{n(C)+1} \frac{\de^{n_T(C)-1}}{L^{n_T(C)d}}    \sum_{\substack{\ka\in\mathcal{D}_k(C)\\ \ka(\mathcal{L}(C)_+)\subset B(0,R)}}    \frac{ |q_j| }{\left| 3\de^{-1} + i \e^{-2} \Om_j \right|},
\end{align}
where we also used that $n(C)=2n_T(C)$. In order to estimate the RHS of this inequality we distinguish between the nature of $j$, i.e. if it is a low node or a high node.

\paragraph{Case 1: $j\in\mathcal{N}_\low(C)$.} We denote by $j_1$, $j_2$ its marked children and by $j_3$ its unmarked child, and we further distinguish between two situations.

\paragraph{Case 1.1: $\mathtt{bush}(j)$ is self-coupled.} In this case, $j$ satisfies satisfies \eqref{assumption of the lemma} and the first part of Lemma \ref{lemma 1 resonant nodes} implies
\begin{align*}
\Om_j & = (\ffi(j) - \ffi(j_3)) \ps{\ka(j_3)} - (\ffi(j_1) + \ffi(j_2)) \ps{\ka(j_1)} .
\end{align*}
Since $j$ is not resonant, then the quantities $\ffi(j) - \ffi(j_3)$ and $\ffi(j_1) + \ffi(j_2)$ (which take value in $\{-2,0,2\}$) cannot vanish simultaneously. If they differ from each other, then we have $|\Om_j|\gtrsim \ps{\ka(j_i)} \gtrsim m$ for some $i=1,3$ so that \eqref{inter bound FC} becomes
\begin{align*}
\left|\widehat{F_C}(\e^{-2}t,k) \right| \leq  \La^{n(C)+1} \de^{n_T(C)-1}    \e^2 ,
\end{align*}
which also concludes the proof since $\de^{n_T(C)-1}    \e^2\leq \de^{n_T(C)}    \e = \de^{\frac{n(C)}{2}}    \e$. Therefore the only remaining case is when they are both equal, so that $\Om_j  = 2\iota \pth{ \ps{\ka(j_3)} - \ps{\ka(j_1)}}$ for some $\iota\in\{\pm\}$. The second part of Lemma \ref{lemma 1 resonant nodes} implies that there exists leaves $\ell_1$ and $\ell_3$ such that $K(\ell_i)=K(j_i)$ for $i=1,3$ and $K(\ell_1)$ and $K(\ell_3)$ are linearly independent. 
\begin{itemize}
\item If $K(\mathtt{root}(A))\notin\mathrm{Span}\{K(\ell_1),K(\ell_3)\}$ then we can complete $\pth{K(\mathtt{root}(A)),K(\ell_1),K(\ell_3) }$ into a basis of $V(C)$ with its initial basis: there exists leaves $\ell_m\in\mathcal{L}(C)_+$ for $m=4,\dots, n_T(C)+1$ such that $\pth{K(\mathtt{root}(A)),K(\ell_1),K(\ell_3), K(\ell_4),\dots,K(\ell_{n_T(C)+1}) }$ is a basis of $V(C)$. By performing a change of variable similar to the one of Section \ref{section basis of V(C) and change of variable} and applying Lemma \ref{lem qj} to $q_j$, we deduce from \eqref{inter bound FC} that
\begin{align*}
&\left|\widehat{F_C}(\e^{-2}t,k) \right| 
\\& \leq  \La^{n(C)+1} \frac{\de^{n_T(C)-1}}{L^{n_T(C)d}}  \sum_{(k_1,k_3)\in\pth{ \Z_L^d\cap B(0,R) }^2}  \frac{ 1}{\ps{k_3}\ps{k_1}^2\left| 3\de^{-1} + i \e^{-2} 2\iota \pth{ \ps{k_3} - \ps{k_1}} \right|}
\\&\hspace{3cm}\times \sum_{(k_r,k_4,\dots,k_{n_T(C)+1})\in\pth{ \Z_L^d\cap B(0,R) }^{n_T(C)-1}}  \mathbbm{1}_{k_r=k}
\\&\leq \La^{n(C)+1} \de^{n_T(C)-1} L^{-2d} \sum_{k_3\in \Z_L^d\cap B(0,R)} \ps{k_3}^{-1} \pth{ \sum_{k_1\in \Z_L^d\cap B(0,R) }  \frac{ 1}{\ps{k_1}^2\left| 3\de^{-1} + i \e^{-2} 2\iota \pth{ \ps{k_3} - \ps{k_1}} \right|} }
\\&\leq \La^{n(C)+1} \de^{n_T(C)-1}  L^{-\pth{1+\frac{1}{p}}d}
\\&\qquad \times \sum_{k_3\in \Z_L^d\cap B(0,R)} \ps{k_3}^{-1} \pth{ \sum_{k_1\in \Z_L^d\cap B(0,R) }  \frac{ 1}{\ps{k_1}^{2p}\left| 3\de^{-1} + i \e^{-2} 2\iota \pth{ \ps{k_3} - \ps{k_1}} \right|^p} }^{\frac{1}{p}}
\end{align*}
where we used Hölder's inequality for some $p>\frac{d}{2}$. Using now Lemma \ref{lem:RiemanntoInt} and a similar proof than the one of \eqref{intupsilon iota=+} in the proof of Lemma \ref{lem:intTheta} we can prove that
\begin{align*}
 \sum_{k_1\in \Z_L^d\cap B(0,R) }  \frac{ 1}{\ps{k_1}^{2p}\left| 3\de^{-1} + i \e^{-2} 2\iota \pth{ \ps{k_3} - \ps{k_1}} \right|^p} \lesssim L^d \pth{ \e^{\b-2} + \de^{p-1}\e^2 },
\end{align*}
so that
\begin{align*}
\left|\widehat{F_C}(\e^{-2}t,k) \right| & \leq \La^{n(C)+1} \de^{n_T(C)-1}   \pth{ \e^{\b-2} + \de^{p-1}\e^2 }^{\frac{1}{p}} L^{-d} \sum_{k_3\in \Z_L^d\cap B(0,R)} \ps{k_3}^{-1} 
\\& \leq   \La^{n(C)+1} \de^{n_T(C)-1}   \pth{ \e^{\b-2} + \de^{p-1}\e^2 }^{\frac{1}{p}} L^{-\frac{d}{q}} 
\end{align*}
where we again used Hölder's inequality on the last sum for some $q>d$. By choosing in addition $q<2d$ and using $\e\leq \de$ one can show that this implies
\begin{align}\label{estim cas 1.1}
\left|\widehat{F_C}(\e^{-2}t,k) \right| & \leq  \La^{n(C)+1} \de^{n_T(C)} \e^{\min\pth{ \frac{\b-2}{p}, \frac{\b d}{q} }}.
\end{align}
\item If now $K(\mathtt{root}(A))\in\mathrm{Span}\{K(\ell_1),K(\ell_3)\}$, then $\#\mathcal{L}(A)=2n_T(A)+1$ and the result of Lemma \ref{lem span} show that $K(\mathtt{root})=K(\ell_i)$ for some $i\in\{1,3\}$, say $i=3$ without loss of generality. We then complete $\pth{ K(\ell_1),K(\ell_3)}$ into a basis of $V(C)$ as above, and the constraint on the root of $A$ translates as a constraint on $K(\ell_3)$. After the usual change of variable we again obtain from \eqref{inter bound FC}:
\begin{align*}
\left|\widehat{F_C}(\e^{-2}t,k) \right| \leq  \La^{n(C)+1} \de^{n_T(C)-1}  L^{-d}  \sum_{(k_1,k_3)\in\pth{\Z_L^d\cap B(0,R) }^2}    \frac{ \mathbbm{1}_{k_3=k} }{\ps{k_1}^2\ps{k_3}\left| 3\de^{-1} + i \e^{-2} 2\iota \pth{ \ps{k_3} - \ps{k_1}} \right|}.
\end{align*}
We can then conclude as above and obtain \eqref{estim cas 1.1} again.
\end{itemize}

\paragraph{Case 1.2: $\mathtt{bush}(j)$ is not self-coupled.} Since every node below $j$ is resonant, we have $K(j_1)+K(j_2)=\mathfrak{K}_{\bush}$ and thus $K(j_1)+K(j_2)$ is not identically zero. As above, Lemma \ref{lemma 1 resonant nodes} implies that there exists leaves $\ell_1$ and $\ell_3$ such that $K(\ell_i)=K(j_i)$ for $i=1,3$. Moreover, we have that $K(\mathtt{root}(A))\notin \mathrm{Span}\{K(\ell_1)+K(\ell_2)\}$ since if $K(\mathtt{root}(A))\in \mathrm{Span}\{K(\ell_1)+K(\ell_2)\}$ then Lemma \ref{lem span} would imply $K(\mathtt{root}(A))=K(\ell_i)$ for some $i=1,2$ and thus $K(\ell_i)\in\mathrm{Span}\{K(\ell_1)+K(\ell_2)\}$ which is absurd. Therefore we complete $\pth{ K(j_1)+K(j_2) , K(\mathtt{root}(A)) }$ into a basis $\pth{ K(j_1)+K(j_2) , K(\mathtt{root}(A)) , K(\ell_3),\dots,K(\ell_{n_T(C)+1} )}$ of $V(C)$ and obtain from \eqref{inter bound FC}:
\begin{align*}
\left|\widehat{F_C}(\e^{-2}t,k) \right| & \leq  \La^{n(C)+1} \frac{\de^{n_T(C)}}{L^{n_T(C)d}}    \sum_{\pth{k_1,k_2,k_3,\dots,k_{n_T(C)+1}} \in \pth{\Z_L^d\cap B(0,R) }^{n_T(C)+1}} \mathbbm{1}_{|k_1|\leq \e^\ga}  \mathbbm{1}_{k_2=k} 
\\&\leq   \La^{n(C)+1} \de^{n_T(C)} \e^{\ga d}.
\end{align*}

\paragraph{Case 2: $j\in\mathcal{N}_\high(C)$.} We denote again by $j_1$, $j_2$ and $j_3$ the children of $j$. As above, Lemma \ref{lemma 1 resonant nodes} implies that there exists distinct leaves $\ell_i$ such that $K(\ell_i)=K(j_i)$ for $i=1,2,3$. Since the leaves are distinct and satisfy $K(\ell_i)\neq 0$ (by definition of $V(C)$), we can assume that $K(j_1) + K(j_2)\neq 0$ and $K(j_1) + K(j_3)\neq 0$. Note that among the five situations characterizing a high node, either all the interactions are high, or $K(j_1) + K(j_2)$ or $K(j_1) + K(j_3)$ are low. Therefore \eqref{inter bound FC} implies
\begin{align*}
\left|\widehat{F_C}(\e^{-2}t,k) \right| & \leq B_1 + B_2 + B_3, 
\end{align*}
where
\begin{align*}
B_1 & \vcentcolon =  \La^{n(C)+1} \de^{n_T(C)} L^{-n_T(C)d}    \sum_{\substack{\ka\in\mathcal{D}_k(C)\\ \ka(\mathcal{L}(C)_+)\subset B(0,R)}}  \mathbbm{1}_{|K(j_1)(\ka) + K(j_2)(\ka) |\leq 2 \e^\ga} ,
\\ B_2 & \vcentcolon = \La^{n(C)+1} \de^{n_T(C)} L^{-n_T(C)d}    \sum_{\substack{\ka\in\mathcal{D}_k(C)\\ \ka(\mathcal{L}(C)_+)\subset B(0,R)}}  \mathbbm{1}_{|K(j_1)(\ka) + K(j_3)(\ka) |\leq 2 \e^\ga}   ,
\\ B_3 & \vcentcolon = \La^{n(C)+1} \de^{n_T(C)-1}L^{-n_T(C)d}   
\\&\quad \times  \sum_{\substack{\ka\in\mathcal{D}_k(C)\\ \ka(\mathcal{L}(C)_+)\subset B(0,R)}}    \frac{  \mathbbm{1}_{|K(j_1)(\ka) + K(j_2)(\ka) |\geq 2 \e^\ga}  \mathbbm{1}_{|K(j_1)(\ka) + K(j_3)(\ka) |\geq 2 \e^\ga}  \mathbbm{1}_{|K(j_3)(\ka) + K(j_2)(\ka) |\geq 2 \e^\ga} }{\ps{K(j_1)(\ka)}\ps{K(j_2)(\ka)}\ps{K(j_3)(\ka)} \left| 3\de^{-1} + i \e^{-2} \Om_j \right|}.
\end{align*} 
For $B_1$, we again note that $K(\mathtt{root}(A))\notin \mathrm{Span}\{K(j_1)+K(j_2)\}$ so that we can complete $\pth{ K(j_1)+K(j_2) , K(\mathtt{root}(A)) }$ into a basis of $V(C)$ and do as in Case 1.2. The case of $B_2$ is identical by considering $K(\ell_1)+K(\ell_3)$ instead. We thus obtain
\begin{align*}
B_1+B_2 \leq  \La^{n(C)+1} \de^{n_T(C)} \e^{\ga d}.
\end{align*}
For $B_3$, we first note that if $\pth{ K(j_1),K(j_2),K(j_3)}$, i.e. $\pth{ K(\ell_1),K(\ell_2),K(\ell_3)}$ is not linearly independent, then there must exists $i_1\neq i_2\in\{1,2,3\}$ such that $\si(\ell_{i_1})=\ell_{i_2}$ so that $K(j_{i_1})+K(j_{i_2})=0$ and $B_3=0$. We can thus assume that $\pth{ K(j_1),K(j_2),K(j_3)}$ is linearly independent.
\begin{itemize}
\item If $K(\mathtt{root}(A))\notin\mathrm{Span}\{K(j_1),K(j_2),K(j_3)\}$, then we complete $\pth{ K(j_1),K(j_2),K(j_3),K(\mathtt{root}(A))}$ into a basis of $V(C)$ and ultimately obtain
\begin{align*}
&B_3 
\\& \leq \La^{n(C)+1} \de^{n_T(C)-1}L^{-3d}   
\\& \times  \sum_{(k_1,k_2,k_3)\in\pth{\Z_L^d\cap B(0,R)}^3}    \frac{ \ps{k_1}^{-1}\ps{k_2}^{-1}\ps{k_3}^{-1} \mathbbm{1}_{|k_3 + k_2 |\geq 2 \e^\ga} }{ \left| 3\de^{-1} + i \e^{-2} \pth{ \ffi(j)\ps{k_1+k_2+k_3} - \ffi(j_1)\ps{k_1} - \ffi(j_2)\ps{k_2} - \ffi(j_3)\ps{k_3} }\right|}.
\end{align*}
By summing first on $k_1$ and using again Lemmas \ref{lem:RiemanntoInt} and \ref{lem:intTheta} we can conclude as in Case 1.1.
\item If $K(\mathtt{root}(A))\in\mathrm{Span}\{K(j_1),K(j_2),K(j_3)\}$, then Lemma \ref{lem span} implies that either $K(\mathtt{root}(A))=K(j_i)$ for some $i=1,2,3$ or $K(\mathtt{root}(A))=K(j_1)+K(j_2)+K(j_3)$. If say $K(\mathtt{root}(A))=K(j_1)$ we complete $\{K(j_1),K(j_2),K(j_3)\}$ into a basis and translate the constraint on the root to a constraint on $K(j_1)(\ka)$ to obtain
\begin{align*}
B_3 & \leq \La^{n(C)+1} \de^{n_T(C)-1}L^{-2d}   
\\&\quad \times  \sum_{(k_2,k_3)\in\pth{\Z_L^d\cap B(0,R)}^2}    \frac{ \ps{k_2}^{-1}\ps{k_3}^{-1} \mathbbm{1}_{|k + k_2 |\geq 2 \e^\ga}  }{ \left| 3\de^{-1} + i \e^{-2} \pth{ \ffi(j)\ps{k+k_2+k_3} - \ffi(j_1)\ps{k} - \ffi(j_2)\ps{k_2} - \ffi(j_3)\ps{k_3}} \right|}.
\end{align*}
If $K(\mathtt{root}(A))=K(j_1)+K(j_2)+K(j_3)$ we again complete $\{K(j_1),K(j_2),K(j_3)\}$ into a basis and translate the constraint on the root to a constraint on $K(j_1)(\ka)+K(j_2)(\ka)+K(j_3)(\ka)$ and obtain
\begin{align*}
B_3 & \leq \La^{n(C)+1} \de^{n_T(C)-1}L^{-2d}   
\\&\quad \times  \sum_{(k_2,k_3)\in\pth{\Z_L^d\cap B(0,R)}^2}    \frac{\ps{k_1}^{-1} \ps{k_2}^{-1}\ps{k_3}^{-1}  \mathbbm{1}_{|k-k_1 |\geq 2 \e^\ga}   }{ \left| 3\de^{-1} + i \e^{-2} \pth{ \ffi(j)\ps{k} - \ffi(j_1)\ps{k_1} - \ffi(j_2)\ps{k_2} - \ffi(j_3)\ps{k-k_1-k_2}} \right|}.
\end{align*}
In both cases, by summing first on $k_2$ we can conclude using Lemmas \ref{lem:RiemanntoInt} and \ref{lem:intTheta} as in Case 1.1
\end{itemize}
This concludes the proof of the proposition.
\end{proof}

\subsection{Resonant couplings}\label{section resonant couplings}

In this section, we introduce the notion of resonant couplings and ultimately show that their recursive structure is the one of binary trees.

\subsubsection{First properties and reduction}

In the following remark we introduce a useful way of denoting signed coloured tree.

\begin{remark}\label{rm:labellednodes} 
Let $(A,\ffi,c)$ be a signed coloured tree with $A \in \tilde{\mathcal{A}}_n$ with $n>0$ and assume that $\mathtt{root}(A)$ is a $*$-ternary node (resp. $*$-binary node). Then $A$ writes $\pth{T_*\pth{\tilde A_1,\tilde A_2,\tilde A_3}, \varphi, c}$ (resp. $\pth{B_*\pth{\tilde A_1,\tilde A_2}, \varphi, c}$). We allow ourselves to also write $A$ as a labbelled tree, that is to say, we allow the abuse of notation $A = (T_*(A_1,A_2,A_3),\iota,\eta)$ (resp. $A = (B_*(A_1,A_2),\iota,\eta)$) where for $i=1,2,3$ (resp. $i=1,2$), we set $A_i = \pth{\tilde A_i , \varphi_{|\mathcal N(\tilde A_i) \cup \mathcal L(\tilde A_i)} , c_{|\mathcal N(\tilde A_i) \cup \mathcal L(\tilde A_i)}}$ and $\iota = \varphi(\mathtt{root}(A))$, $\eta = c(\mathtt{root}(A))$. 
\end{remark}

\begin{definition}
For $n\in\mathbb{N}$, we define $\mathtt{Res}_n$ to be the set of couplings with $n$ nodes that are all resonant (in the sense of Definition \ref{def resonant nodes}). We also define
\begin{align*}
\mathtt{Res}(\iota,\iota',\eta,\eta') = \mathtt{Res}_n \cap \bigcup_{n_1,n_2\in\mathbb{N}} \mathcal C^{\iota,\iota',\eta,\eta'}_{n_1,n_2}.
\end{align*}
\end{definition}

Thanks to 
\[
\int_{I_C(t)} \prod_j dt_j = \# \mathfrak M(C) \frac{t^n}{n!},
\]
we can easily prove from \eqref{fourier of FC} that if $C \in \mathtt{Res}_n$ then we have
\begin{align}\label{FCresonant}
\widehat{F_C}(\varepsilon^{-2}t,k) = (-i)^n L^{-nd} \sum_{\kappa \in \mathcal D_k(C)} \prod_{j\in \mathcal N(C)} q_j \prod_{l\in \mathcal L_-(C)} M^{c(l),c(\sigma(l))}(-\kappa(l)) \# \mathfrak M(C) \frac{t^n}{n!}.
\end{align}

\begin{lemma} \label{lemma res coupling reduction}
Let $C \in \mathtt{Res}_n$ with $n>0$. There exists $n_1,n_2 \in \N$ such that $n_1 + n_2 = n-1$, $\iota \in \{\pm\}$, $\eta \in \{0,1\}$ and $(C_1,C_2) \in \mathtt{Res}_{n_1} \times \mathtt{Res}_{n_2}$ such that if we write $C_i = (A_i,A'_i,\sigma_i)$ and define $\sigma : \mathcal L(C_1)\sqcup \mathcal L(C_2) \rightarrow \mathcal L(C_1)\sqcup \mathcal L(C_2)$ by
\[
\sigma(\ell) = \left \lbrace{\begin{array}{cc}
\sigma_1(\ell)     & \mathrm{if}\, \ell\in \mathcal L(C_1), \\
\sigma_2(\ell)     & \mathrm{if}\, \ell\in \mathcal L(C_2),
\end{array}}\right.
\]
then we are in one of the following situations
\begin{align*}
    C = & ((T_\lowr (A_1,A_1',A_2),\iota,\eta),A'_2,\sigma)\\
    C = & ((T_\lowm (A_1,A_2,A_1'),\iota,\eta),A'_2,\sigma)\\
    C = & ((T_{\lowl} (A_2,A_1,A_1'),\iota,\eta),A'_2,\sigma)\\
    C = & (A_2, (T_\lowr(A_1,A_1',A_2') \iota, \eta),\sigma) \\
    C = & (A_2, (T_\lowm(A_1,A_2',A_1') \iota, \eta),\sigma) \\
    C = & (A_2, (T_\lowr(A_2',A_1,A_1') \iota, \eta),\sigma) .
\end{align*}
Moreover, if $\iota = \iota'$ then $\mathtt{Res}_n(\iota,\iota',\eta,\eta') = \emptyset$.
\end{lemma}

\begin{remark} 
We note that the different situations for $C$ are not exclusive.
\end{remark}

\begin{proof}
Assume $C = (A,A',\sigma)$ belongs to $\mathtt{Res}_n$ with $n>0$. Since $n>0$, either $\mathtt{root}(A)$ or $\mathtt{root}(A')$ is a node, and because this node is resonant, it is a low node. We assume that $\mathtt{root}(A)$ is a $\lowr$ node, which takes us to the first situation above (the other possibilities corresponding to the other cases). In this case, we get that $A = (T_\lowr (A_1,A_1',A_2),\iota,\eta) $ for some $\iota,\eta$ (recall Remark \ref{rm:labellednodes}). We have already seen that $\mathtt{offsprings}(\mathtt{root}(A))$ is self-coupled, which tells us that 
\[
C_1 = \pth{A_1,A_1',\sigma_{|\mathcal L(A_1) \cup \mathcal L(A_1')}} \quad \text{and} \quad C_2 = \pth{A_2,A',\sigma_{|\mathcal L(A_2) \cup \mathcal L(A'_2)}}
\]
are couplings. We remark that the bush of a given node $j$ is a subset of the leaves below $j$. Therefore, if $j\in \mathcal N_\low(A_1)$, the set $\mathtt{bush}(j)$ does not depend on whether we consider it in $A_1$ or in $A$. What we mean by this is that being resonant, although it is not a local property, depends only on what happens below a given node. Therefore the nodes in $C_1$ and $C_2$ are resonant and we deduce $C_i \in \mathtt{Res}_{n_i}$ with $n_1 + n_2 +1 = n$. 

We turn to the second part of the lemma and assume $\mathtt{Res}_n(\iota,\iota',\eta,\eta')\neq \emptyset$. We prove by induction on $n$ that $\iota = - \iota'$. If $n=0$, this is due to the definition of a coupling. If $n>0$, let $C \in \mathtt{Res}_n(\iota,\iota',\eta,\eta')$, thanks to the first part of the lemma we are in one of the situations described above. In all cases, due to the sign condition on resonant nodes, we have that $C_2 \in \mathtt{Res}_{n_2}(\iota,\iota',\eta_2,\eta_2')$. This yields by the induction hypothesis that $\iota' = -\iota$.
\end{proof}

\begin{definition} 
Let $C$ be a coupling.
\begin{itemize}
\item[(i)] We say that a node is a left-resonant (resp. middle-resonant, resp right-resonant) node as a resonant node whose unmarked child is its left (resp. middle, resp. right) child. 
\item[(ii)] We define $\mathcal R_n (\iota,-\iota,\eta,\eta') $ to be the set of couplings $C \in \mathtt{Res}_n(\iota,-\iota,\eta,\eta')$ whose nodes are all right-resonant.
\end{itemize}
\end{definition}

The following proposition shows that we can reduce resonant couplings to resonant couplings with only right-resonant nodes.

\begin{prop}\label{prop:reduceRes} 
Let $C\in \mathtt{Res}_n(\iota,-\iota,\eta,\eta')$. 
\begin{itemize}
\item[(i)] If $C$ has a middle-resonant node, then $F_C=0$.
\item[(ii)] If $C$ has no middle-resonant node, there exists $\tilde C \in \mathcal R_n (\iota,-\iota,\eta,\eta') $ such that $F_C = F_{\tilde C}$ and if $C = (A,A',\sigma)$ and $\tilde C = (\tilde A, \tilde A',\tilde \sigma)$, then $n_T(A) = n_T(\tilde A)$.
\item[(iii)] We have
\[
\sum_{C \in \mathtt{Res}_n(\iota,-\iota,\eta,\eta')} F_C = 2^n\sum_{C\in \mathcal R_n (\iota,-\iota,\eta,\eta')} F_C.
\]
\end{itemize}
\end{prop}

\begin{remark}\label{rm:FinStageZero} 
This proposition reduces considerably the amount of resonant nodes that might contribute to the limit. We remark that without the assumption \eqref{assumption Q}, middle-resonant nodes would come into play. In what follows, we explain how to pass from resonant couplings to binary trees, making emerge the recursive structure of resonant couplings. For this, we use an algorithm that takes into account the type of resonant nodes we consider. With a larger number of types of resonant nodes, we get a larger number of cases in the algorithm.
\end{remark}

\begin{proof} 
For the first part of the proposition, simply note that \eqref{assumption Q} implies that if $k_1+k_3=0$, then $q_\iota(k_1,k_2,k_3)=0$ for all $k_2$ and $\iota$. Therefore, if $C$ has a middle-resonant node $j$, then $q_j=0$ and thus $F_C=0$.

For the second part of the proposition, we first note using \eqref{fourier of FC} that if $C = (A,A',\sigma)$ is a coupling then 
\[
\widehat{F_C}(t,k) = \reallywidehat{F_{(A',A,\sigma)}} (t,-k).
\]
Also note that we can rewrite $\widehat{F_C}(t,k)$ with the help of \eqref{FCresonant} as
\[
\widehat{F_C}(t,k) = \tilde F_C(k) \# \mathfrak M(C)\frac{t^n}{n!},
\]
where 
\begin{align*}
 \tilde F_C(k) & \vcentcolon = (-i)^n L^{-nd} \sum_{\kappa \in \mathcal D_k(C)} \prod_{j\in \mathcal N(C)} q_j \prod_{l\in \mathcal L_-(C)} M^{c(l),c(\sigma(l))}(-\kappa(l)).
\end{align*}
We wish to prove the second part of the proposition on $\tilde F_C$, $\# \mathfrak M(C)\frac{t^n}{n!}$ and $n_T(A)$ by induction. If $n=0$ and $C\in \mathtt{Res}_0(\iota,-\iota,\eta,\eta')$ then it only has right-resonant nodes. If now $n>0$ and $C\in \mathtt{Res}_n(\iota,-\iota,\eta,\eta')$ and $C=(A,A',\si)$, we are in the situation described in Lemma \ref{lemma res coupling reduction}. We describe the first three situations, that is, keeping the notations of the lemma,
\begin{align*}
    C = & ((T_\lowr (A_1,A_1',A_2),\iota,\eta),A'_2,\sigma),\\
    C = & ((T_\lowm (A_1,A_2,A_1'),\iota,\eta),A'_2,\sigma),\\
    C = & ((T_{\lowl} (A_2,A_1,A_1'),\iota,\eta),A'_2,\sigma),
\end{align*}
the others we may deduce by symmetry. In these three situations, we have 
\[
\tilde F_C(k) = -iL^{-d} \sum_{k_1 \in \Z^d_L} q_{\mathtt{root}(A)} \tilde F_{C_1}(k_1) \tilde F_{C_2} (k) .
\]
Indeed, $\kappa$ belongs to $\mathcal D_k(C) $ if and only if there exists $k_1\in  \Z^d_L$, $\kappa_1 \in \mathcal D_{k_1}(C_1)$ and $\kappa_2 \in \mathcal D_k(C_2)$ such that for all $\ell \in \mathcal L(C_1)$, we have $\kappa(\ell) = \kappa_1(\ell)$ and for all $\ell \in \mathcal L(C_2)$, we have that $\kappa(\ell) = \kappa_2(\ell)$. We treat each of the three situations separately:

\begin{itemize}
\item As for the first part of the proposition, if we are in the second situation then $q_{\mathtt{root}(A)} = 0$ and thus $F_C = 0$. 
\item We place ourselves in the first situation, that is to say $\mathtt{root}(A)$ is right-resonant. If $C_1$ or $C_2$ has a middle-resonant node then we have either $\tilde F_{C_1}=0$ or $\tilde F_{C_2} =0$ and thus $ F_C = 0$. Otherwise, the induction assumption implies that there exist couplings $\tilde C_1$ and $\tilde C_2$ with only right-resonant nodes such that $\tilde F_{\tilde C_1} = \tilde F_{C_1}  $ and $\tilde F_{\tilde C_2} =\tilde F_{C_2}$ and $\# \mathfrak M(\tilde C_1) = \# \mathfrak M(C_1)$ and $\# \mathfrak M(\tilde C_2) = \# \mathfrak M(C_2)$. We prove that  
\[
\# \mathfrak M(C) = \frac{n!}{n_T(A) n_1! n_2!}\# \mathfrak M(C_1) \# \mathfrak M(C_2).
\]
Indeed, we have that
\[
\# \mathfrak M(C) = \frac{n!}{n_T(A)! n_T(A_2')!} \# \mathfrak M(A) \# \mathfrak M(A_2').
\]
This last equality is due to the fact that an order $\rho$ on $\mathcal N_T(C)$ is obtained from an order $\rho_1$ on $\mathcal N_T(A)$ and an order $\rho_1'$ on $\mathcal{N}_T(A')$ by chosing $n_T(A)$ integers between $1$ and $n$, that is to say a strictly increasing map $\mu$ from $\llbracket 1,n_T(A) \rrbracket$ into $\llbracket1,n \rrbracket$. Setting $\mu'$ the only strictly increasing map from $\llbracket 1,n_T(A') \rrbracket$ into the complementary set in $\llbracket1,n \rrbracket$ of the image of $\mu$, we define $\rho (j) = \mu(\rho_1(j))$ if $j\in \mathcal N_T(A)$ and $\rho(j) = \mu'(\rho_1'(j))$ otherwise. For similar reasons, we have that 
\[
\# \mathfrak M(A) = \frac{(n_T(A)-1)!}{n_T(A_1) ! n_T(A_2)! n_T(A_3)!} \# \mathfrak M(A_1)\# \mathfrak M(A_2)\# \mathfrak M(A_3).
\]
Rearranging the product, we get that
\[
\# \mathfrak M(C) = \frac{n!}{n_T(A) n_1! n_2 !} \frac{n_1!}{n_T(A_1) ! n_T(A_1')! } \# \mathfrak M(A_1)\# \mathfrak M(A_1') \frac{n_2!}{n_T(A_2)! n_T(A_2')!} \# \mathfrak M(A_2) \# \mathfrak M(A'_2).
\]
We recognize
\[
\# \mathfrak M(C) = \frac{n!}{n_T(A) n_1! n_2!}\# \mathfrak M(C_1) \# \mathfrak M(C_2).
\]
Set $\tilde A = (T_\lowr(\tilde A_1,\tilde A_1', \tilde A_2),\iota,\eta)$ where $\tilde C_1 = (\tilde A_1,\tilde A_1', \tilde \sigma_1)$ and $\tilde C_2 = (\tilde A_2,\tilde A_2',\tilde \sigma_2)$ and $\tilde C = (\tilde A,\tilde A_2,\tilde \sigma)$ where $\tilde \sigma$ is such that $\tilde \sigma(\ell) = \tilde \sigma_1(\ell)$ if $\ell \in \mathcal L(\tilde C_1)$ and $\tilde \sigma(\ell) = \tilde \sigma_2(l)$ otherwise. We have that 
\begin{align*}
n_T(\tilde A)  =& 1+n_T(\tilde C_1) + n_T(\tilde A_2) \\
=& 1+n_T(C_1) + n_T(A_2) \\
= & n_T(A).
\end{align*}
Thanks to the induction formula on $\#\mathfrak M(C)$, we get $\# \mathfrak M(C) = \# \mathfrak M(\tilde C)$. Moreover, we have $q_{\mathtt{root}(A)} = q_{\mathtt{root}(\tilde A)}$ since
\begin{align*}
    q_{\mathtt{root}(A)} = & \chi_\lowr (k_1, -k_1 ,k) q_{\iota \varphi(\mathtt{root}(A_1'))}^{\eta}(k_1,-k_1,k),\\
    q_{\mathtt{root}(\tilde A)} = & \chi_\lowr (k_1, -k_1 ,k) q_{\iota \varphi(\mathtt{root}(\tilde A_1'))}^{\eta}(k_1,-k_1,k).
\end{align*}
By the induction assumption, we also have that $ \varphi(\mathtt{root}(\tilde A_1')) =  \varphi(\mathtt{root}( A_1'))$, so that we have finally proved that
\[
\tilde F_{\tilde C} = \tilde F_{C}.
\]

It remains to check that $\tilde C \in \mathcal R_n$. First, we check that it is indeed a coupling in the sense that it satisfies the colour condition. We have that $c(\mathtt{root}(\tilde A)) = \eta = c(\mathtt{root}(A))$. We deduce that 
\[
(c(\mathtt{root}(A_1)), c(\mathtt{root}(A_1')),c(\mathtt{root}(A_2))) = (\eta, \bar \eta, \eta).
\]
Therefore, $(c(\mathtt{root}(\tilde A_1)), c(\mathtt{root}(\tilde A_1')),c(\mathtt{root}(\tilde A_2))) = (\eta, \bar \eta, \eta)$ which yields that $\tilde C$ is indeed a coupling. We check that $\mathtt{root}(\tilde A)$ is a right-resonant node. By definition, it is a $\lowr$-ternary node. What is more, its offsprings set is self-coupled, which, knowing that all the bushes below are self-coupled, implies that the bush of $\mathtt{root}(\tilde A)$ is self-coupled. The sign of $\mathtt{root}(\tilde A)$ is $\iota$, so is the sign of $\mathtt{root}(A)$. We deduce that the sign of $\mathtt{root}(A_2)$ is $\iota$ and thus so is the sign of $\mathtt{root}(\tilde A_2)$ which makes $\mathtt{root}(\tilde A)$ a right-resonant node. 

\item Assume now that we are in the third situation that is to say that $\mathtt{root}(A)$ is left-resonant. If $C_1$ or $C_2$ has a middle-resonant node, then we have either $\tilde F_{C_1}=0$ or $\tilde F_{C_2} =0$ and thus $\tilde F_C = 0$. Otherwise, there exist couplings $\tilde C_1$ and $\tilde C_2$ with only right-resonant nodes such that $\tilde F_{\tilde C_1} = \tilde F_{C_1}  $ and $\tilde F_{\tilde C_2} =\tilde F_{C_2}$ and $\# \mathfrak M(\tilde C_1) = \# \mathfrak M(C_1)$ and $\# \mathfrak M(\tilde C_2) = \# \mathfrak M(C_2)$. 

Set $\tilde A = (T_\lowr( \tilde A_1', \tilde A_1, \tilde A_2),\iota,\eta)$ where $\tilde C_1 = ( \tilde A_1,\tilde A_1',\tilde \sigma_1)$ and $\tilde C_2 = (\tilde A_2, \tilde A_2',\tilde \sigma_2)$ and $\tilde C = (\tilde A,\tilde A_2',\tilde \sigma)$ where $\tilde \sigma$ is such that $\tilde \sigma(\ell) = \tilde \sigma_1(\ell)$ if $\ell \in \mathcal L(\tilde C_1)$ and $\tilde \sigma(\ell) = \tilde \sigma_2(\ell)$ otherwise. As above we have that $n_T(\tilde A)=n_T(A)$. Thanks to the induction formula on $\#\mathfrak M(C)$, we also get $\# \mathfrak M(C) = \# \mathfrak M(\tilde C)$. We also have that
\[
\tilde F_C(k) = -iL^{-d} \sum_{k_1 \in \Z^d_L} q_{\mathtt{root}(A)} \tilde F_{\tilde C_1}(k_1) \tilde F_{\tilde C_2} (k) 
\]
while
\[
\tilde F_{\tilde C}(k) = -iL^{-d} \sum_{k_1 \in \Z^d_L} q_{\mathtt{root}(\tilde A)} \tilde F_{\tilde C_3}(k_1) \tilde F_{\tilde C_2} (k) 
\]
where $\tilde C_3 = (\tilde A_1',\tilde A_1, \tilde \sigma_1)$. We have  
\[
q_{\mathtt{root}(\tilde A)} = \iota \chi_\lowr (k_1,-k_1,k) q_{\iota \varphi(\mathtt{root}(\tilde A_1))}^{\eta} (k_1,-k_1,k),
\]
where by the induction hypothesis $\varphi(\mathtt{root}(\tilde A_1)) = \varphi(\mathtt{root}( A_1))$. What is more,
\[
q_{\mathtt{root}( A)} = \iota \chi_{\lowl} (k,k_1,-k_1) q_{\iota \varphi(\mathtt{root}( A_1))}^{\eta} (k,k_1,-k_1).
\]
Using that $\tilde F_{\tilde C_1}(k_1) =\tilde F_{\tilde C_3}(-k_1)$, we get that
\[
\tilde F_{\tilde C}(k) = -iL^{-d} \sum_{k_1 \in \Z^d_L} \iota \chi_\lowr (k_1,-k_1,k) q_{\iota \varphi(\mathtt{root}( A_1))}^{\eta} (k_1,-k_1,k) \tilde F_{\tilde C_1}(-k_1) \tilde F_{\tilde C_2} (k) 
\]
We use that ternary $q$ is symmetric in its left and right variables and that 
\[
\chi_\lowr (k_1,-k_1,k) =  (1-\chi) (|k+k_1|\varepsilon^{-\gamma}) (1-\chi )(|k-k_1|\varepsilon^{-\gamma}) = \chi_{\lowl} (k,-k_1,k_1)
\]
to obtain
\[
\tilde F_{\tilde C}(k) = -iL^{-d} \sum_{k_1 \in \Z^d_L} \iota \chi_{\lowl} (k,-k_1,k_1) q_{\iota \varphi(\mathtt{root}( A_1))}^{\eta} (k,-k_1,k_1) \tilde F_{\tilde C_1}(-k_1) \tilde F_{\tilde C_2} (k) 
\]
We perform the change of variables $k_1\leftarrow -k_1$ to get the result. Checking that $\tilde C \in \mathcal R_n$ is done as in the first situation.
\end{itemize}

We now prove the third part of the proposition. We pass from $\mathtt{Res}_n$ to $\mathcal R_n$ by simply transforming right-resonant nodes into left-resonant nodes. Therefore a given $C \in \mathcal R_n$ corresponds to $2^n$ couplings in $\mathtt{Res}_n$ with no middle-resonant node. We obtain these $2^n$ couplings by deciding at each node if we do the reverse operation than the one we have described, that is transforming a left-resonant node into a right-resonant node.
\end{proof}

\begin{example}
We illustrate the transformation we perform in the above proof in the situations 
\begin{align*}
    C = & ((T_\lowr (A_1,A_1',A_2),\iota,\eta),A'_2,\sigma),\\
    C = & ((T_{\lowl} (A_2,A_1,A_1'),\iota,\eta),A'_2,\sigma),\\
    C = & (A_2, (T_\lowr(A_1,A_1',A_2') \iota, \eta),\sigma), \\
    C = & (A_2, (T_\lowr(A_2',A_1,A_1') \iota, \eta),\sigma) .
\end{align*}

\begin{tikzpicture}
\node[label=above:{$\iota$}]{$\textcolor{red}{\lowr}$}
child{node (A1) [label=below:{$A_1$}]{$\textcolor{red}{\bullet}$} }
child{node (A1') [label=below:{$A_1'$}]{$\textcolor{Dandelion}{\bullet}$} }
child{node (A2) [label=below:{$A_2,\iota$}]{$\textcolor{red}{\bullet}$} }
;
\draw (3.5,0) node (A2') [label={[xshift=0.4cm, yshift = -1cm]$A_2',-\iota$}] {$\textcolor{Dandelion}{\bullet}$} ; 
\path (A1) edge [bend right,<->,blue] (A1') ; 
\path (A2) edge [bend right,<->,blue] (A2') ;
\draw (4.5,-1) node {$\rightarrowtail$} ;
\node at (7.5,0) [label=above:{$\iota$}]{$\textcolor{red}{\lowr}$}
child{node (B1) [label=below:{$\tilde A_1$}]{$\textcolor{red}{\bullet}$} }
child{node (B1') [label=below:{$\tilde A_1'$}]{$\textcolor{Dandelion}{\bullet}$} }
child{node (B2) [label=below:{$\tilde A_2,\iota$}]{$\textcolor{red}{\bullet}$} }
;
\draw (10.5,0) node (B2') [label={[xshift=0.4cm, yshift = -1cm]$\tilde A_2',-\iota$}] {$\textcolor{Dandelion}{\bullet}$} ; 
\path (B1) edge [bend right,<->,blue] (B1') ; 
\path (B2) edge [bend right,<->,blue] (B2') ;
\end{tikzpicture}

\begin{tikzpicture}
\node[label=above:{$\iota$}]{$\textcolor{red}{\lowl}$}
child{node (A2) [label={[xshift=-0.2cm, yshift = -1cm]$A_2,\iota$}]{$\textcolor{red}{\bullet}$} }
child{node (A1) [label=below:{$A_1$}]{$\textcolor{Dandelion}{\bullet}$} }
child{node (A1') [label=below:{$A_1'$}]{$\textcolor{red}{\bullet}$} }
;
\draw (3.5,0) node (A2') [label={[xshift=0.6cm, yshift = -1cm]$A_2',-\iota$}] {$\textcolor{Dandelion}{\bullet}$} ; 
\path (A1) edge [bend right,<->,blue] (A1') ; 
\path (A2) edge [bend right=90,<->,blue] (A2') ;
\draw (5,-1) node {$\rightarrowtail$} ;
\node at (7,0) [label=above:{$\iota$}]{$\textcolor{red}{\lowr}$}
child{node (B1') [label=below:{$\tilde A_1'$}]{$\textcolor{red}{\bullet}$} }
child{node (B1) [label=below:{$\tilde A_1$}]{$\textcolor{Dandelion}{\bullet}$} }
child{node (B2) [label=below:{$\tilde A_2,\iota$}]{$\textcolor{red}{\bullet}$} }
;
\draw (10.5,0) node (B2') [label={[xshift=0.4cm, yshift = -1cm]$\tilde A_2',-\iota$}] {$\textcolor{Dandelion}{\bullet}$} ; 
\path (B1') edge [bend right,<->,blue] (B1) ; 
\path (B2) edge [bend right,<->,blue] (B2') ;
\end{tikzpicture}

\begin{tikzpicture}
\node at (3,0) [label=above:{$-\iota$}]{$\textcolor{red}{\lowr}$}
child{node (A1) [label=below:{$A_1$}]{$\textcolor{red}{\bullet}$} }
child{node (A1') [label=below:{$A_1'$}]{$\textcolor{Dandelion}{\bullet}$} }
child{node (A2') [label={[xshift=0.4cm, yshift = -1cm]$A_2',-\iota$}]{$\textcolor{red}{\bullet}$} }
;
\draw (0,0) node (A2) [label={[xshift=-0.5cm, yshift = -1cm]$A_2,\iota$}] {$\textcolor{Dandelion}{\bullet}$} ; 
\path (A1) edge [bend right,<->,blue] (A1') ; 
\path (A2) edge [bend right=90,<->,blue] (A2') ;
\draw (6,-1) node {$\rightarrowtail$} ;
\node at (10.5,0) [label=above:{$-\iota$}]{$\textcolor{red}{\lowr}$}
child{node (B1) [label=below:{$\tilde A_1$}]{$\textcolor{red}{\bullet}$} }
child{node (B1') [label=below:{$\tilde A_1'$}]{$\textcolor{Dandelion}{\bullet}$} }
child{node (B2') [label={[xshift=0.4cm, yshift = -1cm]$\tilde A_2',-\iota$}]{$\textcolor{red}{\bullet}$} }
;
\draw (7.5,0) node (B2) [label={[xshift=-0.6cm, yshift = -1cm]$\tilde A_2,\iota$}] {$\textcolor{Dandelion}{\bullet}$} ; 
\path (B1) edge [bend right,<->,blue] (B1') ; 
\path (B2) edge [bend right=90,<->,blue] (B2') ;
\end{tikzpicture}

\begin{tikzpicture}
\node at (3,0) [label=above:{$-\iota$}]{$\textcolor{red}{\lowl}$}
child{node (A2') [label=below:{$A_2',-\iota$}]{$\textcolor{red}{\bullet}$} }
child{node (A1) [label=below:{$A_1$}]{$\textcolor{Dandelion}{\bullet}$} }
child{node (A1') [label=below:{$A_1'$}]{$\textcolor{red}{\bullet}$} }
;
\draw (0,0) node (A2) [label={[xshift=-0.4cm, yshift = -1cm]$A_2,\iota$}] {$\textcolor{Dandelion}{\bullet}$} ; 
\path (A1) edge [bend right,<->,blue] (A1') ; 
\path (A2) edge [bend right,<->,blue] (A2') ;
\draw (6,-1) node {$\rightarrowtail$} ;
\node at (10.5,0) [label=above:{$-\iota$}]{$\textcolor{red}{\lowr}$}
child{node (B1) [label=below:{$\tilde A_1'$}]{$\textcolor{red}{\bullet}$} }
child{node (B1') [label=below:{$\tilde A_1$}]{$\textcolor{Dandelion}{\bullet}$} }
child{node (B2') [label={[xshift=0.4cm, yshift = -1cm]$\tilde A_2',-\iota$}]{$\textcolor{red}{\bullet}$} }
;
\draw (7.5,0) node (B2) [label={[xshift=-0.6cm, yshift = -1cm]$\tilde A_2,\iota$}] {$\textcolor{Dandelion}{\bullet}$} ; 
\path (B1) edge [bend right,<->,blue] (B1') ; 
\path (B2) edge [bend right=90,<->,blue] (B2') ;
\end{tikzpicture}

\end{example}

\subsubsection{The maps $\mathtt{BtoR}$ and $\mathtt{RtoB}$}\label{section btor rtob}

In the rest of this section, we will relate resonant couplings with only right-resonant nodes to binary trees.

\begin{definition}
Let $n\in\mathbb{N}^*$.
\begin{itemize}
\item We set $\mathcal R_n^{\mathtt{map}}$ the set of couples $(R,\rho)$ such that $R \in \mathcal R_n$ and $\rho$ is a striclty increasing injective map from $\mathcal N_T(R)$ to $\N^*$.
\item We also set $\mathcal R_n^{\mathtt{ord}}$ the set of couples $(R,\rho)$ such that $R \in \mathcal R_n$ and $\rho$ is a striclty increasing injective map from $\mathcal N_T(R)$ to $\llbracket 1,n_T(R) \rrbracket$.
\end{itemize}
\end{definition}

Our goal is to show that $\mathcal R_n^{\mathtt{ord}}$ is in bijection with some set of binary trees with ordered nodes. For that we will use $\mathcal R_n^{\mathtt{map}}$ as an intermediary. First, we introduce our new binary trees.

\begin{definition}
We introduce the following objects.
\begin{itemize}
\item[(i)] We define the set $\tilde{\mathcal B}_n$ of binary trees by induction as follows: we set $\tilde{\mathcal{B}}_0 = \{\perp\}$ and 
\[
\tilde{\mathcal{B}}_{n+1} = \enstq{ N_\lef(B_1,B_2), N_\righ(B_1,B_2) }{ (B_1,B_2)\in \tilde{\mathcal B}_{n_1}\times \tilde{\mathcal B}_{n_2}, \;n_1+n_2=n  },
\]
where $N_\lef$ and $N_\righ$ are new binary nodes. We denote by $\mathcal N_l(B)$ (resp. $\mathcal N_r(B)$) the set of nodes of type $N_\lef$ (resp. $N_\righ$).
\item[(ii)] If $B\in \tilde{\mathcal{B}}_n$, a sign map on $B$ is a map $\mathcal N(B) \sqcup \mathcal L(B)\longrightarrow\{\pm\}$ such that for any node, the sign of its left child matches its own.
\item[(iii)] If $B\in \tilde{\mathcal{B}}_n$, a colour map on $B$ is a map $c:\mathcal N(B) \sqcup \mathcal L(B)\longrightarrow\{0,1\}^2$ satisfying the following recursive rule: if $j\in\mathcal{N}(B)$ with left and right children denoted by $\ell$ and $r$, the colours of $\ell$ and $r$ can be read off in the following table (depending on the nature of $j$ and its colour):
\begin{equation}\label{eq:colortable}
\begin{array}{|c|c|c|c|c|} \hline
 & c(j) = {\bf 0} & c(j) = {\bf 1} & c(j) = \times & c(j) = \overline \times \\\hline
j \in \mathcal N_l(B) & ({\bf 0},\times) & ({\bf 1},\overline \times) & (\times, \times) & (\overline \times , \overline \times) \\\hline
j \in \mathcal N_r(B) & ({\bf 0},\times) & ({\bf 1},\overline \times) & (\times, \overline \times) & (\overline \times , \times) \\\hline
\end{array}
\end{equation}
where we use the notations ${\bf 0} = (0,0)$, ${\bf 1} = (1,1)$, $\times =(0,1)$ and $\overline{\times} = (1,0)$.
\end{itemize}
We denote by $\mathcal B_n$ the set of triplets $(B,\varphi,c)$ where $B \in \tilde{\mathcal{B}}_n$, $\ffi$ is a sign map on $B$ and $c$ a colour map on $B$.
\end{definition}

\begin{definition}
Let $n\in\mathbb{N}^*$.
\begin{itemize}
\item We set $\mathcal B_n^{\mathtt{map}}$ the set of couples $((B,\ffi,c),\rho)$ such that $(B,\ffi,c) \in \mathcal B_n$ and $\rho$ is a striclty increasing injective map from $\mathcal N(B)$ to $\N^*$.
\item We also set $\mathcal B_n^{\mathtt{ord}}$ the set of couples $((B,\ffi,c),\rho)$ such that $(B,\ffi,c) \in \mathcal B_n$ and $\rho$ is a striclty increasing injective map from $\mathcal N(B)$ to $\llbracket1,n\rrbracket$.
\end{itemize}
\end{definition}

We apply Remark \ref{rm:labellednodes} to $\mathcal B_n^{\mathtt{map}}$ in the following sense. If $((B,\ffi,c),\rho) \in \mathcal B_n^{\mathtt{map}}$ with $n>0$, then the tree $B$ writes $B = \pth{N_*\pth{\tilde B_1,\tilde B_2}, \varphi, c, \rho}$. We allow ourselves to also write $B = (N_*(B_1,B_2), \iota,E,M)$ with for $i = 1,2$, $B_i = \pth{\tilde B_i, \varphi_{|\mathcal N(B_i) \cup \mathcal L(B_i)}, c_{|\mathcal N(B_i) \cup \mathcal L(B_i)}, \rho_{|\mathcal N(B_i) }}$ and $(\iota, E, M) = (\varphi, c, \rho)(\mathtt{root}(B))$.

Moreover, we set $\mathcal C_{n_1,n_2}^{\mathtt{map}}$ (res. $\mathcal A_n^{\mathtt{map}}$) the set of couples $(C,\rho)$ (resp. $(A,\rho)$ such that $C \in \mathcal C_{n_1,n_2}$ (resp. $A \in \mathcal A_n$) and $\rho$ is a strictly increasing injective map from $\mathcal N(C)$ (resp. $\mathcal N(A)$) to $\N$. We allow ourselves to write $(A,\rho) \in \mathcal A_n^{\mathtt{map}}$ with $n>0$ as 
\[
(A,\rho) = T_*((A_1,\rho_1),(A_2,\rho_2),(A_3,\rho_3)),\iota, \eta, M)
\]
(resp. $B_*((A_1,\rho_1),(A_2,\rho_2)),\iota, \eta, M)$)  where $A = (T_*(A_1,A_2,A_3),\iota,\eta)$ (resp. $(B_*(A_1,A_2),\iota,\eta)$, where $\rho_i = \rho_{|\mathcal N(A_i)}$ and $M = \rho(\mathtt{root}(A))$. We remark that the images of $\rho_i$ do not intersect. If $(C,\rho) \in \mathcal C_{n_1,n_2}^{\mathtt{map}}$, we allow ourselves to write $(C,\rho)$ as $(C,\rho) = ((A,\rho_{|\mathcal N(A)}), (A', \rho_{|\mathcal N(A')}), \sigma)$ where $C = (A,A',\sigma)$.

\begin{definition}\label{def BtoR}
We define the function 
\begin{align*}
\mathtt{BtoR} : \bigcup_{n\in\mathbb{N}} \mathcal B_n^{\mathtt{map}} \longrightarrow \bigcup_{n_1,n_2\in\mathbb{N}} \mathcal C_{n_1,n_2}^{\mathtt{map}}
\end{align*}
recursively in the following way:
\begin{itemize}
\item If $n=0$, then an element $B$ of $\mathcal B_0$ is a triplet $(\perp, \iota, E)$ where $\iota \in \{\pm\}$ and $E \in \{0,1\}^2$. In that case $\mathtt{BtoR}(B)$ is the coupling 
\[
((\bot, \iota, \eta), (\bot, -\iota, \eta'),\sigma)
\]
where $(\eta,\eta') = E$ and $\sigma$ is the only function possible.
\item If $n>0$ then $B \in \mathcal B_n^{\mathtt{map}}$ writes $B = (N_* (B_1,B_2), \iota, E, M)$ with $* \in \{\lef,\righ\}$ and $B_i \in \mathcal{B}_{n_i}^{\mathtt{map}}$ with $n_1+ n_2 = n-1$. Define $C_i \vcentcolon = \mathtt{BtoR}(B_i)$ and write $C_i = (A_i, A'_i,\sigma_i)$ and $E = (\eta,\eta')$:
\begin{itemize}
\item If $* = \lef$, then we introduce $A = (N_\lowr(A_2,A_2',A_1),\iota,\eta,M)$. We set $C = (A, A'_1,\sigma)$ where $\sigma$ is defined as 
\[
\sigma(l) = \left \lbrace{\begin{array}{cc}
\sigma_1(l)  &\mathrm{if} \, l \in \mathcal L(C_1)\\
\sigma_2(l) &\mathrm{otherwise .}
\end{array}} \right.
\] 
\item If $* = \righ$, then we introduce $ A' = (N_\lowr ( A_2,  A'_2, A'_1), -\iota,\eta',M)$. We set $C = (A_1, A',\sigma)$ where $\sigma$ is defined as 
\[
\sigma(l) = \left \lbrace{\begin{array}{cc}
\sigma_1(l)  &\mathrm{if} \, l \in \mathcal L(C_1)\\
\sigma_2(l) &\mathrm{otherwise.}
\end{array}} \right.
\] 
\end{itemize}
\end{itemize}
\end{definition}

\begin{example}
We draw the transformation that $\mathtt{BtoR}$ performs on a few examples.
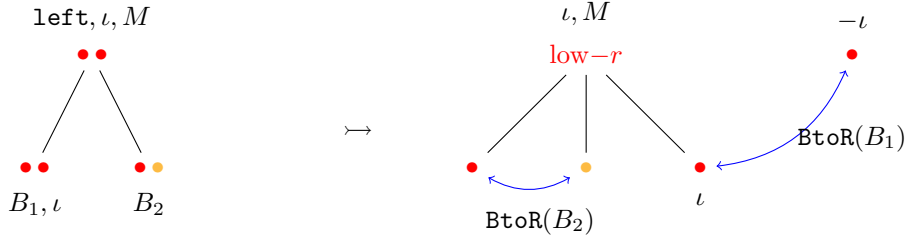
\begin{center}
\begin{tikzpicture}
\node[label=above:{$\lef,\iota,M$}]{$\textcolor{red}{\bullet \, \bullet}$}
child{node [label=below:{$B_1,\iota$}] {\textcolor{red}{$\bullet$}\,\textcolor{red}{$\bullet$}}}
child{node [label=below:{$B_2$}] {$\textcolor{red}{\bullet} \,\textcolor{Dandelion}{ \bullet}$}}
;
\draw (3.5,-1) node {$\rightarrowtail$} ;
\node at (6.5,0) [label=above:{$\iota,M$}]{$\textcolor{red}{\lowr}$}
child{node (B1) {$\textcolor{red}{\bullet}$} }
child{node (B1') {$\textcolor{Dandelion}{\bullet}$} }
child{node (B2) [label=below:{$\iota$}]{$\textcolor{red}{\bullet}$} }
;
\draw (10,0) node (B2') [label=above:{$-\iota$}] {$\textcolor{red}{\bullet}$} ; 
\path (B1) edge [bend right,<->,blue] (B1') ; 
\path (B2) edge [bend right,<->,blue] (B2') ;
\draw (5.9,-2.2) node {$\mathtt{BtoR}(B_2)$} ;
\draw (10,-1.1) node {$\mathtt{BtoR}(B_1)$} ;
\end{tikzpicture}
\end{center}

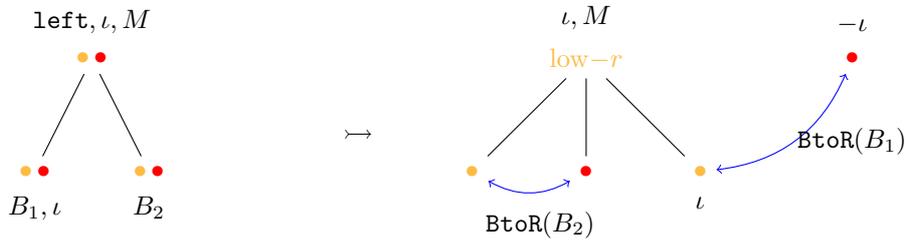
\begin{center}
\begin{tikzpicture}
\node[label=above:{$\lef,\iota,M$}]{$\textcolor{Dandelion}{\bullet} \, \textcolor{red}{\bullet}$}
child{node [label=below:{$B_1,\iota$}] {\textcolor{Dandelion}{$\bullet$}\,\textcolor{red}{$\bullet$}}}
child{node [label=below:{$B_2$}] {$\textcolor{Dandelion}{\bullet} \,\textcolor{red}{ \bullet}$}}
;
\draw (3.5,-1) node {$\rightarrowtail$} ;
\node at (6.5,0) [label=above:{$\iota,M$}]{$\textcolor{Dandelion}{\lowr}$}
child{node (B1) {$\textcolor{Dandelion}{\bullet}$} }
child{node (B1') {$\textcolor{red}{\bullet}$} }
child{node (B2) [label=below:{$\iota$}]{$\textcolor{Dandelion}{\bullet}$} }
;
\draw (10,0) node (B2') [label=above:{$-\iota$}] {$\textcolor{red}{\bullet}$} ; 
\path (B1) edge [bend right,<->,blue] (B1') ; 
\path (B2) edge [bend right,<->,blue] (B2') ;
\draw (5.9,-2.2) node {$\mathtt{BtoR}(B_2)$} ;
\draw (10,-1.1) node {$\mathtt{BtoR}(B_1)$} ;
\end{tikzpicture}
\end{center}

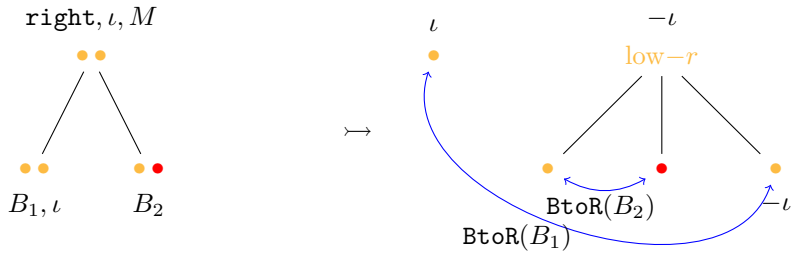
\begin{center}
\begin{tikzpicture}
\node[label=above:{$\righ,\iota,M$}]{$\textcolor{Dandelion}{\bullet} \, \textcolor{Dandelion}{\bullet}$}
child{node [label=below:{$B_1,\iota$}] {\textcolor{Dandelion}{$\bullet$}\,\textcolor{Dandelion}{$\bullet$}}}
child{node [label=below:{$B_2$}] {$\textcolor{Dandelion}{\bullet} \,\textcolor{red}{ \bullet}$}}
;
\draw (3.5,-1) node {$\rightarrowtail$} ;
\node at (7.5,0) [label=above:{$-\iota$}]{$\textcolor{Dandelion}{\lowr}$}
child{node (B1) {$\textcolor{Dandelion}{\bullet}$} }
child{node (B1') {$\textcolor{red}{\bullet}$} }
child{node (B2') [label=below:{$-\iota$}]{$\textcolor{Dandelion}{\bullet}$} }
;
\draw (4.5,0) node (B2) [label=above:{$\iota$}] {$\textcolor{Dandelion}{\bullet}$} ; 
\path (B1) edge [bend right,<->,blue] (B1') ; 
\path (B2) edge [bend right=90,<->,blue] (B2') ;
\draw (6.7,-2) node {$\mathtt{BtoR}(B_2)$} ;
\draw (5.6,-2.4) node {$\mathtt{BtoR}(B_1)$} ;
\end{tikzpicture}
\end{center}

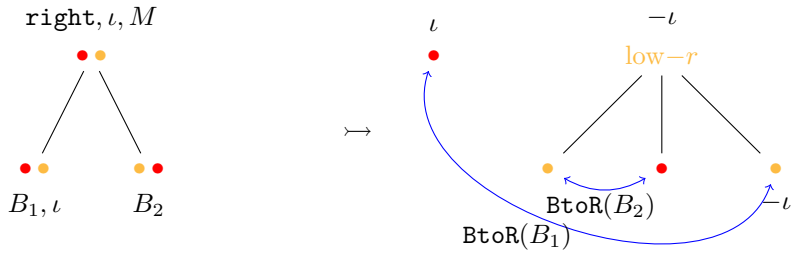
\begin{center}
\begin{tikzpicture}
\node[label=above:{$\righ,\iota,M$}]{$\textcolor{red}{\bullet} \, \textcolor{Dandelion}{\bullet}$}
child{node [label=below:{$B_1,\iota$}] {\textcolor{red}{$\bullet$}\,\textcolor{Dandelion}{$\bullet$}}}
child{node [label=below:{$B_2$}] {$\textcolor{Dandelion}{\bullet} \,\textcolor{red}{ \bullet}$}}
;
\draw (3.5,-1) node {$\rightarrowtail$} ;
\node at (7.5,0) [label=above:{$-\iota$}]{$\textcolor{Dandelion}{\lowr}$}
child{node (B1) {$\textcolor{Dandelion}{\bullet}$} }
child{node (B1') {$\textcolor{red}{\bullet}$} }
child{node (B2') [label=below:{$-\iota$}]{$\textcolor{Dandelion}{\bullet}$} }
;
\draw (4.5,0) node (B2) [label=above:{$\iota$}] {$\textcolor{red}{\bullet}$} ; 
\path (B1) edge [bend right,<->,blue] (B1') ; 
\path (B2) edge [bend right=90,<->,blue] (B2') ;
\draw (6.7,-2) node {$\mathtt{BtoR}(B_2)$} ;
\draw (5.6,-2.4) node {$\mathtt{BtoR}(B_1)$} ;
\end{tikzpicture}
\end{center}
\end{example}

\begin{definition} 
We define the following sets:
\begin{itemize}
\item We define $\mathcal{R}^{\mathtt{map}}_n(\iota,\eta,\eta',M)$ to be the set of couples $(R,\rho)$ such that $(R,\rho) \in \mathcal R_n^{\mathtt{map}}$ with the sign of the root of the left tree of $R$ being $\iota$ and the colour of the root of the left tree of $R$ is $\eta$ and the colour of the root of its right tree is $\eta'$ and $\max \mathrm{Im}\rho = M$. 
\item We define $\mathcal B_n^{\mathtt{map}}(\iota, E,M)$ to be the set of couples $(B,\rho) \in \mathcal B_n^{\mathtt{map}}$ such that the colour of the root of $B$ is $E$, its sign is $\iota$ and $\max \mathrm{Im}\rho = M$.
\end{itemize}
By convention, we set $\mathcal B^{\mathtt{map}}_0(\iota, E,M) = \{(\bot,\iota,E)\}$ and 
\[
\mathcal{R}^{\mathtt{map}}_0(\iota,\eta,\eta',M) = \{((\bot,\iota,\eta),(\bot,-\iota,\eta'),\sigma)\}
\]
where $\sigma$ is trivial.
\end{definition}

At this stage, it is not clear that $\mathtt{BtoR}$ maps indeed $\cup_n \mathcal B_n^{\mathtt{map}}$ to couplings because of the colour condition and the properties that the maps $\rho$ need to satisfy (namely that it is increasing and injective). This the purpose of the next proposition.

\begin{prop}\label{prop:BtoR} 
The image of $\mathcal{B}^{\mathtt{map}}_n(\iota, E,M)$ by $\mathtt{BtoR}$ is included in $\mathcal{R}^{\mathtt{map}}_n(\iota,\eta,\eta',M)$ where $(\eta,\eta') = E$.
\end{prop}

\begin{proof} We proceed by induction on $n$. What we prove by induction is not only the proposition but also the fact that for all $(B,\rho_B) \in \mathcal{B}^{\mathtt{map}}_n(\iota, E,M)$, writing $(R,\rho_R) = \mathtt{BtoR}((B,\rho_B))$, we have that $\mathrm{Im}\rho_R = \mathrm{Im}\rho_B$. If $n = 0$, this is due to the definition of $\mathtt{BtoR}$ and the fact that $\mathrm{Im}\rho_B=\mathrm{Im}\rho_R = \emptyset$. 

If $n>0$, we keep the same notations as in Definition \ref{def BtoR}. We deal with the case $*=\lef$.  We assume that $B \in \mathcal B_n^{\mathtt{map}}(\iota, E,M)$ and $B_i \in \mathcal B^{\mathtt{map}}_{n_i}(\iota_i, E_i,M_i)$. By induction, we have that
$C_i \in \mathcal{R}^{\mathtt{map}}_{n_i}(\iota_i,\eta_i,\eta'_i,M_i)$ where $(\eta_i,\eta'_i) = E_i$. 

First, we get automatically that the number of nodes of $C$ is $n$. Then we have to check that the condition of colour is satisfied by the root $r$ of $A$. We do this in the case $\eta=0$. By definition, the colour of $r$ is $\eta$. In this case, the colour $E$ is either $(0,0)$ or $\times$. This implies that $E_2=\times = (0,1)$ and $E_1$ is either $(0,0)$ or $\times$. In any case, we have that the colour of the root of $A_2$ is $0$, the one of $A_2'$ is $1$ and the one of $A_1$ is $0$, which matches the colour condition on trees. The colour of the root of $A_2'$ is either $0$ or $1$. This yields that $C \in \mathcal C_{m_1,m_2}^{\mathtt{map}}(\iota,\iota',\eta,\eta')$ for some $\iota',m_1,m_2$.

Because $C_2$ is a coupling, the set $\mathtt{offsprings}(r)$ is self-coupled. Using that all the nodes below $r$ are resonant, and thus their bushes are self-coupled, this implies that the bush of $r$ is self-coupled. By definition of $\mathcal B_n$, we have that $\iota_1 = \iota$, which implies that the sign of the root of $A_1$ is $\iota$. Therefore, $r$ is resonant, and $\iota' = - \iota$.

Finally, $\rho_i$ respects the order of parentality and the images of $\rho_1$ and $\rho_2$ are disjoint (because their images are the sames as the images of the maps in $B_1,B_2$). For all $j<r$, we have that $\rho(j) \leq \max (M_1,M_2) < M$. This ensures that $\rho$ respects the order of parentality, is injective and, setting $B = (B',\rho_B)$, we have that $\mathrm{Im}\rho = \mathrm{Im}\rho_B$. Therefore, we have $C \in  \mathcal{R}_{n}^{\mathtt{map}}(\iota,\eta,\eta',M)$. 

The cases $\eta = 1$ and $* = \righ$ follows the same lines. The condition colour on the binary trees $\mathcal B_n$ is made to ensure that the image of $\mathtt{BtoR}$ is indeed included in couplings, and the sign condition is made to match the sign condition of resonant nodes.
\end{proof}


We now define a map which will be ultimately proved to be the inverse of $\mathtt{BtoR}$.

\begin{definition}
We define the function 
\begin{align*}
\mathtt{RtoB} : \bigcup_{n\in\mathbb{N}} \mathcal R_n^{\mathtt{map}} \longrightarrow \bigcup_{n\in\mathbb{N}} \mathcal B_{n}^{\mathtt{map}}
\end{align*}
recursively in the following way:
\begin{itemize}
\item If $n=0$, then an element $R$ of $\mathcal R_0$ is of the form $((\bot, \iota, \eta), (\bot, -\iota, \eta'),\sigma)$ where $\sigma$ is the trivial coupling, where $\iota \in \{\pm\}$ and $(\eta,\eta') \in \{0,1\}^2$. In that case $\mathtt{RtoB}(R)$ is the tree $(\bot,\iota,(\eta,\eta'))$. 
\item If $n>0$ then $R \in \mathcal R_n^{\mathtt{map}}$ writes $R = ((A,A',\sigma),\rho)$. 
\begin{itemize}
\item If $\rho(\mathtt{root}(A)) = \max \mathrm{Im}\rho$, then $C=(A,A',\sigma)$ writes $C =  ((T_\lowr (A_1,A_1',A_2),\iota,\eta),A'_2,\sigma)$ such that $C_i \vcentcolon= (A_i,A_i',\sigma_{|\mathcal L(A_i) \sqcup \mathcal L(A_i')}) \in \mathcal R_{n_i}$ with $n_1 + n_2 = n-1$. We have that $(C_i,\rho_{|\mathcal N(C_i)}) \in \mathcal R^{\mathtt{map}}_{n_i}$. We set $B_i = \mathtt{RtoB}((C_i,\rho_{|\mathcal N(C_i)}))$ and 
\[
\mathtt{RtoB}((C,\rho)) = (N_\lef (B_2,B_1),\iota, (\eta,\eta'),\max \mathrm{Im}\rho)
\]
where $\eta' = c(\mathtt{root}(A'))$. 
\item If $\rho(\mathtt{root}(A')) = \max \mathrm{Im}\rho$, then $C=(A,A',\sigma)$ writes $C =  (A_2,(T_\lowr (A_1,A_1',A_2'),-\iota,\eta'),\sigma)$ such that $C_i \vcentcolon= (A_i,A_i',\sigma_{|\mathcal L(A_i) \sqcup \mathcal L(A_i')}) \in \mathcal R_{n_i}$ with $n_1 + n_2 = n-1$. We have that $(C_i,\rho_{|\mathcal N(C_i)}) \in \mathcal R^{\mathtt{map}}_{n_i}$. We set $B_i = \mathtt{RtoB}((C_i,\rho_{|\mathcal N(C_i)}))$ and 
\[
\mathtt{RtoB}((C,\rho)) = (N_\righ (B_2,B_1),\iota, (\eta,\eta'),\max \mathrm{Im}\rho)
\]
where $\eta = c(\mathtt{root}(A))$. 
\end{itemize}
\end{itemize}
\end{definition}

Note that in the previous definition, we used implicitly the fact that if $n>0$ and $R = ((A,A',\sigma),\rho)\in \mathcal R_n^{\mathtt{map}}$ then either $\rho(\mathtt{root}(A)) = \max \mathrm{Im}\rho$ or $\rho(\mathtt{root}(A')) = \max \mathrm{Im}\rho$. This is true since $\mathrm{Im}\rho \subset \N$ is finite and not empty (because $n>0$) and thus $\max \mathrm{Im}\rho$ is well-defined, and also since, although $\mathtt{root}(A)$ or $\mathtt{root}(A')$ might not be nodes, both cannot be leaves. Hence, either $\rho(\mathtt{root}(A))$ or $\rho(\mathtt{root}(A'))$ is well-defined and maximal.

\begin{example}
We draw the transformation that $\mathtt{RtoB}$ performs on a few examples. Below, we always have $M'<M$.
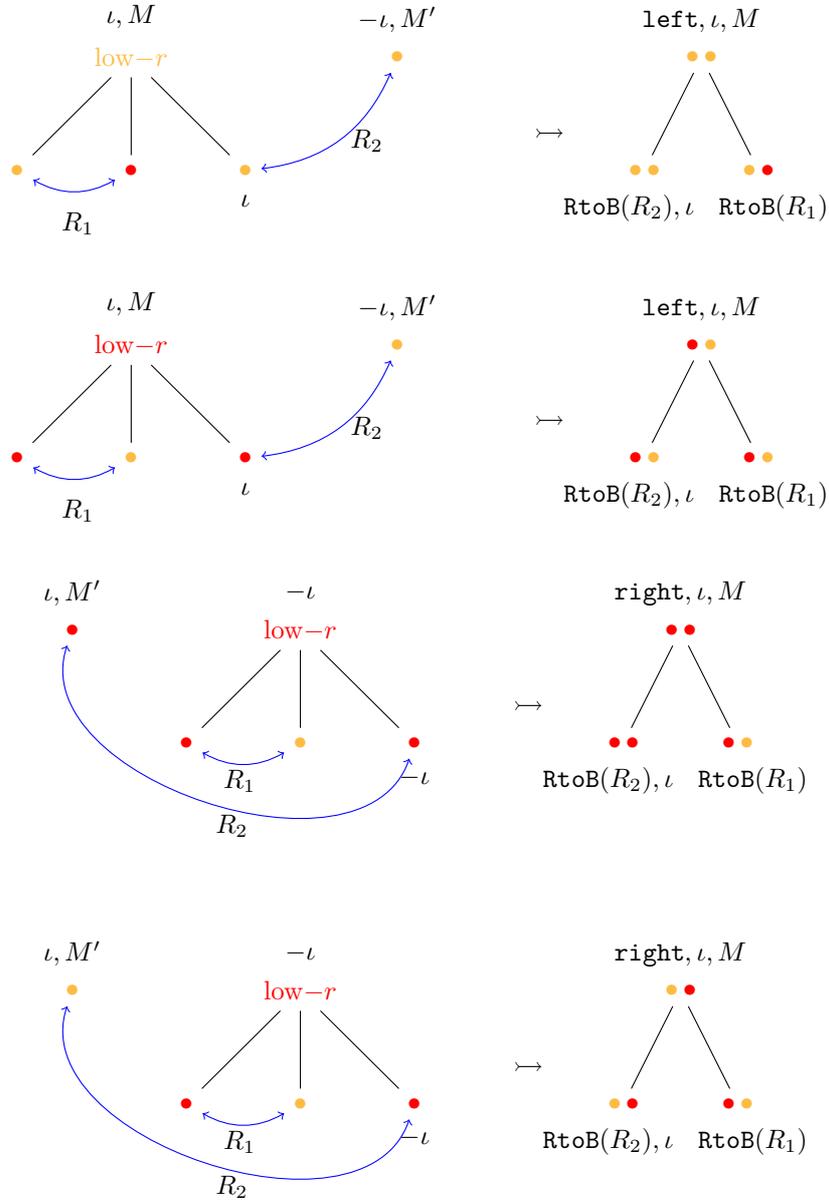
\begin{center}
\begin{tikzpicture}
\node at (7.5,0) [label=above:{$\lef,\iota,M$}]{$\textcolor{Dandelion}{\bullet \, \bullet}$}
child{node [label={[xshift=-0.2cm,yshift=-1cm]$\mathtt{RtoB}(R_2),\iota$}] {\textcolor{Dandelion}{$\bullet$}\,\textcolor{Dandelion}{$\bullet$}}}
child{node [label={[xshift=0.2cm,yshift=-1cm]$\mathtt{RtoB}(R_1)$}] {$\textcolor{Dandelion}{\bullet} \,\textcolor{red}{ \bullet}$}}
;
\draw (5.5,-1) node {$\rightarrowtail$} ;
\node[label=above:{$\iota,M$}]{$\textcolor{Dandelion}{\lowr}$}
child{node (B1) {$\textcolor{Dandelion}{\bullet}$} }
child{node (B1') {$\textcolor{red}{\bullet}$} }
child{node (B2) [label=below:{$\iota$}]{$\textcolor{Dandelion}{\bullet}$} }
;
\draw (3.5,0) node (B2') [label=above:{$-\iota,M'$}] {$\textcolor{Dandelion}{\bullet}$} ; 
\path (B1) edge [bend right,<->,blue] (B1') ; 
\path (B2) edge [bend right,<->,blue] (B2') ;
\draw (-0.7,-2.2) node {$R_1$} ;
\draw (3.1,-1.1) node {$R_2$} ;
\end{tikzpicture}
\end{center}

\begin{center}
\begin{tikzpicture}
\node at (7.5,0) [label=above:{$\lef,\iota,M$}]{$\textcolor{red}{\bullet} \, \textcolor{Dandelion}{\bullet}$}
child{node [label={[xshift=-0.2cm,yshift=-1cm]$\mathtt{RtoB}(R_2),\iota$}] {\textcolor{red}{$\bullet$}\,\textcolor{Dandelion}{$\bullet$}}}
child{node [label={[xshift=0.2cm,yshift=-1cm]$\mathtt{RtoB}(R_1)$}] {$\textcolor{red}{\bullet} \,\textcolor{Dandelion}{ \bullet}$}}
;
\draw (5.5,-1) node {$\rightarrowtail$} ;
\node[label=above:{$\iota,M$}]{$\textcolor{red}{\lowr}$}
child{node (B1) {$\textcolor{red}{\bullet}$} }
child{node (B1') {$\textcolor{Dandelion}{\bullet}$} }
child{node (B2) [label=below:{$\iota$}]{$\textcolor{red}{\bullet}$} }
;
\draw (3.5,0) node (B2') [label=above:{$-\iota,M'$}] {$\textcolor{Dandelion}{\bullet}$} ; 
\path (B1) edge [bend right,<->,blue] (B1') ; 
\path (B2) edge [bend right,<->,blue] (B2') ;
\draw (-0.7,-2.2) node {$R_1$} ;
\draw (3.1,-1.1) node {$R_2$} ;
\end{tikzpicture}
\end{center}

\begin{center}
\begin{tikzpicture}
\node at (8,0)[label=above:{$\righ,\iota,M$}]{$\textcolor{red}{\bullet} \, \textcolor{red}{\bullet}$}
child{node [label={[xshift=-0.2cm,yshift=-1cm]$\mathtt{RtoB}(R_2),\iota$}] {\textcolor{red}{$\bullet$}\,\textcolor{red}{$\bullet$}}}
child{node [label={[xshift=0.2cm,yshift=-1cm]$\mathtt{RtoB}(R_1)$}] {$\textcolor{red}{\bullet} \,\textcolor{Dandelion}{ \bullet}$}}
;
\draw (6,-1) node {$\rightarrowtail$} ;
\node at (3,0) [label=above:{$-\iota$}]{$\textcolor{red}{\lowr}$}
child{node (B1) {$\textcolor{red}{\bullet}$} }
child{node (B1') {$\textcolor{Dandelion}{\bullet}$} }
child{node (B2') [label=below:{$-\iota$}]{$\textcolor{red}{\bullet}$} }
;
\draw (0,0) node (B2) [label=above:{$\iota,M'$}] {$\textcolor{red}{\bullet}$} ; 
\path (B1) edge [bend right,<->,blue] (B1') ; 
\path (B2) edge [bend right=90,<->,blue] (B2') ;
\draw (2.2,-2) node {$R_1$} ;
\draw (2.1,-2.6) node {$R_2$} ;
\end{tikzpicture}
\end{center}

\begin{center}
\begin{tikzpicture}
\node at (8,0)[label=above:{$\righ,\iota,M$}]{$\textcolor{Dandelion}{\bullet} \, \textcolor{red}{\bullet}$}
child{node [label={[xshift=-0.2cm,yshift=-1cm]$\mathtt{RtoB}(R_2),\iota$}] {\textcolor{Dandelion}{$\bullet$}\,\textcolor{red}{$\bullet$}}}
child{node [label={[xshift=0.2cm,yshift=-1cm]$\mathtt{RtoB}(R_1)$}] {$\textcolor{red}{\bullet} \,\textcolor{Dandelion}{ \bullet}$}}
;
\draw (6,-1) node {$\rightarrowtail$} ;
\node at (3,0) [label=above:{$-\iota$}]{$\textcolor{red}{\lowr}$}
child{node (B1) {$\textcolor{red}{\bullet}$} }
child{node (B1') {$\textcolor{Dandelion}{\bullet}$} }
child{node (B2') [label=below:{$-\iota$}]{$\textcolor{red}{\bullet}$} }
;
\draw (0,0) node (B2) [label=above:{$\iota,M'$}] {$\textcolor{Dandelion}{\bullet}$} ; 
\path (B1) edge [bend right,<->,blue] (B1') ; 
\path (B2) edge [bend right=90,<->,blue] (B2') ;
\draw (2.2,-2) node {$R_1$} ;
\draw (2.1,-2.6) node {$R_2$} ;
\end{tikzpicture}
\end{center}
\end{example}

It is not so clear that $\mathtt{RtoB}$ maps indeed $\cup_n \mathcal R_n^{\mathtt{map}}$ to $\cup_n \mathcal B_n^{\mathtt{map}}$ because it is not clear that the colour and sign conditions are respected. It is also unclear why the properties of the maps $\rho$ (namely that it is increasing and injective) should be respected.  This the purpose of next proposition.

\begin{prop}\label{prop:RtoB} 
The image of $\mathcal R_n^{\mathtt{map}}(\iota, \eta,\eta',M)$ by $\mathtt{RtoB}$ is included in $\mathcal{B}^{\mathtt{map}}_n(\iota,(\eta,\eta'),M)$.
\end{prop}

\begin{proof}
The proof follows the same lines as the one of Proposition \ref{prop:BtoR}.
\end{proof}

\begin{prop}
The map $\mathtt{BtoR}$ is the inverse of the map $\mathtt{RtoB}$.
\end{prop}

\begin{proof}
Once Propositions \ref{prop:BtoR} and \ref{prop:RtoB} have been stated, it is explicit that the maps have been constructed to be the inverse of each other but for the fact if $((A,A',\sigma),\rho)$ is the image under $\mathtt{BtoR}$ of 
\[
B = (N_\lef(B_1,B_2),\iota,E,M)
\]
(resp. 
\[
B = (N_\righ(B_1,B_2),\iota,E,M) ),
\]
then $\rho(\mathtt{root}(A)) > \rho(\mathtt{root}(A'))$ (resp. $\rho(\mathtt{root}(A)) < \rho(\mathtt{root}(A'))$) or $A'$ (resp. $A$) is a leaf. However, this is due to the fact that $\rho(\mathtt{root}(A)) = M$ (resp. $\rho(\mathtt{root}(A')) = M$) and that if it exists $\rho(\mathtt{root}(A')) \in \mathrm{Im}(\rho_{B_1})$ (resp. $\rho(\mathtt{root}(A)) \in \mathrm{Im}(\rho_{B_1})$) where $\rho_{B_1}$ is the restriction of $\rho_B$ to $\mathcal N(B_1)$ and $B = (B',\rho_B))$. By definition, $\mathrm{Im}(\rho_{B_1}) \subseteq ([|0,M-1|])$.
\end{proof}

\begin{corollary} 
The set $\mathcal B_n^{\mathtt{ord}}(\iota,(\eta,\eta'))$ is in bijection with $\mathcal R_n(\iota,\eta,\eta')$.
\end{corollary}

\subsubsection{Contributions of resonant couplings}

\begin{definition} 
We define:
\begin{itemize}
\item For any $(R,\rho_R) \in \mathcal R_n^{\mathtt{map}}$, we set 
\[
F_R^{\rho_R} (t,k) \vcentcolon  =  (-i)^n L^{-nd} \sum_{\kappa \in \mathcal D_k(C)} \prod_{j\in \mathcal N(R)} q_j \prod_{l\in \mathcal L_-(R)} M^{c(l),c(\sigma(l))}(-\kappa(l))  \frac{t^n}{n!}.
\]
\item For any $(B,\rho_B) \in \mathcal B_n^{\mathtt{map}}$, we set 
\[
G_B^{\rho_B} \vcentcolon = F_R^{\rho_R} 
\]
where $(R,\rho_R) = \mathtt{BtoR}(B,\rho_B)$.
\item For $B \in \mathcal B_n$, we set 
\[
G_B \vcentcolon = \sum_{\rho_B \in \mathfrak M(B)} G_B^{\rho_B}.
\]
\end{itemize}
\end{definition}

\begin{remark} We note that first, for any $R \in \mathcal{R}_n$, we have that 
\[
\widehat{F_R}(\varepsilon^{-2}t,k) = \sum_{\rho_R \in \mathfrak M(R)} F_R^{\rho_R}(t,k).
\]
What is more, the map $\mathtt{BtoR}$ does not need the order on nodes to be defined. As a consequence if $\rho_1$ and $\rho_2$ are two orders on the nodes of $B$, setting $(R_i,\rho'_i) = \mathtt{BtoR}(B,\rho_i)$, we have that $R_1 = R_2$. The reverse statement is not true.
\end{remark}

In the rest of the subsection, we prove some recursive structure over $G_B$ using the recursive structure over $F_R$.

\begin{prop}\label{prop:recB}
Let $B \in \mathcal{B}_n(\iota,E)$. If $n = 0$, we have that
\[
G_B(t,k) = M^{E,\iota}(k).
\]
Otherwise, $B$ writes $(N_*(B_1,B_2),\iota,E)$ and if $* = \lef$, setting $E = (\eta,\eta')$, we have that
\[
G_B(t,k) = - \iota i  \int_{0}^t G_{B_1}(\tau,k)L^{-d}\sum_{k_2} \chi_\lowr(k_2,-k_2,k) q_{-\iota\varphi(\mathtt{root}(B_2)}^{\eta}(k_2,-k_2,k) G_{B_2}(\tau,k_2)\d\tau
\]
and if $* = \righ$, we have that
\[
G_B(t,k) =  \iota i  \int_{0}^t G_{B_1}(\tau,k)L^{-d}\sum_{k_2} \chi_\lowr(k_2,-k_2,k) q_{\iota\varphi(\mathtt{root}(B_2)}^{\eta'}(k_2,-k_2,-k) G_{B_2}(\tau,k_2)\d\tau . 
\]
\end{prop}

\begin{proof} The case $n=0$ follows from the definition of the different objects. Indeed, if $\iota = +$, then by definition 
\[
G_B(t,k) = M^{ E}(k) = M^{E,+}(k)
\]
and if $\iota=-$, we have 
\[
G_B(t,k) = \overline{M^E(-k)} = M^{E,-}(k).
\]
If $n>0$, let $\rho$ be any element of $\mathfrak{M}(B)$ and $R$ be such that there exists $\rho_R$ such that $(R,\rho_R) = \mathtt{BtoR}(B,\rho)$. We have that 
\[
G_B = \# \mathfrak{M}(B) G_B^\rho = \# \mathfrak{M}(B) F_R^{\rho_R},
\]
and
\[
F_R^{\rho_R} = \frac{t^n}{n!} \tilde F_R.
\]
We set 
\[
G_B = \# \mathfrak{M}(B) \frac{t^n}{n!} \tilde G_B.
\]
We therefore have that $\tilde G_B = \tilde F_R$. Assume that $B = (N_\lef (B_1,B_2),\iota,E)$. Then $R$ writes $(A,A'_1,\sigma)$ where $A = N_\lowr(A_2,A'_2,A_1)$ and where $R_i := (A_i,A_i',\sigma_{|\mathcal L(A_i) \cup \mathcal L(A'_i)}) = \mathtt{BtoR(B_i)}$. As we have already seen, we have that
\[
\tilde F_R(k) = -i \tilde F_{R_1}(k) L^{-d} \sum_{k_2} q_{\mathtt{root}(A)} \tilde F_{R_2}(k_2) 
\]
where we used that $\mathcal D_k(R)$ is the set of maps $\kappa : \mathcal L(R) \cup \mathcal N(R) \longrightarrow \Z_L^d$ such that $\kappa_{|\mathcal L(R_1) \cup \mathcal N(R_1)} \in \mathcal D_k(R_1)$ and such that there exists $k_2 \in \Z_L^d$ such that $\kappa_{|\mathcal N(R_2) \cup \mathcal L(R_2)} \in \mathcal D_{k_2}(R_2)$. With these notations, we have 
\[
q_{\mathtt{root}(A)} = \varphi(\mathtt{root}(A)) \chi_\lowr(k_2,-k_2,k) q^{c(\mathtt{root}(A)}_{\varphi(\mathtt{root}(A))\varphi(\mathtt{root}(A'_2))}(k_2,-k_2,k).
\]
Thanks to Proposition \ref{prop:BtoR}, we have that $\varphi(\mathtt{root}(A)) = \iota$, $\varphi(\mathtt{root}(A'_2)) = -\varphi(\mathtt{root}(A_2)) = -\varphi(\mathtt{root}(B_2))$ and $c(\mathtt{root}(A)) = \eta$.
We therefore have 
\[
\tilde G_B(k) = - \iota i \tilde G_{B_1}(k) L^{-d} \sum_{k_2} \chi_\lowr(k_2,-k_2,k) q^{\eta}_{-\iota\varphi(\mathtt{root}(B_2))}(k_2,-k_2,k) \tilde G_{B_2}(k_2).
\]
We now use that $\# \mathfrak{M}(B) \frac{t^n}{n!} = \int_{I_B(t)} \prod_{j\in \mathcal N(B)} dt_j$ where 
\[
I_B(t) = \{ (t_j)_{j\in \mathcal N(B)} \;|\; j<j' \Longrightarrow t_j < t_{j'} \}
\]
with the order of parentality on the nodes of $B$ to get the induction relation
\[
G_B(t,k) = - \iota i  \int_{0}^t G_{B_1}(\tau,k)L^{-d}\sum_{k_2} \chi_\lowr(k_2,-k_2,k) q_{-\iota\varphi(\mathtt{root}(B_2))}^{\eta}(k_2,-k_2,k) G_{B_2}(\tau,k_2)\d\tau
\]

We turn to the case $B = (N_\righ (B_1,B_2),\iota,E)$. Then $R= (A_1,A',\sigma)$ where $A' = N_\lowr(A_2,A'_2,A'_1)$ and where $R_i := (A_i,A_i',\sigma_{|\mathcal L(A_i) \cup \mathcal L(A'_i)}) = \mathtt{BtoR(B_i)}$. We have that
\[
\tilde F_R(k) = -i \tilde F_{R_1}(k) L^{-d} \sum_{k_2} q_{\mathtt{root}(A')} \tilde F_{R_2}(k_2) 
\]
where we used that $\mathcal D_k(R)$ is the set of maps $\kappa : \mathcal L(R) \cup \mathcal N(R) \rightarrow \Z_L^d$ such that $\kappa_{|\mathcal L(R_1) \cup \mathcal N(R_1)} \in \mathcal D_k(R_1)$ and such that there exists $k_2 \in \Z_L^d$ such that $\kappa_{|\mathcal N(R_2) \cup \mathcal L(R_2)} \in \mathcal D_{k_2}(R_2)$. With these notations, we have 
\[
q_{\mathtt{root}(A')} = \varphi(\mathtt{root}(A')) \chi_\lowr(k_2,-k_2,-k) q^{c(\mathtt{root}(A'))}_{\varphi(\mathtt{root}(A'))\varphi(\mathtt{root}(A'_2))}(k_2,-k_2,-k).
\]
Thanks to Proposition \ref{prop:BtoR}, we have that $\varphi(\mathtt{root}(A')) = -\iota$, $\varphi(\mathtt{root}(A'_2)) = -\varphi(\mathtt{root}(A_2)) = -\varphi(\mathtt{root}(B_2))$ and $c(\mathtt{root}(A')) = \eta'$.
We therefore have 
\[
\tilde G_B(k) =  \iota i \tilde G_{B_1}(k) L^{-d} \sum_{k_2} \chi_\lowr(k_2,-k_2,k) q^{\eta'}_{\iota\varphi(\mathtt{root}(B_2))}(k_2,-k_2,-k) \tilde G_{B_2}(k_2),
\]
and we conclude as before.
\end{proof}

\begin{notation} \label{notation X res n} We introduce the notation 
\[
\Xres{n}^{\iota,(\eta,\eta')}(t) = \sum_{C \in \mathtt{Res}_n(\iota,-\iota,\eta,\eta')} \hat F_C(\varepsilon^{-2}t).
\] 
\end{notation}

\begin{prop}\label{prop:resonantinductionformula} We have that $\Xres{0}^{\iota,E}(t,k) = M^{E,\iota}(k)$ and if $n>0$, we have
\begin{align*}
&\Xres{n}^{\iota,E} (t,k)
\\& = -2 \iota i \sum_{n_1 + n_2 = n-1} \sum_{\iota_2}  \int_{0}^t \Xres{n_1}^{\iota,E}(\tau,k) L^{-d} \sum_{k_2} \chi_\lowr(k_2,-k_2,k) 
\\& \hspace{3cm}\times \Big( q_{-\iota \iota_2}^\eta (k_2,-k_2,k)  \Xres{n_2}^{\iota_2,E_2^\lef}  - q_{\iota \iota_2}^{\eta'} (k_2,-k_2,-k)  \Xres{n_2}^{\iota_2,E_2^\righ}(\tau,k_2)  \Big)\d\tau.
\end{align*}
where $E = (\eta,\eta')$ and $E_2^*$ is the color of the right child of a $*$ node with color $E$.
\end{prop}

\begin{proof} By definition 
\[
\Xres{n}^{\iota,E} (t)= \sum_{C\in \mathtt{Res}_n(\iota,-\iota,\eta,\eta')}\hat F_C(\e^{-2}t).
\]
According to Proposition \ref{prop:reduceRes}, we have that
\[
\Xres{n}^{\iota,E} (t)= 2^n\sum_{R\in \mathcal R_n(\iota,-\iota,\eta,\eta')}\hat F_R(\e^{-2}t).
\]
We may chose to order the nodes in each $R$ and get
\[
\Xres{n}^{\iota,E} = 2^n\sum_{(R,\rho_R)\in \mathcal R_n^{\mathtt{ord}}(\iota,\eta,\eta')} F_R^{\rho_R}.
\]
We use that $\mathcal R_n^{\mathtt{ord}}(\iota,\eta,\eta')$ and $\mathcal B_n^{\mathtt{ord}}(\iota,E)$ are in bijection to get
\[
\Xres{n}^{\iota,E} = 2^n\sum_{(B,\rho_B)\in \mathcal B_n^{\mathtt{ord}}(\iota,E)} G_B^{\rho_B}.
\]
We may sum back on $\rho_B$ and get
\begin{equation}\label{eq:Xresintermediaire}
\Xres{n}^{\iota,E} = 2^n\sum_{B\in \mathcal B_n(\iota,E)} G_B.
\end{equation}

For $n=0$, we use that $ \mathcal B_0(\iota,E)$ is reduced to one element and we use Proposition \ref{prop:recB}.

If $n>0$, we use the induction property in Proposition \ref{prop:recB}. For this, we use that $B\in \mathcal B_n(\iota,E)$ write $B = (N_*(B_1,B_2),\iota,E)$ in a unique way, with $B_i \in \mathcal B_{n_i}(\iota_i,E_i^*)$ where $n_1+n_2 = n-1$, where $\iota_1 = \iota$ and $(E_1^*, E_2^*)$ are the colors prescribed by Table \ref{eq:colortable}. The sign $\iota_2$ may be chosen freely. We note that according to the table, we have that $E_1^* = E$. We deduce
\begin{multline*}
\Xres{n}^{\iota,E} = \\
2^n\sum_{n_1 + n_2 = n-1} \sum_{\iota_2}\Big(\hspace{-0.2cm} \sum_{\substack{B_1 \in \mathcal B_{n_1}(\iota,E)\\B_2 \in \mathcal B_{n_2}(\iota_2,E_2^\lef)}}\hspace{-0.2cm} (- \iota i ) \int_{0}^t G_{B_1}(\tau,k) L^{-d} \sum_{k_2} \chi_\lowr(k_2,-k_2,k) q_{-\iota \iota_2}^\eta (k_2,-k_2,k)  G_{B_2}(\tau,k_2) d \tau \\
+\sum_{\substack{B_1 \in \mathcal B_{n_1}(\iota,E)\\B_2 \in \mathcal B_{n_2}(\iota_2,E_2^\righ)}} \iota i  \int_{0}^t G_{B_1}(\tau,k) L^{-d} \sum_{k_2} \chi_\lowr(k_2,-k_2,-k) q_{\iota \iota_2}^{\eta'} (k_2,-k_2,-k)  G_{B_2}(\tau,k_2) d \tau \Big) .
\end{multline*}
Using Equation \eqref{eq:Xresintermediaire}, this becomes
\begin{multline*}
\Xres{n}^{\iota,E} = 2\sum_{n_1 + n_2 = n-1} \sum_{\iota_2}\Big( - \iota i  \int_{0}^t \Xres{n_1}^{\iota,E}(\tau,k) L^{-d} \sum_{k_2} \chi_\lowr(k_2,-k_2,k) q_{-\iota \iota_2}^\eta (k_2,-k_2,k)  \Xres{n_2}^{\iota_2,E_2^\lef} d \tau \\
 \iota i  \int_{0}^t \Xres{n_1}^{\iota,E}(\tau,k) L^{-d} \sum_{k_2} \chi_\lowr(k_2,-k_2,-k) q_{\iota \iota_2}^{\eta'} (k_2,-k_2,-k)  \Xres{n_2}^{\iota_2,E_2^\righ}(\tau,k_2) d \tau  \Big).
\end{multline*}
Using the fact that 
\[
\chi_\lowr(k_2,-k_2,-k) = (1-\chi)(|k-k_2|\varepsilon^{-\gamma})(1-\chi)(|k+k_2|\varepsilon^{-\gamma}) = \chi_\lowr(k_2,-k_2,-k)
\]
and factorising yields
\begin{multline*}
\Xres{n}^{\iota,E} = -2 \iota i \sum_{n_1 + n_2 = n-1} \sum_{\iota_2}  \int_{0}^t \Xres{n_1}^{\iota,E}(\tau,k) L^{-d} \sum_{k_2} \chi_\lowr(k_2,-k_2,k) \\
\times \Big( q_{-\iota \iota_2}^\eta (k_2,-k_2,k)  \Xres{n_2}^{\iota_2,E_2^\lef}(\tau,k_2)  - q_{\iota \iota_2}^{\eta'} (k_2,-k_2,-k)  \Xres{n_2}^{\iota_2,E_2^\righ}(\tau,k_2) d \tau  \Big),
\end{multline*}
which concludes the proof of the proposition.
\end{proof}

\begin{lemma} 
For all $E = (\eta,\eta')$ we set $\bar E = (\eta',\eta)$. 
We have that
\[
\Xres{n}^{+,\bar E}(t,k) = \Xres{n}^{-,E}(t,-k) = \overline{\Xres{n}^{+,E}(t,k)}.
\]
\end{lemma}

\begin{proof}
We prove this lemma by induction on $n$. For $n=0$, we have 
\[
\Xres{n}^{+,\bar E}(t,k) = \overline{M^{\bar E}(k)} = M^E(k) =  \overline{\Xres{n}^{+,E}(t,k)}
\]
and 
\[
\Xres{n}^{-,E}(t,-k) = \overline{M^{ E,-}(-k)} = M^E(k).
\]
By the induction relation of Proposition \ref{prop:resonantinductionformula}, we have that  
\begin{multline*}
\overline{\Xres{n}^{+,E}(t,k)} =  2 i \sum_{n_1 + n_2 = n-1} \sum_{\iota_2}  \int_{0}^t \overline{ \Xres{n_1}^{+,E}(\tau,k)} L^{-d} \sum_{k_2} \chi_\lowr(k_2,-k_2,k) \\
\times \Big( \overline{q_{-\iota_2}^\eta (k_2,-k_2,k)} \overline{\Xres{n_2}^{\iota_2,E_2^\lef}}(\tau,k_2)  - \overline{q_{ \iota_2}^{\eta'} (k_2,-k_2,-k)} \overline{ \Xres{n_2}^{\iota_2,E_2^\righ}(\tau,k_2)} d \tau  \Big).
\end{multline*}
We can use the induction hypothesis and obtain $\overline{\Xres{n_2}^{\iota_2,E_2^\lef}}(\tau,k_2) = \Xres{n_2}^{-\iota_2,E_2^\lef} (\tau,-k_2)$. We get
\begin{multline*}
\overline{\Xres{n}^{+,E}(t,k)} =  2 i \sum_{n_1 + n_2 = n-1} \sum_{\iota_2}  \int_{0}^t \overline{ \Xres{n_1}^{+,E}(\tau,k)} L^{-d} \sum_{k_2} \chi_\lowr(k_2,-k_2,k) \\
\times \Big( \overline{q_{-\iota_2}^\eta (k_2,-k_2,k)}  \Xres{n_2}^{-\iota_2,E_2^\lef}(\tau,-k_2)  - \overline{q_{ \iota_2}^{\eta'}} (k_2,-k_2,-k) \Xres{n_2}^{-\iota_2,E_2^\righ}(\tau,-k_2) d \tau  \Big).
\end{multline*}
We perform the change of variable $\iota_2 \leftarrow -\iota_2 $ and $k_2 \leftarrow -k_2$ and get 
\begin{multline}\label{eq:symétrieinter}
\overline{\Xres{n}^{+,E}(t,k)} =  2 i \sum_{n_1 + n_2 = n-1} \sum_{\iota_2}  \int_{0}^t \overline{ \Xres{n_1}^{+,E}(\tau,k)} L^{-d} \sum_{k_2} \chi_\lowr(-k_2,k_2,k) \\
\times \Big( \overline{q_{\iota_2}^\eta (-k_2,k_2,k) } \Xres{n_2}^{\iota_2,E_2^\lef}(\tau,k_2)  - \overline{q_{- \iota_2}^{\eta'}} (-k_2,k_2,-k) \Xres{n_2}^{\iota_2,E_2^\righ}(\tau,k_2) d \tau  \Big).
\end{multline}
The induction hypothesis then implies $\overline{ \Xres{n_1}^{+,E}(\tau,k)} = \Xres{n_1}^{-,E}(\tau,-k)$ and we use $\chi_\lowr(-k_2,k_2,k) = \chi_\lowr(k_2,-k_2,-k)$ and $\overline{q_{\iota_2}^\eta (-k_2,k_2,k)}=q_{\iota_2}^\eta (k_2,-k_2,-k)$ to obtain
\begin{multline*}
\overline{\Xres{n}^{+,E}(t,k)} =  2 i \sum_{n_1 + n_2 = n-1} \sum_{\iota_2}  \int_{0}^t  \Xres{n_1}^{-,E}(\tau,-k) L^{-d} \sum_{k_2} \chi_\lowr(k_2,-k_2,-k) \\
\times \Big( q_{\iota_2}^\eta (k_2,-k_2,-k)  \Xres{n_2}^{\iota_2,E_2^\lef}(\tau,k_2)  - q_{- \iota_2}^{\eta'} (k_2,-k_2,k) \Xres{n_2}^{\iota_2,E_2^\righ}(\tau,k_2) d \tau  \Big).
\end{multline*}
We recognize
\[
\overline{\Xres{n}^{+,E}(t,k)} = \Xres{n}^{-,E}(t,-k).
\]

We now go back to \eqref{eq:symétrieinter}, use the induction hypothesis to get that $\overline{ \Xres{n_1}^{+,E}(\tau,k)} = \Xres{n_1}^{+,\bar E}(\tau,k)$, use that $\chi_\lowr(-k_2,k_2,k)=\chi_\lowr(k_2,-k_2,k)$ and rearrange the sum to get
\begin{multline*}
\overline{\Xres{n}^{+,E}(t,k)} =  - 2 i \sum_{n_1 + n_2 = n-1} \sum_{\iota_2}  \int_{0}^t \Xres{n_1}^{+,\bar E}(\tau,k) L^{-d} \sum_{k_2} \chi_\lowr(k_2,-k_2,k) \\
\times \Big( \overline{q_{- \iota_2}^{\eta'} (-k_2,k_2,-k)} \Xres{n_2}^{\iota_2,E_2^\righ}(\tau,k_2) - \overline{q_{\iota_2}^\eta (-k_2,k_2,k) } \Xres{n_2}^{\iota_2,E_2^\lef}(\tau,k_2)  d \tau  \Big).
\end{multline*}
Using now $\overline{q_{- \iota_2}^{\eta'} (-k_2,k_2,-k)} = q_{- \iota_2}^{\eta'} (k_2,-k_2,k)$ and $E_2^\righ = (\bar E)_2^\lef$ for any $E$ (according to Table \eqref{eq:colortable}) we obtain
\begin{multline*}
\overline{\Xres{n}^{+,E}(t,k)} =  - 2 i \sum_{n_1 + n_2 = n-1} \sum_{\iota_2}  \int_{0}^t \Xres{n_1}^{+,\bar E}(\tau,k) L^{-d} \sum_{k_2} \chi_\lowr(k_2,-k_2,k) \\
\times \Big( q_{- \iota_2}^{\eta'} (k_2,-k_2,k) \Xres{n_2}^{\iota_2,(\bar E)_2^\lef}(\tau,k_2) - q_{\iota_2}^\eta (k_2,-k_2,-k)  \Xres{n_2}^{\iota_2,(\bar E)_2^\righ}(\tau,k_2)  d \tau  \Big),
\end{multline*}
and finally recognize
\[
\overline{\Xres{n}^{+,E}(t,k)} = \Xres{n}^{+,\bar E}(t,k).
\]
This concludes the proof of the induction.
\end{proof}

This lemma implies that a lot of quantities are redondant and we define the main ones:
\begin{align*}
    \rho^\eta_{L,n} & \vcentcolon = \Xres{n}^{+,(\eta,\eta)},
    \\ \rho^\times_{L,n} & \vcentcolon = \Xres{n}^{+, \times}.
\end{align*}
The following proposition gathers key recursive properties of $\rho^\eta_{L,n}$ and $\rho^\times{L,n}$.

\begin{prop}\label{recursive expression rhoLn}
The function $\rho^\eta_{L,n}$ is real-valued and $\pth{\rho^0_{L,n},\rho^1_{L,n}, \rho^{\times}_{L,n}}$ satisfies the following induction property: if $n=0$ we have $\rho^\eta_{L,0} = M^\eta$ and $\rho^\times_{L,0} = M^\times $ and if $n>0$ we have
\begin{align*}
\rho^0_{L,n}(t,k) = & 4 \sum_{n_1 + n_2 = n-1} \int_{0}^t \rho_{L,n_1}^{0}(\tau,k) L^{-d} \sum_{k_2} \chi_\lowr(k_2,-k_2,k) \mathrm{Im}\pth{ N^0(k,k_2)  \overline{\rho_{L,n_2}^{\times}(\tau,k_2)}} \d \tau,\\
\rho^1_{L,n}(t,k) = & 4 \sum_{n_1 + n_2 = n-1} \int_{0}^t \rho_{L,n_1}^{1}(\tau,k) L^{-d} \sum_{k_2} \chi_\lowr(k_2,-k_2,k) \mathrm{Im}\pth{ N^1(k,k_2)  \rho_{L,n_2}^{\times}(\tau,k_2)} \d \tau, \\
\rho^\times_{L,n} = &
2  i \sum_{n_1 + n_2 = n-1}  \int_{0}^t \rho_{L,n_1}^{ \times}(\tau,k) L^{-d} \sum_{k_2} \chi_\lowr(k_2,-k_2,k) \\
& \hspace{3cm}  \times \pth{ P(k,k_2)  \rho_{L,n_2}^{ \times}(\tau,k_2)+M(k,k_2) \overline{ \rho_{L,n_2}^{\times}(\tau,k_2)} } \d\tau .
\end{align*}
\end{prop}

\begin{proof} We start with $\rho^0$. We have 
\begin{multline*}
\rho^0_{L,n}(t,k) = \Xres{n}^{+,(0,0)} = -2 i \sum_{n_1 + n_2 = n-1} \sum_{\iota_2}  \int_{0}^t \rho_{L,n}^{0}(\tau,k) L^{-d} \sum_{k_2} \chi_\lowr(k_2,-k_2,k) \\
\times \Big( q_{- \iota_2}^0 (k_2,-k_2,k)  \Xres{n_2}^{\iota_2,E_2^\lef}(\tau,k_2)  - q_{ \iota_2}^{0} (k_2,-k_2,-k)  \Xres{n_2}^{\iota_2,E_2^\righ}(\tau,k_2) \d \tau  \Big).
\end{multline*}
We remark that $E_2^\lef = E_2^\righ = \times$ and get
\begin{multline*}
\rho^0_{L,n}(t,k) = -2 i \sum_{n_1 + n_2 = n-1} \sum_{\iota_2}  \int_{0}^t \rho_{L,n}^{0}(\tau,k) L^{-d} \sum_{k_2} \chi_\lowr(k_2,-k_2,k) \\
\times \Big( q_{- \iota_2}^0 (k_2,-k_2,k)    - q_{ \iota_2}^{0} (k_2,-k_2,-k) \Big)  \Xres{n_2}^{\iota_2,\times}(\tau,k_2) \d \tau .
\end{multline*}
We develop the sum in $\iota_2$ and get
\begin{multline*}
\rho^0_{L,n}(t,k) = -2 i \sum_{n_1 + n_2 = n-1} \int_{0}^t \rho_{L,n}^{0}(\tau,k) L^{-d} \sum_{k_2} \chi_\lowr(k_2,-k_2,k) \\
\times \Big[\Big( q_{- }^0 (k_2,-k_2,k)    - q_{+}^{0} (k_2,-k_2,-k) \Big)  \Xres{n_2}^{+,\times}(\tau,k_2) + \Big( q_{+}^0 (k_2,-k_2,k)    - q_{-}^{0} (k_2,-k_2,-k) \Big)  \Xres{n_2}^{-,\times}(\tau,k_2)\Big] \d \tau .
\end{multline*}
We remark that $\Xres{n_2}^{+,\times}(\tau,k_2) = \rho^\times (\tau,k_2)$ and that $ \Xres{n_2}^{-,\times}(\tau,k_2) = \overline{\rho^\times (\tau,-k_2)}$ and we perform a change of variables into the second part of the sum to get
\begin{multline*}
\rho^0_{L,n}(t,k) = -2 i \sum_{n_1 + n_2 = n-1} \int_{0}^t \rho_{L,n_1}^{0}(\tau,k) L^{-d} \sum_{k_2} \chi_\lowr(k_2,-k_2,k) \\
\times \Big[\Big( q_{- }^0 (k_2,-k_2,k)    - q_{+}^{0} (k_2,-k_2,-k) \Big)  \rho_{L,n_2}^\times(\tau,k_2) + \Big( q_{+}^0 (-k_2,k_2,k)    - q_{-}^{0} (-k_2,k_2,-k) \Big)  \overline{\rho_{L,n_2}^{\times}(\tau,k_2)}\Big] \d \tau .
\end{multline*}
We rearrange the sum using that $q(-k_1,-k_2,-k_3) = \overline{q(k_3,k_2,k_1)}$ and obtain
\begin{multline*}
\rho^0_{L,n}(t,k) = -2 i \sum_{n_1 + n_2 = n-1} \int_{0}^t \rho_{L,n_1}^{0}(\tau,k) L^{-d} \sum_{k_2} \chi_\lowr(k_2,-k_2,k) \\
\times \Big[\Big( q_{- }^0 (k,-k_2,k_2)    - \overline{q_{+}^{0} (k,k_2,-k_2)} \Big)  \rho_{L,n_2}^\times(\tau,k_2) + \Big( q_{+}^0 (k,k_2,-k_2)    - \overline{q_{-}^{0} (k,-k_2,k_2)} \Big)  \overline{\rho_{L,n_2}^{\times}(\tau,k_2)}\Big] \d \tau .
\end{multline*}
that is to say, using the definition of $N^0(k,k_2)$,
\begin{multline*}
\rho^0_{L,n}(t,k) = 2 i \sum_{n_1 + n_2 = n-1} \int_{0}^t \rho_{L,n_1}^{0}(\tau,k) L^{-d} \sum_{k_2} \chi_\lowr(k_2,-k_2,k) \\
\times \pth{\overline{N^0(k,k_2)} \rho_{L,n_2}^\times(\tau,k_2) - N^0(k,k_2)  \overline{\rho_{L,n_2}^{\times}(\tau,k_2)}} \d \tau ,
\end{multline*}
which may be further expressed as
\[
\rho^0_{L,n}(t,k) = 4 \sum_{n_1 + n_2 = n-1} \int_{0}^t \rho_{L,n_1}^{0}(\tau,k) L^{-d} \sum_{k_2} \chi_\lowr(k_2,-k_2,k) \mathrm{Im}\pth{ N^0(k,k_2)  \overline{\rho_{L,n_2}^{\times}(\tau,k_2)}} \d \tau .
\]
The computation for $\rho^1$ is similar. 

We now turn to $\rho^\times$. We have 
 \begin{multline*}
\rho^\times_{L,n} = \Xres{n}^{+, \times} 
-2  i \sum_{n_1 + n_2 = n-1} \sum_{\iota_2}  \int_{0}^t \Xres{n_1}^{+,\bar \times}(\tau,k) L^{-d} \sum_{k_2} \chi_\lowr(k_2,-k_2,k) \\
\times \Big( q_{- \iota_2}^0 (k_2,-k_2,k)  \Xres{n_2}^{\iota_2,E_2^\lef}(\tau,k_2)  - q_{\iota_2}^{1} (k_2,-k_2,-k)  \Xres{n_2}^{\iota_2,E_2^\righ}(\tau,k_2) \d \tau  \Big).
\end{multline*}
We use that $E_2^\lef =  \times$ and that $E_2^\righ = \bar \times$ and get
 \begin{multline*}
\rho^\times_{L,n} = \Xres{n}^{+,\bar \times} =
-2  i \sum_{n_1 + n_2 = n-1} \sum_{\iota_2}  \int_{0}^t \rho_{L,n_1}^{ \times}(\tau,k) L^{-d} \sum_{k_2} \chi_\lowr(k_2,-k_2,k) \\
\times \Big( q_{- \iota_2}^0 (k_2,-k_2,k)  \Xres{n_2}^{\iota_2, \times}(\tau,k_2)  - q_{\iota_2}^{1} (k_2,-k_2,-k)  \Xres{n_2}^{\iota_2,\bar \times}(\tau,k_2)  \Big)\d \tau .
\end{multline*}
We develop the sum in $\iota_2$ and get
 \begin{multline*}
\rho^\times_{L,n} =
-2  i \sum_{n_1 + n_2 = n-1}  \int_{0}^t \rho_{L,n_1}^{ \times}(\tau,k) L^{-d} \sum_{k_2} \chi_\lowr(k_2,-k_2,k) \\
\times \Big( q_{-}^0 (k_2,-k_2,k)  \Xres{n_2}^{+, \times}(\tau,k_2)  - q_{+}^{1} (k_2,-k_2,-k)  \Xres{n_2}^{+,\bar \times}(\tau,k_2) \\
q_{+}^0 (k_2,-k_2,k)  \Xres{n_2}^{-, \times}(\tau,k_2)  - q_{-}^{1} (k_2,-k_2,-k)  \Xres{n_2}^{-,\bar\times}(\tau,k_2)\Big)\d\tau.
\end{multline*}
We write everything in terms of $\rho^\times$ and get
 \begin{multline*}
\rho^\times_{L,n} = 
-2  i \sum_{n_1 + n_2 = n-1}  \int_{0}^t \rho_{L,n_1}^{ \times}(\tau,k) L^{-d} \sum_{k_2} \chi_\lowr(k_2,-k_2,k) \\
\times \Big( q_{-}^0 (k_2,-k_2,k)  \rho_{L,n_2}^{ \times}(\tau,k_2)  - q_{+}^{1} (k_2,-k_2,-k) \overline{ \rho_{L,n_2}^{\times}(\tau,k_2)} \\
q_{+}^0 (k_2,-k_2,k)  \overline{\rho_{L,n_2}^{ \times}(\tau,-k_2)}  - q_{-}^{1} (k_2,-k_2,-k)  \rho_{L,n_2}^{\times}(\tau,-k_2)\Big)\d\tau.
\end{multline*}
We perform a change of variable in the second part of the sum and use the symmetries of $q$ to get
 \begin{multline*}
\rho^\times_{L,n} =
-2  i \sum_{n_1 + n_2 = n-1}  \int_{0}^t \rho_{L,n_1}^{ \times}(\tau,k) L^{-d} \sum_{k_2} \chi_\lowr(k_2,-k_2,k) \\
\times \Big( q_{-}^0 (k_2,-k_2,k)  \rho_{L,n_2}^{ \times}(\tau,k_2)  - \overline{q_{+}^{1} (k,k_2,-k_2)} \overline{ \rho_{L,n_2}^{\times}(\tau,k_2)} \\
q_{+}^0 (k,k_2,-k_2)  \overline{\rho_{L,n_2}^{ \times}(\tau,k_2)}  - \overline{q_{-}^{1} (k,-k_2,k_2)}  \rho_{L,n_2}^{\times}(\tau,k_2)\Big)\d\tau.
\end{multline*}
We finally obtain
 \begin{multline*}
\rho^\times_{L,n} =
-2  i \sum_{n_1 + n_2 = n-1}  \int_{0}^t \d\tau \rho_{L,n_1}^{ \times}(\tau,k) L^{-d} \sum_{k_2} \chi_\lowr(k_2,-k_2,k) \\
\times \Big[\Big( q_{-}^0 (k,-k_2,k_2)- \overline{q_{-}^{1} (k,-k_2,k_2)}\Big)  \rho_{L,n_2}^{ \times}(\tau,k_2)+\Big(q_{+}^0 (k,k_2,-k_2)  - \overline{q_{+}^{1} (k,k_2,-k_2)}\Big) \overline{ \rho_{L,n_2}^{\times}(\tau,k_2)}\Big] .
\end{multline*}
Using the definitions of $P_\pm$ this becomes
\begin{multline*}
\rho^\times_{L,n} =
2  i \sum_{n_1 + n_2 = n-1}  \int_{0}^t  \rho_{L,n_1}^{ \times}(\tau,k) L^{-d} \sum_{k_2} \chi_\lowr(k_2,-k_2,k) \\
\times \Big( P_-(k,k_2)  \rho_{L,n_2}^{ \times}(\tau,k_2)+P_+(k,k_2) \overline{ \rho_{L,n_2}^{\times}(\tau,k_2)}\Big)\d\tau ,
\end{multline*}
which concludes the proof of the proposition.
\end{proof}

\subsection{Analysis of the resonant system}\label{section WP cross system}

We consider the Cauchy problem for \eqref{cross system} with initial data $(\rho^\eta,\rho^\times)_{|_{t=0}} = (M^\eta,  M^\times)$. Note that by definition, for all $\zeta\in\R^d$ the maps $ \xi \longmapsto  N^\eta(\xi,\zeta)$ and $ \xi \longmapsto  P_\pm(\xi,\zeta)$ are continuous and bounded. The following proposition in particular implies the second part of Theorem \ref{main theorem}.

\begin{prop}
The Cauchy problem \eqref{cross system} is locally well-posed in $ W^{1,\infty}(\R^d;\R\times \R \times \C) \cap W^{1,1}(\R^d;\R\times \R \times \C)$. What is more, defining $(\rho_n^\eta,\rho_n^\times)_{n\in\mathbb{N}}$ by induction as
\[
\rho_0^\eta = M^\eta, \quad \rho_0^\times = M^\times
\]
and for $n\in \N$,
\begin{align}
    \rho^\eta_{n+1}(t,\xi) = &  4 \sum_{n_1 + n_2 = n}\int_{0}^t d\tau \rho_{n_1}^\eta (\tau,\xi) \int_{\R^d} d\zeta \mathrm{Im}\pth{N^\eta (\xi,\zeta) \rho_{n_2}^\times (\tau,\zeta)^{\iota(\eta)}} \label{rho eta n}\\
    \rho^\times_{n+1}(t,\xi) = & 2i \sum_{n_1+n_2 = n} \int_{0}^t d\tau \rho_{n_1}^\times (\tau,\xi) \int_{\R^d} d\zeta \Big[ P_-(\xi,\zeta) \rho_{n_2}^\times (\tau,\zeta) + P_+ (\xi,\zeta) \overline{\rho_{n_2}^\times(\tau,\zeta)} \Big], \label{rho x n}
\end{align}
where $\iota(\eta)$ is the sign of $(-1)^{\eta+1}$, we get that there exists $\Lambda >0$ such that for all $n$ and for 
\begin{align*}
\delta < \pth{ 64\Lambda \pth{\max_\eta\|N^\eta\|_{W^{1,\infty}_\xi L^\infty_\zeta} + \max_\pm \|P_\pm\|_{W^{1,\infty}_\xi L^\infty_\zeta}}}^{-1}    
\end{align*}
we have
\begin{align} 
   \|\rho_n^*\|_{L^\infty([0,\delta],W^{1,\infty}(\R^d) \cap W^{1,1}(\R^d))} &\leq \Lambda^{n+1} \delta^n c_n, \label{estimate rho* cauchy pb1} \\
 \l\rho^* - \sum_{k=0}^n\rho_k^*\r_{L^\infty([0,\delta],W^{1,\infty}(\R^d) \cap W^{1,1}(\R^d))}& \leq (\Lambda \delta)^{n} , \label{estimate rho* cauchy pb2}
\end{align}
where $c_n$ is the $n$-th Catalan number and $(\rho^0,\rho^1,\rho^\times)$ is the unique maximal solution to \eqref{cross system}.
\end{prop}
\begin{proof}
The local well-posedness of \eqref{cross system} follows from the fact that the quadratic function
\[
F(\rho^\eta, \rho^\times) := \begin{pmatrix}
4 \rho^\eta  \int d\zeta \mathrm{Im}\pth{N^\eta (\cdot,\zeta) \rho^\times (\zeta)^{\iota(\eta)}}\\ 
2i \rho^\times \int d\zeta \Big[ P_-(\cdot,\zeta) \rho^\times (\zeta) + P_+ (\cdot,\zeta) \overline{\rho^\times(\zeta)} \Big]
\end{pmatrix}
\]
is Lipschitz continuous on $W^{1,\infty}(\R^d)\cap W^{1,1}(\R^d)$. In particular, we observe that
\[
\aligned
&\|F(\rho^\eta, \rho^\times)^\eta\|_{L^\infty_tW^{1,\infty}_\xi\cap W^{1,1}_\xi}\le 8 \|\rho^\eta\|_{L^\infty_tW^{1,\infty}_\xi\cap W^{1,1}_\xi}\|N^\eta\|_{W^{1,\infty}_\xi L^\infty_\zeta}\|\rho^\times\|_{L^\infty_t L^1_\xi}\\
& \|F(\rho^\eta, \rho^\times)^\times\|_{L^\infty_tW^{1,\infty}_\xi\cap W^{1,1}_\xi}\le 8\|\rho^\times\|_{L^\infty_tW^{1,\infty}_\xi\cap W^{1,1}_\xi}\max_{\pm}\|P_\pm\|_{W^{1,\infty}_\xi L^\infty_\zeta}\|\rho^\times\|_{L^\infty_tL^1_\xi}.
\endaligned
\]
There exists then a unique maximal solution $(\rho^0, \rho^1, \rho^\times)\in L^\infty([0, \delta]; W^{1,\infty}(\R^d)\cap W^{1,1}(\R^d))$, where $\delta>0$ depends on the size of the data. 

As concerns the estimate for $\rho^*_n$, we prove by induction that
\[
\sup_{t\in [0,\delta]}\|\rho^*_n(t)\|_{W^{1,\infty}_\xi\cap W^{1,1}_\xi}\le c_n 8^n\delta^n\pth{\max_\eta\|N^\eta\|_{W^{1,\infty}_\xi L^\infty_\zeta} + \max_\pm \|P_\pm\|_{W^{1,\infty}_\xi L^\infty_\zeta}}^n \max_*\pth{\|\rho^*_0\|_{W^{1,\infty}_\xi\cap W^{1,1}_\xi}}^{n+1}
\]
where $c_n$ is the n-th Catalan number. Estimate \eqref{estimate rho* cauchy pb1} will follow by recalling that $c_n\le 4^n$ and choosing $\Lambda$ any fixed constant grater than $32(\max_\eta\|N^\eta\|_{W^{1,\infty}_\xi L^\infty_\zeta} + \max_\pm \|P_\pm\|_{W^{1,\infty}_\xi L^\infty_\zeta}) \max_*(\|\rho^*_0\|_{W^{1,\infty}_\xi\cap W^{1,1}_\xi})$. The above inequality is trivially true for $n=0$ and we suppose it holds true for any $\rho^*_m$ with $0\le m \le n$. From definitions \eqref{rho eta n}, \eqref{rho x n} and the above estimates on $F$, we deduce that for $t\in [0,\delta]$
\[
\aligned
& \|\rho^\eta_{n+1}(t)\|_{W^{1,\infty}_\xi\cap W^{1,1}_\xi} \le 8t\|N^\eta\|_{W^{1,\infty}_\xi L^\infty_\zeta}\sum_{n_1+n_2 = n}\|\rho^\eta_{n_1}\|_{L^\infty_tW^{1,\infty}_\xi\cap W^{1,1}_\xi}\|\rho^\times_{n_2}\|_{L^\infty_t L^1_\xi} \\
& \le 8^{n+1}t\delta^n\|N^\eta\|_{W^{1,\infty}_\xi L^\infty_\zeta}(\|N^\eta\|_{W^{1,\infty}_\xi L^\infty_\zeta} + \max_\pm \|P_\pm\|_{W^{1,\infty}_\xi L^\infty_\zeta})^n \max_*(\|\rho^*_0\|_{W^{1,\infty}_\xi\cap W^{1,1}_\xi})^{n+2} \sum_{n_1+n_2 = n}c_{n_1}c_{n_2} \\
& \le c_{n+1} (8\delta)^{n+1}  (\|N^\eta\|_{W^{1,\infty}_\xi L^\infty_\zeta} + \max_\pm \|P_\pm\|_{W^{1,\infty}_\xi L^\infty_\zeta})^{n+1}\max_*(\|\rho^*_0\|_{W^{1,\infty}_\xi\cap W^{1,1}_\xi})^{n+2}
\endaligned
\]
and similarly
\[
\aligned
\|\rho^\times_{n+1}(t)\|_{W^{1,\infty}_\xi\cap W^{1,1}_\xi} & \le 8t\max_\pm \|P_\pm\|_{W^{1,\infty}_\xi L^\infty_\zeta}\sum_{n_1+n_2=n} \|\rho^\times_{n_1}\|_{L^\infty_t W^{1,\infty}_\xi\cap W^{1,1}_\xi}\|\rho^\times_{n_2}\|_{L^\infty_tL^1_\xi} \\
& \le c_{n+1}(8\delta)^{n+1}  (\|N^\eta\|_{W^{1,\infty}_\xi L^\infty_\zeta} + \max_\pm \|P_\pm\|_{W^{1,\infty}_\xi L^\infty_\zeta})^{n+1}\max_*(\|\rho^*_0\|_{L^\infty_\xi\cap L^1_\xi})^{n+2}.
\endaligned
\]

Up to shortening the time interval to $[0,\delta]$ with $\delta<1/(2\Lambda)$, we immediately obtain from \eqref{estimate rho* cauchy pb1} that the series $\sum_n \rho^*_n$ converges in $L^\infty_t([0,\delta],W^{1,\infty}(\R^d)\cap W^{1,1}(\R^d))$. It remains to show that it converges to $\rho^*$ solution of \eqref{cross system} and to prove \eqref{estimate rho* cauchy pb2}. Let us define $v = (v^0, v^1, v^\times)$ as follows
\[
v^*:= \rho^* - \sum_{n\le N} \rho^*_n, \quad *\in \{0,1,\times\}
\]
From the equations in \eqref{cross system} and Definitions \eqref{rho eta n}, \eqref{rho x n}, we derive that $v$ solves the following system of equations
\[
v^* = \mathfrak{S}^* + \mathfrak{L}^*(v) + \mathfrak{Q}^*(v,v), \quad *\in \{0,1,\times\}
\]
with source term
\[
\aligned
& \mathfrak{S}^\eta = 4\sum_{\substack{0\le n_1, n_2\le N\\ N\le n_1+n_2\le 2N}} \int_0^t d\tau \, \rho_{n_1}^\eta (\tau,\xi) \int d\zeta \mathrm{Im}\pth{N^\eta (\xi,\zeta) \rho_{n_2}^\times (\tau,\zeta)^{\iota(\eta)}} \\
& \mathfrak{S}^\times = 2i\sum_{\substack{0\le n_1, n_2\le N\\ N\le n_1+n_2\le 2N}} \int_{0}^t d\tau \rho_{n_1}^\times (\tau,\xi) \int d\zeta \Big[ P_-(\xi,\zeta) \rho_{n_2}^\times (\tau,\zeta) + P_+ (\xi,\zeta) \overline{\rho_{n_2}^\times(\tau,\zeta)} \Big],
\endaligned
\]
linear term
\[
\aligned
\mathfrak{L}^\eta(v) = &4\sum_{0\le n\le N} \int_0^t d\tau \Big[\rho^\eta_n(\tau,\xi)\int d\zeta \Im \pth{N^\eta(\xi, \zeta) v^\times(\tau, \zeta)^{\iota(\eta)}} + v^\eta(\tau, \xi)\int d\zeta \Im \pth{N^{\eta}(\xi, \zeta)\rho^\times_n(\tau, \zeta)^{\iota(\eta)}}\Big]  \\
\mathfrak{L}^\times(v)  = &  2i\sum_{0\le n\le N}\int_0^t d\tau \rho^\times_n(\tau, \xi)\int d\zeta \Big[P_{-}(\xi, \zeta)v^\times(\tau, \zeta) + P_+(\xi, \zeta)\overline{v^\times(\tau, \zeta)}\Big] \\
& +2i\sum_{n\le N}\int_0^t d\tau v^\times(\tau, \xi)\int d\zeta \Big[P_{-}(\xi, \zeta)\rho^\times_n(\tau, \zeta) + P_+(\xi, \zeta)\overline{\rho^\times_n(\tau, \zeta)}\Big],
\endaligned
\]
and quadratic term
\[
\aligned
\mathfrak{Q}^\eta(v,v) & = 4\int_0^t v^\eta(\tau, \xi) \int d\zeta \Im \pth{ N^\eta(\xi, \zeta) v^\times(\tau, \zeta)^{\iota(\eta)}} \\
\mathfrak{Q}^\times(v,v) & = 2i \int_0^t v^\times(\tau, \xi)\int d\zeta \Big[P_{-}(\xi,\zeta) v^\times(\tau, \zeta) + P_+(\xi, \zeta)\overline{v^\times(\tau,\zeta)}\Big].
\endaligned
\]
From \eqref{estimate rho* cauchy pb1} and the fact that $\Lambda\delta<1/2$, we immediately derive that
\[
\|\mathfrak{S}^\eta\|_{L^\infty_tW^{1,\infty}_\xi\cap W^{1,1}_\xi} + \|\mathfrak{S}^\times\|_{L^\infty_tL^\infty_\xi\cap L^1_\xi}\le 8(\|N^\eta\|_{W^{1,\infty}_\xi L^\infty_\zeta} + \max_\pm\|P_\pm\|_{W^{1,\infty}_\xi L^\infty_\zeta})(\Lambda\delta)^{N+1}\sum_{n=0}^\infty (\Lambda \delta)^n
\]
\[
\aligned
\|\mathfrak{L}^\eta(v)\|_{L^\infty_tW^{1,\infty}_\xi\cap W^{1,1}_\xi} + \|\mathfrak{L}^\times(v)\|_{L^\infty_tW^{1,\infty}_\xi\cap W^{1,1}_\xi} & \le 8 (\|N^\eta\|_{{W^{1,\infty}_\xi L^\infty_\zeta}} + \max_\pm \|P_\pm\|_{{W^{1,\infty}_\xi L^\infty_\zeta}})\|v\|_{L^\infty_t W^{1,\infty}_\xi\cap W^{1,1}_\xi}\sum_{0\le n\le N}(\Lambda\delta)^{n+1}\\
& \le 16\Lambda\delta(\|N^\eta\|_{{W^{1,\infty}_\xi L^\infty_\zeta}} + \max_\pm \|P_\pm\|_{{W^{1,\infty}_\xi L^\infty_\zeta}})\|v\|_{L^\infty_t W^{1,\infty}_\xi\cap W^{1,1}_\xi}
\endaligned
\]
\[
\|\mathfrak{Q}^{\eta}(v,v)\|_{L^\infty_tW^{1,\infty}_\xi\cap W^{1,1}_\xi} + \|\mathfrak{Q}^{\times}(v,v)\|_{L^\infty_tW^{1,\infty}_\xi\cap W^{1,1}_\xi}\le 8\delta (\|N^\eta\|_{{W^{1,\infty}_\xi L^\infty_\zeta}} + \max_\pm \|P_\pm\|_{{W^{1,\infty}_\xi L^\infty_\zeta}})\|v\|^2_{L^\infty_tW^{1,\infty}_\xi\cap W^{1,1}_\xi}.
\]
Up to choosing a smaller $\delta$ as in the statement, we see that the norm of $\mathfrak{L}$ as an operator from $L^\infty_t(W^{1,\infty}_\xi\cap W^{1,1}_\xi)$ to itself is smaller than $1/2$. Therefore $\text{Id}-\mathfrak{L}$ is invertible, the norm of its inverse is bounded by 2 and the equation for $v$ turns into
\[
v = (\text{Id}-\mathfrak{L})^{-1}(\mathfrak{S} + \mathfrak{Q}(v,v)).
\]
Using the estimates recovered above, one can easily check that the operator in the above right hand side is a contraction on the ball 
\begin{align*}
    \enstq{v\in L^\infty_t([0,\de],W^{1,\infty}(\R^d)\cap W^{1,1}(\R^d))}{\| v\|_{L^\infty_t([0,\de], W^{1,\infty}\cap W^{1,1})}\le (\Lambda\delta)^N},
\end{align*}
which implies \eqref{estimate rho* cauchy pb2} and concludes the proof.
\end{proof}

\subsection{Final comparison}\label{section comparison}

In this section, we conclude the proof of Theorem \ref{main theorem} by proving its third point.

For $A>0$ and $L>0$, we denote by $\mathcal E_{L,A}$ the set of measure bigger than $1-L^{-A}$ on which the results of Proposition \ref{prop fixed point} hold. Restricting ourselves to this set, the solution $X^\eta$ of \eqref{fixed point} is (uniquely) well-defined up to times of order $\delta \varepsilon^{-2}$. We compare 
\[
\E\pth{\mathbbm{1}_{\mathcal E_{L,A}}\overline{\widehat{ X^\eta}(\varepsilon^{-2}t,k)} \widehat{X^{\eta'}}(\varepsilon^{-2}t,k)}
\] 
to
\[
\rho^{[\eta,\eta']}(t,k)
\]
where $[\eta,\eta'] = \eta$ if $\eta=\eta'$ or $[\eta,\eta'] = (\eta,\eta') = \times$ if $(\eta,\eta') = (0,1) = \times$, and where $t\in [0,\delta]$ and $k\in \Z_L^d$. We remark that $\rho^{[\eta,\eta']}(t,k)$ is well-defined since $\rho^*$ is continuous in time and phase space. More precisely, we prove that there exists $\nu>0$ such that
\begin{align*}
    \left| \E\Big(\mathbbm{1}_{\mathcal E_{L,A}}\overline{\widehat{X^\eta}(\varepsilon^{-2}t,k)} \widehat{X^{\eta'}}(\varepsilon^{-2}t,k)\Big) - \rho^{[\eta,\eta']}(t,k)  \right| \lesssim L^{-\nu},
\end{align*}
for all $t\in[0,\de]$ and all $k\in\Z_L^d$.

We have that 
\begin{align}\non
\E\Big(\mathbbm{1}_{\mathcal E_{L,A}}&\overline{\widehat{X^\eta}(\varepsilon^{-2}t,k)} \widehat{X^{\eta'}}(\varepsilon^{-2}t,k)\Big) - \rho^{[\eta,\eta']}(t,k)  \\\label{eq:remainder}
= & \E\pth{\mathbbm{1}_{\mathcal E_{L,A}}\overline{\widehat{ X^\eta}(\varepsilon^{-2}t,k)} \widehat{ X^{\eta'}}(\varepsilon^{-2}t,k)}  - \E\pth{\mathbbm{1}_{\mathcal E_{L,A}}\sum_{n_1,n_2 \leq N(L)}\overline{\widehat{X_{n_1}^\eta}(\varepsilon^{-2}t,k)} \widehat{ X_{n_2}^{\eta'}}(\varepsilon^{-2}t,k)} \\ \label{eq:smallmeasure}
+ & \E\pth{\mathbbm{1}_{\mathcal E_{L,A}}\sum_{n_1,n_2 \leq N(L)}\overline{\widehat{ X_{n_1}^\eta}(\varepsilon^{-2}t,k)} \widehat{ X_{n_2}^{\eta'}}(\varepsilon^{-2}t,k)} - \E\pth{\sum_{n_1,n_2 \leq N(L)}\overline{\widehat{ X_{n_1}^\eta}(\varepsilon^{-2}t,k)} \widehat{ X_{n_2}^{\eta'}}(\varepsilon^{-2}t,k)} \\\label{eq:resonantcouplings}
+ & \E\pth{\sum_{n_1,n_2 \leq N(L)}\overline{\widehat{ X_{n_1}^\eta}(\varepsilon^{-2}t,k)} \widehat{ X_{n_2}^{\eta'}}(\varepsilon^{-2}t,k)} - \sum_{n \leq N(L)}\Xres{n}^{+,(\eta,\eta')}(t,k) \\\label{eq:Riemannsumandcutoffs}
+ & \sum_{n \leq  N(L)}\rho_{L,n}^{[\eta,\eta']}(t,k) - \sum_{n \leq  N(L)}\rho_{n}^{[\eta,\eta']}(t,k)\\\label{eq:wellposedness}
+ & \sum_{n \leq  N(L)}\rho_{n}^{[\eta,\eta']}(t,k) - \rho^{[\eta,\eta']}(t,k)
\end{align}
with $\Xres{n}^{+,(\eta,\eta')}=\rho_{L,n}^{[\eta,\eta']}$ is introduced in Notation \ref{notation X res n} and expressed recursively in Proposition \ref{recursive expression rhoLn}, and $\rho_{n}^{[\eta,\eta']}$ is defined in \eqref{rho eta n}, \eqref{rho x n}. We discuss how to bound each of the above lines.

\paragraph{Line \eqref{eq:remainder}: smallness due to the remainder.} Using the result of Theorem \ref{thm X as series}, it corresponds to
\[
\E \pth{\mathbbm{1}_{\mathcal E_{L,A}} \overline{\widehat{{W}^\eta_1}(\e^{-2}t, k)}\widehat{{W}_2^{\eta'}}(\e^{-2}t,k)}
\]
where $W_1, W_2 \in \{ X_{\le N(L)}, v\}$, with at least one of the two equal to $v$. Using the estimates derived in Corollary \ref{coro Xn} and Theorem \ref{thm X as series}, together with the fact that $H^s$ is an algebra when $s>d/2$, we derive that
\[
\E \pth{\mathbbm{1}_{\mathcal E_{L,A}} \overline{\widehat{{W}^\eta_1}(\e^{-2}t, k)}\widehat{{W}_2^{\eta'}}(\e^{-2}t,k)} \lesssim L^{-\pth{d+\frac14+2A}},
\]
for some implicit constant independent of $ L$.

\paragraph{Line \eqref{eq:smallmeasure}: smallness due to high probability.} It corresponds to
\[
\sum_{n_1,n_2 \leq N(L)} \E\Big(\mathbbm{1}_{\mathcal E^\complement_{L,A}}\overline{\widehat{X_{n_1}^\eta}(\varepsilon^{-2}t,k)} \widehat{X_{n_2}^{\eta'}}(\varepsilon^{-2}t,k)\Big) 
\]
where $\mathcal E^\complement_{L,A}$ denotes the complement of $\mathcal E_{L,A}$. We estimate the above terms using H\"older's inequality twice and Gaussian hypercontractivity. Recalling also estimate \eqref{expectation Xn(k)}, we get that the above sum is
\[
\aligned
& \le \sum_{n_1, n_2\le N(L)}L^{-\frac{A}{2}} \Big(\E\big|\overline{\widehat{X_{n_1}^\eta}(\varepsilon^{-2}t,k)} \widehat{X_{n_2}^{\eta'}}(\varepsilon^{-2}t,k)\big|^2\Big)^{\frac12} \\
&\le \sum_{n_1, n_2\le N(L)}L^{-\frac{A}{2}} (\E|\widehat{X^\eta_{n_1}}(\e^{-2}t,k)|^4)^\frac14 (\E|\widehat{X^\eta_{n_2}}(\e^{-2}t,k)|^4)^\frac14\\
&\le  L^{-\frac{A}{2}} \sum_{n_1, n_2\le N(L)} c^{n_1}c^{n_2}(\E|\hat{X}^\eta_{n_1}(\e^{-2}t,k)|^2)^\frac12(\E|\hat{X}^\eta_{n_2}(\e^{-2}t,k)|^2)^\frac12 \\
& \le \Lambda  L^{-\frac{A}{2}} \pth{\sum_{n\le N(L)} (\Lambda\delta)^\frac{n}{4} c^n }^2
\endaligned
\]
for some $\Lambda, c>0$. If $\delta$ is small enough so that $\Lambda c^4\delta<1/2$, we conclude that 
\[
\sum_{n_1,n_2 \leq N(L)} \E\Big(\mathbbm{1}_{\mathcal E^\complement_{L,A}}\overline{\widehat{X_{n_1}^\eta}(\varepsilon^{-2}t,k)} \widehat{X_{n_2}^{\eta'}}(\varepsilon^{-2}t,k)\Big) \leq \Lambda L^{-\frac{A}{2}},
\]
for some new $\Lambda>0$.

\paragraph{Line \eqref{eq:resonantcouplings}: smallness due to resonances.} From \eqref{prop:correlation}, it corresponds to 
\[
\sum_{n_1, n_2\le N(L)}\sum_{C \in \mathcal{C}^{\eta,\eta', -,+}_{n_1,n_2}}\widehat{{F}_C}(\e^{-2}t,k) - \sum_{n\le N(L)}\sum_{C \in \mathtt{Res}_n(\iota,-\iota,\eta,\eta')} \widehat{F_C}(\e^{-2}t,k)
\]
where we recall that $\mathtt{Res}_n(\iota,\iota',\eta,\eta') = \mathtt{Res}_n \cap \cup_{n_1n_2} \mathcal C^{\iota,\iota',\eta,\eta'}_{n_1,n_2}$ and $\mathtt{Res}_n$ is the set of couplings with $n$ nodes, all resonant. The above difference can be rearranged as follows
\begin{equation}\label{eq:estimatingResCoup}
\sum_{n_1, n_2\le N(L)}\sum_{\substack{C\in \mathcal{C}^{\eta, \eta', -, +}_{n_1, n_2}\\ C \text{ has a non-resonant node}}} \widehat{F_C}(\e^{-2}t, k) - \sum_{\substack{n_1+n_2\le 2N(L)\\ n_1>N(L) \text{ or } n_2> N(L)}}\sum_{C\in \mathcal{C}^{\eta, \eta', -, +}_{n_1, n_2} \cap \cup_n \mathtt{Res}_n}\widehat{F_C}(\e^{-2}t, k).
\end{equation}
On the one hand, the first sum in \eqref{eq:estimatingResCoup} can be split into three additional terms:
\begin{align*}
S_1 & = \sum_{n_1+n_2<80} \sum_{\substack{C\in \mathcal{C}^{\eta, \eta', -, +}_{n_1, n_2}\\ C \text{ has a non-resonant node}}} \widehat{F_C}(\e^{-2}t, k),
\\ S_2 & = \sum_{\substack{n_1, n_2\le N(L)\\ n_1+n_2\ge 80}}\sum_{\lfloor \frac{n_1+n_2}{2}\rfloor-40<q\le \lfloor \frac{n_1+n_2}{2}\rfloor} \sum_{\substack{C\in \mathcal{C}^{\eta, \eta', -, +}_{n_1, n_2}\\ C \text{ has a non-resonant node}\\ n_{\scb}(C)=q}} \widehat{F_C}(\e^{-2}t, k),
\\ S_3  & = \sum_{\substack{n_1, n_2\le N(L)\\ n_1+n_2\ge 80}}\sum_{0\le q\le \lfloor \frac{n_1+n_2}{2}\rfloor-40} \sum_{\substack{C\in \mathcal{C}^{\eta, \eta', -, +}_{n_1, n_2}\\ C \text{ has a non-resonant node}\\ n_{\scb}(C)=q}} \widehat{F_C}(\e^{-2}t, k)
\end{align*}
We use Proposition \ref{prop nonresonant} and Lemma \ref{counting self-coupled bushes} (which counts the couplings with a given number of self-coupled bushes) to deduce that $S_1$ and $S_2$ are bounded by $\Lambda \e^\alpha$ for some large constant $\Lambda$, while we use \eqref{estimate FC} and again Lemma \ref{counting self-coupled bushes} to get the same bound for $S_3$ (cf. the same estimate in the proof of Corollary \ref{coro Xn}).

On the other hand, the second sum in \eqref{eq:estimatingResCoup} may be estimated using that if $C \in \mathtt{Res}_n$ then $n_\scb(C) = \frac{n}2$ and therefore the cardinal of $\mathtt{Res}_n(\iota,\iota',\eta,\eta')$ is bounded by some $\Lambda^n$. We deduce that it can be estimated using \eqref{estimate good} and the fact that either $n_1$ or $n_2$ are bigger than $N(L)$ as less than
\[
\aligned
S'_1 & = \sum_{\substack{n_1+n_2\le 2N(L)\\ n_1>N(L) \text{ or } n_2> N(L)}} \Lambda (\Lambda \delta)^{\frac{n_1+n_2}{2}} \le \Lambda (\Lambda\delta)^{\frac{N(L)}{2}}
\endaligned
\]
Given that $\Lambda\delta<1$ and that $N(L) = \lfloor{(\log L)\rfloor}$, $(\Lambda\delta)^{\frac{N(L)}{2}}$ is smaller than a negative power of $L$.

\paragraph{Line \eqref{eq:Riemannsumandcutoffs}: smallness due to recursive properties of resonances.} We recall the recursive expression obtained for $\rho^{[\eta, \eta']}_{L, n}$ in Proposition \ref{recursive expression rhoLn} as well as the definition \eqref{rho eta n}-\eqref{rho x n} of $\rho^{[\eta, \eta']}_{n}$. It can be easily proved by induction that the maps $\xi\mapsto \rho^{[\eta, \eta']}_{L,n}(t,\xi)$ and $\xi\mapsto \rho^{[\eta, \eta']}_{n}(t,\xi)$ are compactly supported in a ball $B(0, R)$. Furthermore, for any $n\in \N$ and $L>0$
\begin{equation} \label{estimate rhoLn}
    \sup_{t\in [0, \delta]}\sup_{k\in \mathbb{Z}^d_L}|\rho^{[\eta, \eta']}_{L, n}(t, k)|\le \Lambda (\Lambda\delta)^n c_n,
\end{equation}
where $\Lambda\gg 1$ is some fixed constant and $c_n$ denotes the $n$th Catalan number. 
Such estimate is the analogue of \eqref{estimate rho* cauchy pb1} and obtained using the same type of argument.

We next prove by induction on $n\le N(L)$ that there exists some $\Lambda, \alpha>0$ such that
\begin{equation}\label{estimate difference rhoLn rhon}
    \sup_{t\in [0, \delta]}\sup_{k\in \mathbb Z^d_L}|\rho^{[\eta, \eta']}_{L, n}(t,k) - \rho^{[\eta, \eta']}_{n}(t,k)| \le \Lambda L^{-\alpha} c_n (\Lambda \delta)^n.
\end{equation}
From this estimate it will follow that line \eqref{eq:Riemannsumandcutoffs} is bounded by $\Lambda L^{-\alpha}(\log L)$, for some new $\Lambda>0$.
For $n=0$, the two quantities coincide so let us assume that the above statement is true for any $0\le m\le n-1$. First, we observe that
\[
\aligned
&\rho^{[\eta, \eta]}_{L, n}(t, k)  = 4\sum_{n_1+n_2=n-1}\int_0^t \rho^{[\eta, \eta]}_{L,n_1}(\tau, k) L^{-d}\sum_{k_2}\Im \pth{N^{[\eta, \eta]}(k_2, k) \rho^\times_{n_2,L}(\tau, k_2)^{\iota(\eta)}}\ d\tau \\
& + 4\sum_{n_1+n_2=n-1}\int_0^t \rho^{[\eta, \eta]}_{L, n_1}(\tau, k) L^{-d}\sum_{k_2}(1-\chi_\lowr)(k_2, -k_2, k)\Im \pth{N^{[\eta, \eta]}(k_2, k) \rho^\times_{L, n_2}(\tau, k_2)^{\iota(\eta)}}\ d\tau
\endaligned\]
and that $(1-\chi_\lowr)(k_2, -k_2, k)$ vanishes when $|k\pm k_2|\ge 2\e^{\gamma}$. Therefore, by triangle inequality we have 
\begin{align}\non
\Big| &\rho^{[\eta, \eta]}_{L, n}(t, k)  - \rho^{[\eta, \eta]}_{ n}(t, k) \Big| \leq \\ 
& 4 \Big|\sum_{n_1+n_2=n-1}\int_0^t (\rho^{[\eta, \eta]}_{L,n_1}(\tau, k) - \rho^{[\eta, \eta]}_{n_1}(\tau, k)) L^{-d}\sum_{k_2}\Im \pth{N^{[\eta, \eta]}(k_2, k) \rho^\times_{n_2,L}(\tau, k_2)^{\iota(\eta)}}\ d\tau \Big|\label{line 1} \\
& + 4 \Big|\sum_{n_1+n_2=n-1}\int_0^t  \rho^{[\eta, \eta]}_{n_1}(\tau, k) L^{-d}\sum_{k_2}\Im \pth{N^{[\eta, \eta]}(k_2, k) (\rho^\times_{n_2,L}(\tau, k_2)^{\iota(\eta)} - \rho^\times_{n_2}(\tau, k_2)^{\iota(\eta)} )}\ d\tau \Big| \label{line 2}\\
& +4 \Big|\sum_{n_1+n_2=n-1}\int_0^t  \rho^{[\eta, \eta]}_{n_1}(\tau, k) L^{-d}\sum_{k_2}\Im \pth{N^{[\eta, \eta]}(k_2, k)  \rho^\times_{n_2}(\tau, k_2)^{\iota(\eta)} }\ d\tau  \label{line 3} \\\non
& \hspace{1cm} - \sum_{n_1+n_2=n-1}\int_0^t  \rho^{[\eta, \eta]}_{n_1}(\tau, k) \int d\zeta \Im \pth{N^{[\eta, \eta]}(\zeta, k)  \rho^\times_{n_2}(\tau, \zeta)^{\iota(\eta)} }\ d\tau  \Big| \\
& + 4\Big|\hspace{-0.4cm}\sum_{n_1+n_2=n-1}\int_0^t \rho^{[\eta, \eta]}_{L, n_1}(\tau, k) L^{-d}\sum_{k_2}(1-\chi_\lowr)(k_2, -k_2, k)\Im \pth{N^{[\eta, \eta]}(k_2, k) \rho^\times_{L, n_2}(\tau, k_2)^{\iota(\eta)}}\ d\tau \Big| \label{line 4}
\end{align}
To bound \eqref{line 1} and \eqref{line 2} we use the induction hypothesis, the fact that $\rho^{[\eta, \eta']}_{L, m}$ and $\rho^{[\eta, \eta']}_{m}$ are bounded by $ \Lambda^{m+1} \delta^m c_m$ (cf. \eqref{estimate rhoLn} and \eqref{estimate rho* cauchy pb1}) and that the sum over $k_2$ is restricted to a ball $B(0, R)$ of $\mathbb Z^d_L$. For \eqref{line 3}, we remark that $\zeta\mapsto N^{[\eta, \eta']}(\zeta, k)\Im \rho^\times_{n}(t, \zeta)$ satisfies the hypothesis of Lemma \ref{lem:RiemanntoInt}, hence using \eqref{estimate rho* cauchy pb1} we deduce that for $n_2\le n-1$
\begin{multline*}
\sup_{t\in [0, \delta]}\sup_{k\in \mathbb Z^d_L}\left| L^{-d}\sum_{k_2}\Im\Big(N^{[\eta, \eta']}(k_2, k)\rho^\times_{n_2}(\tau, k_2)\Big) - \int_{\mathbb R^d} d\zeta \ \Im\Big(N^{[\eta, \eta']}(\zeta, k)\rho^\times_{n_2}(\tau, \zeta)\Big)\right| \\
\lesssim L^{-d-1}\sum_{k_2\in \mathbb Z^d_L}\sup_{C_{k_2, L}} |\nabla\Big(N^{[\eta, \eta']}(\cdot, k )\rho^\times_{n_2}(\cdot)\Big)|\le L^{-1}\|N^{[\eta, \eta']}\|_{W^{1,\infty}_\xi L^\infty_\zeta} \Lambda (\Lambda\delta)^{n_2}c_{n_2}  R^d.
\end{multline*}
Plugging this into \eqref{line 3}, we find it is bounded by
\[
4\sum_{n_1+n_2=n-1} \Lambda (\Lambda\delta)^{n_1+n_2+1}c_{n_1}c_{n_2} L^{-1}\|N^{[\eta, \eta']}\|_{W^{1,\infty}_\xi L^\infty_\zeta} R^d 
\]
and finally by $\Lambda(\Lambda\delta)^n c_n L^{-\alpha}$ given the fact that $4L^{-1+\alpha}\|N^{[\eta, \eta']}\|_{W^{1,\infty}_\xi L^\infty_\zeta} R^d \le 1$ if $L$ is sufficiently large and $\alpha<1$.

Finally, to bound \eqref{line 4} we use estimate \eqref{estimate rhoLn} on $\rho^*_{L,n}$ and the fact that $(1-\chi_\lowr)(k_2, -k_2, k) = 0$ whenever $|k\pm k_2|\ge 2\e^{\gamma d}$ and derive that for $L$ sufficiently large
\[
\le \Lambda(\Lambda\delta)^n 2^d\e^{\gamma d}c_n \|N^{[\eta, \eta']}\|_{L^\infty}\le \Lambda (\Lambda\delta)^n L^{-\alpha},
\]
for some $\alpha>0$. This concludes the proof of \eqref{estimate difference rhoLn rhon}.

\paragraph{Line \eqref{eq:wellposedness}: smallness due to the iterative scheme.} It follows from \eqref{estimate rho* cauchy pb2} that this line is bounded by $(\Lambda \delta)^{2N(L)+2}$. Since $\Lambda\delta<1/2$, this tends to zero uniformly in $t, k$ as $L\rightarrow\infty$.


\appendix

\section{Proof of Proposition \ref{prop LukSpo}}\label{appendix LukSpo}

We give the different steps in the proof of the Section 5 of \cite{LS11} to get this result.

\noindent\textbf{Construction of the momentum graph of $C$.} Given a coupling $C$ and an order $\rho$, we build a \emph{momentum graph} $G$ in the following way:
\begin{itemize}
    \item The vertices of $G$ are the vertices of $C$ (its nodes and leaf vertices), to which we add a vertex that we call a \emph{root vertex} and $n$ vertices that we call \emph{fusion vertex}. 
    \item The edges of $G$ are the edges of $C$ to which we add two edges from the roots of $C$ to the root vertex and $n+2$ edges build in the following way: we associate to each minus leaf $\ell$ a fusion vertex $v_\ell$ in a bijective way, and we add the edges $\{\ell,v_\ell\}$ and $\{\sigma(\ell),v_\ell\}$.
\end{itemize}
 Note that $G$ is a connected graph. We write its vertex set $V$ and its edges set $E$.

\saut
\noindent\textbf{Signed vertices of the momentum graph.} For each vertex $v\in V$, we write $E_+(v)$ and $E_-(v)$ the following collections of edges: 
\begin{itemize}
    \item If $v$ is the root vertex, $E_+(v) = \emptyset$ and $E_-(v)$ are the edges from $v$ to the roots of $C$.
    \item If $v$ is one of the roots of $C$, $E_+(v)$ contains only the edge connecting $v$ to the root vertex and $E_-(v)$ is the collection of the edges connecting $v$ to its children in $C$.
    \item If $v$ is a node vertex that is not a root of $C$, $E_+(v)$ is the set containing only the edge connecting $v$ to its parent in $C$ and $E_-(v)$ is the collection of edges connecting $v$ to its children in $C$.
    \item If $v$ is a leaf vertex, $E_+(v)$ is the set containing only the edge connecting $v$ to its parent in $C$ and $E_-(v)$ contains only the unique edge connecting $v$ to its fusion vertex.
    \item If $v$ is a fusion vertex, $E_+(v)$ is the collection of edges connecting $v$ to its two leaves and $E_-(v)=\emptyset$.
\end{itemize}

\saut\noindent\textbf{Kirchoff rule.} If $j$ is a node or a leaf vertex of $C$, we set $e(j)$ the unique edge in $E_+(j)$. Note that each edge of $G$ is either a $e(j)$ or an edge connecting a leaf vertex to a fusion vertex. We associate elements of $V$ to the edges of $G$ in the following way:
\begin{itemize}
    \item If the edge $e\in E$ is of the form $e=e(j)$ for some $j$ then define $\mathfrak K(e)=K(j)$.
    \item If $e$ is connecting a leaf $\ell$ to its fusion vertex then define $\mathfrak K(e)=K(\ell)$.
\end{itemize}
The relationships between the $(K(j))_{j\in \mathcal N(C) \sqcup \mathcal L(C)}$ of Definition \ref{def linear map} can be expressed as a \emph{Kirchhoff rule} at each vertex $v\in V$:
\begin{align*}
\sum_{e\in E_+(v)} \mathfrak K(e) = \sum_{e\in E_-(v)} \mathfrak K(e).    
\end{align*}
Finally, we define an order on the set of edges $E$ of $G$. Define the map $\rho_G$ in the following way
\begin{itemize}
    \item If $e\in E$ is of the form $e(j)$ set $\rho_G(e)=\rho^{-1}(j)$.
    \item Denote $E_{fusion}$ the set of edges connecting a leaf vertex of $C$ to a fusion vertex, note that $|E_{fusion}|=n+2$. Define $\rho_{G|_{E_{fusion}}}$ to be any bijection from $E_{fusion}$ to $\llbracket -n-1,0\rrbracket$.
\end{itemize}
This defines a bijection $\rho_G:E\rightarrow \llbracket -n-1,n(C)+n+2\rrbracket$ and thus an order denoted $<_G$ with the following properties: 
\begin{itemize}
    \item[(i)] if $e\in E_{fusion}$ then $e<_G e(j)$ for all $j$ which is not the root vertex or any fusion vertex,
    \item[(ii)] if $j,j'$ are not the root vertex nor any fusion vertices then
    \begin{align*}
        j<_{\rho}j' \Longrightarrow j<_G j'.
    \end{align*}
\end{itemize}

\saut\noindent\textbf{Spanning tree of the momentum graph.} We build the \emph{spanning tree} of $G$. For this, we enumerate the edges of $G$, in the appropriate order
\[
e_1 <_G e_2 <_G \hdots <_G e_R
\]
where $R=n(C)+2n+4$ is the cardinal of $E$. We build a sequence of connected graphs $G_0 =(V_0,E_0)$, $G_1 =(V_1,E_1)$ up to $G_R=(V_R,E_R)$ such that $E_0 = V_0= \emptyset$ and $G_{r+1}$ is obtained from $G_r$ in the following way:
\begin{itemize}
    \item If adding the edge $e_{r+1}$ to $G_r$ creates a loop in $G_r$, we set $G_{r+1} = G_r$.
    \item If adding the edge $e_{r+1}$ to $G_r$ does not create a loop in $G_r$, we set $E_{r+1} = E_r \sqcup \{e_{r+1}\}$ and $V_{r+1} = V_r \cup \{v,v'\}$ where $e_{r+1} = \{v,v'\}$ in $G$.
\end{itemize}
This builds a graph $G_R=(E_R,V_R)$ with the following properties:
\begin{itemize}
    \item[(i)] By definition, $G_R$ contains no loop. Moreover, since $G$ is connected, so is $G_R$. Therefore, $G_R$ is a tree.
    \item[(ii)] Since $G$ is connected, we have $V_R=V$. 
    \item[(iii)] The property $V_R=V$ implies that $G_R$ is a maximal (for the inclusion order) tree included in $G$. The number of vertices of $G_R$ is by definition $4+n(C)+\frac32 n$. Since it is the tree, its number of edges, the cardinal of $E_R$ is $3+n(C)+\frac32 n$. Finally, the cardinal of $E$ is by definition $R = n(C)+2n+4$. Therefore, if $F \vcentcolon = E \setminus E_R$ then $|F|=\frac{n+2}{2}$. 
\end{itemize}
We root $G_R$ by choosing as its root the root vertex. This induces a (partial) order of parentality on $G_R$ (\emph{a priori} not similar to the order of parentality in $C$ or the order that we put on the edges). We denote this order on $V$ by $<_{G_R}$.

\saut\noindent\textbf{Generating family for $(\mathfrak K(e))_{e\in E}$.}
We claim that for any $e\in E$, the linear map $\mathfrak K(e)$ is a signed combination of $(\mathfrak K(f))_{f\in F}$. For this, we set $E(v) = E_+(v)\sqcup E_-(v)$ and define $\sigma_v(e)$ for $e\in E(v)$ as $e\in E_{\sigma_v(e)}(v)$. This allows to write the Kirchoff rule as
\begin{align}\label{kirchoff}
\forall v\in V, \quad \sum_{e\in E(v)} \sigma_v(e) \mathfrak K(e) = 0.
\end{align}
For all $v\in V$, we set $\mathcal P(v) = \{v' \in V\; |\; v'\leq_{G_R} v\}$. If $e\in E_R$ then $e$ is an edge in $G_R$ therefore $e$ writes $\{v,u\}$ with $v<_{G_R} u$, we claim that
\begin{equation}\label{eq:ke}
\mathfrak K(e) = -\sigma_v(e)\sum_{v'\leq_{G_R} v} \; \sum_{f\in  E(v')\cap F} \sigma_{v'}(f)\mathfrak K(f).
\end{equation}
This is proved by induction on the cardinal of $\mathcal P(v)$:
\begin{itemize}
    \item If $\mathcal P(v)$ is reduced to $\{v\}$, this means that $v$ is a leaf in $G_R$. Therefore we have that $E(v) \cap E_R = \{e\}$. In other words, $E(v)\cap F = E(v)\smallsetminus \{e\}$. The formula \eqref{eq:ke} then simply follows from the Kirchoff rule \eqref{kirchoff} at the vertex $v$.  
    \item Otherwise, the Kirchoff rule at $v$ yields
\begin{align*}
\mathfrak K(e) = -\sigma_v(e) \sum_{e' \in E(v)\setminus\{e\}} \sigma_v(e')\mathfrak K(e').
\end{align*}
Since $e\in E_R$, we can divide this sum as 
\begin{align*}
\mathfrak K(e) = -\sigma_v(e) \sum_{e' \in (E(v)\cap E_R)\setminus\{e\} } \sigma_v(e') \mathfrak K(e') - \sigma_v(e) \sum_{f\in E(v)\cap F}\sigma_{v}(f)\mathfrak K(f) .
\end{align*}
If $e'$ belongs to $E(v)\smallsetminus (F\sqcup\{e\}) = E(v) \cap E_R$, since $v$ can only have one parent in $G_R$, this implies that $e' = \{v',v\}$ with $v' <_{G_R} v$. We deduce by transitivity of $<_{G_R}$, that $\mathcal P(v') \subsetneq \mathcal P(v)$ and thus we can apply the induction hypothesis to $e'$. We deduce
\begin{align}
\mathfrak K(e) & = -\sigma_v(e) \sum_{e' \in (E(v)\cap E_R)\setminus\{e\} } \sigma_v(e') \pth{-\sigma_{v'}(e')\sum_{v''\leq_{G_R} v'} \; \sum_{f\in  E(v'')\cap F} \sigma_{v''}(f)\mathfrak K(f)} \label{intermediate}
\\&\quad - \sigma_v(e) \sum_{f\in E(v)\cap F}\sigma_{v}(f)\mathfrak K(f) . \nonumber
\end{align}
Note that by definition we have $\sigma_v(e')=-\sigma_{v'}(e')$ and thus $\sigma_v(e') (-\sigma_{v'}(e')) = 1$. Finally, note that summing over the $v''\leq_{G_R}v'$ for each child $v'$ of $v$ is equivalent to summing over the $v''<_{G_R}v$. This implies
\begin{align*}
    \sum_{e' \in (E(v)\cap E_R)\setminus\{e\} } \sigma_v(e') \pth{-\sigma_{v'}(e')\sum_{v''\leq_{G_R} v'} \; \sum_{f\in  E(v'')\cap F} \sigma_{v''}(f)\mathfrak K(f) }=  \sum_{v'<_{G_R} v} \; \sum_{f\in  E(v')\cap F} \sigma_{v'}(f)\mathfrak K(f).
\end{align*}
Plugging this into \eqref{intermediate} gives the desired formula. 
\end{itemize}
We have proved that \eqref{eq:ke} holds, i.e that each $\mathfrak K(e)$ is a linear combination of the $(\mathfrak K(f))_{f\in F}$. We can prove in addition that each $\mathfrak K(e)$ is a signed combination of the $(\mathfrak K(f))_{f\in F}$. For this, let us prove that a given $\mathfrak K(f)$ can only appear once in \eqref{eq:ke}. If $f\in F$ belongs to $E(v')\cap E(v'')$ for some $v'\leq_{G_R} v$ and $v''\leq_{G_R} v$ with $v'\neq v''$ then without loss of generality we can assume that $f_0\in E_+(v')\cap E_-(v'')$. This implies that $\si_{v'}(f)=-\si_{v''}(f)$ so that the contributions of $\mathfrak K(f)$ in \eqref{eq:ke} cancel. This proves indeed that each $\mathfrak K(e)$ is a signed combination of the $(\mathfrak K(f))_{f\in F}$.

\saut\noindent\textbf{Triangular structure for $(\mathfrak K(e))_{e\in E}$.}
We claim that for any $e\in E$, the linear map $\mathfrak K(e)$ is a signed combination of $(\mathfrak K(f))_{f\in F,f>_G e}$. Assume $e\in E_R$, write $e=\{v,u\}$ with $v<_{G_R} u$ and take $f<_G e$ such that $\mathfrak K(f)$ appears in the formula \eqref{eq:ke}, that is, that there exists $v'\leq_{G_R} v <_{G_R} u$ such that $v'$ is an extremity of $f$. Write $f=\{v',u'\}$, the vertex $v'$ and $u'$ are not necessarily comparable in $G_R$. There exists a unique path between $u'$ and $v$ in $G_R$. If $u'>_{G_R} v$ or if $u'$ is incomparable in $G_R$ to $v$ then the path between $u'$ and $v$ necessarily pass by their smallest common ancestor and thus by $u$. This implies that the path from $u'$ to $v$ concatenated with the path from $v$ to $v'$ and $f$ forms a loop that contains $e$. But this is not possible because $e>_G f$, which means that when we decided whether to add $f$ in the spanning tree, $e$ was not in the graph and thus $f$ was not closing this loop (and not any other, otherwise there would be a loop in $G_R$). Therefore $u'\leq_{G_R} v$. But since it is so, $\mathfrak K(f)$ appears twice in the formula \eqref{eq:ke}. As we have already noted, in this case, the contributions where $\mathfrak K(f)$ is involved cancel each other. We deduce that $\mathfrak K(e)$ is a signed combination of $(\mathfrak K(f))_{f\in F,f>_G e}$.

\saut\noindent\textbf{Final properties of $(\mathfrak K(f))_{f\in F}$.} Since the cardinal of $F$ is $\frac{n+2}{2}$, since the family $(\mathfrak K(f))_{f\in F}$ generates the vector subspace $V$ of $\mathcal L(\R^{d(n+2)/2}, \R^d)$ of dimension $\frac{n+2}{2}$, we deduce that $(\mathfrak K(f))_{f\in F}$ is a basis of $V$ and thus the $(\mathfrak K(f))_{f\in F}$ are linearly independant.

It remains to prove that if $f$ belongs to $F$, then there exists $j\in \mathcal N(C)\sqcup \mathcal L(C)$ such that $f=e(j)$. This is due to the construction of $G_R$. Since the edges that are not of this form are smaller than all the other, we start by deciding whether to add them to the spanning tree. But they are edges connecting a leaf to a fusion vertex and those by themselves can never form a loop. This means that they all belong to $E_R$ and thus all the $f\in F$ are of the form $e(j)$. We deduce that for all $j$, the linear map $K(j) = \mathfrak K(e(j))$ is a signed combination of the family $(K(q))_{e(q) \in F, q>_\rho j }$. Setting $\mathcal N'(C) = \{j\in \mathcal N(C)\cup \mathcal L(C)\; |\; e(j) \in F\}$ concludes the proof of Proposition \ref{prop LukSpo}.


\newcommand{\etalchar}[1]{$^{#1}$}


\begin{thebibliography}{DKMV23}

\bibitem[ABC{\etalchar{+}}23]{beck}
Gabriel~B. Apolin{\'a}rio, Geoffrey~M. Beck, Laurent Chevillard, Isabelle
  Gallagher, and Ricardo Grande.
\newblock A {Linear} {Stochastic} {Model} of {Turbulent} {Cascades} and
  {Fractional} {Fields}.
\newblock Preprint, {arXiv}:2301.00780 [math-ph] (2023), 2023.

\bibitem[ACG21]{ampatzoglou2021derivation}
Ioakeim Ampatzoglou, Charles Collot, and Pierre Germain.
\newblock Derivation of the kinetic wave equation for quadratic dispersive
  problems in the inhomogeneous setting.
\newblock {\em arXiv preprint arXiv:2107.11819}, 2021.

\bibitem[BDNY24]{bringmann2024invariant}
Bjoern Bringmann, Yu~Deng, Andrea~R Nahmod, and Haitian Yue.
\newblock Invariant gibbs measures for the three dimensional cubic nonlinear
  wave equation.
\newblock {\em Inventiones mathematicae}, 236(3):1133--1411, 2024.

\bibitem[BGHS21]{BGHS}
T.~Buckmaster, P.~Germain, Z.~Hani, and J.~Shatah.
\newblock Onset of the wave turbulence description of the longtime behavior of
  the nonlinear {S}chr\"{o}dinger equation.
\newblock {\em Invent. Math.}, 225(3):787--855, 2021.

\bibitem[Bou94]{bouPer}
J.~Bourgain.
\newblock Periodic nonlinear {Schr{\"o}dinger} equation and invariant measures.
\newblock {\em Commun. Math. Phys.}, 166(1):1--26, 1994.

\bibitem[BP56]{Brout-Prigo}
R.~Brout and I.~Prigogine.
\newblock Statistical mechanics of irreversible processes part viii: general
  theory of weakly coupled systems.
\newblock {\em Physica}, 22(6):621--636, 1956.

\bibitem[BSB66]{Benney}
D.~J. Benney, Philip~Geoffrey Saffman, and George~Keith Batchelor.
\newblock Nonlinear interactions of random waves in a dispersive medium.
\newblock {\em Proceedings of the Royal Society of London. Series A.
  Mathematical and Physical Sciences}, 289(1418):301--320, 1966.

\bibitem[BT08a]{burq1}
Nicolas Burq and Nikolay Tzvetkov.
\newblock Random data {C}auchy theory for supercritical wave equations. {I}.
  {L}ocal theory.
\newblock {\em Invent. Math.}, 173(3):449--475, 2008.

\bibitem[BT08b]{burq2}
Nicolas Burq and Nikolay Tzvetkov.
\newblock Random data {C}auchy theory for supercritical wave equations. {II}.
  {A} global existence result.
\newblock {\em Invent. Math.}, 173(3):477--496, 2008.

\bibitem[BT24]{bruned2024cancellations}
Yvain Bruned and Leonardo Tolomeo.
\newblock Cancellations for dispersive pdes with random initial data.
\newblock {\em arXiv preprint arXiv:2412.17051}, 2024.

\bibitem[CDG24]{collot2022stability}
Charles Collot, Helge Dietert, and Pierre Germain.
\newblock Stability and cascades for the kolmogorov--zakharov spectrum of wave
  turbulence.
\newblock {\em Archive for Rational Mechanics and Analysis}, 248(1):7, 2024.

\bibitem[CG25]{CoG19}
Charles Collot and Pierre Germain.
\newblock On the derivation of the homogeneous kinetic wave equation.
\newblock {\em Communications on Pure and Applied Mathematics}, 78(4):856--909,
  2025.

\bibitem[CN18]{clough2018difficulty}
Katy Clough and Jens~C Niemeyer.
\newblock On the difficulty of generating gravitational wave turbulence in the
  early universe.
\newblock {\em Classical and Quantum Gravity}, 35(18):187001, 2018.

\bibitem[DGHG23]{dubach}
Guillaume Dubach, Pierre Germain, and Benjamin Harrop-Griffiths.
\newblock On the derivation of the homogeneous kinetic wave equation for a
  nonlinear random matrix model.
\newblock {\em Ars Inveniendi Analytica}, 2023.

\bibitem[DH21]{denghani19}
Yu~Deng and Zaher Hani.
\newblock On the derivation of the wave kinetic equation for {NLS}.
\newblock {\em Forum Math. Pi}, 9:Paper No. e6, 37, 2021.

\bibitem[DH23a]{denghani2023}
Yu~Deng and Zaher Hani.
\newblock Derivation of the wave kinetic equation: full range of scaling laws.
\newblock {\em arXiv preprint arXiv:2301.07063}, 2023.

\bibitem[DH23b]{denghani2021}
Yu~Deng and Zaher Hani.
\newblock Full derivation of the wave kinetic equation.
\newblock {\em Inventiones mathematicae}, 233(2):543--724, 2023.

\bibitem[DH23c]{deng2023long}
Yu~Deng and Zaher Hani.
\newblock Long time justification of wave turbulence theory.
\newblock {\em arXiv preprint arXiv:2311.10082}, 2023.

\bibitem[DH24]{DHPropag}
Yu~Deng and Zaher Hani.
\newblock Propagation of chaos and higher order statistics in wave kinetic
  theory.
\newblock {\em Journal of the European Mathematical Society}, 2024.

\bibitem[DHM24]{deng2024long}
Yu~Deng, Zaher Hani, and Xiao Ma.
\newblock Long time derivation of the boltzmann equation from hard sphere
  dynamics.
\newblock {\em arXiv preprint arXiv:2408.07818}, 2024.

\bibitem[DHM25]{deng2025hilbert}
Yu~Deng, Zaher Hani, and Xiao Ma.
\newblock Hilbert's sixth problem: derivation of fluid equations via
  boltzmann's kinetic theory.
\newblock {\em arXiv preprint arXiv:2503.01800}, 2025.

\bibitem[DK20]{DyKuk1}
Andrey Dymov and Sergei Kuksin.
\newblock On the {Z}akharov-{L}'vov stochastic model for wave turbulence.
\newblock {\em Dokl. Math}, 101:102--109, 2020.

\bibitem[DK21]{Dykuk3}
Andrey Dymov and Sergei Kuksin.
\newblock Formal expansions in stochastic model for wave turbulence 1:
  {K}inetic limit.
\newblock {\em Comm. Math. Phys.}, 382(2):951--1014, 2021.

\bibitem[DK23]{DyKuk2}
Andrey Dymov and Sergei Kuksin.
\newblock Formal expansions in stochastic model for wave turbulence 2: {M}ethod
  of diagram decomposition.
\newblock {\em J. Stat. Phys.}, 190(1):Paper No. 3, 42, 2023.

\bibitem[DKMV23]{DymKukDiscrete}
Andrey Dymov, Sergei Kuksin, Alberto Maiocchi, and Sergei Vl{\u{a}}du{\c{t}}.
\newblock The large-period limit for equations of discrete turbulence.
\newblock In {\em Annales Henri Poincar{\'e}}, volume~24, pages 3685--3739.
  Springer, 2023.

\bibitem[DLN07]{expFiniteBox}
Petr Denissenko, Sergei Lukaschuk, and Sergey Nazarenko.
\newblock Gravity wave turbulence in a laboratory flume.
\newblock {\em Phys. Rev. Lett.}, 99:014501, Jul 2007.

\bibitem[DNY22]{DNY22}
Yu~Deng, Andrea~R Nahmod, and Haitian Yue.
\newblock Random tensors, propagation of randomness, and nonlinear dispersive
  equations.
\newblock {\em Inventiones mathematicae}, 228(2):539--686, 2022.

\bibitem[dS15]{ASkineq}
Anne-Sophie de~Suzzoni.
\newblock On the use of normal forms in the propagation of random waves.
\newblock {\em J. Math. Phys.}, 56(2):021501, 27, 2015.

\bibitem[dS22a]{PropagChaos}
Anne-Sophie de~Suzzoni.
\newblock General remarks on the propagation of chaos in wave turbulence and
  application to the incompressible euler dynamics.
\newblock {\em arXiv eprints 2206.14744}, 2022.

\bibitem[dS22b]{SchroQuint}
Anne-Sophie de~Suzzoni.
\newblock Singularities in the weak turbulence regime for the quintic
  schr{\"o}dinger equation.
\newblock {\em Documenta Mathematica}, 27:2491--2561, 2022.

\bibitem[dS23]{de2023discrete}
Anne-Sophie de~Suzzoni.
\newblock Discrete wave turbulence for the benjamin-bona-mahony equation, part
  i: oscillations for the correlations between the solutions and its initial
  datum.
\newblock {\em arXiv preprint arXiv:2306.05249}, 2023.

\bibitem[dST14]{dSTont}
Anne-Sophie de~Suzzoni and Nikolay Tzvetkov.
\newblock On the propagation of weakly nonlinear random dispersive waves.
\newblock {\em Arch. Ration. Mech. Anal.}, 212(3):849--874, 2014.

\bibitem[EMV07]{EMV1}
Miguel Escobedo, St{\'e}phane Mischler, and Juan~JL Velazquez.
\newblock On the fundamental solution of a linearized uehling--uhlenbeck
  equation.
\newblock {\em Archive for rational mechanics and analysis}, 186:309--349,
  2007.

\bibitem[EMV08]{EMV2}
Miguel Escobedo, Stephane Mischler, and Juan~JL Velazquez.
\newblock Singular solutions for the uehling--uhlenbeck equation.
\newblock {\em Proceedings of the Royal Society of Edinburgh Section A:
  Mathematics}, 138(1):67--107, 2008.

\bibitem[ESY08]{ELSMY}
L\'{a}szl\'{o} Erd\H{o}s, Manfred Salmhofer, and Horng-Tzer Yau.
\newblock Quantum diffusion of the random {S}chr\"{o}dinger evolution in the
  scaling limit.
\newblock {\em Acta Math.}, 200(2):211--277, 2008.

\bibitem[EV15]{EVbook}
Miguel Escobedo and Juan~JL Vel{\'a}zquez.
\newblock {\em On the theory of weak turbulence for the nonlinear Schrödinger
  equation}.
\newblock American Mathematical Soc., 2015.

\bibitem[EY00]{ELY}
L{\'a}szl{\'o} Erd{\H{o}}s and Horng-Tzer Yau.
\newblock Linear boltzmann equation as the weak coupling limit of a random
  schr{\"o}dinger equation.
\newblock {\em Communications on Pure and Applied Mathematics: A Journal Issued
  by the Courant Institute of Mathematical Sciences}, 53(6):667--735, 2000.

\bibitem[Fao20]{Faou}
Erwan Faou.
\newblock Linearized wave turbulence convergence results for three-wave
  systems.
\newblock {\em Communications in Mathematical Physics}, 378, 09 2020.

\bibitem[FGH16]{faouCRequation}
Erwan Faou, Pierre Germain, and Zaher Hani.
\newblock The weakly nonlinear large-box limit of the 2d cubic nonlinear
  {Schr{\"o}dinger} equation.
\newblock {\em J. Am. Math. Soc.}, 29(4):915--982, 2016.

\bibitem[FM24a]{faou2024scattering}
Erwan Faou and Antoine Mouzard.
\newblock Scattering, random phase and wave turbulence.
\newblock {\em Communications in Mathematical Physics}, 405(4):109, 2024.

\bibitem[FM24b]{FaouMouzard}
Erwan Faou and Antoine Mouzard.
\newblock Scattering, random phase and wave turbulence.
\newblock {\em Commun. Math. Phys.}, 405(4):42, 2024.
\newblock Id/No 109.

\bibitem[Gal23]{Galtier}
S{\'e}bastien Galtier.
\newblock {\em Physics of wave turbulence}.
\newblock Cambridge: Cambridge University Press, 2023.

\bibitem[GH24]{grande2024rigorous}
Ricardo Grande and Zaher Hani.
\newblock Rigorous derivation of damped-driven wave turbulence theory.
\newblock {\em arXiv preprint arXiv:2407.10711}, 2024.

\bibitem[GN17]{galtier2017turbulence}
S{\'e}bastien Galtier and Sergey~V Nazarenko.
\newblock Turbulence of weak gravitational waves in the early universe.
\newblock {\em Physical review letters}, 119(22):221101, 2017.

\bibitem[Has62]{Hass1}
K.~Hasselmann.
\newblock On the non-linear energy transfer in a gravity-wave spectrum part 1.
  general theory.
\newblock {\em Journal of Fluid Mechanics}, 12(4):481–500, 1962.

\bibitem[Has63]{Hass2}
K.~Hasselmann.
\newblock On the non-linear energy transfer in a gravity wave spectrum part 2.
  conservation theorems; wave-particle analogy; irrevesibility.
\newblock {\em Journal of Fluid Mechanics}, 15(2):273–281, 1963.

\bibitem[HRST22]{HRST}
Amirali Hannani, Matthew Rosenzweig, Gigliola Staffilani, and Minh-Binh Tran.
\newblock On the wave turbulence theory for a stochastic kdv type
  equation--generalization for the inhomogeneous kinetic limit.
\newblock {\em arXiv preprint arXiv:2210.17445}, 2022.

\bibitem[HSZ24]{hani2024inhomogeneous}
Zaher Hani, Jalal Shatah, and Hui Zhu.
\newblock Inhomogeneous turbulence for the wick nonlinear schr{\"o}dinger
  equation.
\newblock {\em Communications on Pure and Applied Mathematics},
  77(11):4100--4162, 2024.

\bibitem[Jan97]{Janson}
Svante Janson.
\newblock {\em Gaussian {H}ilbert spaces}, volume 129 of {\em Cambridge Tracts
  in Mathematics}.
\newblock Cambridge University Press, Cambridge, 1997.

\bibitem[Kar91]{KARTASHOVA1}
Elena~A. Kartashova.
\newblock On properties of weakly nonlinear wave interactions in resonators.
\newblock {\em Physica D: Nonlinear Phenomena}, 54(1):125--134, 1991.

\bibitem[Kar94]{KARTASHOVA2}
Elena~A. Kartashova.
\newblock Weakly nonlinear theory of finite-size effects in resonators.
\newblock {\em Phys. Rev. Lett.}, 72:2013--2016, Mar 1994.

\bibitem[Lan76]{Boltzmann}
Oscar~E. Lanford, III.
\newblock On a derivation of the {B}oltzmann equation.
\newblock In {\em International {C}onference on {D}ynamical {S}ystems in
  {M}athematical {P}hysics ({R}ennes, 1975)}, Ast\'{e}risque, No. 40, pages
  117--137. Soc. Math. France, Paris, 1976.

\bibitem[LS11]{LS11}
Jani Lukkarinen and Herbert Spohn.
\newblock Weakly nonlinear {S}chr\"odinger equation with random initial data.
\newblock {\em Invent. Math.}, 183(1):79--188, 2011.

\bibitem[Ma22]{ma2022almost}
Xiao Ma.
\newblock Almost sharp wave kinetic theory of multidimensional kdv type
  equations with $ d\geq 3$.
\newblock {\em arXiv preprint arXiv:2204.06148}, 2022.

\bibitem[Men24]{menegaki2022l2stability}
Angeliki Menegaki.
\newblock $l^2$-stability near equilibrium for the 4 waves kinetic equation.
\newblock {\em Kinetic and Related Models}, 17(4):514--532, 2024.

\bibitem[Naz11]{Naz}
Sergey Nazarenko.
\newblock {\em Wave turbulence}, volume 825 of {\em Lecture Notes in Physics}.
\newblock Springer, Heidelberg, 2011.

\bibitem[Pei29]{Peierls1}
R.~Peierls.
\newblock Zur kinetischen theorie der wärmeleitung in kristallen.
\newblock {\em Annalen der Physik}, 395(8):1055--1101, 1929.

\bibitem[Spo06]{Spohnphonon}
Herbert Spohn.
\newblock The phonon {Boltzmann} equation, properties and link to weakly
  anharmonic lattice dynamics.
\newblock {\em J. Stat. Phys.}, 124(2-4):1041--1104, 2006.

\bibitem[ST21]{staffilanitran}
Gigliola Staffilani and Minh-Binh Tran.
\newblock On the wave turbulence theory for stochastic and random
  multidimensional {KdV} type equations.
\newblock {\em arXiv eprints 2106.09819}, 2021.

\bibitem[Ved67]{Vedenov1967}
A.~A. Vedenov.
\newblock Theory of a weakly turbulent plasma.
\newblock {\em Reviews of Plasma Physics}, 3:229--276, 1967.
\newblock edition by Leontovich, M. A.

\bibitem[{Zak}65]{KZspectra}
V.~E. {Zakharov}.
\newblock {Weak turbulence in media with a decay spectrum}.
\newblock {\em Journal of Applied Mechanics and Technical Physics},
  6(4):22--24, July 1965.

\bibitem[ZF67]{Zakharov1967}
V.~E. Zakharov and N.~N. Filonenko.
\newblock Weak turbulence of capillary waves.
\newblock {\em Journal of Applied Mechanics and Technical Physics},
  8(5):37--40, Sep 1967.

\bibitem[ZLF92]{zakharovBook}
V.~E. Zakharov, V.~S. L'vov, and G.~Falkovich.
\newblock {\em Kolmogorov spectra of turbulence {I}. {Wave} turbulence}.
\newblock Berlin: Springer-Verlag, 1992.

\bibitem[ZS67]{ZS67}
GM~Zaslavskii and RZ~Sagdeev.
\newblock Limits of statistical description of a nonlinear wave field.
\newblock {\em Soviet physics JETP}, 25:718--724, 1967.

\end{thebibliography}
\end{document}